\documentclass[12pt]{amsart}

\usepackage{amssymb,latexsym}
\usepackage{bm}
\usepackage{graphicx}
\usepackage{psfrag}
\usepackage{wrapfig}
\usepackage{color}
\usepackage{enumitem}
\setlist[enumerate]{parsep=0pt plus 4pt,topsep=0pt plus 4pt}
\usepackage{url}
\usepackage{hyperref}
\usepackage{tikz}
\definecolor{darkblue}{RGB}{0,0,160}
\hypersetup{colorlinks,citecolor=black,filecolor=black,%
            linkcolor=darkblue,urlcolor=darkblue}
\allowdisplaybreaks

\newcommand{\excise}[1]{}%{$\star$\textsc{#1}$\star$}

\newtheorem{thm}{Theorem}[section]
\newtheorem{lemma}[thm]{Lemma}

\newtheorem{cor}[thm]{Corollary}
\newtheorem{prop}[thm]{Proposition}
\newtheorem{conj}[thm]{Conjecture}

\theoremstyle{definition}
\newtheorem{example}[thm]{Example}
\newtheorem{remark}[thm]{Remark}
\newtheorem{defn}[thm]{Definition}
\newtheorem{conv}[thm]{Convention}

\numberwithin{equation}{section}

\newcounter{separated}

\newcommand{\Ring}[1]{\ensuremath{\mathbb{#1}}}

\renewcommand\>{\rangle}
\newcommand\<{\langle}
\newcommand\0{\mathbf{0}}
\newcommand\NN{\Ring{N}}
\newcommand\OO{\mathcal{O}}
\newcommand\RR{{\mathbb R}}
\newcommand\ZZ{{\mathbb Z}}
\newcommand\bb{{\mathbf b}}
\newcommand\cc{{\mathbf c}}
\newcommand\ee{{\mathbf e}}
\newcommand\ff{{\mathbf f}}
\newcommand\ii{{\mathbf i}}
\newcommand\jj{{\mathbf j}}
\newcommand\kk{\Bbbk}
\newcommand\mm{{\mathfrak m}}
\newcommand\pp{{\mathfrak p}}
\newcommand\qq{{\mathbf q}}
\newcommand\rr{{\mathbf r}}
\newcommand\vv{{\mathbf v}}

\newcommand\xx{{\mathbf x}}
\newcommand\zz{{\mathbf z}}
\newcommand\cA{{\mathcal A}}
\newcommand\cB{{\mathcal B}}
\newcommand\cC{{\mathcal C}}
\newcommand\cD{{\mathcal D}}
\newcommand\cF{{\mathcal F}}
\newcommand\cH{H}%{{\mathcal H}}

\newcommand\cM{M}%{{\mathcal M}}
\newcommand\cN{N}%{{\mathcal N}}
\newcommand\cP{P}%{{\mathcal P}}
\newcommand\cQ{Q}%{{\mathcal Q}}
\newcommand\cS{\mathcal{S}}
\newcommand\cZ{Z}%{{\mathcal Z}}

\newcommand\oD{\hspace{.3ex}\ol{\hspace{-.3ex}D\hspace{-.05ex}}\hspace{.05ex}}

\newcommand\oS{\hspace{.3ex}\ol{\hspace{-.3ex}\cS\hspace{-.05ex}}\hspace{.05ex}}

\newcommand\vC{\check{\mathcal C}}

\newcommand\del{\partial}
\newcommand\sss{\mathbf{s}}

\renewcommand\aa{{\mathbf a}}
\renewcommand\phi{\varphi}
\newcommand\QR{\del}

\newcommand\op{\mathrm{op}}

\newcommand\bnq{\cB_{\!-\cQ}}
\newcommand\brp{\bb'\!\hspace{-.2ex}_{\rho'}}
\newcommand\cfq{\cF_\cQ}
\newcommand\cfr{\cF_\RR}
\newcommand\cmr{\cM^{\rho\hspace{-.2ex}}}

\newcommand\cmt{\cM\hspace{-1pt}/\tau}
\newcommand\cnr{\cN\hspace{-1pt}/\rho}
\newcommand\cnt{\cN\hspace{-1pt}/\tau}
\newcommand\dda{\Delta^{\!D}{\hspace{-1.1ex}}_\aa}
\newcommand\ddb{\Delta^{\!D}_\bb}

\newcommand\dst{\delta^{\sigma/\tau}}
\newcommand\dzt{D/\hspace{.3ex}\ZZ\tau}
\newcommand\fqo{\cfq^\op}
\newcommand\fro{\cfr^{\hspace{.2ex}\op}}
\newcommand\mtd{\textstyle \max_{\hspace{.3ex}\tau\hspace{-.3ex}} D}
\newcommand\mvt{\cM^\vee\hspace{-2pt}/\tau}
\newcommand\nda{\nabla_{\!D}^\aa}

\newcommand\oas{\OO_{\hspace{-.05ex}\aa}^{\hspace{.09ex}\sigma}}
\newcommand\qpc{\cQ_+^\circ}
\newcommand\qrs{\cQ/\hspace{.2ex}\RR\sigma}
\newcommand\qrr{\cQ/\hspace{.2ex}\RR\rho}
\newcommand\qrt{\cQ/\hspace{.2ex}\RR\tau}

\newcommand\qzr{\cQ/\hspace{.2ex}\ZZ\rho}
\newcommand\qzt{\cQ/\hspace{.2ex}\ZZ\tau}
\newcommand\qrsp{\cQ_+/\hspace{.2ex}\RR\sigma}
\newcommand\qrtp{\cQ_+/\hspace{.2ex}\RR\tau}
\newcommand\qztp{\cQ_+/\hspace{.2ex}\ZZ\tau}

\newcommand\into{\hookrightarrow}

\newcommand\kats{\kk_\sigma[\aa+\tau]}
\newcommand\kbrx{\kk_\xi[\bb+\rho]}
\newcommand\nabk{\nabla\hspace{-.2ex}\sigma_k}
\newcommand\nabr{\nabla\hspace{-.2ex}\rho}
\newcommand\nabs{\nabla\hspace{-.2ex}\sigma}
\newcommand\nabt{\nabla\hspace{-.2ex}\tau}
\newcommand\nabx{\nabla\hspace{-.2ex}\xi}
\newcommand\onto{\twoheadrightarrow}

\newcommand\spot{{\hbox{\raisebox{1.5pt}{\large\bf .}}\hspace{-.5pt}}}

\newcommand\minus{\smallsetminus}
\newcommand\nabro{\Delta\rho}
\newcommand\nabto{\Delta\tau}
\newcommand\simto{\mathrel{\!\ooalign{$\fillrightmap$\cr\raisebox{.75ex}{$\,\sim\ \hspace{.2ex}$}}}}
\newcommand\cupdot{\ensuremath{\mathbin{\mathaccent\cdot\cup}}}
\newcommand\goesto{\rightsquigarrow}
\newcommand\dirlim{\varinjlim}
\newcommand\invlim{\varprojlim}

\newcommand\nothing{\varnothing}

\newcommand\bigcupdot{\makebox[0pt][l]{$\hspace{1.05ex}\cdot$}\textstyle\bigcup}

\newcommand\filleftmap{\mathord\leftarrow \mkern-6mu
	\cleaders\hbox{$\mkern-2mu \mathord- \mkern-2mu$}\hfill
	\mkern-6mu \mathord-}
\newcommand\fillrightmap{\mathord- \mkern-6mu
	\cleaders\hbox{$\mkern-2mu \mathord- \mkern-2mu$}\hfill
	\mkern-6mu \mathord\rightarrow}
\newcommand\fillonto{\mathord- \mkern-6mu
	\cleaders\hbox{$\mkern-2mu \mathord- \mkern-2mu$}\hfill
	\mkern-6mu \mathord\twoheadrightarrow}
\renewcommand\iff{\Leftrightarrow}
\renewcommand\epsilon{\varepsilon}
\renewcommand\implies{\Rightarrow}

\newcommand\dd[2][\!D]{\Delta^{#1}_{#2}}
\newcommand\dr[1][]{\delta_{\hspace{-.2ex}\rho\hspace{.15ex}}^{\hspace{.15ex}#1}}
\newcommand\ds[1][\ ]{\delta^{\sigma\hspace{-1.1ex}}{}_{#1\hspace{.2ex}}}
\newcommand\dt[1][\ ]{\operatorname{\delta_{\tau\hspace{-1ex}}{}^{#1}\hspace{-.2ex}}}
\newcommand\dx[1][]{\delta^{\hspace{.1ex}\xi}}
\newcommand\lr[1][]{\del_{\hspace{-.2ex}\rho\hspace{.2ex}}^{\hspace{.15ex}#1}}

\newcommand\lx[1][\xi]{\del^{\hspace{.15ex}#1}}
\newcommand\nd[1][\ ]{\nabla_{\!D}^{#1}}
\newcommand\ol[1]{{\overline{#1}}}

\newcommand\wh[1]{{\widehat{#1}}}
\newcommand\wt[1]{{\widetilde{#1}}}
\newcommand\ats[1][\sigma]{\operatorname{\mathit{A}}_\tau^{#1}\hspace{-.2ex}}
\newcommand\brx[1][\xi]{(\bb_\rho, #1)}
\newcommand\brxp[1][\xi']{(\brp,#1)}
\newcommand\cbm[1][]{\cB_{\hspace{-.2ex}\cM}^{\hspace{.2ex}#1}}
\newcommand\cbq[1][]{\cB_\cQ^{\hspace{.2ex}#1}}
\newcommand\cbr[1][]{\cB_\RR^{\hspace{.2ex}#1}}
\newcommand\cde[1][]{\cD_{\hspace{-.2ex}\cM/\cM_{\prec\beta}}^{\hspace{.1ex}#1}}
\newcommand\cdm[1][]{\cD_{\hspace{-.2ex}\cM}^{\hspace{.1ex}#1}}
\newcommand\cdq[1][]{\cD_{\hspace{-.2ex}\cQ}^{\hspace{.1ex}#1}}
\newcommand\cdr[1][]{\cD_{\hspace{-.2ex}\RR}^{\hspace{.1ex}#1}}
\newcommand\cds[1][]{\cD_{\kk[S]}^{\hspace{.1ex}#1}}
\newcommand\cqb{\cQ_\beta}

\newcommand\csp[1][\tau']{\operatorname{\mathcal S}_{#1\!}}
\newcommand\cst[1][]{\operatorname{\mathcal S}_{\tau\!}^{#1\hspace{-.3ex}}}
\newcommand\mrx[1][\xi]{\cM^{\rho,\hspace{.1ex}#1\hspace{-.15ex}}}
\newcommand\qnk{\cQ_{\nabk}}
\newcommand\qnp{\cQ_{\nabs'}}
\newcommand\qns[1][]{\cQ_{\nabs#1}}
\newcommand\qnt{\cQ_{\nabt}}
\newcommand\qnx[1][]{\cQ_{\nabx#1}}
\newcommand\qny[1][]{\cQ_{\naby[#1]}}
\newcommand\naby[1][y]{\nabla\hspace{-.2ex}#1}

\newcommand\socc[1][]{\ol{\mathrm{soc}}_{#1\,}} % closed socle
\newcommand\socp[1][\tau']{\operatorname{soc}_{\hspace{.02ex}#1\hspace{-.15ex}}}
\newcommand\socr[1][]{\operatorname{soc}_{\hspace{.02ex}\rho}^{#1\!\!}\hspace{-.15ex}}
\newcommand\soct[1][]{\operatorname{soc}_{\hspace{.02ex}\tau}^{#1\!\!}\hspace{-.15ex}}
\newcommand\topc[1][]{\ol{\mathrm{top}}_{#1\,}} % closed top

\newcommand\topr[1][]{\operatorname{top}_{\hspace{-.1ex}\rho}^{\hspace{.15ex}#1\!\!}}

\newcommand\socct[1][\tau]{\ol{\operatorname{soc}}_{\hspace{.02ex}#1}^{\,}\hspace{.1ex}}

\newcommand\topcr[1][\ ]{\ol{\mathrm{top}}{}_{\rho\!}^{\hspace{.15ex}#1}} %
\newcommand\topct[1][\ ]{\ol{\mathrm{top}}{}_{\,\tau\!}{}^{\hspace{-.7ex}#1\,}} %
\newcommand\cpsoc{\cP\hbox{\rm-}\socc}
\newcommand\cptop{\cP\hbox{\rm-}\topc}
\newcommand\fqsoc{\fqo\hspace{-.1ex}\hbox{\rm-}\hspace{.1ex}\socc}
\newcommand\nrsoc{\nabr\hspace{.25ex}\hbox{\rm-}\hspace{.1ex}\socc}
\newcommand\nrtop{\nabro\hspace{.25ex}\hbox{\rm-}\hspace{.1ex}\topc}
\newcommand\ntsoc{\nabt\hbox{\rm-}\hspace{.1ex}\socc}

\newcommand\qtsoc{(\qrt)\hspace{-.05ex}\hbox{\rm-}\hspace{.1ex}\socc}

\newcommand\rnsoc{\cQ\hspace{-.1ex}\hbox{\rm-}\hspace{.1ex}\socc}

\renewcommand\top{\operatorname{top}}

\newcommand\qrcode{QR~code}

\newcommand{\aoverb}[2]{{\genfrac{}{}{0pt}{1}{#1}{#2}}}
\def\twoline#1#2{\aoverb{\scriptstyle {#1}}{\scriptstyle {#2}}}

\DeclareMathOperator\gr{gr} % associated graded object
\DeclareMathOperator\ass{Ass} % associated faces
\DeclareMathOperator\att{Att} % attached faces
\DeclareMathOperator\Gen{Birth} % Generator module
\DeclareMathOperator\Hom{Hom} % Hom
\DeclareMathOperator\life{Life} % life module
\DeclareMathOperator\soc{soc} % socle of a module
\DeclareMathOperator\Soc{Death} % total socle module
\DeclareMathOperator\Top{Top} % top of a module ("\top" already defined)
\DeclareMathOperator\hhom{
	\hspace{.6pt}{\underline{\hspace{-.6pt}{\rm Hom}\hspace{-.6pt}}\hspace{1pt}}}
\DeclareMathOperator\coker{coker} % cokernel
\newcommand\monomialmatrix[3]{{
\begin{array}{@{}r@{\:}r@{}c@{}l@{}}
  \begin{array}{@{}c@{}}		%leftmost column
	\begin{array}{@{}r@{}}
	\\
	#1
	\end{array}\!
  \end{array}						
&
  \begin{array}{@{}c@{}}		%big left parenthesis	%(
	\begin{array}{@{}l@{}}\\				%(
	\end{array}						%(
	\\							%(
	\left[\begin{array}{@{}l@{}}				%(
	#3							%(
	\end{array}\!						%(
	\right.							%(
  \end{array}							%(
&
  #2					%top border row
&
  \begin{array}{@{}c@{}}		%big right parenthesis	%)
	\begin{array}{@{}l@{}}\\				%)
	\end{array}						%)
	\\							%)
	\left.\!\begin{array}{@{}l@{}}				%)
	#3							%)
	\end{array}						%)
	\right]							%)
  \end{array}							%)
\end{array}
}}

\begin{document}%%%%%%%%%%%%%%%%%%%%%%%%%%%%%%%%%%%%%%%%%%%%%%%%%%%%%%%%
%%%%%%%%%%%%%%%%%%%%%%%%%%%%%%%%%%%%%%%%%%%%%%%%%%%%%%%%%%%%%%%%%%%%%%%%

\mbox{}
\vspace{-4ex}
\title[Data structures for real multiparameter persistence modules]
      {Data structures for real multiparameter persistence modules}
%author{Justin Curry}
%address{Mathematics Department\\Duke University\\Durham, NC 27708}
\author{\vspace{-1ex}Ezra Miller}
\address{Mathematics Department\\Duke University\\Durham, NC 27708}
\urladdr{\url{http://math.duke.edu/people/ezra-miller}}
%author{Ashleigh Thomas}%\\\\(DRAFT---DO NOT DISTRIBUTE)}
%address{Mathematics Department\\Duke University\\Durham, NC 27708}
%urladdr{\url{https://fds.duke.edu/db/aas/math/grad/athomas}}

\makeatletter
  \@namedef{subjclassname@2010}{\textup{2010} Mathematics Subject Classification}
\makeatother

\date{21 September 2017}

\begin{abstract}
A theory of modules over posets is developed to define computationally
feasible, topologically interpretable data structures, in terms of
birth and death of homology classes, for persistent homology with
multiple real parameters.  To replace the noetherian hypothesis in the
general setting of modules over posets, for theoretical as well as
computational purposes, a \emph{finitely encoded} condition is defined
combinatorially and developed algebraically.  The finitely encoded
hypothesis captures topological tameness of persistent homology, and
poset-modules satisfying it can be specified by \emph{fringe
presentations} that reflect birth-and-death descriptions of persistent
homology.

The homological theory of modules over posets culminates in a syzygy
theorem characterizing finitely encoded modules as those that admit
finite presentations or resolutions by direct sums of upset modules or
downset modules, which are analogues over posets of flat and injective
modules over multigraded polynomial rings.

The geometric and algebraic theory of modules over posets focuses on
modules over real polyhedral groups (partially ordered real vector
spaces whose positive cones are polyhedral), with a parallel theory
over discrete polyhedral groups (partially ordered abelian groups
whose positive cones are finitely generated) that is simpler but still
largely new in the generality of finitely encoded modules.  Existence
of primary decomposition is proved over arbitrary polyhedral partially
ordered abelian groups, but the real and discrete cases carry enough
geometry and, crucially in the real case, topology to induce complete
theories of minimal primary and secondary decomposition, associated
and attached faces, minimal generators and cogenerators, socles and
tops, minimal upset covers and downset hulls, Matlis duality, and
minimal fringe presentation.  In particular, when the data are real
semialgebraic, that property is preserved by functorial constructions.
And when the modules in question are subquotients of the group itself,
minimal primary and secondary decompositions are~canonical.

Tops and socles play the roles of functorial birth and death spaces
for multiparameter persistence modules.  They yield functorial
\emph{\qrcode s} and \emph{elder morphisms} for modules over real and
discrete polyhedral groups that generalize and categorify the bar code
and elder rule for persistent homology in one parameter.  The
disparate ways that \qrcode s and elder morphisms model bar codes
coalesce, in ordinary persistence with one parameter, to a functorial
bar code.
\end{abstract}
\maketitle

\vspace{-2.9ex}
\setcounter{tocdepth}{2}
\tableofcontents

%mbox{}\vspace{-10ex}\mbox{}
%%%%%%%%%%%%%%%%%%%%%%%%%%%%%%%%%%%%%%%%%%%%%%%%%%%%%%%%%%%%%%%%%%%%%%%%%
\section{Introduction}\label{s:intro}%%%%%%%%%%%%%%%%%%%%%%%%%%%%%%%%%%%%
%addtocontents{toc}{\protect\setcounter{tocdepth}{1}}%%%%%%%%%%%%%%%%%%%%
%%%%%%%%%%%%%%%%%%%%%%%%%%%%%%%%%%%%%%%%%%%%%%%%%%%%%%%%%%%%%%%%%%%%%%%%%

\noindent
Families of topological spaces in data analysis often arise from
filtrations: collections of subspaces of a single topological
space. Inclusions of subspaces induce a partial order on such
collections.  Applying the homology functor then yields a commutative
diagram~$\cM$ of vector spaces indexed by the partially ordered
set~$\cQ$ of subspaces.  This \emph{$\cQ$-module}~$\cM$ is called the
\emph{persistent homology} of the filtration, referring to how classes
are born, persist for a while, and then die as the parameter moves up
in the poset~$\cQ$.

Ordinary persistent homology, in which $\cQ$ is totally
ordered---usually the real numbers~$\RR$, the integers~$\ZZ$, or a
subset $\{1,\ldots,m\}$---is well studied; see \cite{edels-harer2010},
for example.  Persistence with multiple parameters was introduced by
Carlsson and Zomorodian \cite{multiparamPH} for $\cQ = \NN^n$, and it
has been developed since then in various ways, assuming that the
module~$\cM$ is finitely generated.  In contrast, the application that
drives the developments here has real parameters, and it fails to be
finitely generated in other fundamental ways.  It is therefore the goal
here to define computationally feasible, topologically interpretable
data structures, in terms of birth and death of homology classes, for
persistent homology with multiple real parameters.

The data structures are made possible by a \emph{finitely encoded}
hypothesis that captures topological tameness of persistent homology.
A syzygy theorem characterizes finitely encoded modules as those that
admit appropriately finite presentations and resolutions, all amenable
to computation.  The technical heart of the paper is a development of
basic commutative algebra---especially socles and minimal primary
decomposition---for modules over the poset~$\RR^n$, or equivalently,
multigraded modules over rings of polynomials with real instead of
integer exponents.  This algebraic theory yields functorial
\emph{\qrcode s} and \emph{elder morphisms} for modules over real and
discrete polyhedral groups that generalize and categorify the bar code
and elder rule for ordinary single-parameter persistence.  The
disparate ways that \qrcode s and elder morphisms model bar codes
coalesce, in ordinary persistence with one parameter, to make ordinary
bar codes~\mbox{functorial}.

\vspace{-.5ex}
\addtocontents{toc}{\protect\setcounter{tocdepth}{2}}%%%%%%%%%%%%%%%%%%%%
\subsection{Acknowledgements}

First, special acknowledgements go to Ashleigh Thomas and Justin
Curry.  Both have been and continue to be long-term collaborators on
this project.  They were listed as authors on earlier drafts, but
their contributions lie more properly beyond these preliminaries, so
they declined in the end to be named as authors on this installment.
Early in the development of the ideas here, Thomas put her finger on
the continuous rather than discrete nature of multiparameter
persistence modules for fly wings.  She computed the first examples
explicitly, namely those in Example~\ref{e:toy-model-fly-wing}, and
produced the biparameter persistence diagrams there.  And she
suggested the term ``\qrcode'' (Remark~\ref{r:qr-etymology}).  Curry
contributed, among many other things, clarity and intuition regarding
the topology of endpoints---the limits defining upper and lower
boundary functors---as well as regarding the elder rule.  He also
pointed out connections from the combinatorial viewpoint taken here,
in terms of modules over posets, to higher notions in algebra and
category theory, particularly those involving constructible sheaves,
which are in the same vein as Curry's proposed uses of them in
persistence \cite{curry-thesis}; see Remarks~\ref{r:curry},
\ref{r:lurie}, \ref{r:indicator}, and~\ref{r:kan-extension}.

The author is indebted to David Houle, whose contribution to this
project was seminal and remains ongoing; in particular, he and his lab
produced the fruit fly wing images.  Paul Bendich and Joshua Cruz took
part in the genesis of this project, including early discussions
concerning ways to tweak persistent (intersection) homology for the
fly wing investigation.  Banff International Research Station provided
an opportunity for valuable feedback and suggestions at the workshop
there on Topological Data Analysis (August, 2017) as this research was
being completed; many participants, especially the organizers, Uli
Bauer and Anthea Monod, as well as Michael Lesnick, shared important
perspectives and insight.  Thomas Kahle requested that
\mbox{Proposition}~\ref{p:determined} be an equivalence instead of merely the
one implication it had stated.  Hal Schenck
% and Rachel Levanger
gave helpful comments on the Introduction.  Some passages in
Section~\ref{sub:biological} are taken verbatim, or nearly so, from
\cite{fruitFlyModuli}.

%%%%%%%%%%%%%%%%%%%%%%%%%%%%%%%%%%%%%%%%%%%%%%%%%%%%%%%%%%%%%%%%%%%%%%%%%
\subsection{Biological origins}\label{sub:biological}%%%%%%%%%%%%%%%%%%%

This investigation of data structures for real multiparameter
persistence modules intends both senses of the word ``real'':
actual---from genuine data, with a particular dataset in mind---and
with parameters taking continuous as opposed to discrete values.
Instead of reviewing the numerous possible reasons for considering
multiparameter persistence, many already having been present from the
outset \cite[Section~1.1]{multiparamPH}, what follows is an account of
how real multiparameter persistence arises in the biological problem
that the theory here is specifically designed to serve.

The normal \textsl{Drosophila melanogaster} fruit fly wing depicted on
the left%
\begin{figure}[ht]
% See
% http://www.sflorg.com/sciencenews/scn041906_02.html
% for images of different species' wings.
\vspace{-2ex}
$$%
\includegraphics[height=25mm]{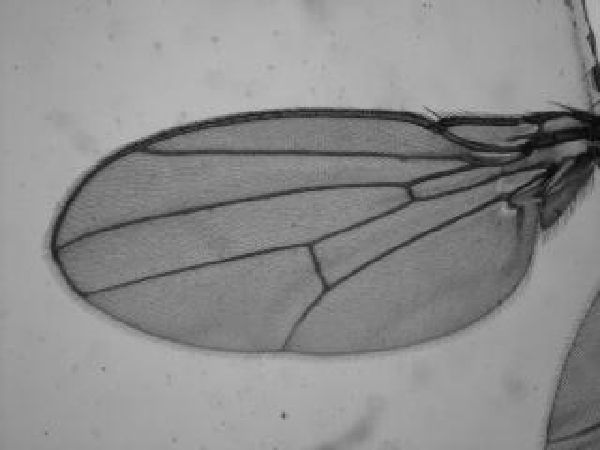}
\qquad
\includegraphics[height=25mm]{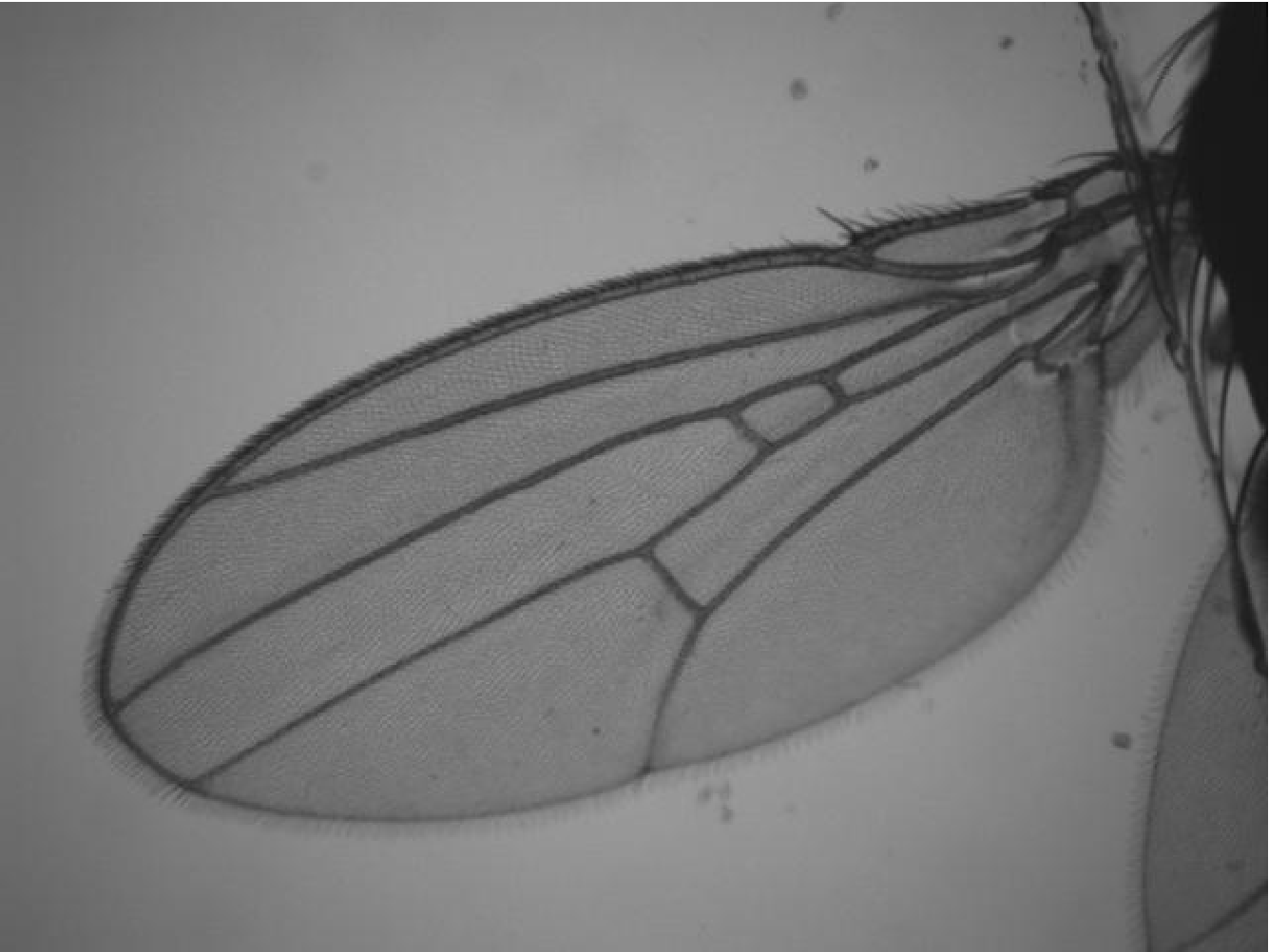}
\qquad
\includegraphics[height=25mm]{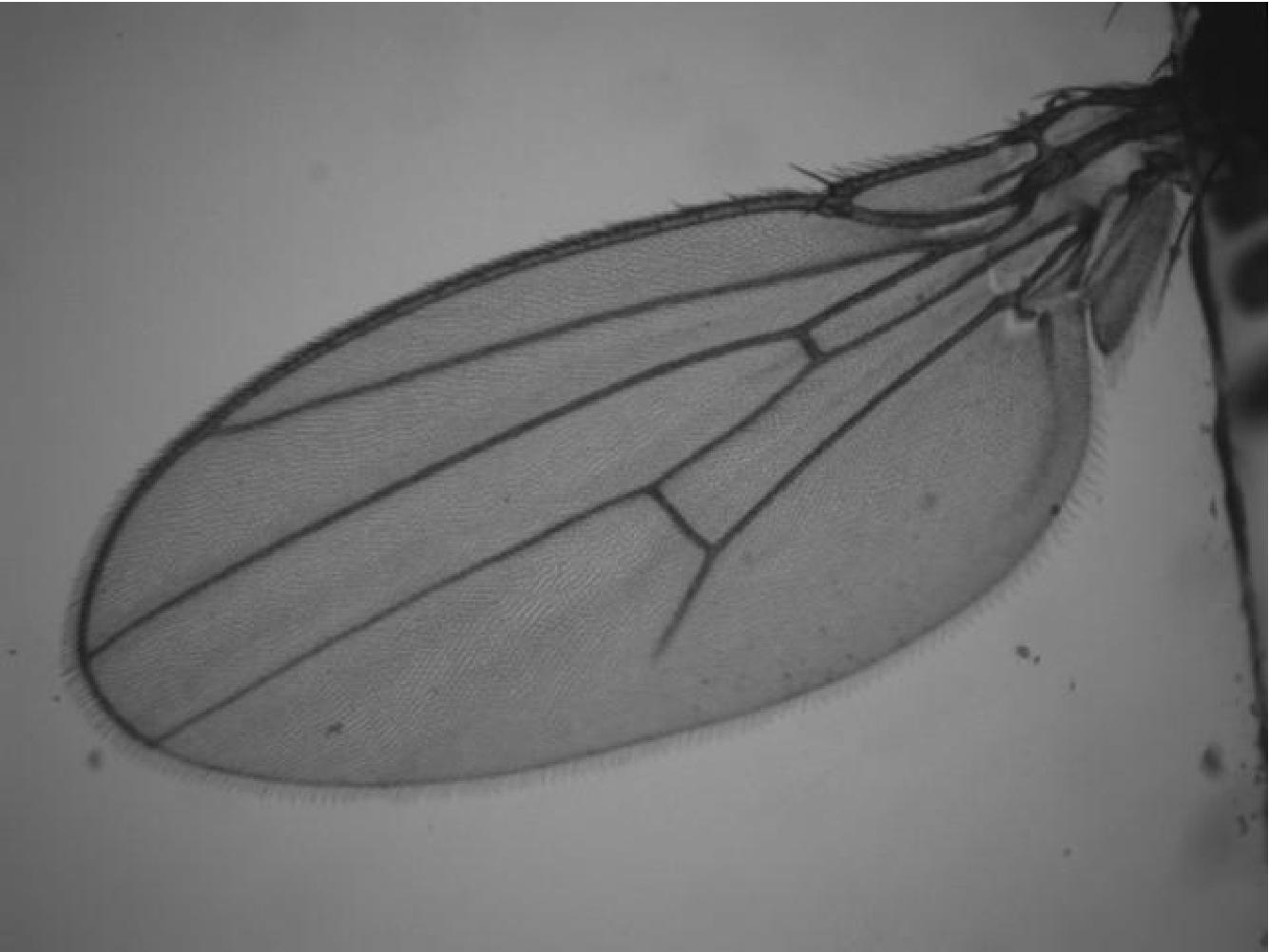}
\vspace{-2.5ex}
$$
\end{figure}
differs from the abnormal other two in topology as well as geometry.
The dataset, as provided by biologist David Houle \cite{houle03} as
part of ongoing work with his lab on the ideas here, presents the
veins in each wing as an embedded planar graph, with a location for
each vertex and an algebraic curve for each arc.  The graph in the
middle has an extra edge, and hence two extra vertices, while the
graph on the right is lacking an intersection.  These topological
variants, along with many others, occur in natural
\textsl{D.$\:$melanogaster} populations, but rarely.  On the other
hand, different species of \textsl{Drosophila} exhibit a range of wing
vein topologies.  How did that come to be?  Wing veins serve several
key purposes, as structural supports as well as conduits for airways,
% http://www.sdbonline.org/sites/fly/aimorph/trachia.htm#dafka
nerves, and blood cells~\cite{blair07}.  Is it possible that some
force causes aberrant vein topologies to occur more frequently than
would otherwise be expected in a natural population---frequently
enough for evolutionary processes to~take~over?

Waddington \cite{waddington1953} famously observed hereditary
topological changes in wing vein phenotype (loss of crossveins) after
breeding flies selected for crossveinless phenotype due to embryonic
heat shock.  Results generated by Weber, and later with more power by
Houle's lab, show that selecting for continuous wing deformations
results in skews toward deformed wings with normal vein topology
\cite{weber90,weber92,houle03}.  But this selection also unexpectedly
yields higher rates of topological novelty; this finding, as yet
unpublished and not yet precisely formulated, is what needs to be
tested statistically.

There are many options for statistical methods to test the hypothesis,
some of them elementary, such as a linear model taking into account a
weighted sum of (say) the number of vertices and the total edge
length.  Whatever the chosen method, it has to grapple with the
topological vein variation, giving appropriate weight to new or
deleted singular points in addition to varying shape.  Real
multiparameter persistence in its present form was conceived to serve
the biology in this way, but the problem has since turned around: fly
wings supply a testing ground for the feasibility and effectiveness of
multiparameter persistent homology as a statistical tool.

\begin{example}\label{e:fly-wing-filtration}
Let $\cQ = \RR_- \times \RR_+$ with the coordinatewise partial order,
so $(r,s) \in \cQ$ for any nonnegative real numbers~$-r$ and~$s$.  Let
$X = \RR^2$ be the plane in which the fly wing is embedded and define
$X_{rs} \subseteq X$ to be the set of points at distance at least~$-r$
from every vertex and within~$s$ of some edge.  Thus $X_{rs}$ is
obtained by removing the union of the balls of radius~$r$ around the
vertices from the union of $s$-neighborhoods of the edges.  In the
following portion of a fly wing, $-r$ is approximately twice~$s$:
$$%
\includegraphics[height=23mm]{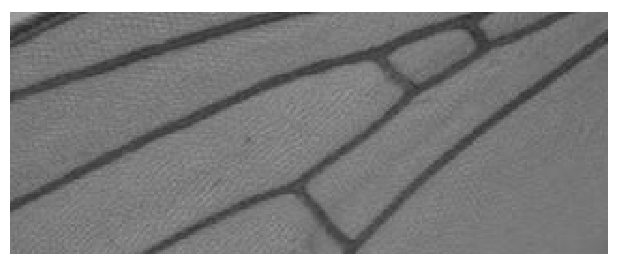}
\quad
\raisebox{10mm}{$\goesto$}
\quad
\includegraphics[height=23mm]{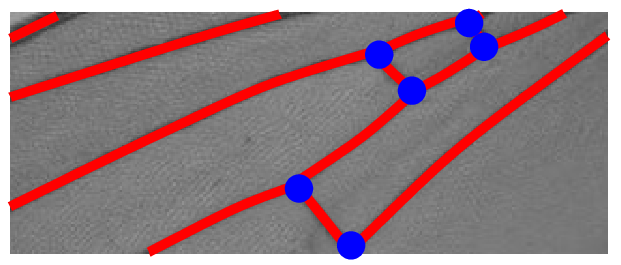}
\vspace{-1ex}
$$
The biparameter persistent homology module $\cM_{rs} = H_0(X_{rs})$
% (or~$H_1(X_r^s)$, which is basically Poincar\'e dual)
summarizes wing vein structure for our purposes.  And the nature of
wing vein formation from gene expression levels during embryonic
development (see \cite{blair07} for background) gives reason to
believe that this type of persistence models biological reality
reasonably faithfully.
\end{example}

\begin{remark}\label{r:biparameterPH}
The parameters in Example~\ref{e:fly-wing-filtration} govern
intersections between edges and vertices.  The biparameter persistence
here can therefore be viewed as a multiscale generalization of
persistent intersection homology \cite{bendich-harer2011} in which
interactions among strata are tuned by relations among the parameters.
% Why might I say this?  Because this generalization was genuinely
% inspired by the fly wing phenomics project.  I would in fact write
% the first sentence of this comment in the article if there was a
% source to cite (beyond the Bendich-Harer definition of PIH)
% concerning the relation between this multiparameter persistence and
% PIH.
\end{remark}

%%%%%%%%%%%%%%%%%%%%%%%%%%%%%%%%%%%%%%%%%%%%%%%%%%%%%%%%%%%%%%%%%%%%%%%%%
\subsection{Encoding modules over arbitrary posets}\label{sub:over-posets}

Biparameter persistence can only serve as an effective summary of a
fly wing for statistical purposes if
% (i)~it can be computed from the initial spline data and
% (ii)~statistically meaningful invariants can be isolated from it.
it can be computed from the initial spline data.  In general,
computation with persistent homology is only conceivable in the
presence of some finiteness condition on the $\cQ$-module or tameness
on the topology that gives rise~to~it.  The prior standard for
finiteness in multiparameter persistence has been the setting where
$\cQ = \NN^n$ and the $\cQ$-module is finitely generated.  Those
conditions quickly fail for fly wings.

\begin{example}\label{e:toy-model-fly-wing}
Using the setup from Example~\ref{e:fly-wing-filtration}, the zeroth
persistent homology for the toy-model ``fly wing'' at left in
Figure~\ref{f:toy-model-fly-wing} is the $\RR^2$-module $\cM$ shown at
center.  Each point of $\RR^2$ is colored according to the dimension
of its associated vector space in~$\cM$, namely $3$, $2$, or~$1$
proceeding up (increasing~$s$) and to the right (increasing~$r$).  The
structure homomorphisms $\cM_{rs} \to \cM_{r's'}$ are all surjective.
This $\RR^2$-module fails to be finitely presented for three
fundamental reasons: first, the three generators sit infinitely far
back along the $r$-axis.  (Fiddling with the minus sign on~$r$ does
not help: the natural maps on homology proceed from infinitely large
radius to~$0$ regardless of how the picture is drawn.)  Second, the
relations that specify the transition from vector spaces of
dimension~$3$ to those of dimension~$2$ or~$1$ lie along a real
algebraic curve, as do those specifying the transition from
dimension~$2$ to dimension~$1$.  These curves have uncountably many
points.  Third, even if the relations are discretized---restrict~$\cM$
to a lattice $\ZZ^2$ superimposed on~$\RR^2$, say---the relations
\begin{figure}[ht]
\vspace{-8ex}
$$%
\begin{array}[b]{@{}c@{}}
\mbox{}\\[28.5pt]
\includegraphics[height=30mm]{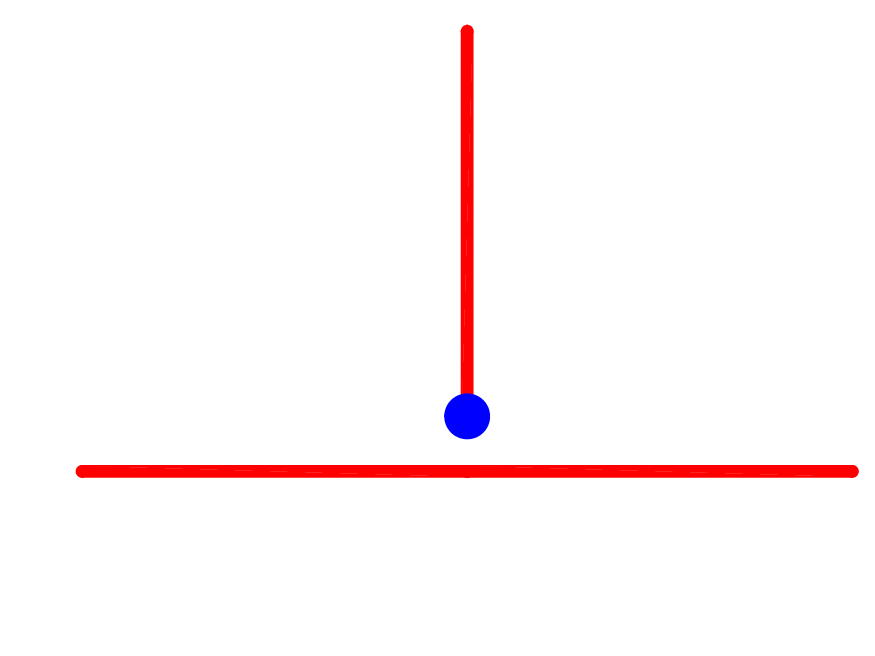}
\\[-22.5pt]
\end{array}
\begin{array}[b]{@{\ \ }c@{}}
\hspace{-3pt}\goesto\\[7mm]\mbox{}
\end{array}
\qquad
\begin{array}[b]{@{\hspace{-10pt}}r@{\hspace{-10pt}}|@{}l@{}}
\begin{array}{@{}c@{}}
\psfrag{r}{\tiny$r \to$}
\psfrag{s}{\tiny$\begin{array}{@{}c@{}}\uparrow\\[-.5ex]s\end{array}$}
\includegraphics[height=30mm]{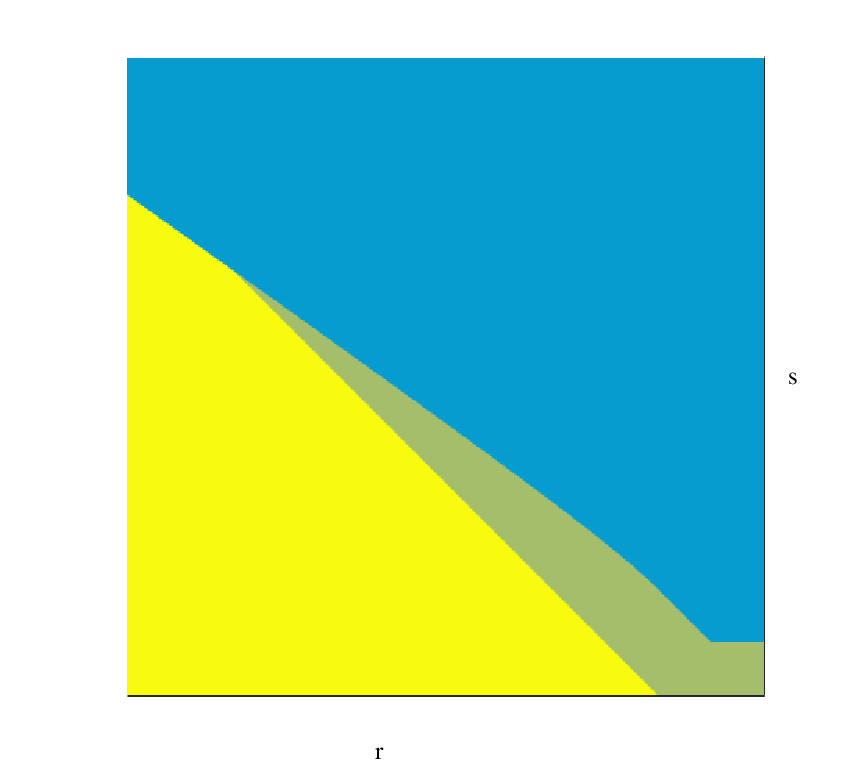}\\[-7.7pt]
\end{array}
&\,\,\,\\[-6pt]\hline
\end{array}
\qquad
\begin{array}[b]{@{\ \ }c@{}}
\hspace{-3pt}\goesto\\[7mm]\mbox{}
\end{array}
\qquad
\begin{array}[b]{@{\hspace{-10pt}}r@{\hspace{-10pt}}|@{}l@{}}
\begin{array}{@{}c@{}}
\psfrag{1}{\tiny$\kk$}
\psfrag{2}{\tiny$\kk^2$}
\psfrag{3}{\tiny$\kk^3$}
\includegraphics[height=30mm]{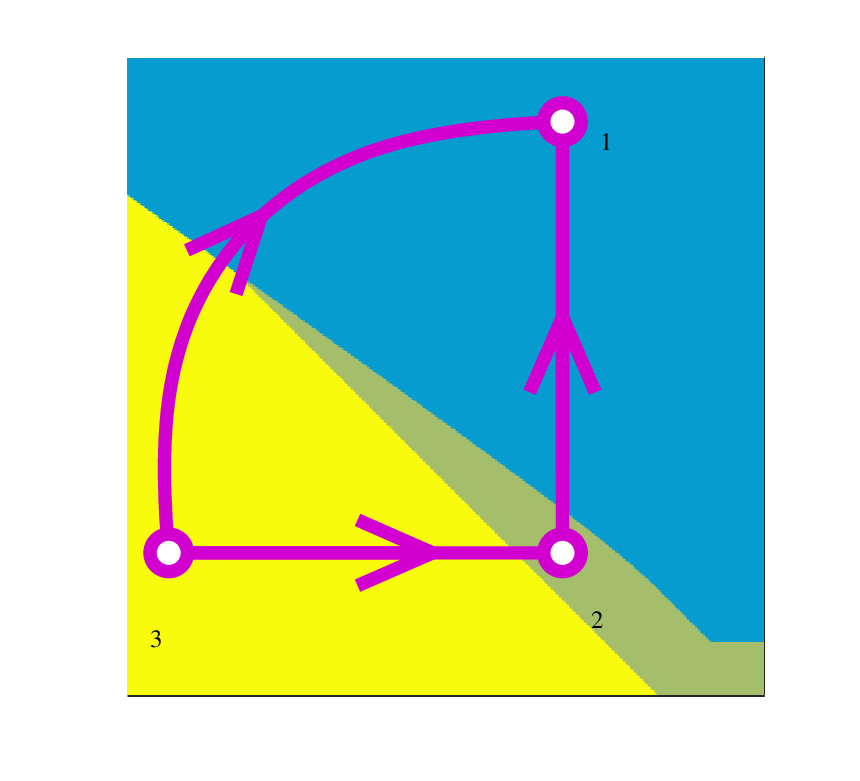}\\[-7.7pt]
\end{array}
&\,\,\,\\[-6pt]\hline
\end{array}
\vspace{-.5ex}
$$
\caption{Biparameter persistence module and finite encoding}
\label{f:toy-model-fly-wing}
\end{figure}
march off to infinity roughly diagonally away from the origin.  (The
remaining image is explained in Example~\ref{e:encoding}.)
\end{example}

To replace the noetherian hypothesis in the setting of modules over
arbitrary posets, for theoretical as well as computational purposes,
the \emph{finitely encoded} condition is introduced combinatorially
(Definition~\ref{d:encoding}).  It stipulates that the module~$\cM$
should be pulled back from a $\cP$-module along a poset morphism $\cQ
\to \cP$ in which $\cP$ is a finite poset and the $\cP$-module has
finite dimension as a vector space over the field~$\kk$.

\begin{example}\label{e:encoding}
The right-hand image in Example~\ref{e:toy-model-fly-wing} is a finite
encoding of~$\cM$ by a three-element poset $\cP$ and the $\cP$-module
$H = \kk^3 \oplus \kk^2 \oplus \kk$ with each arrow in the image
corresponding to a full-rank map between summands of~$H$.
Technically, this is only an encoding of~$\cM$ as a module over $\cQ =
\RR_- \times \RR_+$.  The poset morphism $\cQ \to \cP$ takes all of
the yellow rank~$3$ points to the bottom element of~$\cP$, the olive
rank~$2$ points to the middle element of~$\cP$, and the blue rank~$1$
points to the top element of~$\cP$.  (To make this work over all
of~$\RR^2$, the region with vector space dimension~$0$ would have to
be subdivided, for instance by introducing an antidiagonal extending
downward from the origin, thus yielding a morphism from~$\RR^2$ to a
five-element poset.)  This encoding is \emph{semialgebraic}
(Definition~\ref{d:alg-finite}): its fibers are real semialgebraic
sets.
\end{example}

The finitely encoded hypothesis captures topological tameness of
persistent homology in situations from data analysis, making precise
what it means for there to finitely many topological transitions as
the parameters vary.  But there is nuance: the \emph{isotypic regions}
(Definition~\ref{d:subdivide}), on which the homology remains
constant, need not be situated in a manner that makes them the fibers
of a poset morphism (Example~\ref{e:subdivide}).  Nonetheless, over
arbitrary posets, modules with finitely many isotypic regions always
admit finite encodings (Theorem~\ref{t:isotypic}), although the
isotypic regions are typically subdivided by the encoding poset
morphism.  In the case where the poset is a real vector space, if the
isotypic regions are semialgebraic then a semialgebraic encoding is
possible.

In ordinary totally ordered persistence, finitely encoded means simply
that the bar code should have finitely many bars: the poset being
finite precludes infinitely many non-overlapping bars (the bar code
can't be ``too long''), while the vector space having finite dimension
precludes a parameter value over which lie infinitely many bars (the
bar code can't be ``too thick'').

Finite encoding has its roots in combinatorial commutative algebra in
the form of sector partitions \cite{injAlg} (or
see~\cite[Chapter~13]{cca}).  Like sector partitions, finite encoding
is useful, theoretically, for its enumeration of all topologies
encountered as the parameters vary.  However, enumeration leads to
combinatorial explosion outside of the very lowest numbers of
parameters.  And beyond its inherent inefficiency, poset encoding
lacks many of the features that have come to be expected from
persistent homology, including the most salient: a description of
homology classes in terms of their persistence, meaning birth, death,
and lifetime.

%%%%%%%%%%%%%%%%%%%%%%%%%%%%%%%%%%%%%%%%%%%%%%%%%%%%%%%%%%%%%%%%%%%%%%%%%%
\subsection{Discrete persistent homology by birth and death}\label{sub:disc-PH}

The perspective on finitely generated $\ZZ^n$-modules arising from
their equivalence with multiparameter persistence is relatively new to
commutative algebra.  Initial steps have included descriptions of the
set of isomorphism classes \cite{multiparamPH}, presentations
\cite{csv14} and algorithms for computing \cite{computMultiPH,csv12}
or visualizing \cite{lesnick-wright2015} them, as well as interactions
with homological algebra of modules, such as persistence invariants
\cite{knudson2008} and certain notions of multiparameter noise
\cite{clrso15}.

Algebraically, viewing persistent homology as a module rather than
(say) a diagram or a bar code, a birth is a generator.  In ordinary
persistence, with one parameter, a death is more or less the same as a
relation.  However, in multiparameter persistence the notion of death
diverges from that of relation.  The issue is partly one of geometric
shape in the parameter poset, say~$\ZZ^n$ (the shaded regions indicate
where classes die):
\vspace{-.2ex}
$$%
\psfrag{birth}{\tiny birth}
\psfrag{death}{\tiny death}
\psfrag{relation}{\tiny relation}
\includegraphics[height=30mm]{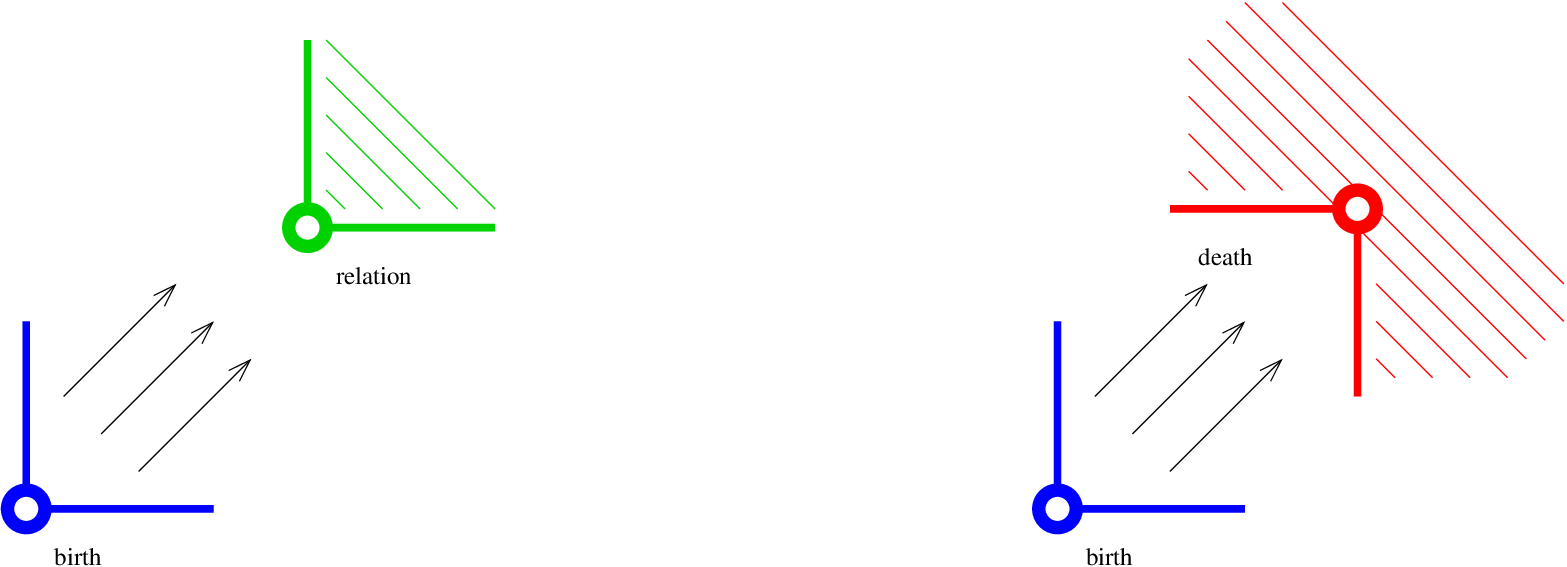}
$$
\vspace{-.25ex}
If death is to be dual to birth, then a nonzero homology class at some
parameter should die if it moves up along any direction in the poset.
Birth is not the bifurcation of a homology class into two independent
ones; it is the creation of a new class from zero.  Likewise, genuine
death is not the joining of two classes into one; it is annihilation.
And death should be stable, in the sense that wiggling the parameter
and then pushing up should still kill the homology class.

In algebraic language, death is a \emph{cogenerator} rather than a
relation.  For finitely generated $\NN^n$-modules, or slightly more
generally for \emph{finitely determined} modules
(Example~\ref{e:convex-projection} and Definition~\ref{d:determined}),
cogenerators are irreducible components, cf.~\cite[Section~5.2]{cca}.
Indeed, irreducible decomposition suffices as a dual theory of death
in the finitely generated case; this is more or less the content of
Theorem~\ref{t:finitely-determined}.  The idea there is that
surjection from a free module covers the module by sending basis
elements to births in the same (or better, dual) way that inclusion
into an injective module envelops the module by capturing deaths as
injective summands.  The geometry of this process in the parameter
poset on the injective side is as well understood as it is on the free
side \cite[Chapter~11]{cca}, and in finitely generated situations it
is carried out theoretically or algorithmically by finitely generated
truncations of injective modules~\cite{irredRes,injAlg}.

Combining birth by free cover and death by injective hull leads
naturally to \emph{flange presentation} (Definition~\ref{d:flange}),
which composes the augmentation map $F \onto \cM$ of a flat resolution
with the augmentation map $\cM \into E$ of an injective resolution to
get a homomorphism $F \to E$ whose image is~$\cM$.  The indecomposable
summands of~$F$ capture births and those of~$E$ deaths.  Flange
presentation splices a flat resolution to an injective one in the same
way that Tate resolutions (see \cite{coanda2003}, for example) segue
from a free resolution to an injective one over a Gorenstein local
ring of dimension~$0$.

Why a flat cover $F \onto \cM$ instead of a free one?  There are two
related reasons: first, flat modules are dual to injective ones
(Remark~\ref{r:matlis-pair}), so in the context of finitely determined
modules the entire theory is self-dual; and second, births can lie
infinitely far back along axes, as in the toy-model fly wing from
Example~\ref{e:toy-model-fly-wing}.

%%%%%%%%%%%%%%%%%%%%%%%%%%%%%%%%%%%%%%%%%%%%%%%%%%%%%%%%%%%%%%%%%%%%%%%%%
\subsection{Discrete to continuous: fringe presentation}\label{sub:disc-to-cont}

That multiparameter persistence modules can fail to be finitely
generated, like Example~\ref{e:toy-model-fly-wing} does, in situations
reflecting reasonably realistic data analysis was observed by
Patriarca, Scolamiero, and Vaccarino \cite[Section~2]{psv12}.  Their
``monomial module'' view of persistence covers births much more
efficiently, for discrete parameters, by keeping track of generators
not individually but gathered together as generators of monomial
ideals.  Huge numbers of predictable syzygies among generators are
swallowed: monomial ideals have known syzygies, and there are lots of
formulas for them, but nothing new is learned from them topologically,
in the persistent sense.

Translating to the setting of continuous parameters, and including the
dual view of deaths, which works just as well, suggests an uncountably
more efficient way to cover births and deaths than listing them
individually.  This urge to gather births or deaths comes
independently from the transition to continuous parameters from
discrete ones.  To wit, any $\RR^n$-module~$\cM$ can be approximated
by a~$\ZZ^n$-module, the result of restricting $\cM$ to, say, the
rescaled lattice~$\epsilon\ZZ^n$.  Suppose, for the sake of argument,
that~$\cM$ is bounded, in the sense of being zero at parameters
outside of a bounded subset of~$\RR^n$; think of
Example~\ref{e:toy-model-fly-wing}, ignoring those parts of the module
there that lie outside of the depicted square.
\begin{figure}[h]
\vspace{-2ex}
$$%
\begin{array}[b]{@{\hspace{-10pt}}r@{\hspace{-10pt}}|@{}l@{}}
\begin{array}{@{}c@{}}
\psfrag{r}{}
\psfrag{s}{}
\includegraphics[height=30mm]{toy-model}\\[-7.7pt]
\end{array}
&\,\,\,\\[-6pt]\hline
\end{array}
\quad\
\begin{array}[b]{@{\ \ }c@{}}
\hspace{-3pt}\goesto\\[7mm]\mbox{}
\end{array}
\qquad
\begin{array}[b]{@{\hspace{-10pt}}r@{\hspace{-10pt}}|@{}l@{}}
\begin{array}{@{}c@{}}
\includegraphics[height=30mm]{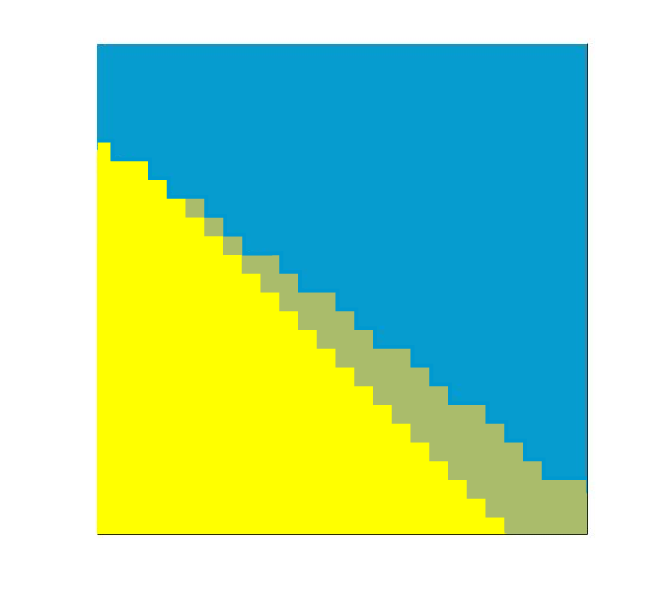}\\[-7.7pt]
\end{array}
&\,\,\,\\[-6pt]\hline
\end{array}
$$
\end{figure}
Ever better approximations, by smaller~$\epsilon \to 0$, yield sets of
lattice points ever more closely hugging an algebraic curve.
Neglecting the difficulty of computing where those lattice points lie,
how is a computer to store or manipulate such a set?  Listing the
points individually is an option, and perhaps efficient for
particularly coarse approximations, but in~$n$ parameters the
dimension of this storage problem is~$n-1$.  As the approximations
improve, the most efficient way to record such sets of points is
surely to describe them as the allowable ones on one side of an
algebraic curve.  And once the computer has the curve in memory, no
approximation is required: just use the (points on the) curve itself.
In this way, even in cases where the entire topological filtration
setup can be approximated by finite simplicial complexes,
understanding the continuous nature of the un-approximated setup is
both more transparent and more efficient.

Combining flange presentation with this monomial module view of births
and deaths yields \emph{fringe presentation}
(Definition~\ref{d:fringe}), the analogue for modules over an
arbitrary poset~$\cQ$ of flange presentation for finitely determined
modules over $\cQ = \ZZ^n$.  The role of indecomposable free or flat
modules is played by \emph{upset modules} (Example~\ref{e:indicator})
which have $\kk$ in degrees from an upset~$U$ and~$0$ elsewhere.  The
role of indecomposable injective modules is played similarly by
\emph{downset modules}.

Fringe presentation is expressed by \emph{monomial matrix}
(Definition~\ref{d:monomial-matrix-fr}), an array of scalars with rows
labeled by upsets and columns labeled by downsets.  For example,
\vspace{7ex}
$$%
\monomialmatrix
	{\begin{array}[t]{@{}r@{\hspace{-3.1pt}}|@{}l@{}}
	 \includegraphics[height=15mm]{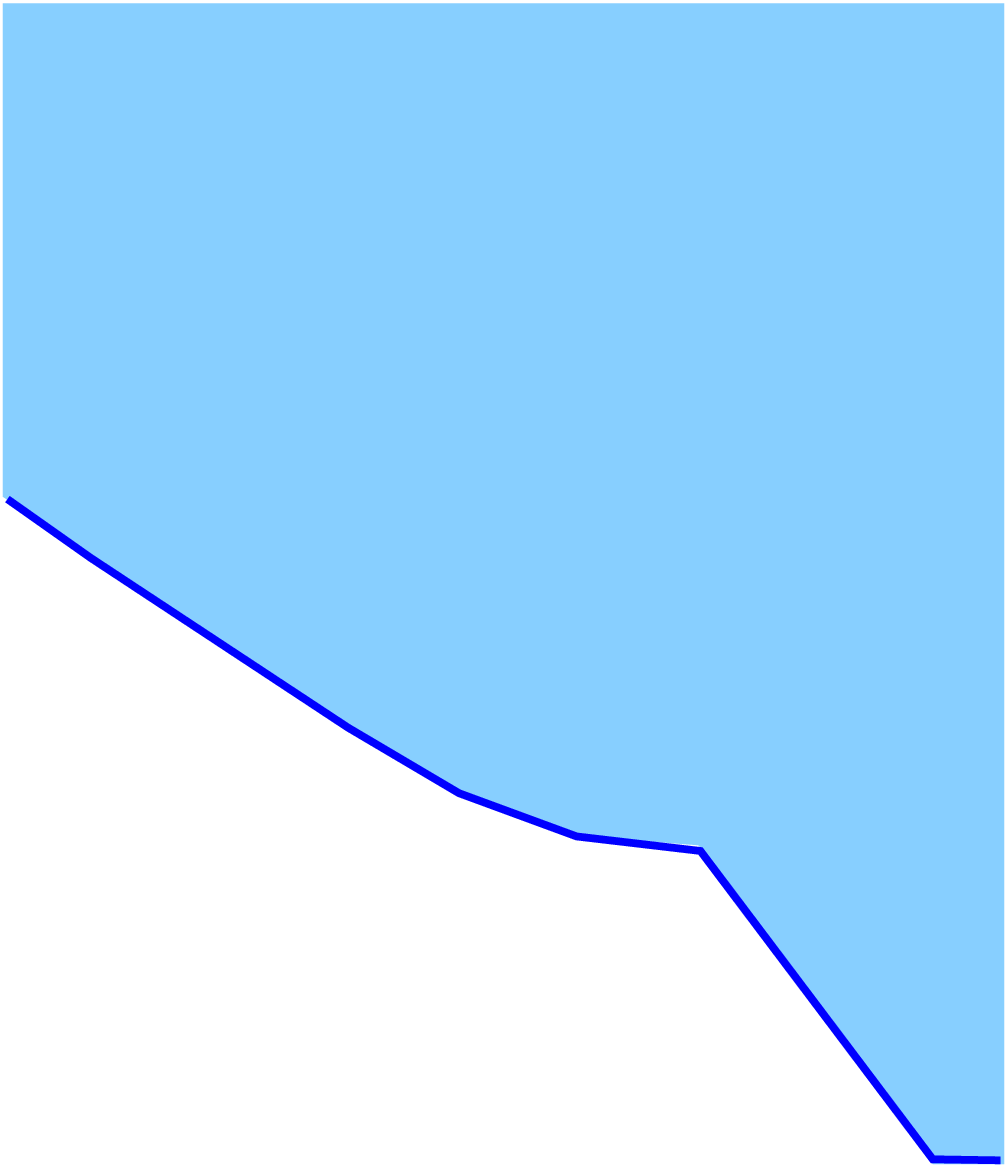}&\ \hspace{-.2pt}\\[-4.3pt]\hline
	\end{array}}
	{\!\!
	 \begin{array}{c}
	 \\[-10ex]
	 \begin{array}[b]{@{}r@{\hspace{-.2pt}}|@{}l@{}}
		\raisebox{-5mm}{\includegraphics[height=17mm]{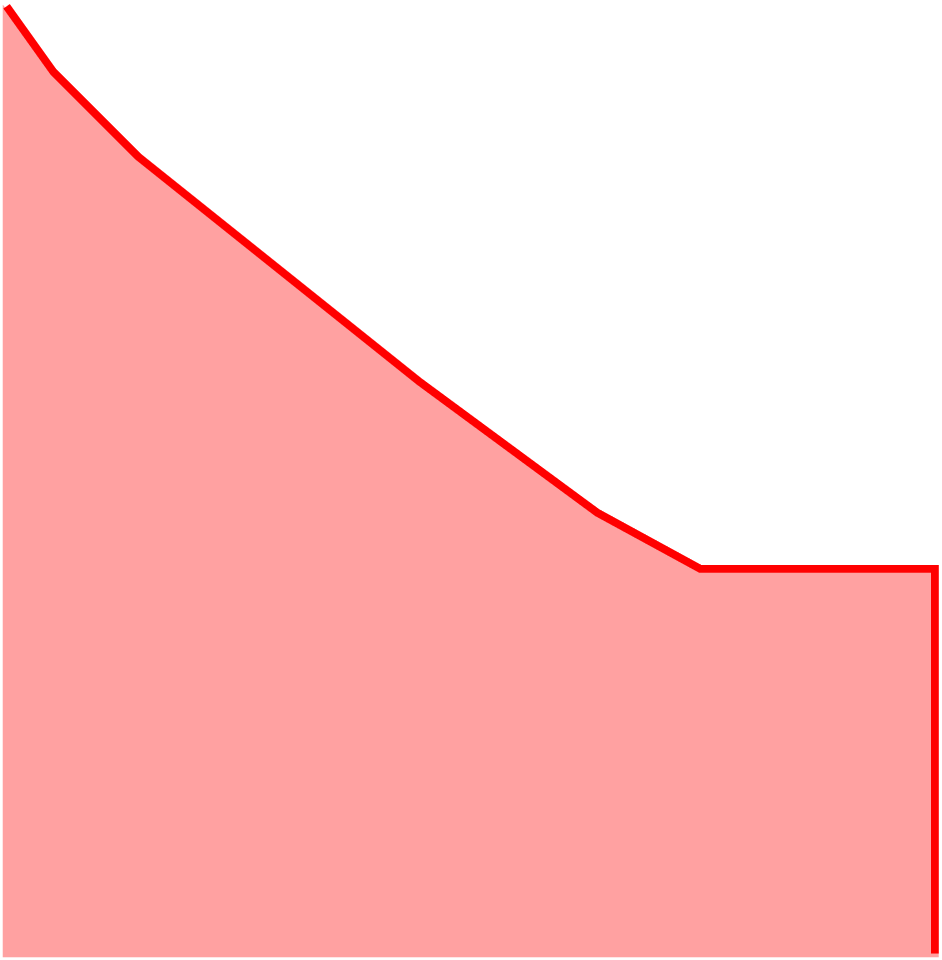}}
		&\ \,\\[-6.3pt]\hline
	 \end{array}
	 \!\!
	 \\[4ex]
	 \phi_{11}
	 \\[3ex]
	 \end{array}}
	{\\\\\\}
\begin{array}{@{}c@{}}
\hspace{.1pt}\ \text{represents a fringe presentation of}\ \hspace{.2pt}
\cM = \kk\!\!
\left[
\begin{array}{@{\ }c@{\,}}
\\[-2.2ex]
\begin{array}{@{}r@{\hspace{-.4pt}}|@{}l@{}}
\includegraphics[height=15mm]{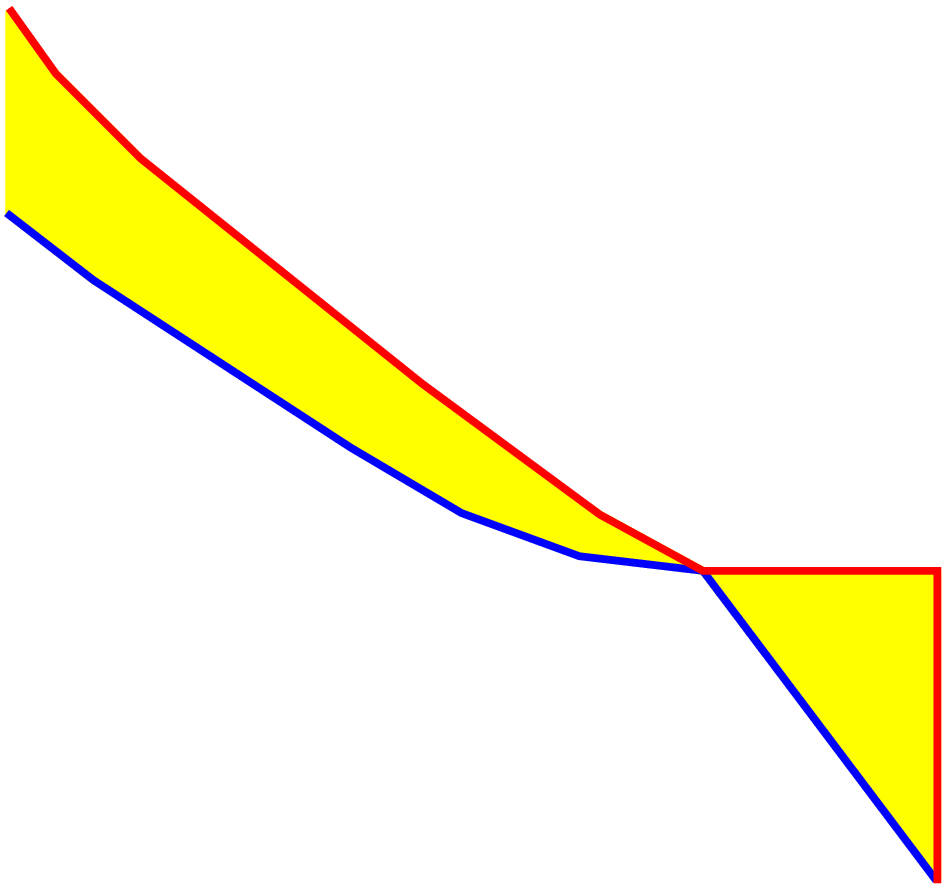}&\ \,\hspace{-.3pt}\\[-4.7pt]\hline
\end{array}
\\[-1ex]\mbox{}
\end{array}
\right]
\\[-3ex]
\end{array}
\vspace{.5ex}
$$
as long as $\phi_{11} \in \kk$ is nonzero.  The monomial matrix
notation specifies a homomorphism
\(
  \kk\bigl[\,
  \begin{array}{@{}c@{}}
  \\[-3ex]
  \begin{array}{@{}r@{\hspace{-1.4pt}}|@{}l@{}}
  \raisebox{-2.2pt}{\includegraphics[height=5mm]{blue-upset}}
  &\hspace{1.7pt}\\[-2pt]\hline
  \end{array}\,
  \end{array}
  \bigr]\!
\to
  \kk\bigl[\,
  \begin{array}{@{}c@{}}
  \\[-3ex]
  \begin{array}{@{}r@{\hspace{-.2pt}}|@{}l@{}}
  \raisebox{-6.2pt}{\includegraphics[height=5mm]{red-downset}}
  &\hspace{1.5pt}\\[-2pt]\hline
  \end{array}
  \end{array}
  \,\bigr]
\)
whose image is~$\cM$, which has $\cM_\aa = \kk$ over the yellow
parameters~$\aa$ and~$0$ elsewhere.  The blue upset specifies the
births at the lower boundary of~$\cM$; unchecked, the classes would
persist all the way up and to the right.  But the red downset
specifies the deaths along the upper boundary of~$\cM$.

When the birth upsets and death downsets are semialgebraic, or
otherwise manageable algorithmically, monomial matrices render fringe
presentations effective data structures for real multiparameter
persistence.  Fringe presentations have the added benefit of being
topologically interpretable in terms of birth and death.

Although the data structure of fringe presentation is aimed at
$\RR^n$-modules, it is new and lends insight already for finitely
generated $\NN^n$-modules (even when $n = 2$), where monomial matrices
have their origins \cite[Section~3]{alexdual}.  The context there is
more or less that of finitely determined modules; see
Definition~\ref{d:monomial-matrix-fl}, in particular, which is really
just the special case of fringe presentation in which the upsets are
localizations of~$\NN^n$ and the downsets are duals---that is,
negatives---of those.

%%%%%%%%%%%%%%%%%%%%%%%%%%%%%%%%%%%%%%%%%%%%%%%%%%%%%%%%%%%%%%%%%%%%%%%%%
\subsection{Homological algebra of modules over posets}\label{sub:homalg}

Even in the case of filtrations of finite simplicial complexes by
products of intervals---that is, \emph{multifiltrations}
(Example~\ref{e:RR+}) of finite simplicial complexes---persistent
homology is not naturally a module over a polynomial ring in $n$ (or
any finite number of) variables.  This is for the same reason that
single-parameter persistent homology is not naturally a module over a
polynomial ring in one variable: though there can only be finitely
many topological transitions, they can (and often do) occur at
incommensurable real numbers.  That said, observe that filtering a
finite simplicial complex automatically induces a finite encoding.
Indeed, the parameter space maps to the poset of simplicial
subcomplexes of the original simplicial complex by sending a parameter
to the simplicial subcomplex it represents.  That is not the smallest
poset, of course, but it leads to a fundamental point: one can and
should do homological algebra over the finite encoding poset rather
than (only) over the original parameter space.

This line of thinking culminates in a syzygy theorem
(Theorem~\ref{t:syzygy}) to the effect that finitely encoded modules
are characterized as those admitting, equivalently,
\begin{itemize}
\item%
finite fringe presentations,
\item%
finite resolutions by finite direct sums of upset modules, or
\item%
finite resolutions by finite direct sums of downset modules.
\end{itemize}
This result directly reflects the closer-to-usual syzygy theorem for
finitely determined $\ZZ^n$-modules
(Theorem~\ref{t:finitely-determined}), with upset and downset
resolutions being the arbitrary-poset analogues of free and injective
resolutions, respectively, and fringe presentation being the
arbitrary-poset analogue of flange presentation.

The moral is that the finitely encoded condition over arbitrary posets
appears to be the right notion to stand in lieu of the noetherian
hypothesis over~$\ZZ^n$: the finitely encoded condition is robust, has
separate combinatorial, algebraic, and homological characterizations,
and makes algorithmic computation possible, at least in principle.

The proof of the syzygy theorem works by reducing to the finitely
determined case over~$\ZZ^n$.  The main point is that given a finite
encoding of a module over an arbitrary poset~$\cQ$, the encoding poset
can be embedded in~$\ZZ^n$.  The proof is completed by pushing the
data forward to~$\ZZ^n$, applying the syzygy theorem there, and
pulling back to~$\cQ$.

%%%%%%%%%%%%%%%%%%%%%%%%%%%%%%%%%%%%%%%%%%%%%%%%%%%%%%%%%%%%%%%%%%%%%%%%%
\subsection{Geometric algebra over partially ordered abelian groups}\label{sub:pogroup}

One of the most prominent features of multiparameter persistence that
differs from the single-parameter theory is that elements in modules
over $\RR^n$ or~$\ZZ^n$ can die in many ways.  A hint of this
phenomenon occurs already over one parameter: an element can die after
persisting finitely---that is, its bar can have a right endpoint---or
it can persist indefinitely.  Over two parameters, a class could
persist infinitely along the $x$-axis but die upon moving up
sufficiently along the $y$-axis, or vice versa, or it could persist
infinitely in both directions, or die in both directions.

\begin{example}\label{e:hyperbola-GD}
Let $D$ be the downset in~$\RR^2$ consisting of all points beneath the
upper branch of the hyperbola $xy = 1$.  Then $D$ canonically
decomposes as the union
$$%
\psfrag{x}{\tiny$x$}
\psfrag{y}{\tiny$y$}
  \begin{array}{@{}c@{}}\includegraphics[height=25mm]{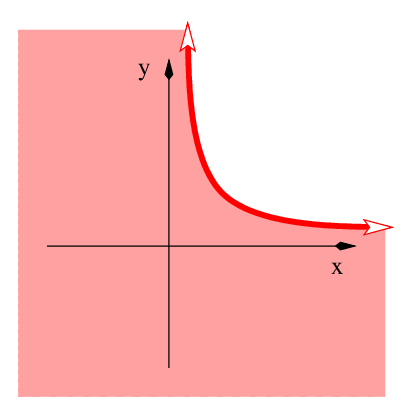}\end{array}
\ =\
  \begin{array}{@{}c@{}}\includegraphics[height=25mm]{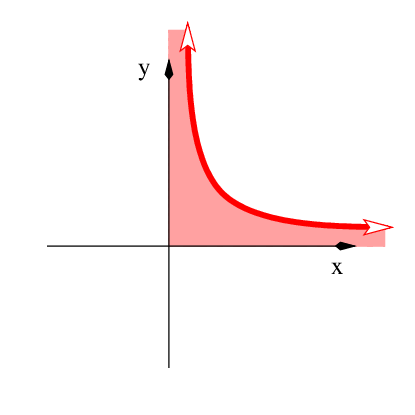}\end{array}
\cup\,
  \begin{array}{@{}c@{}}\includegraphics[height=25mm]{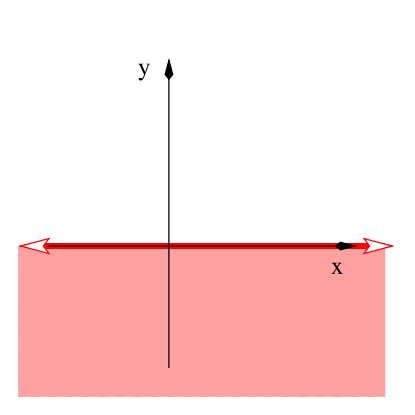}\end{array}
\cup\,
  \begin{array}{@{}c@{}}\includegraphics[height=25mm]{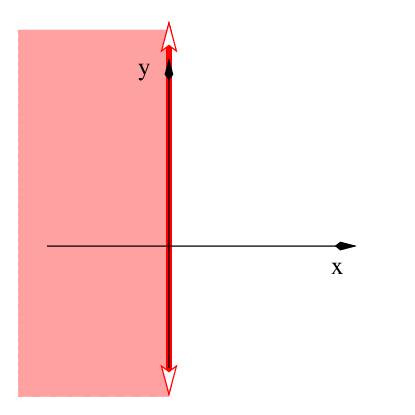}\end{array}
$$ %
of its subsets that die pure deaths of some type (Theorem~\ref{t:PF}):
every red point in the
\begin{itemize}
\item%
leftmost subset on the right dies when pushed over to the right or up far enough;
\item%
middle subset dies in the \emph{localization} of~$D$ along the
$x$-axis (Definition~\ref{d:PF} or Definition~\ref{d:support}) when
pushed up far enough; and
\item%
rightmost subset dies locally along the $y$-axis when pushed over far
enough.
\end{itemize}
\end{example}

\noindent
Generally, in multiparameter situations, elements can persist forever
along any face of the positive cone and die upon exiting far enough
from that face.  Worse, single elements can do combinations of these
things; take, for example, any nonzero element of degree~$\0$ in the
$\ZZ^2$-module $\kk[x,y]/\<xy\>$, which persists along the $x$-axis
but dies upon leaving it, while it also persists along the $y$-axis
but dies upon leaving it.

These phenomena have already been discussed implicitly in prior
sections: types of death in, say, the category of finitely determined
$\ZZ^n$-modules correspond to isomorphism classes of indecomposable
injective $\ZZ^n$-modules, which come in $2^n$ flavors, up to
translation along~$\ZZ^n$ (also known as $\ZZ^n$-graded shift).  The
injective half of a finite flange presentation places a module~$\cM$
inside of a finite direct sum of such modules, thereby decomposing
every element of~$\cM$ as a sum of components each of which dies a
pure death of one of these $2^n$ types.  Gathering injective summands
of the same type, every element of~$\cM$ is a sum of elements each of
which dies a pure death of some distinct type.  In algebraic language,
this is a \emph{primary decomposition} of~$\cM$
(Definition~\ref{d:primDecomp'}).

The natural setting in which to carry out primary decomposition is, in
the best tradition of classical mathematics
\cite{birkhoff42,clifford40,riesz40}, over partially ordered abelian
groups (Definition~\ref{d:pogroup}).  Those provide an optimally
general context in which posets have some notion of ``face'' along
which one can localize without altering the ambient poset.  That is, a
partially ordered group~$\cQ$ has an origin---namely its
identity~$\0$---and hence a positive cone~$\cQ_+$ of elements it
precedes.  A~\emph{face} of~$\cQ$ is a submonoid of the positive cone
that is also a downset therein (Definition~\ref{d:face}).  And as
everything takes place inside of the ambient group~$\cQ$,
% the positive cone is a submonoid of its Grothendieck group, which
% is~$\cQ$.  Therefore
every localization of a $\cQ$-module along a face
(Definition~\ref{d:support}) remains a $\cQ$-module; that much is
needed to functorially isolate all module elements that die in only
one way, which is accomplished by local support functors, as in
ordinary commutative algebra and algebraic geometry
(Definition~\ref{d:primDecomp'}).  Preoccupation with the potential
for algorithmic computation draws the focus to the case where $\cQ$ is
\emph{polyhedral}, meaning that it has only finitely many faces
(Definition~\ref{d:face}).  This notion is apparently new.  Its role
here is to guarantee finiteness of primary decomposition of finitely
encoded modules (Theorem~\ref{t:primDecomp}).

It bears mentioning that primary decomposition of downset modules, or
equivalently, expressions of downsets as unions of \emph{coprimary}
downsets (Definition~\ref{d:primDecomp}), is canonical
(Theorem~\ref{t:PF} and Corollary~\ref{c:PF}), generalizing the
canonical primary decomposition of monomial ideals in ordinary
polynomial rings.  Topologically speaking, coprimary downsets are
those cogenerated by poset elements that die pure deaths, so this
canonical expression as a union tells the fortune of every downset
element.

Notably lacking from primary decomposition theory over arbitrary
polyhedral partially ordered abelian groups is a notion of
minimality---alas, a lack that is intrinsic.

\begin{example}\label{e:hyperbola-PD}
The union in Example~\ref{e:hyperbola-GD} results in a canonical
primary decomposition
$$%
\psfrag{x}{\tiny$x$}
\psfrag{y}{\tiny$y$}
  \kk\!
  \left[
  \begin{array}{@{}c@{}}\includegraphics[height=25mm]{hyperbola}\end{array}
  \right]
\:\into\ 
  \kk\!
  \left[
  \begin{array}{@{}c@{}}\includegraphics[height=25mm]{hyperbola}\end{array}
  \right]
\oplus\,
  \kk\!
  \left[
  \begin{array}{@{}c@{}}\includegraphics[height=25mm]{x-component}\end{array}
  \right]
\oplus\,
  \kk\!
  \left[
  \begin{array}{@{}c@{}}\includegraphics[height=25mm]{y-component}\end{array}
  \right]
$$
of the downset module $\kk[D]$ over~$\RR^2$ (Corollary~\ref{c:PF}).
Although all three of the pure death types are required in the union
decomposing~$D$ (Example~\ref{e:hyperbola-GD}), the final two summands
in the decomposition of~$\kk[D]$ are redundant.  One can, of course,
simply omit the redundant summands, but for arbitrary polyhedral
partially ordered groups no criterion is known for detecting a~priori
which summands should be omitted.
\end{example}

The failure of minimality here stems from geometry that can only occur
in partially ordered groups more general than finitely generated free
ones.  More specifically, $D$ contains elements that die pure deaths
of type ``$x$-axis'' but the boundary of~$D$ fails to contain an
actual translate of the face of~$\RR^2_+$ that is the positive
$x$-axis.  This can be seen as a certain failure of localization to
commute with taking homomorphisms into~$\kk[D]$; this is the content
of the crucial Remark~\ref{r:soc-vs-supp}, which highlights the
difference between real-graded algebra and integer-graded algebra.  It
is the source of much of the subtlety in the theory developed in this
paper, particularly Sections~\ref{s:socle}--\ref{s:hulls}, which is a
development of substantial portions of basic commutative algebra of
finitely determined $\RR^n$-modules with one hand (the noetherian one)
tied behind the back.

The purpose of that theory is partly to rectify, for real
multiparameter persistence, the failure of minimality in
Example~\ref{e:hyperbola-PD} by pinpointing what it means for
$\RR^n$-modules to have minimal cogenerators (Section~\ref{s:socle}),
and what it means for a primary decomposition to be minimal
(Section~\ref{s:hulls}).  But more importantly, thinking of
cogenerators topologically as deaths, the minimal cogenerator theory
and its dual for minimal generators and births
(Section~\ref{s:gen-functors}) provide exactly the ingredients needed
for multiparameter functorial generalizations of bar codes and elder
rules (Sections~\ref{s:qr}, \ref{s:elder}, and~\ref{s:barcodes}).

%%%%%%%%%%%%%%%%%%%%%%%%%%%%%%%%%%%%%%%%%%%%%%%%%%%%%%%%%%%%%%%%%%%%%%%%%
\subsection{Minimal generators and cogenerators}\label{sub:min}

Even in ordinary, single-parameter persistence, the intervals in a bar
code need not be closed: they are usually half-open, being typically
closed on the left (at birth) and open on the right (at death).  This
subtlety becomes more delicate with multiple parameters.
Sections~\ref{s:socle}--\ref{s:tops} build theory to handle the
notions of generator and cogenerator for modules over \emph{real
polyhedral groups} (Definition~\ref{d:polyhedral}), which is~$\RR^n$
with an arbitrary polyhedral positive~cone.

Intuitively, a generator of a module (or upset) is an element that is
not present when approached from below but present when approached
from above.  Dually, a cogenerator of a module (or downset) is an
element that is present when approached from below but not present
when approached from above---think of the right endpoint of a bar in
ordinary persistence, be it a closed endpoint or an open one.  With
multiple parameters, (co)generators can also have positive dimension,
being parallel to any face of the positive cone.  For death, this is
what it means to persist along the face; for birth, the dual concept
is less familiar but occurs already in the toy-model fly wing
(Example~\ref{e:toy-model-fly-wing}).  The single-parameter case,
namely bars with infinite length, have in practice been handled in an
ad hoc manner, but that is not an option in multiparameter
persistence, where infinite bars come in distinct polyhedral flavors.

Phrased more geometrically, upsets and downsets in a real polyhedral
group~$\cQ$ need not be closed subsets of~$\cQ$.  Points along the
\emph{frontier}---in the topological closure but outside of the
% https://en.wikipedia.org/wiki/Boundary_(topology)
original set---feel like they are minimal or maximal, respectively,
but in reality they only are so in a limiting sense.  For this reason,
a $\cQ$-module need not be minimally generated or cogenerated, even if
it is $\cQ$-finite (Definition~\ref{d:Q-finite}), has finite isotypic
subdivision (Definition~\ref{d:subdivide}), and is bounded, in the
sense of having nonzero homogeneous elements only in degrees from a
set that is bounded in~$\cQ \cong \RR^n$.  The indicator module for
the interior of the unit cube in~$\RR^n_+$ provides a specific
example.

The main idea is to express frontier elements using limits and
colimits.  To get a feel for the theory, it is worth a leisurely tour
through the single-parameter case, where $\cQ$ is the real polyhedral
group~$\RR$ with positive cone~$\RR_+$.  To that end, fix a
module~$\cM$~over~$\RR$.
\begin{itemize}
\item%
Closed right endpoint at $a \in \RR$.  Let $\kk_a$ be the $\RR$-module
that is zero outside of degree~$a$ and has a copy of~$\kk$ in
degree~$a$.  Each nonzero element $\kk_a$ dies whenever it is pushed
to the right by any positive distance along~$\RR$.  A closed right
endpoint at~$a$ is a submodule of~$\cM$ isomorphic to~$\kk_a$.
Equivalently a closed right endpoint at~$a$ is a nonzero homomorphism
$\kk_a \to \cM$, so $1 \in \kk_a$ lands on an element of degree~$a$
in~$\cM$ that dies when pushed any distance to the right along~$\RR$.
Thus closed right endpoints of~$\cM$ at~$a$ are detected functorially
by $\Hom_\RR(\kk_a,\cM)$.

\item%
Open right endpoint at $a \in \RR$.  Consider $\dirlim_{a' < a}
\cM_a$.  This vector space sits at~$a$ but records only what happens
strictly to the left of~$a$.  This direct limit being nonzero means
only that $a$ is not a leftmost endpoint of~$\cM$.  More precisely,
this direct limit detects exactly those bars that extend strictly to
the left from~$a$.  Now compare the direct limit with the vector space
$\cM_a$ itself via the natural homomorphism $\phi_a: \dirlim_{a' < a}
\cM_a \to \cM_a$ induced by the universal property of direct limit.
Elements outside of the image of~$\phi_a$ are left endpoints at~$a$.
(Hence, without meaning to at this stage, we have stumbled upon what
it means to be a closed generator at~$a$; see
Section~\ref{s:generators}.)  Elements in the image of~$\phi_a$
persist from ``just before~$a$'' to~$a$ itself and hence could be
closed right endpoints at~$a$ but not open ones.  Elements in the
kernel of~$\phi_a$, on the other hand, persist until ``just
before~$a$'' and no further; these are the open right endpoints
sought.

\item%
Infinite right endpoint.  A class persists indefinitely if its bar
contains a translate of the ray~$\RR_+$.  Functorially, the goal is to
stick the upset module $\kk[a + \RR_+]$ into~$\cM$.  An element of
$\Hom\bigl(\kk[a + \RR_+],\cM\bigr)$ is nothing more or less than
simply an element of~$\cM_a$, since $\kk[a + \RR_+]$ is a free
$\kk[\RR_+]$-module of rank~$1$ generated in degree~$a$.  How is the
homomorphism ensured to be injective?  By localizing along the
face~$\RR_+$, which annihilates torsion and hence non-injective
homomorphisms.  Note that any other choice of~$a'$ on the ray $a +
\RR_+$ should morally select the very same bar, so these injective
homomorphisms should be taken modulo translation along the face~$\tau$
in question, in this case $\tau = \RR_+$ itsef.  Thus, functorially,
infinite right endpoints are detected by
$\hhom_\RR\bigl(\kk[\RR_+],\cM\bigr){}_{\RR_+}/\hspace{.2ex}\RR_+$,
where the underline on $\Hom$ means to try sticking all translates
(graded shifts) of~$\kk[\RR_+]$ into~$\cM$, the subscript~$\RR_+$
means localization, and the quotient means modulo the translation
action of~$\RR_+$ on the localization.
\end{itemize}

To unify the closed and open right endpoints, and hence to indicate
the generalization to multiple parameters, consider the \emph{upper
boundary module} $\delta\cM = \delta^{\{0\}}\cM \oplus
\delta^{\,\RR_+}\cM$, where $\delta^{\{0\}}\cM = \cM$ is viewed in
each degree $a \in \RR$ as the trivial direct limit of~$\cM_{a'}$ over
$a' \in a - \sigma^\circ$ for the relative interior~$\sigma^\circ$ of
the face $\sigma = \{0\}$ of the positive cone~$\RR_+$, and
$\delta^{\,\RR_+}\cM$ in degree~$a$ is the direct limit over $a -
\sigma^\circ$ for $\sigma = \RR_+$ itself
(Definition~\ref{d:upper-boundary}).  The universal property of direct
limits induces a homomorphism from the $\RR_+$-summand to the
$\{0\}$-summand.  Therefore $\delta\cM$ carries an action of $\RR
\times \cF$, where $\RR$ acts on each summand and the face poset $\cF$
of the cone~$\RR_+$ takes the $\sigma$-summand to the
$\sigma'$-summand for any $\sigma' \subseteq \sigma$, which in this
case means only the $\RR_+$-summand to the $\{0\}$-summand.
Functorially, a right endpoint---be it closed or open, without
specification---is detected at~$a$ by the \emph{cogenerator functor}
$\Hom_{\RR \times \fro}\bigl(\kk_a,\delta\cM\bigr)$, which computes
the degree~$a$ piece of the \emph{socle} of~$\cM$
(Definition~\ref{d:soc}).

The general real multiparameter case (Definition~\ref{d:soct}) is the
main contribution of Section~\ref{s:socle}.  It takes the join of the
finite and infinite endpoint cases, combining the upper boundary
functor with homomorphisms from~$\kk[\tau]$, localization, and
quotient modulo the span of the face~$\tau$, being careful to do these
operations in the correct order, because some of them fail to commute
(Remark~\ref{r:soc-vs-supp}).  The definition comes with a small
galaxy of useful foundations; see the opening of Section~\ref{s:socle}
for an overview of those.

The functorial treatment of left endpoints is dual to that of right
endpoints presented here, with inverse instead of direct limits, and
so on.  Armed with geometric intuition from this detailed
single-parameter discussion of right endpoints as a guide to
Section~\ref{s:socle}, reflecting the picture left-to-right should
suffice as intuition for its dual, Section~\ref{s:gen-functors}.
There are subtleties stemming from the asymmetry between the exactness
properties of direct and inverse limits---see the opening of
Section~\ref{s:gen-functors} for an overview of those---but they
affect only the outer confines of the theoretical development and
should cause no concern in practice, when all of the vector spaces in
sight have finite dimension.

The exposition remains concerned solely with cogenerator theory and
its consequences from Section~\ref{s:socle} through
Section~\ref{s:hulls}, as opposed to the dual theory of tops in
Section~\ref{s:gen-functors} and its consequences in
Section~\ref{s:tops}.  Generators are left until
Section~\ref{s:generators} because they require more, namely elder
morphisms (Section~\ref{s:elder}).

%%%%%%%%%%%%%%%%%%%%%%%%%%%%%%%%%%%%%%%%%%%%%%%%%%%%%%%%%%%%%%%%%%%%%%%%%
\subsection{Infinitesimal geometry of real polyhedral modules}\label{sub:inf}

Although primary decomposition and local support work over arbitrary
polyhedral partially ordered groups (Section~\ref{s:decomp}), the
generality must be restricted to the setting of real polyhedral groups
for much of the new theory of socles and cogenerators in
Sections~\ref{s:socle}--\ref{s:hulls}.  (The simpler parallel theory
over \emph{discrete polyhedral groups}---see
Definition~\ref{d:discrete-polyhedral}---holds and has value, but it
is barely new, being based on more elementary foundations; it is
recorded for posterity in Section~\ref{s:discrete}.)  The specific
technical reason for the restriction is detailed in
Remark~\ref{r:quantum}.  It is related to the omnipresent phenomenon
that drives the novelty, namely that boundaries of upsets and downsets
need not be closed or open, but can a~priori be anything in between.

The miracle, however, is that ``anything in between'' is hardly so: it
turns out to be rigidly constrained.  The distinct approaches to a
boundary point of a downset in a real polyhedral group~$\cQ$ are
indexed by the faces of~$\cQ$ (Proposition~\ref{p:shape}).  This
rigidity renders the cogenerator theory finite.  In particular,
viewing upper boundary modules as gathering all ways of taking direct
limits of vector spaces~$\cM_{\aa'}$ beneath a fixed degree~$\aa$,
this rigidity is what confines upper boundary modules to only
finitely~many~summands.

It might rightly be observed that in most current uses of ordinary
persistent homology with real parameters, the left endpoints are
closed,
% thinking of \cite{kashiwara-schapira2017} here
so there is no need to develop notions of closed and open generators.
However, for the same reason the right endpoints are usually open, so
the notion of open socle is critical.  In addition, allowing only
closed births and open deaths would obscure the natural self-duality
of the theory via Matlis duality for finitely encoded modules
(Section~\ref{sub:matlis}).

Infinitesimal geometry of real polyhedral groups prevents minimal
generating sets in the usual sense---or dually, minimal irreducible
decomposition (see Remark~\ref{r:irred-decomp})---from necessarily
existing in this setting.  That remains true even if one is willing to
accept uncountably many summands.  Most of the this trouble is
alleviated by the concept of \emph{shape} (Proposition~\ref{p:shape}),
which codifies how closed is the principal upset of a generator or
coprincipal downset of a cogenerator.

\begin{example}\label{e:open-generator}
Consider the following upsets and downsets in~$\RR^2$.  (The choice of
which of these to make an upset and which a downset was arbitrary: the
picture dualizes by reflection through the origin.)
$$%
\begin{array}{@{}*3{c@{\qquad}}c}
\\[-3.8ex]
\begin{array}{@{}c@{}}\includegraphics[height=20mm]{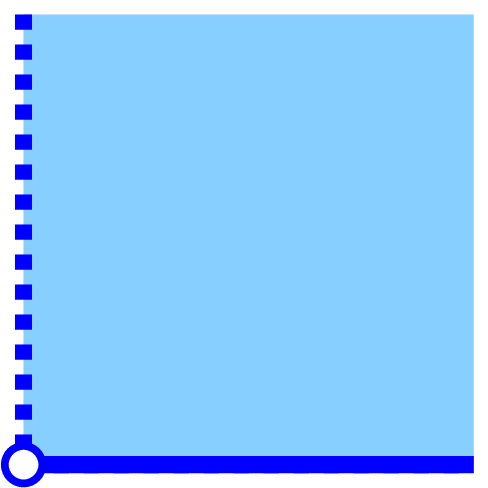}\end{array}
&
\begin{array}{@{}c@{}}\includegraphics[height=20mm]{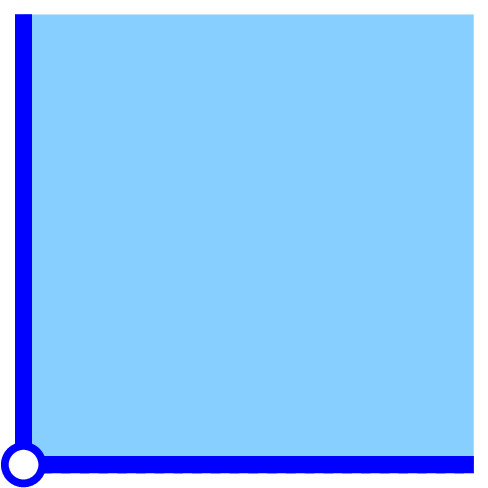}\end{array}
&
\begin{array}{@{}c@{}}\includegraphics[height=30mm]{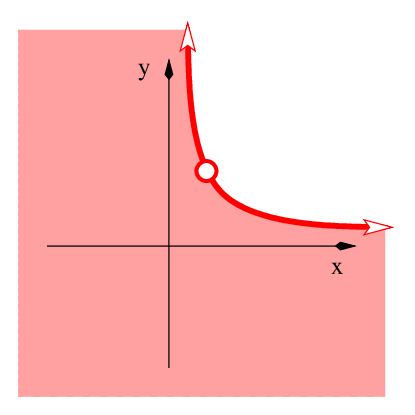}\end{array}
&
\begin{array}{@{}c@{}}\includegraphics[height=30mm]{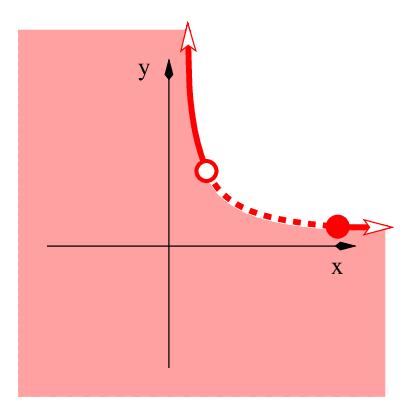}\end{array}
\\
1 & 2 & 3 & 4
\\[-.5ex]
\end{array}
$$
\begin{enumerate}
\item\label{i:along-y-open}%
The upset $U_{\hbox{\tiny\ref{i:along-y-open}}} = \RR^2_+ \minus
y$-axis is a ``half-open'' positive quadrant.  Writing it as a union
of (translated) closed positive quadrants
in~$U_{\hbox{\tiny\ref{i:along-y-open}}}$ requires infinitely many.
As long as the $x$-axis contains enough of their lower corners so
that~$\0$ is an accumulation point, that suffices.  However, the
theory of generators makes $U_{\hbox{\tiny\ref{i:along-y-open}}}$
principal: its sole open generator lies infinitesimally near the
origin on the~\mbox{$x$-axis}.

\item\label{i:open-generator}%
The upset $U_{\hbox{\tiny\ref{i:open-generator}}} = \RR^2_+ \minus
\{\0\}$ is a positive quadrant that is closed away from its missing
origin.  It is a union of closed positive quadrants as
$U_{\hbox{\tiny\ref{i:along-y-open}}}$ is, except the lower corners
must have~$\0$ as an accumulation point in each of the axes.
Generator theory detects two open generators, one
being~$U_{\hbox{\tiny\ref{i:along-y-open}}}$ and the other being its
reflection across the diagonal.  Note the similarity with the discrete
case: the maximal ideal in $\NN^2$ has two generators, one as close to
the origin as possible on each axis.  That description is made precise
in~$\RR^2_+$ by Theorem~\ref{t:downset=union}.

\item\label{i:plucked-hyperbola}%
Plucking out a single point from the hyperbolic boundary of the
downset in Example~\ref{e:hyperbola-GD} has an odd effect.  At the
frontier point, cogenerator theory (Definition~\ref{d:soct}) detects
two open cogenerators, dual to the generators of
of~$U_{\hbox{\tiny\ref{i:open-generator}}}$, but they are redundant:
$D_{\hbox{\tiny\ref{i:plucked-hyperbola}}}$ equals the union of the
closed negative orthants cogenerated by the points along the rest of
the hyperbola.

\item\label{i:half-open-hyperbola}%
The ability to omit (co)generators is even more striking upon deleting
an interval from the hyperbola, instead of a single point, to get the
downset $D_{\hbox{\tiny\ref{i:half-open-hyperbola}}}$.  Along the
deleted curve, cogenerator theory detects cogenerators of the same
shape as those for~$D_{\hbox{\tiny\ref{i:plucked-hyperbola}}}$.  Hence
$D_{\hbox{\tiny\ref{i:half-open-hyperbola}}}$ is the union of
coprincipal downsets of these shapes along the deleted curve together
with closed coprincipal downsets along the rest of the hyperbola.
However, any finite number of the coprincipal downsets along the
deleted curve can be omitted, as can be checked directly.  In fact,
any subset of them that is dense in the deleted curve can be omitted
(Theorem~\ref{t:downset=union}).  Note that a closed negative orthant
is required at the lower endpoint of the deleted curve, because the
endpoint has not been deleted, whereas the cogenerator at the upper
endpoint of the deleted curve can always be omitted because of open
negative orthants hanging from points along the hyperbola just below
it.
\end{enumerate}
\end{example}

%%%%%%%%%%%%%%%%%%%%%%%%%%%%%%%%%%%%%%%%%%%%%%%%%%%%%%%%%%%%%%%%%%%%%%%%%
\subsection{Topologies on generators and cogenerators}\label{sub:topologies}

What, in the end, does it take to express a downset in a real
polyhedral group as a union of coprincipal downsets?  This question
stems from practical considerations about birth and death in
multiparameter persistence: any generalization of bar code, functorial
or otherwise, requires knowing at least what are the analogues of the
set of left endpoints and the set of right endpoints.  The general
answer comes down to density in what is more or less a topology on the
set of cogenerators (Theorem~\ref{t:downset=union}) that arises by
viewing certain cones related to faces of the positive cone~$\cQ_+$ as
being open in~$\cQ$.  These ideas and their consequences are akin to
those in \cite{kashiwara-schapira2017}, but the main results lie
along~different~lines.

The same concept of density governs criteria for injectivity of
homomorphisms between finitely encoded modules.  In noetherian
commutative algebra, socles are essential submodules, meaning that any
nonzero submodule---the kernel of a homomorphism at hand, say---must
intersect the socle nontrivially.  Here, socles of~$\cM$ are no longer
submodules, having been constructed from auxiliary modules derived
from~$\cM$, namely upper boundary modules~$\delta\cM$.  Nonetheless,
socles are still functorial, so the question of whether they detect
injectivity remains valid, and the answer is positive
(\mbox{Theorem~\ref{t:injection}}).  Furthermore, in noetherian
commutative algebra, no proper submodule of the socle detects
injectivity.  But over real polyhedral groups, the same geometric
considerations just discussed for cogeneration of downsets have the
consequence that a subfunctor of~$\soc$ detects injectivity precisely
when it is dense in the same sense (Theorem~\ref{t:dense}).

The fact that socles are not submodules raises another relevant point:
in the closed cogenerator case there is a submodule containing the
socle element, but for open cogenerators there must still be a
submodule to witness the injectivity, because injectivity comes from
there being no actual submodule that goes to~$0$.  The cogenerator
merely indicates the presence of such a submodule, rather than being
an element of it.  Reconstructing an honest submodule in
Section~\ref{s:density} requires much of the theory in earlier
sections.  It has the consequence that a submodule is an essential
submodule of an ambient module precisely when the socle of the
submodule is dense in that of the ambient one
(Theorem~\ref{t:essential-submodule}), a result pivotal in the
characterization of coprimary modules via socles
(Theorem~\ref{t:coprimary}), the first step in minimal primary
decomposition.

%%%%%%%%%%%%%%%%%%%%%%%%%%%%%%%%%%%%%%%%%%%%%%%%%%%%%%%%%%%%%%%%%%%%%%%%%
\subsection{Primary decomposition over real polyhedral groups}\label{sub:RRprim-decomp}

Detecting injectivity of homomorphisms is vital to primary
decomposition because such decompositions can be expressed as
inclusions into direct sums of coprimary modules
(Definition~\ref{d:primDecomp'}), as already seen in
Example~\ref{e:hyperbola-PD}.  What makes such a decomposition
minimal?  Given that injectivity of homomorphisms is characterized by
arbitrary inclusion of \mbox{socles}, primary decompositions ought to
be considered minimal if their socle inclusions are dense.  However,
even though it is a~priori harder to achieve isomorphism on socles
than density of inclusion, the stronger conclusion nonetheless holds:
minimal primary decompositions exist for finitely determined modules
over real polyhedral groups with socle isomorphism as the minimality
criterion (Definition~\ref{d:minimal-primary} and
Theorem~\ref{t:minimal-primary}).  Moreover, these minimal
decompositions are canoncial for downset modules or, more generally,
subquotients of~$\kk[\cQ]$ (Theorem~\ref{t:hull-D} or
Corollary~\ref{c:hull-D}, respectively), just as they are for monomial
ideals in ordinary polynomial rings.

\begin{example}\label{e:hyperbola-min}
When $\RR^2$ is considered as a real polyhedral group, the two
redundant components in Example~\ref{e:hyperbola-PD} are not part of
the minimal primary decomposition in Theorem~\ref{t:hull-D} because
the downset~$D$ has no cogenerators, in the functorial sense of
Definition~\ref{d:soct}, along the $x$-axis or $y$-axis.  The
illustrations in Example~\ref{e:hyperbola-GD} show that $\kk[D]$ has
elements supported on the face of~$\RR^2$ that is the (positive)
$x$-axis, and it has elements supported on the face of~$\RR^2$ that is
the (positive) $y$-axis, but the socle of~$\kk[D]$ along each axis
is~$0$, as the entire boundary curve of~$\kk[D]$ is supported on the
origin.
\end{example}

In finitely generated situations, socle-minimality as in
Definition~\ref{d:minimal-primary} is equivalent, when applied to an
injective hull or irreducible decomposition, to there being a minimal
number of summands.  In contrast, minimal primary decompositions in
noetherian commutative algebra are not required to be socle-minimal in
any sense: they stipulate only minimal numbers of summands, with no
conditions on socles.  This has unfortunate consequences for
uniqueness: components for embedded (i.e., nonminimal) primes are far
from unique.  Requiring socle-minimality as in
Definition~\ref{d:minimal-primary} recovers a modicum of uniqueness
over arbitrary noetherian rings, as socles are functorial even if
primary components themselves need not be.  This tack is more commonly
taken in combinatorial commutative algebra, typically involving
objects such as monomial or binomial ideals.  In particular, the
``witnessed'' forms of minimality for mesoprimary decomposition
\cite[Definition~13.1 and Theorem~13.2]{mesoprimary} and irreducible
decomposition of binomial ideals \cite{soccular} serve as models for
the type of minimality in primary decompositions considered here.

That said, in ordinary noetherian commutative algebra a socle-minimal
primary decomposition is automatically produced by the usual existence
proof, which leverages the noetherian hypothesis to create an
irreducible decomposition.  Indeed, a primary decomposition is
socle-minimal if and only if each primary component is obtained by
gathering some of the components in a minimal irreducible
decomposition.  In lieu of truly minimal irreducible decompositions
over real polyhedral groups, whose impossibility stems from
infinitesimal geometry, one is forced to settle for the density
theorems reviewed in Section~\ref{sub:topologies}.  The existence of
primary decompositions at all, let alone ones with a mite of canonical
minimality, should therefore come as a relief.

The Matlis duals of (minimal) primary decomposition and associated
primes are (minimal) \emph{secondary decomposition} and \emph{attached
primes}, covered briefly over real and discrete polyhedral groups in
Section~\ref{s:gen-functors}.  Prior theory surrounding those is
lesser known, even to algebraists, but has existed for decades
\cite{kirby1973, macdonald-secondary-rep1973,northcott-gen-koszul1972}
(see \cite[Section~1]{sharp-secondary1976} for a brief summary of the
main concepts).  Not enough of secondary decomposition is required in
later sections to warrant going into much detail about it, especially
as Matlis duality provides such a clear dual picture that is
equivalent for finitely encoded modules over real or discrete
polyhedral groups, but some is needed.  The fly wing situation in
Example~\ref{e:toy-model-fly-wing} highlights the naturality of
allowing births to extend backward indefinitely in one or more
directions, for instance.  The unfamiliarity of secondary
decomposition and its related functors is another reason why the bulk
of the technical development over real polyhedral groups is carried
out in terms of cogenerators and socles instead of generators and
tops.

It would seem that birth and generators are in adamantine antisymmetry
with death and cogenerators, but when it comes to interactions between
the two, the symmetry of the theory is broken by the partial order
on~$\cQ$: elements in $\cQ$-modules move from birth inexorably toward
death.  It is possible to treat elements of modules over partially
ordered groups functorially, of course, since they are homomorphisms
from the monoid algebra of the positive cone, but the dual of an
element is not an element.  (It is instead a homomorphism to the
injective hull of the residue field.)  That makes generators more
complicated to deal with than cogenerators, cementing the choice to
develop the theory in terms of cogenerators.

%%%%%%%%%%%%%%%%%%%%%%%%%%%%%%%%%%%%%%%%%%%%%%%%%%%%%%%%%%%%%%%%%%%%%%%%%
\subsection{Minimal homological algebra over real polyhedral groups}\label{sub:minhom}

Minimal primary decomposition provides an opportunity to revisit the
syzygy theorem for modules over posets (see Section~\ref{sub:homalg})
in the context of real polyhedral groups
(Theorem~\ref{t:presentations-minimal}): fringe presentations, upset
presentations, and downset presentations of modules from data analysis
applications can be chosen minimal, and semialgebraic if the modules
start off that way.  This line of thought raises issues concerning
minimal upset and downset resolutions, especially termination of
minimal resolutions and length bounds on resolutions that are not
necessarily minimal, discussed more in Sections~\ref{s:min}
and~\ref{s:future}.

%%%%%%%%%%%%%%%%%%%%%%%%%%%%%%%%%%%%%%%%%%%%%%%%%%%%%%%%%%%%%%%%%%%%%%%%%
\subsection{Functorial invariants of multiparameter persistence}\label{sub:invariants}

Fringe presentation is a convenient, compact, computable
representation of multiparameter persistent homology, but it is not
canonical---not unique up to isomorphism, not necessarily minimal, not
overtly exhibiting invariants of the module.  For these things, in
ordinary persistence with one parameter, one turns to bar codes.

The bar code in single-parameter persistence has two interpretations
that diverge in the presence of multiple parameters.  One is that it
records birth and genuine death of persistent homology classes,
meaning parameters where classes go to~$0$
\cite{computingPH,crawley-boevey}.  The other is that it records birth
and elder-death, meaning parameters where homology classes join to
persisting classes born earlier.  These two interpretations yield the
same result up to isomorphism for modules over the poset~$\RR$
or~$\ZZ$.  The first interpretation generalizes functorially to
\qrcode s (Section~\ref{s:qr}), while the second generalizes
functorially to elder morphisms (Section~\ref{s:elder}).  Both have
rightful claims to be considered ``the'' functorial multiparameter
analogue of bar code, but neither exactly captures the essence of bar
code, which in single-parameter situations should functorially be
\begin{itemize}
\item%
a ``top'' vector space spanned by open, closed, and infinite left
endpoints;
\item%
a ``socle'' vector space spanned by open, closed, and infinite right
endpoints; and
\item%
a linear map from the top vector space to the socle vector space.
\end{itemize}
\qrcode s get the latter two items right but fail on the first, being
forced by functoriality to use birth spaces that are entire graded
pieces of~$\cM$ (technically:~$\mrx/\rho$ from Definition~\ref{d:mrx})
that merely surject onto the top spaces of~$\cM$.  Elder morphisms get
the first one right but fail on the second, being forced by the
filtered nature of birth spaces to use death spaces that come not
from~$\cM$ itself but rather from the elder quotients of~$\cM$.  This
divergence of \qrcode\ and elder morphism traces back to the divergent
notions of birth and death as opposed to generator and relation in
Section~\ref{sub:disc-PH}.

What is a \qrcode?  Bar codes for ordinary persistence specify births,
deaths, and matchings between them.  \qrcode s in multiparameter
situations categorify these notions as a partially ordered direct sum
of birth vector spaces (Definitions~\ref{d:birth-ZZ}
and~\ref{d:birth-RR} over discrete and real polyhedral groups,
respectively), a partially ordered direct product of death vector
spaces (Definitions~\ref{d:death-ZZ} and~\ref{d:death-RR} over
discrete and real polyhedral groups, respectively), with a functorial
linear map to relate them (Theorem~\ref{t:death-functor} and
Definition~\ref{d:qr-code}).  The death vector spaces in the product
are simply socles.  It would be desirable for the birth direct
summands to be tops, but alas there are no functorial homomorphisms
from tops to socles, so more or less entire graded pieces of~$\cM$ are
used instead; see Remark~\ref{r:qr-code} for why this is both
necessary and not so bad.  The \qrcode\ is functorial and faithful
enough that $\cM$ is easily recovered canonically from its \qrcode\
if~$\cM$ is finitely encoded (Theorem~\ref{t:recover}; see also
Remarks~\ref{r:unwind} and~\ref{r:life}).

The claim that there are no functorial homomorphisms from tops to
socles merits clarification: it means no functorial homomorphisms from
top spaces of~$\cM$ to socles of the same module~$\cM$.  The reason is
that given a generator of~$\cM$ in birth degree~$\beta$, it determines
an element of~$\top(\cM)$ in degree~$\beta$ that is only well defined
modulo elements born at parameters earlier than~$\beta$.  But elements
born earlier than~$\beta$ can die earlier, at the same parameter, or
after the given generator dies, so modifying the given generator by
adding an elder element can alter the generator's map to the death
module.  This analysis demonstrates why the birth summands of \qrcode
s are bigger than top spaces of~$\cM$, but it also points the way to
functorial elder morphisms.

The partial order on the birth poset is utilized to define, in the
previous paragraph, what it means for an element to be ``elder'' than
a degree~$\beta$ generator (Definition~\ref{d:elder-module}).  With
that precision in hand, the previous paragraph shows why there
\emph{is} a well defined linear map from the top of~$\cM$ in
degree~$\beta$ to the socle of the quotient $\cM/\cM_{\prec\beta}$
of~$\cM$ modulo its \emph{elder submodule} at~$\beta$
(Theorem~\ref{t:life-top}; see also Remark~\ref{r:elder-rule}).  That
is the \emph{elder morphism} (Definition~\ref{d:elder-morphism}).  It
functorializes the well known the elder rule from single-parameter
persistence, which says that a class dies when it ``joins with an
older class''.  The elder objects in Section~\ref{s:elder} make this
joining precise, particularly what happens the instant before an
extant class joins an elder one: the dying ember of an extant class is
a socle element of the elder quotient $\cM/\cM_{\prec\beta}$.  In
multiparameter persistence, a~single extant class can, when pushed up
along the parameter poset, give rise to dying embers over many death
parameters of the elder quotient~$\cM/\cM_{\prec\beta}$.

The failures of \qrcode s and elder morphisms to satisfy all three
desiderata are solely multiparameter phenomena; in one parameter,
functorial \qrcode s and elder morphisms can be combined to succeed on
all three items simultaneously.  The result is the functorial bar code
in Theorem~\ref{t:elder-projection}, whose minor cost is that the
socle in death degree~$\alpha$ must---by a combination of the two
forces already encountered, namely functoriality and the filtered
nature of birth spaces---be replaced by an associated graded vector
space.  In other words, the vector space spanned by all of the closed
right endpoints over $a \in \RR$ is functorial, as is the the vector
space spanned by all of the open right endpoints over~$a$, but their
subspaces spanned by right endpoints of intervals born
% (either openly or closedly)
over~$b$ is only functorial modulo the right endpoints of intervals
born strictly earlier than~$b$.

The idea of the proof is enlightening.  An extant element suffers
elder-death if pushing it infinitesimally forward along the real line
lands the extant element in the elder submodule.  Subtracting the
image elder element (``elder projection'') from the extant one yields
an extant element that suffers genuine death and not merely
elder-death.  This proof works over~$\ZZ$ but fails in the presence of
multiple parameters for a telling reason: when the extant element
suffers elder-death in independent directions, it can land on
different elder elements when pushed up in different directions, so
there is no way to modify the extant element to die genuinely in all
directions simultaneously.

\begin{example}\label{e:elder}
Let $\cM$ be the $\ZZ^2$-module generated by $\{e_x,e_y,e_{xy}\}$,
where $\deg e_m = \deg m \in \NN^2$, with relations $\<xy e_x - x
e_{xy}, xy e_y - y e_{xy}\>$.  If $\beta = \big[\twoline 11\big]$ then
$\cM_{\prec\beta} = \<e_x,e_y\>$,~so
$$%
  \cM/\cM_{\prec\beta}
  =
  \cM/\<e_x,e_y\>
  =
  \<e_{xy}\>/\<x e_{xy},y e_{xy}\>
  \cong
  \kk_{xy}
$$
is the indicator subquotient for the point $\big[\twoline 11\big] \in
\ZZ^2$.  Therefore $\cM/\cM_{\prec\beta}$ is its own socle, and it is
supported on the face~$\{\0\}$.  But $\cM$ itself has no socle
elements along~$\{\0\}$ of degree~$\big[\twoline 11\big]$.  (In fact,
$\cM$ has no elements supported on~$\{\0\}$, regardless of degree.)
Indeed, every element of degree~$\big[\twoline 11\big]$ is expressible
as $\alpha y e_x + \beta x e_y + \gamma e_{xy}$.  If this element lies
in $\socp[\{\0\}]\cM$ then multiplying it by~$xy$ forces $\alpha +
\beta = \gamma$, but multiplying it by~$x$ forces $\beta = 0$, while
multiplying it by~$y$ forces $\alpha = 0$.  The relations of~$\cM$ are
specifically designed so that any generator of~$\cM$ with degree
$\big[\twoline 11\big]$ maps to different elements of
$\cM_{\prec\beta} = \<e_x,e_y\>$ upon multiplication by~$x$ and
by~$y$, \mbox{because $xe_{xy} = xye_x$ whereas $ye_{xy} = xye_y$}.
\end{example}

This example brings up one final point: in the presence of multiple
parameters, generically it is impossible to peel off summands using an
elder rule.  This fuels an intuition that the Krull--Schmidt theorem
on existence of decompositions as direct sums of indecomposables
carries little power for persistence over posets not of finite
representation type.
% (e.g., ADE quivers).
Heuristically, if the filtered topological space is connected or has
otherwise trivial topology, then although lots of interesting things
may happen along the way to triviality, the infinitely persisting
class must lie in some summand.  That doesn't a~priori mean
anything---it is true in one real parameter, after all---but not being
able to peel off summands by elder rule obstructs nontrivial direct
sum decomposition.

\addtocontents{toc}{\protect\setcounter{tocdepth}{1}}%%%%%%%%%%%%%%%%%%%%
\subsection*{Advice for the reader}

The heart of this paper is pictorial algebra.  Nearly every statement
and proof was written with visions of upsets and downsets and cones
and faces at the forefront.  It is likely best understood with that
kind of geometry in mind.

\addtocontents{toc}{\protect\setcounter{tocdepth}{2}}%%%%%%%%%%%%%%%%%%%%

%%%%%%%%%%%%%%%%%%%%%%%%%%%%%%%%%%%%%%%%%%%%%%%%%%%%%%%%%%%%%%%%%%%%%%%%%
\section{Encoding poset modules}\label{s:encoding}%%%%%%%%%%%%%%%%%%%%%%%
%%%%%%%%%%%%%%%%%%%%%%%%%%%%%%%%%%%%%%%%%%%%%%%%%%%%%%%%%%%%%%%%%%%%%%%%%

%%%%%%%%%%%%%%%%%%%%%%%%%%%%%%%%%%%%%%%%%%%%%%%%%%%%%%%%%%%%%%%%%%%%%%%%%
\subsection{Multiparameter persistence}\label{sub:multiparam}\mbox{}%%%%%

\begin{defn}\label{d:poset-module}
Let $\cQ$ be a partially ordered set (\emph{poset}) and~$\preceq$ its
partial order.  A \emph{module over~$\cQ$} (or a \emph{$\cQ$-module})
is
\begin{itemize}
\item%
a $Q$-graded vector space $\cM = \bigoplus_{q\in Q} \cM_q$ with
\item%
a homomorphism $\cM_q \to \cM_{q'}$ whenever $q \preceq q'$ in~$Q$
such that
\item%
$\cM_q \to \cM_{q''}$ equals the composite $\cM_q \to \cM_{q'} \to
\cM_{q''}$ whenever $q \preceq q' \preceq q''$.
\end{itemize}
A \emph{homomorphism} $\cM \to \cN$ of $\cQ$-modules is a
degree-preserving linear map, or equivalently a collection of vector
space homomorphisms $\cM_q \to \cN_q$, that commute with the structure
homomorphisms $\cM_q \to \cM_{q'}$ and $\cN_q \to \cN_{q'}$.
\end{defn}

The last bulleted item is \emph{commutativity}: it reflects that
inclusions of subspaces induce functorial homology morphisms in the
motivating examples of $\cQ$-modules.

\begin{defn}\label{d:multifiltration}
Let $X$ be a topological space and $\cQ$ a poset.
\begin{enumerate}
\item\label{i:filtration}%
A \emph{filtration of~$X$ indexed by~$\cQ$} is a choice of subspace
$X_q \subseteq X$ for each $q \in \cQ$ such that $X_q \subseteq
X_{q'}$ whenever $q \preceq q'$.

\item\label{i:PH}%
The \emph{$i^\mathrm{th}$ persistent homology} of the \emph{filtered
space}~$X$ is the associated homology module, meaning the $\cQ$-module
$\bigoplus_{q \in \cQ} H_i X_q$.
\end{enumerate}
\end{defn}

\begin{conv}\label{c:kk}
The homology here could be taken over an arbitrary ring, but for
simplicity it is assumed throughout that homology is taken with
coefficients in a field~$\kk$.
\end{conv}

\begin{remark}\label{r:curry}
There are a number of abstract, equivalent ways to phrase
Definition~\ref{d:multifiltration}.  For example, a filtration is a
functor from $\cP$ to the category $\mathcal{S}$ of subspaces of~$X$,
or a natural transformation from the category~$\cP$ to~$\mathcal{S}$,
or an $\mathcal{S}$-valued sheaf on the topological space~$\cP$, where
a base for the topology is the set of principal dual order~ideals.
For background on and applications of many of these perspectives, see
Curry's dissertation \cite{curry-thesis}, particularly \S4.2 there.
\end{remark}

\begin{example}\label{e:RR+}
A \emph{real multifiltration} of~$X$ is a filtration indexed
by~$\RR^n$, with its partial order by coordinatewise comparison.
Example~\ref{e:fly-wing-filtration} is a real multifiltration of $X =
\RR^2$ with $n = 2$.  The monoid $\RR_+^n \subset \RR^n$ of
nonnegative real vectors under addition has monoid algebra
$\kk[\RR_+^n]$ over the field~$\kk$, a ``polynomial'' ring whose
elements are (finite) linear combinations of monomials $\xx^\aa$ with
real, nonnegative exponent vectors $\aa = (a_1,\dots,a_n) \in
\RR_+^n$.  It contains the usual polynomial ring $\kk[\NN^n]$ as a
$\kk$-subalgebra.  The persistent homology of a real $n$-filtered
space~$X$ is an $\RR^n$-graded module over\/~$\kk[\RR_+^n]$, which is
the same thing as an $\RR^n$-module.
\end{example}

%%%%%%%%%%%%%%%%%%%%%%%%%%%%%%%%%%%%%%%%%%%%%%%%%%%%%%%%%%%%%%%%%%%%%%%%%
\subsection{Finite encoding}\label{sub:encoding}\mbox{}%%%%%%%%%%%%%%%%%%

\noindent
For practical purposes, it is important to encode finiteness
properties of poset-graded vector spaces that at least provide hope
for effective computation.

\begin{defn}\label{d:encoding}
Fix a poset~$\cQ$.  An \emph{encoding} of a $\cQ$-module $\cM$ by a
poset~$\cP$ is a poset morphism $\pi: \cQ \to \cP$ together with a
$\cP$-module $\cH$ such that $\cM \cong \pi^*\cH =
\bigoplus_{q\in\cQ}H_{\pi(q)}$, the \emph{pullback of~$\cH$
along~$\pi$}, which is naturally a $\cQ$-module.  The encoding is
\emph{finite} if
\begin{enumerate}
\item%
the poset $\cP$ is finite, and
\item%
the vector space $H_p$ has finite dimension for all $p \in \cP$.
\end{enumerate}
\end{defn}

\begin{remark}\label{r:lurie}
Encoding of a $\cQ$-module~$\cM$ by a poset morphism to~$\cP$ is
equivalent to viewing~$\cM$ as a sheaf on~$\cP$ that is constructible
in the Alexandrov topology, in the sense defined independently by
Lurie \cite[Definitions~A.5.1 and~A.5.2]{lurie2017}.
% on page 1064
The focus here is on the characterizations and consequences of the
finitely encoded condition rather than on arbitrary poset encodings.
\end{remark}

\begin{example}\label{e:convex-projection}
Take $\cQ = \ZZ^n$ and $\cP = \NN^n$.  The \emph{convex projection}
$\ZZ^n \to \NN^n$ sets to~$0$ every negative coordinate.  The pullback
under convex projection is the \v Cech hull
\cite[Definition~2.7]{alexdual}.  More generally, suppose $\aa \preceq
\bb$ in~$\ZZ^n$.  The interval $[\aa,\bb] \subseteq \ZZ^n$ is a box
(rectangular parallelepiped) with lower corner at~$\aa$ and upper
corner at~$\bb$.  The \emph{convex projection} $\pi: \ZZ^n \to
[\aa,\bb]$ takes every point in~$\ZZ^n$ to its closest point in the
box.  A $\ZZ^n$-module is \emph{finitely determined} if it is finitely
encoded by~$\pi$.
\end{example}

Effectively computing a finite encoding of a real multifiltered space
requires keeping track of the fibers of the morphism $\RR^n \to \cP$
in addition to the data on~$\cP$.  The fact that applications of
persistent homology often arise from metric considerations, which are
semialgebraic in nature, suggests the following condition for
algorithmic developments.

\begin{defn}\label{d:alg-finite}
Fix a partially ordered real vector space~$\cQ$ of finite dimension
(see Definition~\ref{d:pogroup}, or just take $\cQ = \RR^n$ for now).
A finite encoding \mbox{$\pi: \cQ \to \cP$} is \emph{semialgebraic} if
its fibers are real semialgebraic varieties.  A module over~$\cQ$ is
\emph{semialgebraic} if it has a semialgebraic encoding.
\end{defn}

\begin{lemma}\label{l:abelian-category}
For any poset~$\cQ$, the category of finitely encoded $\cQ$-modules
is~abelian.  If $\cQ$ is a partially ordered real vector space of
finite dimension, then the category of semialgebraic modules
is~abelian.
\end{lemma}
\begin{proof}
The category in question is a subcategory of the category of
$\cQ$-modules.  It suffices to prove that it is a full subcategory,
meaning that the kernel and cokernel of any homomorphism $\cM_1 \to \cM_2$
of finitely encoded $\cQ$-modules is finitely encoded, as is any finite
direct sum of finitely encoded $\cQ$-modules.  For all of these it
suffices to show that any two finitely encoded $\cQ$-modules can be
finitely encoded by a single poset morphism.  To that end, suppose
that $\cM_i$ is finitely encoded by $\pi_i: \cQ \to P_i$ for $i = 1,2$.
Then $\cM_1$ and~$\cM_2$ are both finitely encoded by $\pi_1 \times \pi_2:
\cQ \to P_1 \times P_2$.  The semialgebraic case has the same proof.
\end{proof}

%%%%%%%%%%%%%%%%%%%%%%%%%%%%%%%%%%%%%%%%%%%%%%%%%%%%%%%%%%%%%%%%%%%%%%%%%
\subsection{Upsets, downsets, and indicator modules}\label{sub:upsets}\mbox{}

\begin{example}\label{e:indicator}
Given a poset~$\cQ$, the vector space $\kk[\cQ] = \bigoplus_{q\in\cQ}
\kk$ that assigns $\kk$ to every point of~$\cQ$ is a $\cQ$-module.  It
is encoded by the morphism from $\cQ$ to the trivial poset~$\cP$ with
one point and vector space $\cH = \kk$.  This leads to less trivial
examples in a way that works for an arbitrary poset~$\cQ$, although
our main example is $\cQ =\RR^n$.
\begin{enumerate}
\item%
An \emph{upset} (also called a \emph{dual order ideal}) $U \subseteq
\cQ$, meaning a subset closed under going upward in~$\cQ$ (so $U +
\RR_+^n = U$, when $\cQ = \RR^n$) determines an \emph{indicator
submodule} or \emph{upset module} $\kk[U] \subseteq \kk[\cQ]$.
\item%
Dually, a \emph{downset} (also called an \emph{order ideal}) $D
\subseteq \cQ$, meaning a subset closed under going downward in~$\cQ$
(so $D - \RR_+^n = D$, when $\cQ = \RR^n$) determines an
\emph{indicator quotient module} or \emph{downset module} $\kk[\cQ]
\onto \kk[D]$.
\end{enumerate}
When $\cQ = \RR^n$, an indicator module of either sort is
semialgebraic if the corresponding upset or downset is a semialgebraic
subset of~$\RR^n$.
\end{example}

\begin{remark}\label{r:indicator}
Indicator submodules $\kk[U]$ and quotient modules $\kk[D]$ are
$\cQ$-modules, not merely $U$-modules or $D$-modules, by setting the
graded components indexed by elements outside of the ideals to~$0$.
It is only by viewing indicator modules as $\cQ$-modules that they are
forced to be submodules or quotients, respectively.  For relations
between these notions and those in Remark~\ref{r:curry}, again see
Curry's thesis \cite{curry-thesis}.  For example, upsets form the open
sets in the topology from Remark~\ref{r:curry}.
\end{remark}

\begin{example}\label{e:melting}
Ising crystals at zero temperature, with polygonal boundary conditions
and fixed mesh size, are semialgebraic upsets in~$\RR^n$.  That much
is by definition: fixing a mesh size means that the crystals in
question are (staircase surfaces of finitely generated) monomial
ideals in $n$ variables.  Remarkably, such crystals remain
semialgebraic in the limit of zero mesh size; see \cite{okounkov16}
for an exposition and references.
\end{example}

\begin{example}\label{e:asw}
Monomial ideals in polynomial rings with real exponents, which
correspond to upsets in $\RR^n_+$, are considered in
\cite{andersen--sather-wagstaff2015}, including aspects of primality,
irreducible decomposition, and Krull dimension.  Upsets in $\RR^n$ are
also considered in \cite{madden-mcguire2015}, where the combinatorics
of their lower boundaries, and topology of related simplicial
complexes, are investigated in cases with locally finite generating
sets.
\end{example}

For future reference, here are some basic facts about upset and
downset modules.

\begin{defn}\label{d:connected-poset}
A poset~$\cQ$ is
\begin{enumerate}
\item%
\emph{connected} if for every pair of elements $q,q' \in \cP$ there is
a sequence $q = q_0 \preceq q'_0 \succeq q_1 \preceq q'_1 \succeq
\dots \succeq q_k \preceq q'_k = q'$ in~$\cQ$;

\item%
\emph{upper-connected} if every pair of elements in~$\cQ$ has an
upper bound in~$\cQ$;

\item%
\emph{lower-connected} if every pair of elements in~$\cQ$ has a
lower bound in~$\cQ$; and

\item%
\emph{strongly connected} if $\cQ$ is upper-connected and
lower-connected.
\end{enumerate}
\end{defn}

\begin{example}\label{e:connected-poset}
$\RR^n$ is strongly connected.  The same is true of any partially
ordered abelian group (see Section~\ref{sub:polyhedral} for basic
theory of those posets).
\end{example}

\begin{example}\label{e:bounded-poset}
A poset~$\cQ$ is upper-connected if (but not only if) it has a maximum
element---one that is preceded by every element of~$\cQ$.  Similarly,
$\cQ$ is lower-connected if it has a minimum element---one that
precedes every element of~$\cQ$.
\end{example}

\begin{defn}\label{d:connected}
For an upset~$U$ and a downset~$D$ in a poset~$\cQ$, write $U \preceq
D$ if $U$ and~$D$ have nonempty intersection: $U \cap D \neq
\nothing$.
\end{defn}

\begin{lemma}\label{l:U->D}
Fix a poset~$\cQ$.
\begin{enumerate}
\item\label{i:U->D}%
For an upset~$U$ and a downset~$D$, a nonzero homomorphism $\kk[U] \to
\kk[D]$ of $\cQ$-modules exists if and only if $U \preceq D$.

\item\label{i:kk}%
$\Hom_\cQ(\kk[U], \kk[D]) = \kk$ if $U \preceq D$ and either $U$ is
lower-connected as a subposet of~$\cQ$ or $D$ is upper-connected as a
subposet of~$\cQ$.

\item\label{i:U}%
If $U$ and~$U'$ are upsets and $\cQ$ is upper-connected, then
$\Hom_\cQ(\kk[U'],\kk[U]) = \kk$ if $U' \subseteq U$ and~$0$
otherwise.

\item\label{i:D}%
If $D$ and~$D'$ are upsets and $\cQ$ is lower-connected, then
$\Hom_\cQ(\kk[D],\kk[D']) = \kk$ if $D \supseteq D'$ and~$0$
otherwise.
\end{enumerate}
\end{lemma}
\begin{proof}
The first claim is immediate from the definitions.  For the second,
compare the action of $\phi: \kk[U] \to \kk[D]$ on the copy of~$\kk$
in any two degrees $q$ and~$q'$ with the action of~$\phi$ on the copy
of~$\kk$ in any degree that is a lower bound for~$q$ and~$q'$.  The
proof of the last claim is the same, after replacing $D$ with~$D'$ and
$U$ with~$D$.  The proof of the remaining claim is dual to that of the
last claim.
\end{proof}

\begin{example}\label{e:disconnected-homomorphism}
Consider the poset $\NN^2$, the upset $U = \NN^2 \minus \{\0\}$, and
the downset $D$ consisting of the origin and the two standard basis
vectors.  Then $\kk[U] = \mm = \<x,y\>$ is the graded maximal ideal of
$\kk[\NN^2] = \kk[x,y]$ and $\kk[D] = \kk[\NN^2]/\mm^2$.  Now
calculate
$$%
  \Hom_{\NN^2}(\kk[U],\kk[D])
  =
  \Hom_{\NN^2}(\mm,\kk[\NN^2]/\mm^2)
  =
  \kk^2,
$$
a vector space of dimension~$2$: one basis vector preserves the
monomial~$x$ while killing the monomial~$y$, and the other basis
vector preserves~$y$ while killing~$x$.  In general,
$\Hom_\cQ(\kk[U],\kk[D])$ is the vector space spanned by the set
$\pi_0(U \cap D)$ of connected components
(Definition~\ref{d:connected-poset}) of $U \cap D$ as a subposet
of~$\cQ$.  This proliferation of homomorphisms is undesirable for both
our computational and theoretical purposes.  Indeed, for an extreme
example, consider the case over~$\RR^2$ in which $U$ is the closed
half-plane above the antidiagonal line $y = -x$ and where $D = -U$, so
that $U \cap D$ is totally disconnected: $\pi_0(U \cap D) = U \cap D$.
This example motivates the following concept.
\end{example}

\begin{defn}\label{d:connected-homomorphism}
For an upset~$U$ and a downset~$D$ in a poset~$\cQ$ with $U \preceq
D$, a homomorphism $\phi: \kk[U] \to \kk[D]$ is \emph{connected} if
there is a scalar $\lambda \in \kk$ such that $\phi$ acts as
multiplication by~$\lambda$ on the copy of~$\kk$ in degree $q$ for all
$q \in U \cap D$.
\end{defn}

\begin{remark}\label{r:connected-homomorphism}
Equivalently, $\kk[U] \to \kk[D]$ is connected if it factors
through~$\kk[Q]$.
\end{remark}

%%%%%%%%%%%%%%%%%%%%%%%%%%%%%%%%%%%%%%%%%%%%%%%%%%%%%%%%%%%%%%%%%%%%%%%%%
\subsection{Isotypic subdivision}\label{sub:isotypic}\mbox{}%%%%%%%%%%%%%

\noindent
The main result of Section~\ref{s:encoding}, namely
Theorem~\ref{t:isotypic}, says that finite encodings always exist for
$\cQ$-modules with finitely many domains of constancy, defined as
follows.

\begin{defn}\label{d:subdivide}
Fix a $\cQ$-module~$\cM$.  The \emph{isotypic subdivision} of~$\cQ$
\emph{induced by~$\cM$} is the equivalence relation $\cQ/\cM$
generated by the relation that sets $\aa \sim \bb$ whenever $\aa
\preceq \bb$ in~$\cQ$ and the induced homomorphism $\cM_\aa \to
\cM_\bb$ is an isomorphism.  The equivalence classes of the isotypic
subdivision are called \emph{regions}, or \emph{$\cM$-regions} if the
context does not make~it~clear.
\end{defn}

\begin{example}\label{e:subdivide}
The quotient map $\cQ \to \cQ/\cM$ of sets need not be a morphism of
posets.  Indeed, there is no natural way to impose a poset structure
on the set of isotypic regions.  Take, for example, $\cQ = \RR^2$ and
$\cM = \kk_\0 \oplus \kk[\RR^2]$, where $\kk_\0$ is the $\RR^2$-module
whose only nonzero component is at the origin, where it is a vector
space of dimension~$1$.  This module~$\cM$ induces only two isotypic
regions, namely the origin and its complement.  Neither of the two
isotypic regions has a stronger claim to precede the other, but at the
same time it would be difficult to justify forcing the isotypic
regions to be incomparable.
\end{example}

\begin{example}\label{e:wing-subdivision}
Example~\ref{e:toy-model-fly-wing} shows an isotypic subdivision of
$\RR^2$ which happens to form a poset and therefore produces an
encoding.
\end{example}

%%%%%%%%%%%%%%%%%%%%%%%%%%%%%%%%%%%%%%%%%%%%%%%%%%%%%%%%%%%%%%%%%%%%%%%%%
\subsection{Uptight posets}\label{sub:uptight}\mbox{}%%%%%%%%%%%%%%%%%%%%

\noindent
Constructing finite encodings relies on combinatorics that works for
all posets.

\begin{defn}\label{d:uptight}
Fix a poset~$\cQ$ and a set $\Upsilon$ of upsets.  For each poset
element $\aa \in \cQ$, let $\Upsilon_\aa \subseteq \Upsilon$ be the
set of upsets from~$\Upsilon$ that contain~$\aa$.  Two poset elements
$\aa,\bb \in \cQ$ lie in the same \emph{uptight region} if
$\Upsilon_\aa = \Upsilon_\bb$.
\end{defn}

\begin{remark}\label{r:iso-uptight}
Every uptight region is the intersection of a single upset (not
necessarily one of the ones in~$\Upsilon$) with a single downset.
Indeed, the intersection of any family of upsets is an upset, the
complement of an upset is a downset, and the intersection of any
family of downsets is a downset.  Hence the uptight region
containing~$\aa$ equals $\bigl(\bigcap_{U \in \Upsilon_\aa} U\bigr) \cap
\bigl(\bigcap_{U \not\in \Upsilon_\aa} \ol U\bigr)$, with the first
intersection being an upset and the second being a~downset.
\end{remark}

\begin{prop}\label{p:posetQuotient}
In the situation of Definition~\ref{d:uptight}, the uptight regions
form a poset in which $A \preceq B$ whenever $\aa \preceq \bb$ for
some $\aa \in A$ and $\bb \in B$.
\end{prop}
\begin{proof}
The stipulated relation on the set of uptight regions is
\begin{itemize}
\item%
reflexive because $\aa \preceq \aa$ for any element~$\aa$ in any
uptight region~$A$;
\item%
transitive because the relation on~$\cQ$ is transitive; and
\item%
antisymmetric for the following reason.  Suppose uptight regions $A$
and~$B$ satisfy $A \preceq B$ and $B \preceq A$, so $\aa \preceq \bb$
and $\bb' \preceq \aa'$ for some $\aa,\aa' \in A$ and $\bb,\bb' \in
B$.  Then $\Upsilon_\aa \subseteq \Upsilon_\bb = \Upsilon_{\bb'}
\subseteq \Upsilon_{\aa'} = \Upsilon_\aa$, with the containments
following from $\aa \preceq \bb$ and $\bb' \preceq \aa'$.  Therefore
$\Upsilon_\aa = \Upsilon_\bb$, so $A = B$.\qedhere
\end{itemize}
\end{proof}

\begin{defn}\label{d:iso-uptight}
Fix a $\cQ$-module~$\cM$.  An \emph{isotypic upset} of~$\cQ$
\emph{induced by~$\cM$} is either
\begin{enumerate}
\item%
an upset $U_I$ generated by an isotypic region~$I$ of~$\cM$, or
\item%
the complement of a downset $D_I$ cogenerated by an isotypic
region~$I$ of~$\cM$.
\end{enumerate}
Write $\cP_\cM$ for the poset of uptight regions determined by the
set~$\Upsilon_\cM$ of isotypic upsets.
\end{defn}

\begin{prop}\label{p:iso-uptight}
Each uptight region determined by the set~$\Upsilon_\cM$ of isotypic
upsets is contained a single $\cM$-isotypic region.
\end{prop}
\begin{proof}
Suppose that $A$ is an uptight region that contains points from
isotypic regions $I$ and~$J$.  Any point in $I \cap A$ witnesses the
containments $A \subseteq D_I$ and $A \subseteq U_I$ of~$A$ inside the
isotypic upset and downset generated and cogenerated by~$I$.  Any
point $\jj \in J \cap A$ is therefore sandwiched between elements
$\ii, \ii' \in I$, so $\ii \preceq \jj \preceq \ii'$, because $\jj \in
U_I$ (for~$\ii$) and $\jj \in D_I$ (for~$\ii'$).  By symmetry,
switching $I$ and~$J$, there exists $\jj' \in J$ with $\ii' \preceq
\jj'$.  The sequence $\cM_\ii \to \cM_\jj \to \cM_{\ii'} \to
\cM_{\jj'}$ of $\cQ$-module structure homomorphisms induces
isomorphisms $\cM_\ii \to \cM_{\ii'}$ and $\cM_\jj \to \cM_{\jj'}$ by
definition of isotypic region.  Elementary homological algebra implies
that $\cM_\ii \to \cM_\jj$ is an isomorphism, so $I = J$, as desired.
\end{proof}

\begin{example}\label{e:iso-uptight}
Proposition~\ref{p:iso-uptight} does not claim that $I = U_I \cap
D_I$, and in fact that claim is often not true.  Consider $\cQ =
\RR^2$ and $\cM = \kk_\0 \oplus \kk[\RR^2]$, as in
Example~\ref{e:subdivide}, and take $I = \RR^2 \minus \{\0\}$.  Then
$U_I = D_I = \RR^2$, so $U_I \cap D_I$ contains the other isotypic
region $J = \{\0\}$.  The uptight poset $\cP_\cM$ has precisely four
elements:
\begin{enumerate}
\item%
the origin $\{\0\} = U_J \cap D_J$;
\item%
the complement $U_J \minus \{\0\}$ of the origin in $U_J$;
\item%
the complement $D_J \minus \{\0\}$ of the origin in $D_J$; and
\item%
the points $\RR^2 \minus (U_J \cup D_J)$ lying only in~$I$ and in
neither $U_J$ nor~$D_J$.
\end{enumerate}
Oddly, uptight region~4 has two connected components, the second and
fourth quadrants $A$ and~$B$, that are incomparable: any chain of
relations from Definition~\ref{d:subdivide} that realizes the
equivalence $\aa \sim \bb$ for $\aa \in A$ and~$\bb \in B$ must pass
through the positive quadrant or the negative quadrant, each of which
accidentally becomes comparable to the other isotypic region~$J$ and
hence lies in a different uptight region.
\end{example}

\begin{defn}\label{d:Q-finite}
For any poset~$\cQ$, a module $\cM$ is \emph{$\cQ$-finite} if the
vector spaces $\cM_\aa$ have finite dimension for all $\aa \in \cQ$.
\end{defn}

$\cQ$-finiteness is one of two requirements for the existence of a
finite encoding.  The other stipulates that $\cQ$ be decomposable in a
well behaved manner with respect to~$\cM$, and this is implied by
% but not equivalent to (see Example~\ref{e:antidiagonal})
having only finitely many isotypic regions.

\begin{thm}\label{t:isotypic}
A $\cQ$-finite module~$\cM$ over a poset~$\cQ$ admits a finite
encoding if the isotypic subdivision of~$\cQ$ induced by~$\cM$ is
finite.  Moreover, in that case $\cM$ has a finite encoding by its
uptight poset~$\cP_\cM$.  This finite encoding is semialgebraic if
$\cQ = \RR^n$ and every isotypic region is semialgebraic.
\end{thm}
\begin{proof}
The first sentence follows from the second.  Assume that the
$\cM$-isotypic subdivision is finite.  $\cP_\cM$ is a finite set
because the number of uptight regions is bounded above by
$2^{2|\cQ/\cM|}$: every element of~$\cQ$ lies inside or outside of
each isotypic upset and isotypic downset.  The set~$\cP_\cM$ is a
poset by Proposition~\ref{p:posetQuotient}, and hence the set map $\cQ
\to \cP_\cM$ is a poset morphism by definition of the partial order
on~$\cP_\cM$.  The module~$\cM$ is constant on the fibers of this
quotient morphism, in the sense that the vector spaces $\cM_\aa$ are
naturally isomorphic for all $\aa$ in any single uptight region $A \in
\cP_\cM$, by Proposition~\ref{p:iso-uptight}.  For each such
region~$A$, choose a representative $\aa(A) \in A$, and define $H_A =
\cM_{\aa(A)}$.  The vector spaces $H_A$ for $A \in \cP_\cM$ naturally
comprise a $\cP_\cM$-module by virtue of the structure maps for~$\cM$,
and the pullback of~$H_A$ to~$\cQ$ is isomorphic
to~$\cM$~by~construction.

For the final claim, the Minkowski sums $I + \RR_+^n$ and $I -
\RR_+^n$ are semialgebraic whenever $I$ is a semialgebraic isotypic
region, and finite intersections of these and their complements are
semialgebraic.
\end{proof}

By Theorem~\ref{t:isotypic}, the notion of finite encoding extends the
notion of finiteness encapsulated by finite isotypic subdivision.  The
next example demonstrates that this is a proper extension.

\begin{example}\label{e:antidiagonal}
Let $\cM$ be the $\RR^2$-module that has $\cM_\aa = 0$ for all $\aa
\in \RR^2$ except for those on the antidiagonal line spanned by
$\left[\twoline 1{-1}\right] \in \RR^2$, where $\cM_\aa = \kk$.  There
is only one such $\RR^2$-module because all of the degrees of nonzero
graded pieces of~$\cM$ are incomparable, so all of the structure
homomorphisms $\cM_\aa \to \cM_\bb$ are zero.  Thus every point on the
line is a singleton isotypic region.  However, these uncountably many
isotypic regions can all be gathered together: $\cM$ has a poset
encoding by the chain with three elements, where the fiber over the
middle element is the antidiagonal line, and the fibers over the top
and bottom elements are the open half-spaces above and below the line,
respectively.  (This is the uptight poset for the two upsets that are
the closed and open half-spaces bounded below by the antidiagonal.)
In contrast, using the diagonal line spanned by $\left[\twoline
11\right] \in \RR^2$ instead of the antidiagonal line yields a module
with no finite encoding; see Example~\ref{e:diagonal}.
\end{example}

The direction of the line in Example~\ref{e:antidiagonal} is
important: antidiagonal lines, whose points form an antichain
in~$\RR^2$, behave radically differently than diagonal lines.

\begin{example}\label{e:diagonal}
Let $\cM$ be an $\RR^2$-module with $\cM_\aa = \kk$ whenever $\aa$
lies in the closed diagonal strip between the lines of slope~$1$
passing through any pair of points.  The structure homomorphisms
$\cM_\aa \to \cM_\bb$ could all be zero, for instance, or some of them
could be nonzero.  But the length $|\aa - \bb|$ of any nonzero such
homomorphism must in any case be bounded above by the Manhattan (i.e.,
$\ell^\infty$) distance between the two points, since every longer
structure homomorphism factors through a sequence that exits and
re-enters the strip.  In particular, the structure homomorphism
between any pair of points on the upper boundary line of the strip is
zero because it factors through a homomorphism that points upward
first; therefore such pairs of points lie in distinct isotypic
regions.  The same conclusion holds for pairs of points on the lower
boundary line of the strip.  Consequently, in any encoding of~$\cM$,
the poset must be uncountable.
\end{example}

%%%%%%%%%%%%%%%%%%%%%%%%%%%%%%%%%%%%%%%%%%%%%%%%%%%%%%%%%%%%%%%%%%%%%%%%%
\section{Primary decomposition over partially ordered groups}\label{s:decomp}
%%%%%%%%%%%%%%%%%%%%%%%%%%%%%%%%%%%%%%%%%%%%%%%%%%%%%%%%%%%%%%%%%%%%%%%%%

%%%%%%%%%%%%%%%%%%%%%%%%%%%%%%%%%%%%%%%%%%%%%%%%%%%%%%%%%%%%%%%%%%%%%%%%%
\subsection{Polyhedral partially ordered groups}\label{sub:polyhedral}\mbox{}

\noindent
The next definition, along with elementary foundations surrounding it,
can be found in Goodearl's book \cite[Chapter~1]{goodearl86}.

\begin{defn}\label{d:pogroup}
An abelian group~$\cQ$ is \emph{partially ordered} if it is generated by
a submonoid~$\cQ_+$, called the \emph{positive cone}, that has trivial
unit group.  The partial order is: $q \preceq q' \iff q' - q \in \cQ_+$.
\end{defn}

\begin{example}\label{e:ZZn-pogroup}
The finitely generated free abelian group $\cQ = \ZZ^n$ can be
partially ordered with any positive cone~$\cQ_+$, polyhedral or
otherwise, although the free commutative monoid $\cQ_+ = \NN^n$ of
integer vectors with nonnegative coordinates is most common and serves
as a well behaved, well known foundational case to which substantial
parts of the general theory reduce.
\end{example}

\begin{example}\label{e:pogroup}
The group $\cQ = \RR^n$ can be partially ordered with any positive
cone~$\cQ_+$, polyhedral or otherwise, although the orthant $\cQ_+ =
\RR_+^n$ of vectors with nonnegative coordinates is most useful for
multiparameter persistence.
\end{example}

The following allows the free use of the language of either
$\cQ$-modules or $\cQ$-graded $\kk[\cQ_+]$-modules, as appropriate to
the context.

\begin{lemma}\label{l:Q-graded}
A module over a partially ordered abelian group~$\cQ$ is the same
thing as a $\cQ$-graded\/ module over~$\kk[\cQ_+]$.\hfill\qed
\end{lemma}

\begin{example}\label{e:ZZn-graded}
When $\cQ = \ZZ^n$ and $\cQ_+ = \NN^n$, the relevant monoid
algebra is the polynomial ring $\kk[\NN^n] = \kk[\xx]$, where $\xx =
x_1,\dots,x_n$ is a sequence of $n$ commuting~variables.
\end{example}

Primary decomposition of $\cQ$-modules depends on certain finiteness
conditions.  In ordinary commutative algebra, where $\cQ = \ZZ^n$, the
finiteness comes from~$\cQ_+$, which is assumed to be finitely
generated (so it is an \emph{affine semigroup}).  This condition
implies that finitely generated $\cQ$-modules are noetherian: every
increasing chain of submodules stabilizes.  Primary decomposition is
then usually derived as a special case of the theory for finitely
generated modules over noetherian rings.  But the noetherian condition
is stronger than necessary: in the presence of a finitely encoded
hypothesis, it suffices for the positive cone to have finitely many
faces, in the following sense.  To our knowledge, the notion of
polyhedral partially ordered group is new, as there is no existing
literature on primary decomposition in this setting.

\begin{defn}\label{d:face}
A \emph{face} of the positive cone~$\cQ_+$ of a partially ordered
abelian group~$\cQ$ is a submonoid $\sigma \subseteq \cQ_+$ such that
$\cQ_+ \minus \sigma$ is an ideal of the monoid~$\cQ_+$.  Sometimes it
is simpler to say that~$\sigma$ is a \emph{face} of~$\cQ$.  Call~$\cQ$
\emph{polyhedral} if it has only finitely many~faces.
\end{defn}

\begin{example}\label{e:face}
The partially ordered groups in Example~\ref{e:pogroup} are polyhedral
so long as $\cQ_+$ has only finitely many extreme rays.  This case
arises so often in the sequel---it is crucial in
Section~\ref{s:socle}---that it is isolated in
Definition~\ref{d:polyhedral}.  A non-polyhedral partial order is
induced on $\cQ = \RR^3$ by taking $\cQ_+$ to be a cone over a circle.
% such as either half of the cone
% $x^2 + y^2 \leq z^2$.
\end{example}

\begin{defn}\label{d:polyhedral}
A \emph{real polyhedral group} is a group $\cQ \cong \RR^n$, for
some~$n$, partially ordered so that its positive cone~$\cQ_+$ is an
intersection of finitely many closed half-spaces.  For notational
clarity, $\RR_+^n$ always means the nonnegative orthant in~$\RR^n$, so
$\cQ_+ = \RR^n_+$ means the standard componentwise partial order
on~$\RR^n$.
\end{defn}

\begin{remark}\label{r:rational}
The positive cone~$\cQ_+$ in a real polyhedral group $\cQ = \RR^n$ is
the set of nonnegative real linear combinations of finitely many
vectors that all lie in a single open half-space in~$\cQ$.  The cone
need not be rational; that is, the vectors that generate it---or the
linear functions defining the closed half-spaces---need not have
rational entries.
\end{remark}

\begin{defn}\label{d:discrete-polyhedral}
A \emph{discrete polyhedral group} is a finitely generated free
abelian group partially ordered so that its positive cone is a
finitely generated submonoid.
\end{defn}

\begin{remark}\label{r:discrete-polyhedral}
A finitely generated free abelian group can be partially ordered by a
positive cone that is not a finitely generated submonoid, such as $\cQ
= \ZZ^2$ with $\cQ_+ = C \cap \ZZ^2$ for the cone $C \subseteq \RR^2$
generated by $\bigl[\twoline 10\bigr]$ and $\bigl[\twoline
1\pi\bigr]$.  Such objects are ruled out as discrete polyhedral groups
because the faces of the positive cone in $\cQ \otimes \RR$ need not
be in bijection with the faces of~$\cQ_+$ itself, and the image
of~$\cQ$ need not be discrete in the quotient of $\cQ \otimes \RR$
modulo the subgroup spanned by one of those real faces.
\end{remark}

%%%%%%%%%%%%%%%%%%%%%%%%%%%%%%%%%%%%%%%%%%%%%%%%%%%%%%%%%%%%%%%%%%%%%%%%%
\subsection{Primary decomposition of downsets}\label{sub:downsets}%%%%%%%

\begin{defn}\label{d:PF}
Fix a face~$\tau$ of the positive cone~$\cQ_+$ in a polyhedral
partially ordered group~$\cQ$.  Write $\ZZ \tau$ for the subgroup
of~$\cQ$ generated by~$\tau$.  Let $D \subseteq \cQ$ be a downset.
\begin{enumerate}
\item\label{i:localization}%
The \emph{localization} of~$D$ \emph{along~$\tau$} is the subset
$$%
  D_\tau = \{q \in D \mid q + \tau \subseteq D\}.
$$

\item\label{i:globally-supported}%
An element $q \in D$ is \emph{globally supported on~$\tau$} if $q
\not\in D_{\tau'}$ whenever $\tau' \not\subseteq \tau$.

\item%
The part of~$D$ \emph{globally supported on~$\tau$} is
$$%
  \Gamma_{\!\tau} D = \{q \in D \mid q \text{ is supported on }\tau\}.
$$

\item%
An element $q \in D$ is \emph{locally supported on~$\tau$} if $q$ is
globally supported on~$\tau$ in~$D_\tau$.

\item%
The \emph{local $\tau$-support} of~$D$ is the subset
$\Gamma_{\!\tau}(D_\tau) \subseteq D$ consisting of elements globally
supported on~$\tau$ in the localization~$D_\tau$.

\item\label{i:primary-component}%
The \emph{$\tau$-primary component} of~$D$ is the downset
$$%
  P_\tau(D) = \Gamma_{\!\tau}(D_\tau) - \cQ_+
$$
cogenerated by the local $\tau$-support of~$D$.
\end{enumerate}
\end{defn}

\begin{example}\label{e:coprincipal}
The \emph{coprincipal} downset $\aa + \tau - \cQ_+$ inside of $\cQ =
\ZZ^n$ \emph{cogenerated} by~$\aa$ \emph{along~$\tau$} is globally
supported along~$\tau$.  It also equals its own localization
along~$\tau$, so it equals its local $\tau$-support and is its own
$\tau$-primary component.  Note that when $\cQ_+ = \NN^n$, faces
of~$\cQ_+$ correspond to subsets of~$[n] = \{1,\dots,n\}$, the
correspondence being $\tau \leftrightarrow \chi(\tau)$, where
$\chi(\tau) = \{i \in [n] \mid \ee_i \in \tau\}$ is the
\emph{characteristic subset} of~$\tau$ in~$[n]$.  (The vector~$\ee_i$
is the standard basis vector whose only nonzero entry is $1$ in
slot~$i$.)
\end{example}

\begin{remark}\label{r:freely}
The localization of~$D$ along~$\tau$ is acted on freely by~$\tau$.
Indeed, $D_\tau$ is the union of those cosets of~$\ZZ \tau$ each of
which is already contained in~$D$.  The minor point being made here is
that the coset $q + \ZZ \tau$ is entirely contained in~$D$ as soon as
$q + \tau \subseteq D$ because $D$ is a downset: $q + \ZZ \tau = q +
\tau - \tau \subseteq q + \tau - \cQ_+ \subseteq D$ if~\mbox{$q + \tau
\subseteq D$}.
\end{remark}

\begin{remark}\label{r:monomial-localization}
The localization of~$D$ is defined to reflect localization at the
level of $\cQ$-modules: enforcing invertibility of structure
homomorphisms $\kk[D]_q \to \kk[D]_{q+f}$ for $f \in \tau$ results
in a localized indicator module $\kk[D][\ZZ \tau] = \kk[D_\tau]$.
\end{remark}

\begin{example}\label{e:support}
In Definition~\ref{d:PF}, assume that $\cQ$ is a real polyhedral
group.  Then an element $q \in D$ is globally supported on~$\tau$ if
and only if it lands outside of~$D$ when pushed far enough in any
direction outside of~$\tau$---that is, every element $f \in \cQ_{+\!}
\minus \tau$ has a nonnegative integer multiple~$\lambda f$
with~$\lambda f + q \not\in D$.

One implication is easy: if every $f \in \cQ_{+\!} \minus \tau$ has
$\lambda f + q \not\in D$ for some $\lambda \in \NN$, then any element
$f' \in \tau' \minus \tau$ has a multiple $\lambda f' \in \tau'$ such
that $\lambda f' + q \not\in D$, so $q \not\in D_{\tau'}$.  For the
other direction, use that $\cQ_{+\!} \minus \tau$ is generated as an
upset (i.e., an ideal) of~$\cQ_{+\!}$ by the nonzero vectors along the
rays of~$\cQ_{+\!}$ that are not contained in~$\tau$.  By hypothesis,
$q \in \Gamma_{\!\tau} D \implies q \not\in D_\rho$ for all such
rays~$\rho$, so along each~$\rho$ there is a vector~$v_\rho$ with
$v_\rho + q \not\in D$.  Given $f \in \cQ_{+\!} \minus \tau$, choose
$\lambda \in \NN$ big enough so that $\lambda f \succeq v_\rho$ for
some~$\rho$.
\end{example}

\begin{example}\label{e:PF}
The $\tau$-primary component of~$D$ in Definition~\ref{d:PF} need not
be supported on~$\tau$.  Take $D$ to be the ``under-hyperbola''
downset in Example~\ref{e:hyperbola-GD}.  Then $D = P_\0(D)$, where
$\0$ is the face of~$\cQ_+$ consisting of only the origin.  Points
outside of $\cQ_+$ are not supported at the origin, being instead
supported at either the $x$-axis (if the point is below the $x$-axis)
or the $y$-axis (if the point is behind the $y$-axis).  The subsets
depicted in Example~\ref{e:hyperbola-GD} are $D$ itself; the global
support on~$\0$, which equals the local support on~$\0$; the local
support on the $x$-axis; and the local support on the $y$-axis.  In
contrast, the global support on (say) the $y$-axis consists of the
part of the local support that sits strictly above the $x$-axis.
\end{example}

\begin{defn}\label{d:primDecomp}
Fix a downset~$D$ in a polyhedral partially ordered group~$\cQ$.
\begin{enumerate}
\item%
The downset~$D$ is \emph{coprimary} if $D = P_\tau(D)$ for some
face~$\tau$ of the positive cone~$\cQ_+$.  If $\tau$ needs to
specified then $D$ is called \emph{$\tau$-coprimary}.

\item%
A \emph{primary decomposition} of~$D$ is an expression $D =
\bigcup_{i=1}^r D_i$ of coprimary downsets~$D_i$, called
\emph{components} of the decomposition.
\end{enumerate}
\end{defn}

\begin{thm}\label{t:PF}
Every downset $D$ in a polyhedral partially ordered group~$\cQ$ is the
union $\bigcup_\tau \Gamma_{\!\tau}(D_\tau)$ of its local
$\tau$-supports for all faces $\tau$ of the positive cone.
\end{thm}
\begin{proof}
Given an element $q \in D$, finiteness of the number of faces implies
the existence of a face~$\tau$ that is maximal among those such that
$q \in D_\tau$; note that $q \in D = D_\0$ for the trivial face $\0$
consisting of only the identity of~$\cQ$.  It follows immediately that
$q$ is supported on~$\tau$ in~$D_\tau$.
\end{proof}

\begin{cor}\label{c:PF}
Every downset $D$ in a polyhedral partially ordered group~$\cQ$ has a
canonical primary decomposition $D = \bigcup_\tau P_\tau(D)$, the
union being over all faces~$\tau$ of the positive cone with nonempty
support $\Gamma_{\!\tau}(D_\tau)$.
\end{cor}

\begin{remark}\label{r:disjoint}
The union in Theorem~\ref{t:PF} is not necessarily disjoint.  Nor,
consequently, is the union in Corollary~\ref{c:PF}.  There is a
related union, however, that is disjoint: the sets $(\Gamma_{\!\tau}
D) \cap D_\tau$ do not overlap.  Their union need not be all of~$D$,
however; try Example~\ref{e:PF}, where the negative quadrant
intersects none of the sets~$(\Gamma_{\!\tau} D) \cap D_\tau$.

Algebraically, $(\Gamma_{\!\tau} D) \cap D_\tau$ should be interpreted
as taking the elements of~$D$ globally supported on~$\tau$ and then
taking their images in the localization along~$\tau$, which deletes
the elements that aren't locally supported on~$\tau$.  That is,
$(\Gamma_{\!\tau} D) \cap D_\tau$ is the set of degrees where the
image of $\Gamma_{\!\tau}\kk[D] \to \kk[D]_\tau$ is nonzero.
\end{remark}

\begin{example}\label{e:PF'}
The decomposition in Theorem~\ref{t:PF}---and hence
Corollary~\ref{c:PF}---is not necessarily minimal: it might be that
some of the canonically defined components can be omitted.  This
occurs, for instance, in Example~\ref{e:hyperbola-PD}.  The general
phenomenon, as in this hyperbola example, stems from geometry of the
elements in~$D_\tau$ supported on~$\tau$, which need not be bounded in
any sense, even in the quotient $\cQ/\ZZ \tau$.  In contrast, for
(say) quotients by monomial ideals in the polynomial
ring~$\kk[\NN^n]$, only finitely many elements have support at the
origin, and the downset they cogenerate is consequently~artinian.
\end{example}

%%%%%%%%%%%%%%%%%%%%%%%%%%%%%%%%%%%%%%%%%%%%%%%%%%%%%%%%%%%%%%%%%%%%%%%%%
\subsection{Primary decomposition of finitely encoded modules}\label{sub:prim-decomp}

\mbox{}

\noindent
This section leverages the finitely encoded condition to produce
primary decompositions.  First, the support $\Gamma_{\!\tau}$ on a
face~$\tau$ needs to be defined as a functor on modules.

\begin{defn}\label{d:support}
Fix a face~$\tau$ of a polyhedral partially ordered group~$\cQ$.  The
\emph{localization} of a $\cQ$-module~$\cM$ \emph{along~$\tau$} is the
tensor product
$$%
  \cM_\tau = \cM \otimes_{\kk[\cQ_{+\!}]} \kk[\cQ_{+\!} + \ZZ \tau],
$$
viewing~$\cM$ as a $\cQ$-graded $\kk[\cQ_{+\!}]$-module.  The
submodule of~$\cM\hspace{-1.17pt}$ \emph{globally
supported~on~$\tau$}~is
$$%
  \Gamma_{\!\tau} \cM
  =
  \bigcap_{\tau' \not\subseteq \tau}\bigl(\ker(\cM \to \cM_{\tau'})\bigr)
  =
  \ker \bigl(\cM \to \bigoplus_{\tau' \not\subseteq \tau} \cM_{\tau'}\bigr).
$$
\end{defn}

\begin{example}\label{e:Gamma}
Definition~\ref{d:PF}.2 says that $1_q \in \kk[D]_q = \kk$ lies
in~$\Gamma_{\!\tau} \kk[D]$ if and only if $q \in \Gamma_{\!\tau} D$,
because $q \not\in D_{\tau'}$ if and only if $1_q \mapsto 0$ under
localization of~$\kk[D]$ along~$\tau'$.
\end{example}

\begin{lemma}\label{l:left-exact}
The kernel of any natural transformation between two exact covariant
% https://en.wikipedia.org/wiki/Functor_category
functors is left-exact.  In more detail, if $\alpha$ and $\beta$ are
two exact covariant functors $\cA \to \cB$ for abelian categories
$\cA$ and~$\cB$, and $\gamma_X: \alpha(X) \to \beta(X)$ naturally for
all objects~$X$ of~$\cA$, then the association $X \mapsto \ker
\gamma_X$ is a left-exact covariant functor~$\cA \to \cB$.
\end{lemma}
\begin{proof}
This can be checked by diagram chase or spectral sequence.
\end{proof}

\begin{prop}\label{p:support-left-exact}
The global support functor\/ $\Gamma_{\!\tau\!}$ is left-exact.
\end{prop}
\begin{proof}
Use Lemma~\ref{l:left-exact}: global support is the kernel of the
natural transformation from the identity to a direct sum of
localizations.
\end{proof}

\begin{prop}\label{p:support-localizes}
For modules over a polyhedral partially ordered group, localization
commutes with taking support: $(\Gamma_{\!\tau'} \cM)_\tau =
\Gamma_{\!\tau'}(\cM_\tau)$, and both sides are~$0$ unless~$\tau'
\supseteq \tau$.
\end{prop}
\begin{proof}
Localization along~$\tau$ is exact, so
$$%
  \ker(\cM \to \cM_{\tau''})_\tau
  =
  \ker\bigl(\cM_\tau \to (\cM_{\tau''})_\tau\bigr)
  =
  \ker\bigl(\cM_\tau \to (\cM_\tau)_{\tau''}\bigr).
$$
Since localization along~$\tau$ commutes with intersections of
submodules, $(\Gamma_{\!\tau'} \cM)_\tau$ is the intersection of
the leftmost of these modules over $\tau'' \not\subseteq \tau'$.
But $\Gamma_{\!\tau'}(\cM_\tau)$ equals the same intersection of
the rightmost of these modules by definition.  And if $\tau'
\not\supseteq \tau$ then one of these $\tau''$ equals~$\tau$, so
$\cM_\tau \to (\cM_\tau)_{\tau''} = \cM_\tau$ is the identity
map, whose kernel~is~$0$.
\end{proof}

\begin{defn}\label{d:primDecomp'}
Fix a $\cQ$-module $\cM$ for a polyhedral partially ordered group~$\cQ$.
\begin{enumerate}
\item%
The \emph{local $\tau$-support} of~$\cM$ is the module
$\Gamma_{\!\tau} \cM_\tau$ of elements globally supported
on~$\tau$ in the localization~$\cM_\tau$, or equivalently (by
Proposition~\ref{p:support-localizes}) the localization along~$\tau$
of the submodule of~$\cM$ globally supported on~$\tau$.

\item\label{i:coprimary}%
The module~$\cM$ is \emph{coprimary} if for some face~$\tau$, the
localization map $\cM \into \cM_\tau$ is injective and
$\Gamma_{\!\tau} \cM_\tau$ is an essential submodule of~$\cM_\tau$.
If $\tau$ needs to specified then $\cM$ is called
\emph{$\tau$-coprimary}.

\item\label{i:primdecomp}%
A \emph{primary decomposition} of~$\cM$ is an injection $\cM \into
\bigoplus_{i=1}^r \cM/\cM_i$ into a direct sum of coprimary quotients
$\cM/\cM_i$, called \emph{components} of the decomposition.
\end{enumerate}
\end{defn}

\begin{remark}\label{r:primary}
Primary decomposition is usually phrased in terms of \emph{primary
submodules} $\cM_i \subseteq \cM$, which by definition have coprimary
quotients~$\cM/\cM_i$, satisfying $\bigcap_{i=1}^r \cM_i = 0$
in~$\cM$.  This is equivalent to
Definition~\ref{d:primDecomp'}.\ref{i:primdecomp}.
\end{remark}

\begin{example}\label{e:prim-decomp-downset}
A primary decomposition $D = \bigcup_{i=1}^r D_i$ of a downset~$D$
yields a primary decomposition of the corresponding indicator
quotient, namely the injection $ \kk[D] \into \bigoplus_{i=1}^r
\kk[D_i] $ induced by the surjections $\kk[D] \onto \kk[D_i]$.
\end{example}

The existence of primary decomposition in Theorem~\ref{t:primDecomp}
is intended for finitely encoded modules, but because it deals with
essential submodules and not generators, it only requires the downset
half of a fringe presentation.

\begin{defn}\label{d:downset-hull}
A \emph{downset hull} of a module~$\cM$ over an arbitrary poset is an
injection $\cM \into \bigoplus_{j \in J} E_j$ with each $E_j$ being a
downset module.  The hull is \emph{finite} if $J$ is~finite.  The
module~$\cM$ is \emph{downset-finite} if it admits a finite downset
hull.
\end{defn}

\begin{thm}\label{t:primDecomp}
Every downset-finite module over a polyhedral partially ordered group
admits a primary decomposition.
\end{thm}
\begin{proof}
If $\cM \into \bigoplus_{j=1}^k E_j$ is a downset hull of the
module~$\cM$, and $E_j \into \bigoplus_{i=1}^\ell E_{ij}$ is a primary
decomposition for each~$j$ afforded by Corollary~\ref{c:PF} and
Example~\ref{e:prim-decomp-downset}, then let $E^\tau$ be the direct
sum of the downset modules $E_{ij}$ that are $\tau$-coprimary.  Set
$M^\tau = \ker(\cM \to E^\tau)$.  Then $\cM/\cM^\tau$ is coprimary,
being a submodule of a coprimary module.  Moreover, $\cM \to
\bigoplus_{\tau} \cM/\cM^\tau$ is injective because its kernel is the
same as the kernel of $\cM \to \bigoplus_{ij} E_{ij}$, which is a
composite of two injections and hence injective by construction.
Therefore $\cM \to \bigoplus_{\tau} \cM/\cM^\tau$ is a primary
decomposition.
\end{proof}

%%%%%%%%%%%%%%%%%%%%%%%%%%%%%%%%%%%%%%%%%%%%%%%%%%%%%%%%%%%%%%%%%%%%%%%%%
\section{Finitely determined \texorpdfstring{$\ZZ^n$}{Zn}-modules}\label{s:ZZn}
%%%%%%%%%%%%%%%%%%%%%%%%%%%%%%%%%%%%%%%%%%%%%%%%%%%%%%%%%%%%%%%%%%%%%%%%%

Unless otherwise stated, this section is presented over the discrete
polyhedral group $\cQ = \ZZ^n$ with $\cQ_+ = \NN^n$.  It begins by
reviewing the structure of finitely determined $\ZZ^n$-mod\-ules,
including (minimal) injective and flat resolutions.  These serve as
models for the concepts of socle, cogenerator, and downset hull over
real polyhedral groups (Sections~\ref{s:socle}--\ref{s:hulls})---as
well as their dual notions of top, generator, and upset covers
(Section~\ref{s:gen-functors})---and is the foundation underlying the
syzygy theorem for finitely encoded modules
(Section~\ref{sub:syzygy}), including the existence of fringe
presentations (Section~\ref{sub:fringe}).

Our main references for $\ZZ^n$-modules are \cite{alexdual,cca}.  The
development of homological theory for injective and flat resolutions
in the context of finitely determined modules is functorially
equivalent to the development for finitely generated modules, by
\cite[Theorem~2.11]{alexdual}, but it is convenient to have on hand
the statements in the finitely determined case directly.  The
characterization of finitely determined modules in
Proposition~\ref{p:determined} and (hence)
Theorem~\ref{t:finitely-determined} is apparently new.

%%%%%%%%%%%%%%%%%%%%%%%%%%%%%%%%%%%%%%%%%%%%%%%%%%%%%%%%%%%%%%%%%%%%%%%%%
\subsection{Definitions}\label{sub:def-finitely-determined}\mbox{}%%%%%%%

\noindent
The essence of the finiteness here is that all of the relevant
information about the relevant modules should be recoverable from what
happens in a bounded box in~$\ZZ^n$.

\begin{defn}\label{d:determined}
A $\ZZ^n$-finite module~$\cN$ is \emph{finitely determined} if for
each $i = 1,\dots,n$ the multiplication map $\cdot x_i: N_\bb \to
N_{\bb+\ee_i}$ (see Example~\ref{e:ZZn-graded} for notation) is an
isomorphism whenever $b_i$ lies outside of some bounded interval.
\end{defn}

\begin{remark}\label{r:determined}
This notion of finitely determined is the same notion as in
Example~\ref{e:convex-projection}.  A module is finitely determined if
and only if, after perhaps translating its $\ZZ^n$-grading, it is
\emph{$\aa$-determined} for some $\aa \in \NN^n$, as defined in
\cite[Definition~2.1]{alexdual}.
\end{remark}

\begin{remark}\label{r:fg/ZZ^n}
For $\ZZ^n$-modules, the finitely determined condition is
weaker---that is, more inclusive---than finitely generated, but it is
much stronger than finitely encoded.  The reason is essentially
Example~\ref{e:convex-projection}, where the encoding has a very
special nature.  For a generic sort of example, the restriction
to~$\ZZ^n$ of any $\RR^n$-finite $\RR^n$-module with finitely many
isotypic regions of sufficient width is a finitely encoded
$\ZZ^n$-module, and there is simply no reason why the isotypic regions
should be commensurable with the coordinate directions in~$\ZZ^n$.
Already the toy-model fly wing modules in
Examples~\ref{e:toy-model-fly-wing} and~\ref{e:encoding} yield
infinitely generated but finitely encoded $\ZZ^n$-modules, and this
remains true when the discretization $\ZZ^n$ of~$\RR^n$ is rescaled by
any factor.
\end{remark}

\begin{example}\label{e:local-cohomology}
The local cohomology of an affine semigroup rings is finitely encoded
but usually not finitely determined; see \cite{injAlg} and
\cite[Chapter~13]{cca}, particularly Theorem~13.20, Example~13.17, and
Example~13.4 there.
\end{example}

%%%%%%%%%%%%%%%%%%%%%%%%%%%%%%%%%%%%%%%%%%%%%%%%%%%%%%%%%%%%%%%%%%%%%%%%%
\subsection{Injective hulls and resolutions}\label{sub:inj}\mbox{}%%%%%%%

\begin{remark}\label{r:injective}
Every $\ZZ^n$-finite module that is injective in the category of
$\ZZ^n$-modules is a direct sum of downset modules $\kk[D]$ for
downsets $D$ cogenerated by vectors along faces
(Example~\ref{e:coprincipal}).  This statement holds over any discrete
polyhedral group (Definition~\ref{d:discrete-polyhedral}) by
\cite[Theorem~11.30]{cca}.
\end{remark}

Minimal injective resolutions work for finitely determined modules
just as they do for finitely generated modules.  The standard
definitions are as follows.

\begin{defn}\label{d:inj}
Fix a $\ZZ^n$-module~$\cN$.
\begin{enumerate}
\item%
An \emph{injective hull} of~$\cN$ is an injective homomorphism $\cN
\to E$ in which $E$ is an injective $\ZZ^n$-module (see
Remark~\ref{r:injective}).  This injective hull is
\begin{itemize}
\item%
\emph{finite} if $E$ has finitely many indecomposable summands and
\item%
\emph{minimal} if the number of such summands is minimal.
\end{itemize}
\item%
An \emph{injective resolution} of~$\cN$ is a complex~$E^\spot$ of
injective $\ZZ^n$-modules whose differential $E^i \to E^{i+1}$ for $i
\geq 0$ has only one nonzero homology $H^0(E^\spot) \cong\nolinebreak
\cN$ (so $\cN \into E^0$ and $\coker(E^{i-1} \to E^i) \into E^{i+1}$
are injective hulls for all $i \geq 1$).  \nolinebreak$E^\spot$
\begin{itemize}
\item%
has \emph{length~$\ell$} if $E^i = 0$ for $i > \ell$ and $E^\ell \neq
0$;
\item%
is \emph{finite} if $E^\spot = \bigoplus_i E^i$ has finitely many
indecomposable summands; and
\item%
is \emph{minimal} if $\cN \into E^0$ and $\coker(E^{i-1} \to E^i)
\into E^{i+1}$ are minimal injective hulls for all $i \geq 1$.
\end{itemize}
\end{enumerate}
\end{defn}

\begin{prop}\label{p:determined}
The following are equivalent for a $\ZZ^n$-module~$\cN$.
\begin{enumerate}
\item%
$\cN$ is finitely determined.
\item%
$\cN$ admits a finite injective resolution.
\item%
$\cN$ admits a finite minimal injective resolution.
\end{enumerate}
Any finite minimal resolution is unique up to isomorphism and has
length~at~most~$n$.
\end{prop}
\begin{proof}
The proof is based on existence of finite minimal injective hulls and
resolutions for finitely generated $\ZZ^n$-modules, along with
uniqueness and length~$n$ given minimality, as proved by Goto and
Watanabe \cite{GWii}.

First assume $\cN$ is finitely determined.  Translating the
$\ZZ^n$-grading affects nothing about existence of a finite injective
resolution.  Therefore, using Remark~\ref{r:determined}, assume that
$\cN$ is $\aa$-determined.  Truncate by taking the $\NN^n$-graded part
of~$\cN$ to get a positively $\aa$-determined---and hence finitely
generated---module~$\cN_{\succeq\0}$; see
\cite[Definition~2.1]{alexdual}.  Take any minimal injective
resolution $\cN_{\succeq\0} \to E^\spot$.  Extend backward using the
\v Cech hull \cite[Definition~2.7]{alexdual}, which is exact
\cite[Lemma~2.9]{alexdual}, to get a finite minimal injective
resolution $\vC(\cN_{\succeq\0} \to E^\spot) = (\cN \to \vC E^\spot)$,
noting that $\vC$ fixes indecomposable injective modules whose
$\NN^n$-graded parts are nonzero and is zero on all other
indecomposable injective modules \cite[Lemma~4.25]{alexdual}.  This
proves 1~$\implies$~3.

That 3 $\implies$~2 is trivial.  The remaining implication,
2~$\implies$~1, follows because every indecomposable injective is
finitely determined and the category of finitely determined modules is
abelian.  (The category of $\ZZ^n$-modules each of which is nonzero
only in a bounded set of degrees is abelian, and constructions such as
kernels, cokernels, or direct sums in the category of finitely
determined modules are pulled back from there.)
\end{proof}

%%%%%%%%%%%%%%%%%%%%%%%%%%%%%%%%%%%%%%%%%%%%%%%%%%%%%%%%%%%%%%%%%%%%%%%%%
\subsection{Socles of finitely determined modules}\label{sub:ZZn-soc}\mbox{}

\noindent
The nature of minimality in Definition~\ref{d:inj} plays a decisive
role in the theory developed in later sections.  Its homological
manifestation via socles is particularly~essential.  Expressing that
manifestation requires interactions with localization and restriction.

\begin{defn}\label{d:collapse}
Fix a face $\tau$ of~$\NN^n$ and a $\ZZ^n$-module~$\cN$.
\begin{enumerate}
\item%
The \emph{restriction} of~$\cN$ to the subgroup $\ZZ\tau \subseteq
\ZZ^n$ is $\cN|_\tau = \bigoplus_{\aa \in \ZZ\tau} \cN_\aa$.
\item%
If $\cN = \cM_\tau$ is the localization of a $\ZZ^n$-module
along~$\tau$, then $\cN|_\ol\tau$ is the \emph{quotient-restriction}
$\cmt$ of~$\cM$ along~$\tau$, where $\ol\tau$ is the complement $[n]
\minus \tau$.
\end{enumerate}
\end{defn}

\begin{remark}\label{r:restriction}
The restriction that defines $\cmt$ can equivalently be viewed as a
quotient: $\ZZ\tau$ acts freely on~$\cM_\tau$, and $\cmt$ is the
quotient by this action; hence the nomenclature.  Another way of
viewing $\cmt$ as a quotient is via $\cmt = \cM/I_\tau\cM$, where
$I_\tau = \<x_i - 1 \mid i \in \chi(\tau)\>$.  Note that it is not
necessary to localize along~$\tau$ before restricting, if one is
willing to accept restriction along a translate of~$\ZZ^\ol\tau$.
More~pre\-cisely, for any $\bb$ with sufficiently large
$\tau$-coordinates, $\cmt =\nolinebreak \cM|_{\bb + \ZZ\ol\tau}
=\nolinebreak \bigoplus_{\aa\in\ZZ\ol\tau} \cM_{\bb+\aa}$.
% Aha!  QR = ``quotient-restriction'', the $\ol\sigma$-restriction of
% a module ``sufficiently far out in the $\sigma$-direction being the
% same as collapse along~$\sigma$.
This statement depends on the finitely determined condition; it is
false in general.
\end{remark}

$\ZZ^n$-finite injective modules have minimal cogenerators whose
degrees are canonical when taken in $\ZZ\ol\tau =
\ZZ^n\hspace{-.2ex}/\hspace{.2ex}\ZZ\tau$.  This generalizes the dual
notion of minimal generators for $\ZZ^n$-graded free modules (or, more
generally, minimal generators for finitely generated modules over
graded or local rings).  The notion of cogenerator extends to
arbitrary finitely determined $\ZZ^n$-modules by functoriality,
as~follows.

\begin{defn}\label{d:cogenerator}
Fix a $\ZZ^n$-module~$\cM$ and a face $\tau$ of~$\NN^n$.  A
\emph{cogenerator along~$\tau$} of \emph{degree} $\aa \in \ZZ^n$ is a
nonzero element $y \in \cM_\aa$ that generates a submodule isomorphic
to $\kk[\aa + \tau] = \xx^\aa\kk[\tau]$.  The \emph{cogenerator
functor along~$\tau$} takes~$\cM$ to its \emph{socle along~$\tau$}:
$$%
  \socct\cM
  =
  \bigoplus_{\aa\in\ZZ\ol\tau} \Hom_{\ZZ\ol\tau}(\kk_\aa,\cmt)
$$
where the \emph{skyscraper} module $\kk_\aa$ is the indicator
subquotient for the singleton $\{\aa\} \subseteq \ZZ\ol\tau$.
\end{defn}

The bar over $\soc$ is explained after Definition~\ref{d:socc} and in
Remark~\ref{r:notation-soc-bar}.

\begin{prop}\label{p:minimal-inj-hull}
An injective hull $\cM \to E$ is minimal as in Definition~\ref{d:inj}
if and only if\/ $\socct\cM \to \socct E$ is an isomorphism for all
faces~$\tau$ of\/~$\NN^n$.
\end{prop}
\begin{proof}
See \cite[Section~11.5]{cca}.
\end{proof}

%%%%%%%%%%%%%%%%%%%%%%%%%%%%%%%%%%%%%%%%%%%%%%%%%%%%%%%%%%%%%%%%%%%%%%%%%
\subsection{Matlis duality}\label{sub:matlis}\mbox{}%%%%%%%%%%%%%%%%%%%%%

\noindent
Finitely determined flat modules and their properties are best
described and deduced using a duality that turns~$\ZZ^n$ upside down
and takes vector space duals.  However,~since this duality is needed
in greater generality later, the next definitions are more general.

\begin{defn}\label{d:self-dual}
A poset~$\cQ$ is \emph{self-dual} if it is given a poset automorphism
$q \mapsto -q$.
\end{defn}

\begin{example}\label{e:self-dual}
Inversion makes partially ordered abelian groups self-dual as posets.
\end{example}

\begin{defn}\label{d:matlis}
Fix a self-dual poset~$\cQ$.  The \emph{Matlis dual} of a
$\cQ$-module~$\cM$ is the $\cQ$-module~$\cM^\vee$ defined by
$$%
  (\cM^\vee)_q = \Hom_\kk(\cM_{-q},\kk),
$$
so the homomorphism $(\cM^\vee)_q \to (M^\vee)_{q'}$ is transpose to
$\cM_{-q'} \to \cM_{-q}$.
\end{defn}

\begin{example}\label{e:matlis}
The Matlis dual over a partially ordered group~$\cQ$ is equivalently
$$%
  \cM^\vee = \hhom_\cQ(\cM,\kk[\cQ_+]^\vee),
$$
where $\hhom_\cQ(\cM,\cN) = \bigoplus_{q\in\cQ}
\Hom\bigl(\cM,\cN(q)\bigr)$ is the direct sum of all degree-preserving
homomorphisms from~$\cM$ to $\cQ$-graded translates of~$\cN$.  This is
proved using the adjuntion between $\Hom$ and~$\otimes$; see
\cite[Lemma~11.16]{cca}, noting that the nature of the grading group
is immaterial.  And as in \cite[Lemma~11.16]{cca},
$$%
  \hhom_\cQ(\cM,\cN^\vee) = (\cM \otimes_\cQ \cN)^\vee.
$$
\end{example}

\begin{example}\label{e:dual-of-localization}
It is instructive to compute the Matlis dual of localization along a
face~$\tau$ over a partially ordered abelian group: the Matlis dual
of~$\cM_\tau$ is
$$%
  (\cM_\tau)^\vee
  =
  \hhom\bigl(\kk[\cQ_+]_\tau\otimes\cM,\kk\bigr)
  =
  \hhom\bigl(\kk[\cQ_+]_\tau,\hhom(\cM,\kk)\bigr)
  =
  \hhom\bigl(\kk[\cQ_+]_\tau,\cM^\vee\bigr),
$$
the module of homomorphisms from a localization of~$\kk[\cQ_+]$
into~$\cM$.  The unfamiliarity of this functor is one of the reasons
for developing most of the theory in this paper in terms of socles and
cogenerators instead of tops and generators.
\end{example}

\begin{lemma}\label{l:vee-vee}
$(\cM^\vee)^\vee\!$ is canonically isomorphic to~$\cM$ if $\dim_\kk
\cM_q < \infty$ for all~$q \in\nolinebreak \cQ$.~~\hspace{1ex}$\square$
\end{lemma}

\begin{remark}\label{r:flat}
The Matlis dual of Remark~\ref{r:injective} says that every
$\cQ$-finite flat module over a discrete polyhedral group~$\cQ$ is
isomorphic to a direct sum of upset modules~$\kk[U]$ for upsets of the
form $U = \bb + \ZZ\tau + \cQ_+$.  These upset modules are the graded
translates of localizations of~$\kk[\cQ_+]$ along faces.
\end{remark}

\begin{lemma}\label{l:matlis-pair}
$\hhom(\kk[\cQ_+]_\tau,-)$ is exact for all faces~$\tau$ of any
partially ordered group~$\cQ$.
\end{lemma}
\begin{proof}
Apply the final line of Example~\ref{e:matlis}, using that
$\kk[\cQ_+]_\tau$ is flat (it is a localization of~$\kk[\cQ_+]$) and
$(-)^\vee$ is exact and faithful.
\end{proof}

\begin{remark}\label{r:matlis-pair}
What really drives the lemma is the observation that while the
opposite notion to injective is projective (reverse all of the arrows
in the definition), the adjoint notion to injective is flat.  That is,
a module is flat if and only if its Matlis dual is injective.  This is
an instance of a rather general phenomenon that can be phrased in
terms of a monoidal abelian category~$\cC$ possessing a \emph{Matlis
object}~$E$ for a \emph{Matlis dual pair} of subcategories~$\cA$
and~$\cB$ such that $\Hom(-,E)$ restricts to exact contravariant
functors $\cA \to \cB$ and $\cB \to \cA$ that are inverse to one
another.  The idea is to set $\cM^\vee = \Hom(\cM,E)$, the
\emph{Matlis dual} of any object~$\cM$ of~$\cC$, and define an object
of~$\cC$ to be \emph{$\cB$-flat} if $F \otimes -$ is an exact functor
on~$\cB$.  Then an object $F$ of~$\cA$ is $\cB$-flat if and only if
$\Hom(F,-)$ is an exact functor on~$\cA$.  Examples of this situation
include artinian and noetherian modules over a complete local ring;
modules of finite length over any local ring (in both cases $E =
E(R/\mm)$ is the injective hull of the residue field); and of course
$\cQ$-finite modules over a partially ordered abelian group~$\cQ$.
The latter two examples feature a single Matlis self-dual subcategory.
\end{remark}

%%%%%%%%%%%%%%%%%%%%%%%%%%%%%%%%%%%%%%%%%%%%%%%%%%%%%%%%%%%%%%%%%%%%%%%%%
\subsection{Flat covers, flat resolutions, and generators}\label{sub:flat}\mbox{}

\noindent
Minimal flat resolutions are not commonplace, but the notion is Matlis
dual to that of minimal injective resolution.  In the context of
finitely determined modules, flat resolutions work as well as
injective resolutions.  The definitions are as follows.

\begin{defn}\label{d:flat}
Fix a $\ZZ^n$-module~$\cN$.
\begin{enumerate}
\item%
A \emph{flat cover} of~$\cN$ is a surjective homomorphism $F \to \cN$
in which $F$ is a flat $\ZZ^n$-module (see Remark~\ref{r:flat}).  This
flat cover is
\begin{itemize}
\item%
\emph{finite} if $F$ has finitely many indecomposable summands and
\item%
\emph{minimal} if the number of such summands is minimal.
\end{itemize}
\item%
A \emph{flat resolution} of~$\cN$ is a complex~$F_\spot$ of flat
$\ZZ^n$-modules whose differential $F_{i+1} \to F_i$ for $i \geq 0$
has only one nonzero homology $H_0(F_\spot) \cong \cN$ (so $F_0 \onto
\cN$ and $F_{i+1} \onto \ker(F_i \to F_{i-1})$ are flat covers for all
$i \geq 1$).  The flat resolution~$F_\spot$
\begin{itemize}
\item%
has \emph{length~$\ell$} if $F_i = 0$ for $i > \ell$ and $F_\ell \neq
0$;
\item%
is \emph{finite} if $F_\spot = \bigoplus_i F_i$ has finitely many
indecomposable summands; and
\item%
is \emph{minimal} if $F_0 \onto \cN$ and $F_{i+1} \onto \ker(F_i \to
F_{i-1})$ are minimal flat covers for all $i \geq 1$.
\end{itemize}
\end{enumerate}
\end{defn}

\begin{defn}\label{d:generator-ZZ}
Fix a $\ZZ^n$-module~$\cM$ and a face $\tau$ of~$\NN^n$.  A
\emph{generator} of~$\cM$ \emph{along~$\tau$} of \emph{degree} $\aa
\in \ZZ^n$ is a nonzero element $y \in \cM_\aa$ that maps to a socle
element in a quotient of~$\cM$ isomorphic to $\kk[\aa - \tau] =
\kk[-\aa + \tau]^\vee$.  The \emph{generator functor along~$\tau$}
takes~$\cM$ to its \emph{top along~$\tau$}:
$$%
  \topct\cM
  =
  \kk \otimes_{\ZZ\ol\tau} \bigl(\hhom_{\ZZ^n}(\kk[\NN^n + \ZZ\tau],\cM)/\tau\bigr),
$$
where $\kk = \kk_\0$ is the skyscraper $\ZZ\ol\tau$-module in
degree~$\0$.
\end{defn}

\begin{remark}\label{r:generator}
The definition of generator may feel unfamiliar.  Usually a minimal
generator of a module~$\cM$ over (say) a complete local ring~$R$ with
maximal ideal~$\mm$ is an element $y \in \cM$ whose image is part of a
basis for $\cM/\mm\cM$.  Equivalently, $y$ is nonzero in a quotient
of~$\cM$ that is isomorphic to~$R/\mm$.  What has been relaxed here is
that to consider generators along a prime ideal~$\pp$ of positive
dimension, the definition requires $y$ to map to the socle of a
quotient of~$\cM$ isomorphic to the Matlis dual of~$R/\pp$; the
generator~$y$ then ``extends backward in~$\cM$ along~$R/\pp$''.  The
generator functor $\topct$ encapsulates this by first attempting to
insert the localization $\kk[\NN^n]_\tau$ into~$\cM$---which only
works if there is an element of~$\cM$ that can be ``extended backward
along~$\tau$''---and then taking the ordinary generators of the
quotient-restriction along~$\tau$.
\end{remark}

The functorial relationship between top and socle is the main mode of
proof for tops.

\begin{prop}\label{p:top=socvee}
The generator functor over~$\ZZ^n$ along a face~$\tau$ of\/~$\NN^n$ is
Matlis dual to the cogenerator functor over~$\ZZ^n$ along~$\tau$ on
the Matlis dual:
$$%
  \topct\cM = \bigl(\socct(\cM^\vee)\bigr){}^\vee.
$$
\end{prop}
\begin{proof}
This is an exercise in the adjointness of $\hhom$ and~$\otimes$, using
Example~\ref{e:dual-of-localization} (applied to~$\cM^\vee$).  To
simplify notation in the argument, set $\cN =
\hhom\bigl(\kk[\NN^n]_\tau,\cM\bigr)$.  Then
\begin{align*}
\bigl(\socct(\cM^\vee)\bigr){}^\vee
  &=
  \hhom_{\ZZ\ol\tau}(\kk,\mvt)^\vee
\\*&=
  \hhom_{\ZZ^n}\bigl(\kk[\ZZ\tau],(\cM^\vee)_\tau\bigr){}^{\vee\!}\big/\tau
\\*&=
  \hhom_{\ZZ^n}\bigl(\kk[\ZZ\tau],\hhom(\kk[\NN^n]_\tau,\cM)^\vee\bigr){}^{\vee\!}\big/\tau
\\&=
  \hhom_{\ZZ^n}(\kk[\ZZ\tau],\cN^\vee)^{\vee\!}\big/\tau
\\&=
  \hhom_{\ZZ^n}\bigl(\kk[\ZZ\tau],\hhom(\cN,\kk)\bigr){}^{\vee\!}\big/\tau
\\&=
  \hhom_{\ZZ^n}\bigl(\kk[\ZZ\tau] \otimes_{\ZZ^n} \cN,\kk\bigr){}^{\vee\!}\big/\tau
\\&=
  \bigl((\kk[\ZZ\tau] \otimes_{\ZZ^n} \cN)^\vee\bigr){}^{\vee\!}\big/\tau
\\&=
  \bigl(\kk[\ZZ\tau] \otimes_{\ZZ^n} \cN\bigr)\big/\tau
\\*&=
  \kk \otimes_{\ZZ\ol\tau} (\cnt).\qedhere
\end{align*}
\end{proof}

Here is a prototypical example of how to use this dual-to-socle
characterization.

\begin{cor}\label{c:minimal-flat-cover}
A flat cover $F \to \cM$ is minimal as in Definition~\ref{d:flat} if
and only if\/ $\topct F \to \topct\cM$ is an isomorphism for all
faces~$\tau$ of\/~$\NN^n$.
\end{cor}
\begin{proof}
This is Matlis dual to Proposition~\ref{p:minimal-inj-hull}, by
Proposition~\ref{p:top=socvee}.
\end{proof}

%%%%%%%%%%%%%%%%%%%%%%%%%%%%%%%%%of \texorpdfstring{$\ZZ^n$}{Zn}-modules%
\subsection{Flange presentations}\label{sub:flange}\mbox{}%%%%%%%%%%%%%%%

\begin{defn}\label{d:flange}
Fix a $\ZZ^n$-module~$\cN$.
\begin{enumerate}
\item%
A \emph{flange presentation} of~$\cN$ is a $\ZZ^n$-module morphism
$\phi: F \to E$, with image isomorphic to~$\cN$, where $F$ is flat and
$E$ is injective in the category of \mbox{$\ZZ^n$-modules}.
\item%
If $F$ and~$E$ are expressed as direct sums of indecomposables, then
$\phi$ is \emph{based}.
\item%
If $F$ and~$E$ are finite direct sums of indecomposables, then $\phi$
is \emph{finite}.
\item%
If the number of indecomposable summands of~$F$ and~$E$ are
simultaneously minimized then $\phi$ is \emph{minimal}.
\end{enumerate}
\end{defn}

\begin{remark}\label{r:portmanteau-fl}
The term \emph{flange} is a portmanteau of \emph{flat} and
\emph{injective} (i.e., ``flainj'') because a flange presentation is
the composite of a flat cover and an injective hull.
\end{remark}

Flange presentations are effective data structures via the following
notational trick.  Topologically, it highlights that births occur
along generators of the flat summands and deaths occur along
cogenerators of the injective summands, with a linear map, over the
ground field, to relate them.

\begin{defn}\label{d:monomial-matrix-fl}
Fix a based finite flange presentation $\phi:
\bigoplus_p\hspace{-.2pt} F_p = F \to E =
\nolinebreak\bigoplus_q\hspace{-.2pt} E_q$.  A \emph{monomial matrix}
for $\phi$ is an array of \emph{scalar entries}~$\phi_{qp}$ whose
columns are labeled by the indecomposable flat summands~$F_p$ and
whose rows are labeled by the indecomposable injective summands~$E_q$:
$$%
\begin{array}{ccc}
  &
  \monomialmatrix
	{F_1\\\vdots\ \\F_k\\}
	{\begin{array}{ccc}
		   E_1    & \cdots &    E_\ell   \\
		\phi_{11} & \cdots & \phi_{1\ell}\\
		\vdots    & \ddots &   \vdots    \\
		\phi_{k1} & \cdots & \phi_{k\ell}\\
	 \end{array}}
	{\\\\\\}
\\
  F_1 \oplus \dots \oplus F_k = F
  & \fillrightmap
  & E = E_1 \oplus \dots \oplus E_\ell.
\end{array}
$$
\end{defn}

The entries of the matrix $\phi_{\spot\spot}$ correspond to
homomorphisms $F_p \to E_q$.

\begin{lemma}\label{l:F->E}
If $F = \kk[\aa + \ZZ\tau' + \NN^n]$ is an indecomposable flat
$\ZZ^n$-module and $E = \kk[\bb + \ZZ\tau - \NN^n]$ is an
indecomposable injective $\ZZ^n$-module, then $\Hom_{\ZZ^n}(F, E) = 0$
unless $(\aa + \ZZ\tau' + \NN^n) \cap (\bb + \ZZ\tau - \NN^n) \neq
\nothing$, in which case $\Hom_{\ZZ^n}(F, E) = \kk$.
\end{lemma}
\begin{proof}
This is a special case of Lemma~\ref{l:U->D}.
\end{proof}

\begin{defn}\label{d:F<E}
In the situation of Lemma~\ref{l:F->E}, write $F \preceq E$ if their
degree sets have nonempty intersection: $(\aa + \ZZ\tau' + \NN^n) \cap
(\bb + \ZZ\tau - \NN^n) \neq \nothing$.
\end{defn}

\begin{prop}\label{p:scalars-fl}
With notation as in Definition~\ref{d:monomial-matrix-fl}, $\phi_{pq} =
0$ unless $F_p \preceq E_q$.  Conversely, if an array of scalars
$\phi_{qp} \in \kk$ with rows labeled by indecomposable flat modules
and columns labeled by indecomposable injectives has $\phi_{pq} = 0$
unless $F_q \preceq E_q$, then it represents a flange presentation.
\end{prop}
\begin{proof}
Lemma~\ref{l:F->E} and Definition~\ref{d:F<E}.
\end{proof}

The unnatural hypothesis that a persistence module be finitely
generated results in data types and structure theory that are
asymmetric regarding births as opposed to deaths.  In contrast, the
notion of flange presentation is self-dual: their duality interchanges
the roles of births~($F$) and deaths~($E$).

\begin{prop}\label{p:duality}
A $\ZZ^n$-module $\cN$ has a finite flange presentation $F \to E$ if
and only if the Matlis dual $E^\vee \to F^\vee$ is a finite flange
presentation of the Matlis dual $\cN^\vee$.
\end{prop}
\begin{proof}
Matlis duality is an exact, contravariant functor on~$\ZZ^n$-modules
that takes the subcategory of finitely determined $\ZZ^n$-modules to
itself (these properties are immediate from the definitions),
interchanges flat and injective objects therein, and has the property
that the natural homomorphism $(\cN^\vee)^\vee \to \cN$ is an
isomorphism for finitely determined~$\cN$; see
\cite[Section~1.2]{alexdual} for a discussion of these properties.
\end{proof}

%%%%%%%%%%%%%%%%%%%%%%%%%%%%%%%%%%%%%%%%%%%%%%%%%%%%%%%%%%%%%%%%%%%%%%%%%
\subsection{Syzygy theorem for \texorpdfstring{$\ZZ^n$}{Zn}-modules}\label{sub:Zsyzygy}

\begin{thm}\label{t:finitely-determined}
A $\ZZ^n$-module is finitely determined if and only if it admits one,
and hence all, of the following:
\begin{enumerate}
\item\label{i:flange}%
a finite flange presentation; or
\item\label{i:flat-presentation}%
a finite flat presentation; or
\item\label{i:injective-copresentation}%
a finite injective copresentation; or
\item\label{i:flat-res}%
a finite flat resolution; or
\item\label{i:injective-res}%
a finite injective resolution; or
\item\label{i:minimal}%
a minimal one of any of the above.
\end{enumerate}
Any minimal one of these objects is unique up to noncanonical
isomorphism, and the resolutions have length at most~$n$.
\end{thm}
\begin{proof}
The hard work is done by Proposition~\ref{p:determined}.  It implies
that $\cN$ is finitely determined $\iff \cN^\vee$ has a minimal
injective resolution $\iff \cN$ has a minimal flat resolution of
length at most~$n$, since the Matlis dual of any finitely determined
module~$\cN$ is finitely determined.  Having both a minimal injective
resolution and a minimal flat resolution is stronger than having any
of items~\ref{i:flange}--\ref{i:injective-copresentation}, minimal or
otherwise, so it suffices to show that $\cN$ is finitely determined if
$\cN$ has any of
items~\ref{i:flange}--\ref{i:injective-copresentation}.  This follows,
using that the category of finitely determined modules~is~abelian as
in the proof of Proposition~\ref{p:determined}, from the fact that
every indecomposable injective or flat $\ZZ^n$-module is finitely
determined.
\end{proof}

\begin{remark}\label{r:finitely-determined}
Conditions~\ref{i:flange}--\ref{i:minimal} in
Theorem~\ref{t:finitely-determined} remain equivalent for
$\RR^n$-modules, with the standard positive cone $\RR^n_+$, assuming
that the finite flat and injective modules in question are finite
direct sums of localizations of~$\RR^n$ along faces and their Matlis
duals.  (The equivalence, including minimality, is a consequence of
the generator and cogenerator theory over real polyhedral groups,
particularly Theorem~\ref{t:presentations-minimal}.)  The equivalent
conditions do not characterize $\RR^n$-modules that are pulled back
under convex projection from arbitrary modules over an interval
in~$\RR^n$, though, because all sorts of infinite things can can
happen inside of a box, such as having generators~along~a~curve.
\end{remark}

%%%%%%%%%%%%%%%%%%%%%%%%%%%%%%%%%%%%%%%%%%%%%%%%%%%%%%%%%%%%%%%%%%%%%%%%%
\section{Fringe presentations and syzygy theorem for poset modules}\label{s:fringe}
%%%%%%%%%%%%%%%%%%%%%%%%%%%%%%%%%%%%%%%%%%%%%%%%%%%%%%%%%%%%%%%%%%%%%%%%%

For modules over arbitrary posets, the conditions of flatness and
injectivity still make sense, but it is too restrictive to ask for
resolutions or presentations that are finite direct sums of
indecomposables in such generality, as demonstrated by the formidable
infinitude of such objects in the case of $\RR^n$-modules like fly
wing modules.  The idea here, both for theoretical and computational
purposes, is to allow arbitrary upset and downset modules instead of
only flat and injective ones.  The data structures in the title of
this paper refer, in practical terms, to monomial matrices constructed
from fringe presentations and indicator resolutions.

%%%%%%%%%%%%%%%%%%%%%%%%%%%%%%%%%%%%%%%%%%%%%%%%%%%%%%%%%%%%%%%%%%%%%%%%%
\subsection{Fringe presentations}\label{sub:fringe}\mbox{}%%%%%%%%%%%%%%%

\vbox{
\begin{defn}\label{d:fringe}
Fix any poset~$\cQ$.  A \emph{fringe presentation} of a
$\cQ$-module~$\cM$ is
\begin{itemize}
\item%
a direct sum~$F$ of upset modules~$\kk[U]$,
\item%
a direct sum~$E$ of downset modules~$\kk[D]$, and
\item%
a homomorphism $F \to E$ of $\cQ$-modules with
\begin{itemize}
  \item%
  image isomorphic to~$\cM$ and
  \item%
  components $\kk[U] \to \kk[D]$ that are connected
  (Definition~\ref{d:connected-homomorphism}).
\end{itemize}
\end{itemize}
A~fringe presentation~is
\begin{enumerate}
\item%
\emph{finite} if the direct sums are finite;
\item\label{i:subordinate-fringe}%
\emph{subordinate} to an encoding of~$\cM$
(Definition~\ref{d:encoding}) if each summand $\kk[U]$ of~$F$
and~$\kk[D]$ of~$E$ is constant on every fiber of the encoding poset
morphism; and\nopagebreak
\item\label{i:semialgebraic-fringe}%
\emph{semialgebraic} if $\cQ$ is a partially ordered real vector space
of finite dimension and the fringe presentation is subordinate to a
semialgebraic encoding
\hspace{-.3pt}(\hspace{-.3pt}Definition~\hspace{-2pt}\ref{d:alg-finite}).
\end{enumerate}
\end{defn}
}

\begin{lemma}\label{l:constant}
An indicator module is constant on every fiber of a poset morphism
$\pi: \cQ \to \cP$ if and only if it is the pullback along~$\pi$ of an
indicator $\cP$-module.
\end{lemma}
\begin{proof}
The ``if'' direction is by definition.  For the ``only if'' direction,
observe that if $U \subseteq \cQ$ is an upset that is a union of
fibers of~$\cP$, then the image $\pi(U) \subseteq \cP$ is an upset
whose preimage equals~$U$.  The same argument works for downsets.
\end{proof}

\begin{example}[Pullbacks of flat and injective modules]\label{e:pullback}
An indecomposable flat $\ZZ^n$-module $\kk[\bb + \ZZ\tau + \NN^n]$ is
a downset module for the poset~$\ZZ^n$.  Pulling back to any poset
under a poset map to~$\ZZ^n$ therefore yields a downset module for the
given poset.  The dual statement holds for any indecomposable
injective module $\kk[\bb + \ZZ\tau - \NN^n]$: its pullback is a
downset module.
\end{example}

\begin{defn}\label{d:monomial-matrix-fr}
Fix a finite fringe presentation $\phi: \bigoplus_p \kk[U_p] = F \to E
= \bigoplus_q \kk[D_q]$.  A \emph{monomial matrix} for $\phi$ is an
array of \emph{scalar entries}~$\phi_{pq}$ whose columns are labeled
by the \emph{birth upsets}~$U_p$ and whose rows are labeled by the
\emph{death downsets}~$D_q$:
$$%
\begin{array}{ccc}
  &
  \monomialmatrix
	{U_1\\\vdots\ \\U_k\\}
	{\begin{array}{ccc}
		   D_1    & \cdots &    D_\ell   \\
		\phi_{11} & \cdots & \phi_{1\ell}\\
		\vdots    & \ddots &   \vdots    \\
		\phi_{k1} & \cdots & \phi_{k\ell}\\
	 \end{array}}
	{\\\\\\}
\\
  \kk[U_1] \oplus \dots \oplus \kk[U_k] = F
  & \fillrightmap
  & E = \kk[D_1] \oplus \dots \oplus \kk[D_\ell].
\end{array}
$$
\end{defn}

\begin{prop}\label{p:scalars}
With notation as in Definition~\ref{d:monomial-matrix-fr}, $\phi_{pq} =
0$ unless $U_p \preceq D_q$.  Conversely, if an array of scalars
$\phi_{pq} \in \kk$ with rows labeled by upsets and columns labeled by
downsets has $\phi_{pq} = 0$ unless $U_p \preceq D_q$, then it
represents a fringe presentation.
\end{prop}
\begin{proof}
Lemma~\ref{l:U->D}.\ref{i:U->D} and
Definition~\ref{d:connected-homomorphism}.
\end{proof}

Pullbacks have particularly transparent monomial matrix
interpretations.

\begin{prop}\label{p:pullback-monomial-matrix}
Fix a poset~$\cQ$ and an encoding of a $\cQ$-module~$\cM$ via a poset
morphism $\pi: \cQ \to \cP$ and $\cP$-module~$\cH$.  Any monomial
matrix for a fringe presentation of~$\cH$ pulls back to a monomial
matrix for a fringe presentation subordinate to the encoding by
replacing the row labels $U_1,\dots,U_k$ and column labels
$D_1,\dots,D_\ell$ with their preimages, namely $\pi^{-1}(U_1),
\dots, \pi^{-1}(U_k)$ and $\pi^{-1}(D_1), \dots, \pi^{-1}(D_\ell)$.
\hfill$\square$
\end{prop}

\begin{remark}\label{r:portmanteau-fr}
The term ``fringe'' is a portmanteau of ``free'' and ``injective''
(i.e., ``frinj''), the point being that it combines aspects of free
and injective resolutions while also conveying that the data structure
captures trailing topological features at both the birth and death
ends.
\end{remark}

%%%%%%%%%%%%%%%%%%%%%%%%%%%%%%%%%%%%%%%%%%%%%%%%%%%%%%%%%%%%%%%%%%%%%%%%%
\subsection{Indicator resolutions}\label{sub:indicator}\mbox{}%%%%%%%%%%%

\begin{defn}\label{d:resolutions}
Fix any poset~$\cQ$ and a $\cQ$-module~$\cM$.
\begin{enumerate}
\item%
An \emph{upset resolution} of~$\cM$ is a complex~$F_\spot$ of
$\cQ$-modules, each a direct sum of upset submodules of~$\kk[\cQ]$,
whose differential $F_i \to F_{i-1}$ decreases homological degrees and
has only one nonzero homology $H_0(F_\spot) \cong \cM$.
\item%
A \emph{downset resolution} of~$\cM$ is a complex~$E^\spot$ of
$\cQ$-modules, each a direct sum of downset quotient modules
of~$\kk[\cQ]$, whose differential $E^i \to E^{i+1}$ increases
cohomological degrees and has only one nonzero homology
$H^0(E^\spot) \cong\nolinebreak \cM$.
\end{enumerate}\setcounter{separated}{\value{enumi}}%saves \enumi
An upset or downset resolution is called an \emph{indicator
resolution} if the up- or down- nature is unspecified.  The
\emph{length} of an indicator resolution is the largest
(co)homological degree in which the complex is nonzero.  An indicator
resolution is
\begin{enumerate}\setcounter{enumi}{\value{separated}}%restores \enumi
\item%
\emph{finite} if the number of indicator module summands is finite,
\item%
\emph{subordinate} to an encoding of~$\cM$
(Definition~\ref{d:encoding}) if each indicator summand is constant on
every fiber of the encoding poset morphism, and
\item%
\emph{semialgebraic} if $\cQ$ is a partially ordered real vector space
of finite dimension and the resolution is subordinate to a
semialgebraic encoding (Definition~\ref{d:alg-finite}).
\end{enumerate}
\end{defn}

\begin{defn}\label{d:resolution-monomial-matrix}
Monomial matrices for indicator resolutions are defined similarly to
how they are for fringe presentations in
Definition~\ref{d:monomial-matrix-fr}, except that for the
cohomological case the row and column labels are source and target
downsets, respectively, while in the homological case the row and
column labels are target and source upsets, respectively:
$$%
\begin{array}{ccc}
  &
  \monomialmatrix
	{\vdots\ \\D^i_p\\\vdots\ }
	{\begin{array}{ccc}
		 \cdots & D^{i+1}_q & \cdots \\
		        &           &        \\
		        & \phi_{pq} &        \\
		        &           &        \\
	 \end{array}}
	{\\\\\\}
\\
    E^i
  & \fillrightmap
  & E^{i+1}
\end{array}
\qquad\text{and}\qquad
\begin{array}{ccc}
  &
  \monomialmatrix
	{\vdots\ \\U_i^p\\\vdots\ }
	{\begin{array}{ccc}
		 \cdots & U_{i+1}^q & \cdots \\
		        &           &        \\
		        & \phi_{pq} &        \\
		        &           &        \\
	 \end{array}}
	{\\\\\\}
\\
    F_i
  & \filleftmap
  & F_{i+1}.
\end{array}
$$
(Note the switch of source and target from cohomological to
homological, so the map goes from right to left in the homological
case, with decreasing homological indices.)
\end{defn}

As in Proposition~\ref{p:pullback-monomial-matrix}, pullbacks have
transparent of monomial matrix interpretation.

\begin{prop}\label{p:syzygy-monomial-matrix}
Fix a poset~$\cQ$ and an encoding of a $\cQ$-module~$\cM$ via a poset
morphism $\pi: \cQ \to \cP$ and $\cP$-module~$\cH$.  Monomial matrices
for any indicator resolution of~$\cH$ pull back to monomial matrices
for an indicator resolution of~$\cM$ subordinate to the encoding by
replacing the row and column labels with their preimages under~$\pi$.
\hfill$\square$
\end{prop}

\begin{defn}\label{d:indicator-(co)presentation}
Fix any poset~$\cQ$ and a $\cQ$-module~$\cM$.
\begin{enumerate}
\item%
An \emph{upset presentation} of~$\cM$ is an expression of~$\cM$ as the
cokernel of a homomorphism $F_1 \to F_0$ each of which is a direct sum
of upset modules.
\item%
A \emph{downset copresentation} of~$\cM$ is an expression of~$\cM$ as
the kernel of a homomorphism $E^0 \to E^1$ each of which is a direct
sum of downset modules.
\end{enumerate}\setcounter{separated}{\value{enumi}}%saves \enumi
These \emph{indicator presentations} are \emph{finite}, or
\emph{subordinate} to an encoding of~$\cM$, or \emph{semialgebraic} as
in Definition~\ref{d:resolutions}.
\end{defn}

\begin{remark}\label{r:augmentation}
It is tempting to think that a fringe presentation is nothing more
than the concatenation of the augmentation map of an upset resolution
(that is, the surjection at the end) with the augmentation map of a
downset resolution (that is, the injection at the beginning), but
there is no guarantee that the components $F_i \to E_j$ of the
homomorphism thus produced are connected
(Definition~\ref{d:connected-homomorphism}).
\end{remark}

%%%%%%%%%%%%%%%%%%%%%%%%%%%%%%%%%%%%%%%%%%%%%%%%%%%%%%%%%%%%%%%%%%%%%%%%%
\subsection{Syzygy theorem for modules over posets}\label{sub:syzygy}%%%%

\begin{prop}\label{p:pushforward}
For any inclusion $\iota: \cP \to \cZ$ of posets and
$\cP$-module~$\cH$ there is a $\cZ$-module $\iota_*\cH$, the
\emph{pushforward to~$\cZ$}, whose restriction to~$\iota(\cP)$
is~$\cH$ and is universally repelling: $\iota_*\cH$ has a canonical
map to every $\cZ$-module whose restriction to~$\iota(\cP)$ is~$\cH$.
\end{prop}
\begin{proof}
At $z \in \cZ$ the pushforward $\iota_*\cH$ places the colimit
$\dirlim\cH_{\preceq z}$ of the diagram of vector spaces indexed by
the elements of~$\cP$ whose images precede~$z$.  The universal
pro\-perty of colimits implies that $\iota_*\cH$ is a $\cZ$-module
with the desired universal property.
\end{proof}

\begin{remark}\label{r:kan-extension}
With certain perspectives as in Remark~\ref{r:curry}, the pushforward
is a \emph{left Kan extension} \cite[Remark~4.2.9]{curry-thesis}.
\end{remark}

\begin{defn}\label{d:dominates}
An encoding of a $\cQ$-module \emph{dominates} a fringe presentation,
indicator resolution, or indicator presentation if any of these is
subordinate to the encoding.
\end{defn}

\begin{thm}\label{t:syzygy}
A module~$\cM$ over a poset~$\cQ$ is finitely encoded if and only if
it admits one, and hence all, of the following:
\begin{enumerate}
\item\label{i:fringe}%
a finite fringe presentation; or

\item\label{i:upset-presentation}%
a finite upset presentation; or

\item\label{i:downset-copresentation}%
a finite downset copresentation; or

\item\label{i:upset-res}%
a finite upset resolution; or

\item\label{i:downset-res}%
a finite downset resolution; or

\item\label{i:subordinate}%
any of the above subordinate to any given finite encoding; or

\item\label{i:dominating}%
a finite encoding dominating any given one of
items~\ref{i:fringe}--\ref{i:downset-res}.
\end{enumerate}
The equivalences remain true over any partially ordered real vector
space of finite~dimen\-sion if ``finitely encoded'' and all
occurrences of~\mbox{``finite''} are replaced by ``semialgebraic''.
\end{thm}
\begin{proof}
The plan is first to show that a finitely encoded $\cQ$-module~$\cM$
has finite upset and downset resolutions and (co)presentations, as
well as a finite fringe presentation.  Next a dominating finite
encoding is constructed from a given presentation or resolution.  The
semialgebraic case uses the same arguments, mutatis mutandis, and will
not be mentioned further.

Let $\cM$ be a $\cQ$-module finitely encoded by a poset morphism $\pi:
\cQ \to \cP$ and $\cP$-module~$\cH$.  The finite poset~$\cP$ has order
dimension~$n$ for some positive integer~$n$; as such $\cP$ has an
embedding $\iota: \cP \into \ZZ^n$.  The pushforward $\iota_*\cH$ is
finitely determined because it is pulled from any box
containing~$\iota(\cP)$.  The desired presentation or resolution is
pulled back to~$\cQ$ (via $\iota \circ \pi: \cQ \to \ZZ^n$) from the
corresponding flange, flat, or injective presentation or resolution
of~$\iota_*\cH$ afforded by Theorem~\ref{t:finitely-determined}; these
pullbacks to~$\cQ$ are finite indicator resolutions of~$\cM$
subordinate to~$\pi$ by Example~\ref{e:pullback} and
Lemma~\ref{l:constant}.  In the fringe case, the component
homomorphisms are connected because, by Lemma~\ref{l:U->D}.\ref{i:kk},
the components of flange presentations are automatically connected.

To produce a finite encoding~$\pi$ given a fringe presentation~$\phi$,
let $\cP$ be the poset of uptight regions
(Proposition~\ref{p:posetQuotient}) for the set~$\Upsilon(\phi)$ of
upsets comprising the birth upsets and the complements of the death
downsets of~$\phi$.  The proof is similar if indicator presentations
or resolutions are given: each upset in the set $\Upsilon$ is either
the degree set of one of the summands (in the case of upset
presentation or resolution) or is the complement of the degree set of
one of the summands (in the case of downset copresentation or
resolution).
\end{proof}

\begin{example}\label{e:wing-fringe}
The persistent homology of a fly wing
(Example~\ref{e:fly-wing-filtration}) admits a finite isotypic
subdivision of $\RR^2$ into semialgebraic sets.  Such a module
therefore has a finite fringe presentation subordinate to a
semialgebraic encoding by Theorems~\ref{t:isotypic}
and~\ref{t:syzygy}.
\end{example}

\begin{remark}
Comparing Theorems~\ref{t:syzygy} and~\ref{t:finitely-determined},
what happened to minimality?  It is not clear in what generality
minimality can be characterized.  Much of this paper can be seen as a
case study for posets arising from abelian groups that are either
finitely generated and free (the discrete polyhedral case) or real
vector spaces of finite dimension (the real polyhedral case); see in
particular Theorem~\ref{t:presentations-minimal}.  The answer is
rather more nuanced in the real case, obscuring how minimality might
generalize beyond discrete or real polyhedral groups.  For discrete
polyhedral groups, the result is not far beyond what is in the
literature, although it highlights that minimality ought to be phrased
in terms of isomorphisms on tops and socles rather than by counting
summands.  (That point that is especially pertinent for primary
decomposition; see~Section~\ref{sub:RRprim-decomp}.)
\end{remark}

\begin{remark}\label{r:pullback-twice}
In the situation of the proof of Theorem~\ref{t:syzygy}, composing two
applications of Proposition~\ref{p:pullback-monomial-matrix}---one for
the encoding $\pi: \cQ \to \cP$ and one for the embedding $\iota: \cP
\into \ZZ^n$---yields a monomial matrix for a fringe presentation
of~$\cM$ directly from a monomial matrix for a flange presentation.
\end{remark}

\begin{remark}\label{r:RRn-mod}
Lesnick and Wright consider \mbox{$\RR^n$-modules}
\cite[Section~2]{lesnick-wright2015} in finitely presented cases.  As
they indicate,
% by their Proposition~2.9,
homological algebra of such $\RR^n$-modules is no different than
finitely generated \mbox{$\ZZ^n$-modules}.  This can be seen by finite
encoding: any finite poset in~$\RR^n$ is embeddable in~$\ZZ^n$,
because a product of finite chains is all that~is~needed.
\end{remark}

%%%%%%%%%%%%%%%%%%%%%%%%%%%%%%%%%%%%%%%%%%%%%%%%%%%%%%%%%%%%%%%%%%%%%%%%%
\section{Socles and cogenerators}\label{s:socle}%%%%%%%%%%%%%%%%%%%%%%%%%
%%%%%%%%%%%%%%%%%%%%%%%%%%%%%%%%%%%\texorpdfstring{$\RR^n$}{Rn}-modules}%

The main result in this section is Definition~\ref{d:soct}, which
introduces the notions of cogenerator functor and socle along a face
with a given nadir.  Its concomitant foundations include ways to
decompose the cogenerator functors into continuous and discrete parts
(Proposition~\ref{p:either-order}), interactions with localization
(Proposition~\ref{p:local-vs-global}), left-exactness
(Proposition~\ref{p:left-exact-tau}), and preservation of finitely
encoded or semialgebraic conditions (Theorem~\ref{t:soct-encoding}),
along with the crucial calculation of socles in the simplest case,
namely the indicator function of a single face
(Example~\ref{e:soct-k[tau]}).  The theory is built step by step,
starting with local geometry of downsets near their boundaries
(Section~\ref{sub:tangent}) and the functorial view of the geometry
(Section~\ref{sub:upper-boundary}), and then proceeding to cogenerator
functors and socles with increasing levels of complexity
(Sections~\mbox{\ref{sub:socc}--\ref{sub:soc-along}}).  Each stage
includes preliminary versions of the foundations whose take-away
versions are those in Section~\ref{sub:soc-along}, although the
foundations in earlier stages are developed in increasing generality,
over arbitrary posets instead of real polyhedral groups, for
\mbox{example}.

%%%%%%%%%%%%%%%%%%%%%%%%%%%%%%%%%%%%%%%%%%%%%%%%%%%%%%%%%%%%%%%%%%%%%%%%%
\subsection{Tangent cones of downsets}\label{sub:tangent}\mbox{}%%%%%%%%%

\begin{defn}\label{d:tangent-cone}
The \emph{tangent cone} $T_\aa D$ of a downset $D$ in a real
polyhedral group~$\cQ$ (Definition~\ref{d:polyhedral}) at a point $\aa
\in \cQ$ is the set of vectors $\vv \in -\cQ_+$ such that $\aa +
\epsilon\vv \in D$ for all sufficiently small (hence all)~$\epsilon > 0$.
\end{defn}

\begin{remark}\label{r:tangent-cone}
Since the real number $\epsilon$ in Definition~\ref{d:tangent-cone} is
strictly positive, the vector $\vv = \0$ lies in $T_\aa D$ if and only
if $\aa$ itself lies in~$D$, and in that case $T_\aa D = -\cQ_+$.
\end{remark}

\begin{example}\label{e:tangent-cone}
The tangent cone defined here is not the tangent cone of~$D$ as a
stratified space, because the cone here only considers vectors
in~$-\cQ_+$.  A specific simple example to see the difference is the
closed half-plane $D$ beneath the line $y = -x$ in~$\RR^2$, where the
usual tangent cone at any point along the boundary line is the
half-plane, whereas $T_\aa D = -\RR^2_+$.  Furthermore, $T_\aa D$ can
be nonempty for a point~$\aa$ in the boundary of~$D$ even if $\aa$
does not lie in~$D$ itself.  For an example of that, take $D^\circ$ to
be the interior of~$D$; then $T_\aa D^\circ = -\RR^2_+ \minus \{\0\}$
for any $\aa$ on the boundary line.
\end{example}

The most important conclusion concerning tangent cones at points of
downsets, Proposition~\ref{p:shape}, says that such cones are certain
unions of relative interiors of faces.  Some definitions and
preliminary results are required.

\begin{defn}\label{d:cocomplex}
Fix a real polyhedral group $\cQ$.
\begin{enumerate}
\item%
For any face $\sigma$ of the positive cone~$\cQ_+$, write $\sigma^\circ$
for the relative interior of~$\sigma$.
\item%
For any set~$\nabla$ of faces of~$\cQ_+$, write $\cQ_\nabla =
\bigcup_{\sigma \in \nabla} \sigma^\circ$, the \emph{cone of
shape~$\nabla$}.
\item%
A \emph{cocomplex} in~$\cQ_+$ is an upset in the face poset $\cfq$
of~$\cQ_+$, where $\sigma \preceq \tau$ if~$\sigma \subseteq \tau$.
\end{enumerate}
\end{defn}

\begin{example}\label{e:nabla}
The cocomplex $\nabs = \{\text{faces }\sigma'\text{ of }\cQ_+ \mid
\sigma' \supseteq \sigma\}$ is the \emph{open star} of the
face~$\sigma$.  It determines the cone $\qns$ of shape~$\nabs$, which
plays an important role.
\end{example}

\begin{remark}\label{r:quantum}
The next proposition is the reason for specializing this section to
real polyhedral groups instead of arbitrary polyhedral partially
ordered groups, where limits might not be meaningful.  For example,
although limits make formal sense in the integer lattice~$\ZZ^n$ with
the usual discrete topology, it is impossible for a sequence of points
in the relative interior of a face to converge to the origin of the
face.  This quantum separation has genuine finiteness consequences for
the algebra of $\ZZ^n$-modules that usually do not hold for
$\RR^n$-modules.
\end{remark}

\begin{prop}\label{p:<<}
If $\{\aa_k\}_{k\in\NN}$ is any sequence converging to~$\aa$, then	
$\bigcup_{k=0}^\infty (\aa_k - \cQ_+) \supseteq \aa - \qpc$.  If the
sequence is contained in $\aa - \sigma^\circ$, then the union equals
$\aa - \qns$.
\end{prop}
\begin{proof}
For each point $\bb \in \aa - \qpc$, every linear function $\ell:
\RR^n \to \RR$ that is nonnegative on~$\cQ_+$ eventually takes values
on the sequence that are bigger than~$\ell(\bb)$; thus $\bb$ lies in
the union as claimed.  When the sequence is contained in $\aa -
\sigma^\circ$, the union is contained in $\aa - \sigma^\circ - \cQ_+$
by hypothesis, but the union contains $\aa - \sigma^\circ$ by the
first claim applied with $\sigma$ in place of~$\cQ_+$.  The union
therefore equals $\aa - \sigma^\circ - \cQ_+$ because it is a downset.
The next lemma completes the proof.
\end{proof}

\begin{lemma}\label{l:<<}
If $\sigma$ is any face of the positive cone~$\cQ_+$ then $\sigma^\circ
+ \cQ_+ = \qns$.
\end{lemma}
\begin{proof}
Fix $\ff + \bb \in \sigma^\circ + \cQ_+$.  If $\ell(\ff + \bb) = 0$
for some linear function $\ell: \RR^n \to \RR$ that is nonnegative
on~$\cQ_+$, then $\ell(\ff) = 0$, too.  Therefore the support face
of~\mbox{$\ff + \bb$} (the smallest face in which it lies)
contains~$\sigma$.  On the other hand, suppose $\bb$ lies interior to
some face of~$\cQ_+$ that contains~$\sigma$.  Then pick any $\ff \in
\sigma^\circ$.  If~$\ell(\bb) = 0$ then also $\ell(\ff) = 0$, because
the support face of~$\bb$ contains~$\sigma$.  But if $\ell(\bb) > 0$,
then $\ell(\bb) > \ell(\epsilon \ff)$ for any sufficiently small
positive~$\epsilon$.  As~$\cQ_+$ is an intersection of only finitely many
closed half-spaces, a single $\epsilon$ works for all relevant~$\ell$, and
then $\bb = \epsilon \ff + (\bb - \epsilon \ff) \in \sigma^\circ + \cQ_+$.
\end{proof}

For $\cQ_+ = \RR_+^n$ the following is essentially
\cite[Lemma~5.1]{madden-mcguire2015}.

\begin{cor}\label{c:<<}
If $D \subseteq \cQ$ is a downset in a real polyhedral group, then
$\aa$ lies in the closure $\oD$ if and only if $D$ contains the
interior $\aa - \qpc$ of the negative cone with~origin~$\aa$.
\end{cor}

\begin{prop}\label{p:shape}
If $\aa \in \oD$ for a downset $D$, then $T_\aa D = -\cQ_\nabla$ is
the negative cone of shape~$\nabla$ for some cocomplex~$\nabla$
in~$\cQ_+$.  In this case $\nabla = \nda$ is the \emph{shape} of~$D$
at\/~$\aa$.
\end{prop}
\begin{proof}
The result is true when $n = 1$ because there are only three
possibilities for $a \in \RR$: either $a \in D$, in which case $T_a D
= -\RR_+ = \cQ_\nabla$ for $\nabla = \cfq$
(Remark~\ref{r:tangent-cone}); or $a \not\in D$ but $a$ lies in the
closure of~$D$, in which case $T_a D = \cQ_\nabla$ for $\nabla =
\{\qpc\} \subseteq \cfq$; or $a$ is separated from~$D$ by a nonzero
distance, in which case $T_a D = \cQ_{\{\}}$ is empty.

Write $D_\sigma$ for the intersection of~$D$ with the $\aa$-translate
of the linear span of~$\sigma$:
$$%
  D_\sigma = D \cap (\aa + \RR \sigma).
$$
If $\sigma \subsetneqq \cQ_+$, then $T_\aa D_\sigma = \sigma_\nabla$
for some upset $\nabla \subseteq \cF_\sigma$ by induction on the
dimension of~$\sigma$.  In actuality, only the case $\dim \sigma =
n-1$ is needed, as the face posets $\cF_\sigma$ for $\dim \sigma =
n-1$ almost cover~$\cfq$: only the open maximal face $\qpc$ itself
lies outside of their union, and that case is dealt with by
Corollary~\ref{c:<<}.
\end{proof}

%%%%%%%%%%%%%%%%%%%%%%%%%%%%%%%%%%%%%%%%%%%%%%%%%%%%%%%%%%%%%%%%%%%%%%%%%
\subsection{Upper boundary functors}\label{sub:upper-boundary}\mbox{}%%%%

\begin{defn}\label{d:atop-sigma}
For a module~$\cM$ over a real polyhedral group~$\cQ$, a face~$\sigma$
of~$\cQ_+$, and a degree $\aa \in \cQ$, the \emph{upper boundary
atop~$\sigma$} at~$\aa$ in~$\cM$ is the vector space
$$%
  (\ds\cM)_\aa
  =
  \cM_{\aa-\sigma}
  =
  \dirlim_{\aa' \in \aa - \sigma^\circ} \cM_{\aa'}.
$$
\end{defn}

\begin{lemma}\label{l:exact-delta}
The functor $\cM \mapsto \ds\cM = \bigoplus_{\aa\in\cQ} (\ds\cM)_\aa$
is exact.
\end{lemma}
\begin{proof}
Direct limits are exact in categories of vector spaces (or modules
over rings).
\end{proof}

\begin{lemma}\label{l:natural}
The structure homomorphisms of~$\cM$ as a $\cQ$-module induce natural
homomorphisms $\cM_{\aa-\sigma} \to \cM_{\bb-\tau}$ for $\aa \preceq
\bb$ in~$\cQ$ and faces $\sigma \supseteq \tau$ of~$\cQ_+$.
\end{lemma}
\begin{proof}
The natural homomorphisms come from the universal property of
colimits.  First a natural homomorphism $\cM_{\aa-\sigma} \to
\cM_{\bb-\sigma}$ is induced by the composite homomorphisms $\cM_\cc
\to \cM_{\cc+\bb-\aa} \to \cM_{\bb - \sigma}$ for $\cc \in \aa -
\sigma^\circ$ because adding $\bb - \aa$ takes $\aa - \sigma^\circ$ to
\mbox{$\bb - \sigma^\circ$}.  For $\cM_{\bb-\sigma} \to
\cM_{\bb-\tau}$ the argument is similar, except that existence of
natural homomorphisms $\cM_\cc \to \cM_{\bb-\tau}$ for $\cc \in \bb -
\sigma^\circ$ requires Proposition~\ref{p:<<} and Lemma~\ref{l:<<}.
\end{proof}

\begin{remark}\label{r:semilattice=monoid}
The face poset~$\cfq$ of the positive cone~$\cQ_+$ can be made into a
commutative monoid in which faces $\sigma$ and~$\tau$ of~$\cQ_+$ have
sum
$$%
  \sigma + \tau = \sigma \cap \tau.
$$
Indeed, these monoid axioms use only that $\bigl(\cfq,\cap\bigr)$ is a
bounded meet semilattice, the monoid unit element being the maximal
semilattice element---in this case, $\cQ_+$ itself.  When $\cfq$ is
considered as a monoid in this way, the partial order on it has
$\sigma \preceq \tau$ if $\sigma \supseteq \tau$, which is the
opposite of the default partial order on the faces of a polyhedral
cone.  For utmost clarity, $\fqo$ is written when this monoid partial
order is intended.
\end{remark}

\begin{defn}\label{d:upper-boundary}
Fix a module~$\cM$ over a real polyhedral group~$\cQ$ and a degree
$\aa \in \cQ$.  The \emph{upper boundary functor} takes~$\cM$ to the
$\cQ \times \fqo$-module $\delta\cM$ whose fiber over $\aa \in \cQ$ is
the $\fqo$-module
$$%
  (\delta\cM)_\aa
  =
  \bigoplus_{\sigma \in \cfq} \cM_{\aa-\sigma}
  =
  \bigoplus_{\sigma \in \cfq} (\ds\cM)_\aa.
$$
The fiber of $\delta\cM$ over $\sigma \in \fqo$ is the \emph{upper
boundary~$\ds\cM\!$ of~$\cM$ atop~$\sigma$}.
\end{defn}

\begin{remark}\label{r:upper-frontier}
The face of~$\cQ_+$ that contains only the origin~$\0$ is an absorbing
element: it acts like infinity, in the sense that $\sigma + \{\0\} =
\{\0\}$ in the monoid~$\fqo$ for all faces~$\sigma$.  Adding the
absorbing element~$\0$ in the $\fqo$ component therefore induces a
natural\/ $\RR^n \times \fqo$-module projection from the upper
boundary $\delta\cM$ to~$\cM$.  At a degree $\aa \in \RR^n$, this
projection is $\cM_{\aa-\sigma} \to \cM_{\aa - \0} = \cM_\aa$.
Interestingly, the \emph{frontier} of a downset~$D$---those points in
the topological closure but outside of~$D$---is the set of nonzero
degrees of a functor, namely $\ker(\ds\cM \to \cM)$ for $\sigma =
\qpc$.  The proof is by Corollary~\ref{c:<<}.
\end{remark}

There is no natural map $\cM \to \ds\cM$ when $\sigma \neq \{\0\}$ has
positive dimension, because an element of degree~$\aa$ in~$\cM$ comes
from elements of~$\ds\cM$ in degrees less than~$\aa$.  However, that
leaves a way for Lemma~\ref{l:natural} to afford a notion of
divisibility of upper boundary elements by elements of~$\cM$.

\begin{defn}\label{d:divides}
An element $y \in \cM_\bb$ \emph{divides} $x \in (\ds\cM)_\aa$ if $\bb
\in \aa - \qns = \aa - \sigma^\circ - \cQ_+\!$ (Lemma~\ref{l:<<}) and
$y \mapsto x$ under the natural map $\cM_\bb \to \cM_{\aa-\sigma}$
(Lemma~\ref{l:natural}).  The element $y$ is said to
\emph{$\sigma$-divide}~$x$ if, more restrictively, $\bb \in \aa -
\sigma^\circ$.
\end{defn}

\begin{lemma}\label{l:downset-upper-boundary}
If $\sigma \in \cfq$ and $D$ is a downset in~$\cQ$ then $\ds\,\kk[D] =
\kk[\ds D]$,~where
$$%
  \ds D
  =
  \bigcup_{\xx\in\cQ} \ol{D \cap (\xx + \RR\sigma)}
  =
  D \cup \bigcup_{\xx\in\del D} \ol{D \cap (\xx + \RR\sigma)}
$$
It suffices to take the middle union over $\xx$ in any subspace
complement to~$\RR\sigma$.
\end{lemma}
\begin{proof}
For the second displayed equality, observe that the middle union
contains the right-hand union because the middle one contains~$D$.
For the other containment, if $\xx + \RR\sigma$ contains no boundary
point of~$D$, then $D \cap (\xx + \RR\sigma) = \oD \cap (\xx +
\RR\sigma)$ is already closed, so the contribution of $D \cap (\xx +
\RR\sigma)$ to the middle union is contained in~$D$.

For the other equality, $\ds\kk[D]$ is nonzero in degree $\aa$ if and
only if $\aa - \sigma^\circ \subseteq D$.  That condition is
equivalent to $\aa - \sigma^\circ \subseteq D \cap (\aa + \RR\sigma)$
because $\aa - \sigma^\circ \subseteq \aa + \RR\sigma$.  Translating
$D \cap (\aa + \RR\sigma)$ back by~$\aa$ yields a downset in the real
polyhedral group~$\RR\sigma$, with $(\RR\sigma)_+ = \sigma$, thereby
reducing to Corollary~\ref{c:<<}.
\end{proof}

\begin{prop}\label{p:downset-upper-boundary}
If $\sigma \in \cfq$ and $D$ is a downset in a real polyhedral group
then $\ds\,\kk[D] = \kk[\ds D]$ is the indicator quotient for a
downset $\ds D$ satisfying
\begin{enumerate}
\item\label{i:containment}%
$D \subseteq \ds D \subseteq \oD$;
\item\label{i:nabla}%
$\ds D = \{\aa\in \oD \mid \sigma \in \nda\}$; and
\item\label{i:semialgebraic}%
if $D$ is semialgebraic then so is~$\ds\,\kk[D]$.
\end{enumerate}
\end{prop}
\begin{proof}
Item~\ref{i:containment} follows from item~\ref{i:nabla}.  What
remains to show is that $\ds D$ is a downset in~$\oD$ characterized by
item~\ref{i:nabla} and that it is semialgebraic if $D$~is.

First, $\sigma \in \nda$ means that $\aa - \sigma^\circ \subseteq D$,
which immediately implies that $\aa \in \ds D$ by
Lemma~\ref{l:downset-upper-boundary}.  Conversely, suppose $\aa \in
\ds D$.  Lemma~\ref{l:downset-upper-boundary} and
Corollary~\ref{c:<<}, the latter applied to the downset $-\aa + D \cap
(\aa + \RR\sigma)$ in~$\RR\sigma$, imply that $\aa - \sigma^\circ
\subseteq D$, and hence $\sigma \in \nda$ by definition, proving
item~\ref{i:nabla}.  Given that $\aa - \sigma^\circ \subseteq D$,
Proposition~\ref{p:<<} yields $D \cap (\aa - \cQ_+) \supseteq \aa -
\qns$.  Consequently, if $\bb \in \aa - \cQ_+$ then $D \cap (\bb -
\sigma) \supseteq \bb - \sigma^\circ$, whence $\bb \in \ds D$.  Thus
$\ds D$ is a downset.

The semialgebraic claim follows from a general result,
Lemma~\ref{l:semialg-closure}, the case here being $Y = \qrs$ and $X =
D$ by Lemma~\ref{l:downset-upper-boundary}.
\end{proof}

\begin{lemma}\label{l:semialg-closure}
If $X \subseteq \RR^n$ and $X \to Y$ is a morphism of semialgebraic
varieties, then the family $\ol X_Y$ obtained by taking the closure
in~$\RR^n$ of every fiber of~$X$ is semialgebraic.
\end{lemma}
\begin{proof}
This is a consequence of Hardt's theorem \cite[Theorem~4]{hardt80}
(see also \cite[Remark~II.3.13]{shiota97}), which says that over a
subset of~$Y$ whose complement in~$Y$ has dimension less than $\dim
Y$, the family $X \to Y$ is trivial.
\end{proof}

\begin{prop}\label{p:frontier-encoding}
If $\cM$ is finitely encoded over a real polyhedral group then so is
its upper boundary~$\delta \cM$.  If $\cM$ is semialgebraic then so
is~$\delta \cM$.  The same are true with $\ds$ in place of~$\delta$
for a fixed~face~$\sigma$.
\end{prop}
\begin{proof}
Since the relevant categories are abelian by
Lemma~\ref{l:abelian-category}, it suffices to treat the $\ds$ case.
As the direct sums that form a $\cQ$-module from its graded pieces are
exact, as is the formation of $\delta\cM$ from~$\cM$
(Lemma~\ref{l:exact-delta}) the claim reduces to
Proposition~\ref{p:downset-upper-boundary} by
Theorem~\ref{t:syzygy}.\ref{i:downset-copresentation}.
\end{proof}

%%%%%%%%%%%%%%%%%%%%%%%%%%%%%%%%%%%%%%%%%%%%%%%%%%%%%%%%%%%%%%%%%%%%%%%%%
\subsection{Closed socles and closed cogenerator functors}\label{sub:socc}\mbox{}

\noindent
In commutative algebra, the socle of a module over a local ring is the
set of elements annihilated by the maximal ideal.  These elements form
a vector space over the residue field~$\kk$ that can alternately be
characterized by taking homomorphisms from~$\kk$ into the module.
Either characterization works for modules over partially ordered
groups, but only the latter generalizes readily to modules over
arbitrary posets.

\begin{defn}\label{d:socc}
Fix an arbitrary poset~$\cP$.  The \emph{skyscraper} $\cP$-module
$\kk_p$ at $p \in \cP$ has $\kk$ in degree~$p$ and $0$ in all other
degrees.  The \emph{closed cogenerator functor} $\hhom_\cP(\kk,-)$ takes
each $\cP$-module $\cM$ to its \emph{closed socle}: the $\cP$-submodule
$$%
  \socc\cM = \hhom_\cP(\kk,\cM) = \bigoplus_{p\in\cP} \Hom_\cP(\kk_p,\cM).
$$
When it is important to specify the poset, the notation $\cpsoc$ is
used instead of~$\socc\!$.  A \emph{closed cogenerator} of
\emph{degree}~$p \in P$ is a nonzero element in $(\socc\cM)_p$.
\end{defn}

The bar over ``$\soc$'' is meant to evoke the notion of closure or
``closed''.  The bar under $\Hom$ is the usual one in multigraded
commutative algebra for the direct sum of homogeneous homomorphisms of
all degrees (see \cite[Section~I.2]{GWii} or
\cite[Definition~11.14]{cca}, for example).

\begin{example}\label{e:socc}
The closed socle of~$\cM$ consists of the elements that are
annihilated by moving up in any direction, or that have maximal
degree.  In particular, the indicator quotient module $\kk[D]$ for any
downset $D \subseteq \cP$ has closed socle
$$%
  \socc \kk[D] = \kk[\max D],
$$
the indicator subquotient supported on the set of elements of~$D$ that
are maximal in~$D$.
\end{example}

\begin{lemma}\label{l:left-exact-socc}
The closed cogenerator functor over any poset is left-exact.
\end{lemma}
\begin{proof}
A $\cP$-module is the same thing as a module over the path algebra of
(the Hasse diagram of)~$\cP$ with relations to impose commutativity,
namely equality of the morphism induced by $p < p''$ and the composite
morphism induced by $p < p' < p''$ (see \cite{yuzvinsky1981}, for
example).  The purpose of vieweing things this way is merely to
demonstrate that the category of $\cP$-modules is a category of
modules (graded by~$\cP$) over a ring, where $\Hom$ from a fixed
source is automatically~left-exact.
\end{proof}

\begin{lemma}\label{l:semialgebraic}
If $D$ is a semialgebraic downset in a real polyhedral group then
$\max D$ is semialgebraic, as well.
\end{lemma}
\begin{proof}
The proof uses standard operations on semialgebraic subsets that
preserve the semialgebraic property; see \cite[Chapter~II]{shiota97},
for instance.

Inside of $\RR^n \times \RR^n$, consider the subset~$X$ whose fiber
over each point $\aa \in D$ is $\aa + \mm$, where $\mm = \cQ \minus
\{\0\}$ is the maximal monoid ideal of~$\cQ_+$.  Note that $\mm$ is
semialgebraic because it is defined by linear inequalities and a
single linear inequation.  The subset~$X \subseteq \RR^n \times \RR^n$
is semialgebraic because it is the image of the algebraic morphism $D
\times \mm \to D \times \RR^n$ sending $(\aa,q) \mapsto (\aa,\aa+q)$.
The intersection of~$X$ with the semialgebraic subset $D \times D$
remains semialgebraic, as does the projection of this intersection
to~$D$.  The image of the projection is $D \minus \max D$ because
$(\aa + \mm) \cap D = \nothing$ precisely when $\aa \in \max D$.
Therefore $\max D = D \minus (D \minus \max D)$ is semialgebraic.
\end{proof}

\begin{prop}\label{p:socc-encoding}
If a module $\cM$ over any poset is finitely encoded then so is its
closed socle $\socc\cM$.  If $\cM$ is semialgebraic over a real
polyhedral group then so is~$\socc\cM$.
\end{prop}
\begin{proof}
If $\cM$ is finitely encoded then $\cM = \ker(E^0 \to E^1)$ with $E^i$
a finite direct sum of indicator quotients by
Theorem~\ref{t:syzygy}.\ref{i:downset-copresentation}.  Thus $\socc\cM
= \ker(\socc E^0 \to \socc E^1)$ by Lemma~\ref{l:left-exact-socc}.
For the finitely encoded claim it therefore suffices, by
Lemma~\ref{l:abelian-category}, to prove that $\socc\kk[D]$ is
finitely encoded if~$\kk[D]$ is.  But Example~\ref{e:socc} implies
that $\socc\kk[D] = \kk[\max D] = \ker\bigl(\kk[D] \to \kk[D \minus
\max D]\bigr)$ is a downset copresentation, where $D \minus \max D$ is
a downset by definition.

The same argument shows the semialgebraic claim by
Lemma~\ref{l:semialgebraic}.
\end{proof}

%%%%%%%%%%%%%%%%%%%%%%%%%%%%%%%%%%%%%%%%%%%%%%%%%%%%%%%%%%%%%%%%%%%%%%%%%
\subsection{Socles and cogenerator functors}\label{sub:soc}\mbox{}%%%%%%%

\begin{defn}\label{d:soc}
The \emph{cogenerator functor} takes a module over a real polyhedral
group to its \emph{socle}:
$$%
  \soc\cM = \socc\delta\cM,
$$
the closed socle, computed over the poset $\cQ \times \fqo$, of the
upper boundary module~of~$\cM$.
\end{defn}

\begin{remark}\label{r:notation-soc-bar}
Notationally, the lack of a bar over ``$\soc$'' serves as a visual cue
that the functor is over a real polyhedral group, as the upper
boundary $\delta$ is not defined in more generality.  This visual cue
persists throughout the more general notions of socle.
\end{remark}

\begin{lemma}\label{l:soc-left-exact}
The cogenerator functor $\cM \mapsto \soc\cM$ is left-exact.
\end{lemma}
\begin{proof}
Use exactness of upper boundaries atop~$\sigma$
(Lemma~\ref{l:exact-delta}), exactness of the direct sums forming
$\delta\cM$ from $\ds\cM$, and left-exactness of closed socles
(Lemma~\ref{l:left-exact-socc}).
\end{proof}
  
Sometimes is it useful to apply the closed socle functor
to~$\delta\cM$ over $\cQ \times \fqo$ in two steps, first over one
poset and then over the other.  These yield the same result.

\begin{lemma}\label{l:either-order}
The functors $\rnsoc\!$ and $\fqsoc$ commute.  In particular,
$$%displaystyle
  \fqsoc(\rnsoc\delta\cM)
  \cong
  \soc\cM\!
  \cong
  \rnsoc(\fqsoc\delta\cM).
$$
\end{lemma}
\begin{proof}
By taking direct sums over $\aa$ and~$\sigma$, this follows from the
natural isomorphisms
$
  \Hom_{\fqo}\bigl(\kk_\sigma,\Hom_{\cQ}(\kk_\aa,-)\bigr)
  \cong
  \Hom_{\cQ\times\fqo}(\kk_{\aa,\sigma},-)
  \cong
  \Hom_\cQ\bigl(\kk_\aa,\Hom_{\fqo}(\kk_\sigma,-)\bigr).
$
\end{proof}

The fundamental examples---indicator quotients for downsets---require
a no(ta)tion.

\begin{defn}\label{d:del-nabla}
In the situation of Definition~\ref{d:cocomplex}, write $\del\nabla$
for the antichain of faces of~$\cQ_+$ that are minimal under inclusion
in~$\nabla$.
\end{defn}

\pagebreak
\begin{remark}\label{r:del-nda}
The reason for writing $\del\nabla$ instead of $\max$ or~$\min$ is
that it would be ambiguous either way, since both $\cfq$ and~$\fqo$
are natural here.  Taking the ``op'' perpsective, the $\fqo$-module
$\kk[\del\nabla]$ with basis~$\del\nabla$ resulting from
Definition~\ref{d:del-nabla} is really just a $\fqo$-graded vector
space: the antichain condition ensures that every non-identity element
of~$\fqo$ acts by~$0$, unless $\del\nabla = \{\0\}$, in which case all
of~$\fqo$ acts~by~$1$.
\end{remark}

The case of most interest here is $\nabla = \nda$, the shape of~$D$
at~$\aa$ (Proposition~\ref{p:shape}).

\begin{example}\label{e:soc-Rn-downset}
For a downset $D$ in a real polyhedral group, $\fqsoc\delta\kk[D]$ has
$\kk[\del\nda]$ in each degree~$\aa$, because $\delta\kk[D]$ itself
has $\kk[\nda]$ in each degree~$\aa$ by
Definition~\ref{d:upper-boundary} and
Proposition~\ref{p:downset-upper-boundary}.  What $\rnsoc$ then does
is find the degrees~$\aa \in \cQ$ maximal among those where $\sigma
\in \del\nda$, by
Proposition~\ref{p:downset-upper-boundary}.\ref{i:nabla} and
Example~\ref{e:socc}.

Taking socles in the other order, first $\rnsoc\kk[\ds D]$ asks
whether $\sigma \in \nda$ but $\sigma \not\in \nd[\bb]$ for any $\bb
\succ \aa$ in~$\cQ$.  That can happen even if $\sigma$ contains a
smaller face where it still happens.  What $\fqsoc$ then does is
return the smallest faces at~$\aa$ where it happens.
\end{example}

\begin{cor}\label{c:soc-Rn-downset}
The socle of the indicator quotient $\kk[D]$ for any downset $D$ in a
real polyhedral group~$\cQ$ is nonzero only in degrees lying in the
topological boundary~$\del D$.
\end{cor}
\begin{proof}
By Proposition~\ref{p:downset-upper-boundary}.\ref{i:containment},
$\delta\kk[D]$ is a direct sum of indicator quotients.
Example~\ref{e:socc} and Proposition~\ref{p:shape} show that the socle
of an indicator quotient over a real polyhedral group lies along the
boundary of the downset in question.
\end{proof}

\begin{prop}\label{p:soc-encoding}
If a module~$\cM$ over a real polyhedral group is finitely encoded
then so is its socle $\soc\cM$.  If~$\cM$ is semialgebraic then so
is~$\soc\cM$.
\end{prop}
\begin{proof}
By Lemma~\ref{l:either-order}, $\soc\cM$ is the closed socle of the
$\cQ$-module $\fqsoc\delta\cM$.  The $\fqo$-graded component of this
$\cQ$-module in degree~$\sigma$ is the intersection of the kernels of
the $\cQ$-module homomorphisms $\ds\cM \to
\delta^{\sigma\hspace{-1.1ex}}{}_{\sigma'\hspace{.1ex}}\cM$ for
$\sigma \supseteq \sigma'$.  Now apply
Proposition~\ref{p:frontier-encoding} and
Lemma~\ref{l:abelian-category}.
\end{proof}

%%%%%%%%%%%%%%%%%%%%%%%%%%%%%%%%%%%%%%%%%%%%%%%%%%%%%%%%%%%%%%%%%%%%%%%%%
\subsection{Closed socles along faces of positive dimension}\label{sub:socc-along}

\begin{defn}\label{d:qrt}
For a partially ordered abelian group~$\cQ$ and a face~$\tau$
of~$\cQ_+$, write $\qzt$ for the quotient of~$\cQ$ modulo the subgroup
generated by~$\tau$.  If $\cQ$ is a real polyhedral group then write
$\qrt = \qzt$.
\end{defn}

\begin{remark}\label{r:qrt}
The image $\qztp$ of~$\cQ_+\!$ in~$\qzt$ is a submonoid that generates
$\qzt$, but $\qztp$ can have units, so $\qzt$ need not be partially
ordered in a natural way.  However, if $\cQ$ is a real polyhedral
group then the group of units (lineality space) of the cone $\cQ_+ \!+
\RR\tau$ is just~$\RR\tau$ itself, because $\cQ_+$ is pointed, so
$\qrt$ is a real polyhedral group whose positive cone $(\qrt)_+ =
\qrtp$ is the image of~$\cQ_+$.  Similar reasoning applies to the
intersection of the real polyhedral situation with any subgroup
of~$\cQ$; this includes the case of normal affine semigroups, where
the subgroup of~$\cQ$ is discrete.
\end{remark}

\pagebreak
\begin{lemma}\label{l:quotient-restriction}
The subgroup $\ZZ\tau \subseteq \cQ$ of a partially ordered
group~$\cQ$ acts freely on~the localization $\cM_\tau$ of any
$\cQ$-module~$\cM$ along a face~$\tau$.  Consequently, if $I_\tau
\subseteq \kk[\cQ_+]$ is the augmentation ideal $\<m - 1 \mid m \in
\kk[\tau]$ is a monomial\/$\>$, then the $\qzt$-graded module $\cmt =
\cM/I_\tau\cM$ over the monoid algebra $\kk[\qztp]$ satisfies
$$%
  \cM_\tau \cong \bigoplus_{\aa \mapsto \wt\aa} (\cM/\tau)_\wt\aa.
$$
\end{lemma}
\begin{proof}
The monomials of $\kk[\cQ_+]$ corresponding to elements of~$\tau$ are
units on~$\cM_\tau$ acting as translations along~$\tau$.  Since the
augmentation ideal sets every monomial equal to~$1$, the quotient $\cM
\to \cmt$ factors through~$\cM_\tau$.
\end{proof}

\begin{defn}\label{d:quotient-restriction}
The $\kk[\qzt]$-module $\cM_\tau$ in
Lemma~\ref{l:quotient-restriction} is the \emph{quotient-restriction}
$\cmt$ of~$\cM$ along~$\tau$.
\end{defn}

\begin{remark}\label{r:quotient-functor}
Over (any subgroup of) a real polyhedral group~$\cQ$, the functor
$\cM_\tau \mapsto \cmt$ has a ``section'' $\cmt \mapsto
\cM_\tau|_\tau{}_{{}^{\!\perp}}$, where $\cN|_{\tau^\perp} =
\bigoplus_{\aa\in\tau^\perp} \cN_\aa$ is the \emph{restriction}
of~$\cN$ to any linear subspace~$\tau^\perp$ complementary
to~$\RR\tau$.  (When $\cQ_+ = \RR^n_+$, a complement is canonically
spanned by the face orthogonal to~$\tau$, or equivalently, the unique
maximal face of~$\RR^n_+$ intersecting~$\tau$ trivially.)  The
restriction is a module over the real polyhedral group $\tau^\perp$
with positive cone $(\cQ_+ \!+ \RR\tau) \cap \RR\tau^\perp$, which
projects isomorphically to the positive cone of~$\qrt$.  Thus the
quotient-restriction is both a quotient and a restriction
of~$\cM_\tau$.  While a section can exist over polyhedral partially
ordered groups that are not real, it need not.  For example, when $\cQ
= \ZZ^2$ and the columns of $\big[\twoline {2\ 1\ 0}{0\ 1\ 1}\big]$
generate~$\cQ_+$, taking $\tau = \big\<\big[\twoline 20\big]\big\>$ to
be the face along the $x$-axis yields a quotient monoid $\cQ_+/\ZZ\tau
\cong \ZZ/2\ZZ \oplus \NN$ with torsion, preventing $\kk[\cQ_+]_\tau
\mapsto \kk[\cQ_+]/\tau$ from having a section to any category of
modules over~a~subgroup~of~$\cQ$.
\end{remark}

\begin{lemma}\label{l:exact-qr}
The quotient-restriction functors $\cM \mapsto \cM/\tau$ are exact.
\end{lemma}
\begin{proof}
Localizing along~$\tau$ (Definition~\ref{d:support}) is exact because
the localization $\kk[\cQ_+\! +\nolinebreak \ZZ\tau]$ of~$\kk[\cQ_+]$
is flat as a $\kk[\cQ_+]$-module.  The exactness of the functor that
takes each $\kk[\qztp]$-module $\cM_\tau$ to~$\cmt$ can be checked on
each $\qzt$-degree individually.
\end{proof}

The definition of closed socle and closed cogenerator are expressed in
terms of Hom functors analogous to those in Definition~\ref{d:socc}.
They are more general in that they occur along faces of~$\cQ$, but
more restrictive in that $\cQ$ needs to be a partially ordered group
instead of an arbitrary poset for the notion of face to make sense.

\begin{defn}\label{d:hhom}
Fix a face~$\tau$ of a partially ordered group~$\cQ$.  The
\emph{skyscraper} $\cQ$-module at $\aa \in \cQ$ along~$\tau$ is
$\kk[\aa+\tau]$, the subquotient $\kk[\aa + \cQ_+]/\kk[\aa +
\mm_\tau]$ of~$\kk[\cQ]$, where $\mm_\tau = \cQ_+\! \minus \tau$.  Set
$$%
  \hhom_\cQ\bigl(\kk[\tau],-\bigr)
  =
  \bigoplus_{\aa \in \cQ} \Hom_\cQ\bigl(\kk[\aa+\tau],-\bigr).
$$
\end{defn}

\pagebreak
\begin{defn}\label{d:socct}
Fix a partially ordered group~$\cQ$, a face~$\tau$, and a
$\cQ$-module~$\cM$.
\begin{enumerate}
\item\label{i:global-socc-tau}%
The \emph{(global) closed cogenerator functor along~$\tau$} takes
$\cM$ to its \emph{(global) closed socle along~$\tau$}: the
$\kk[\qztp]$-module
$$%
  \socct\cM = \hhom_\cQ\bigl(\kk[\tau],\cM\bigr)/\tau.
$$

\item\label{i:local-socc-tau}%
If $\qzt$ is partially ordered, then the \emph{local closed
cogenerator functor along~$\tau$} takes $\cM$ to its \emph{local
closed socle along~$\tau$}: the $\qzt$-module
$$%
  \socc(\cmt) = \hhom_{\qzt}(\kk,\cM/\tau).
$$
Elements of $\socc(\cmt)$ are identified with elements of~$\cmt$ via
$\phi \mapsto \phi(1)$.

\item\label{i:global-closed-cogen}%
Regard the $\cQ$-module $\hhom_\cQ\bigl(\kk[\tau],\cM\bigr)$ naturally
as contained in~$\cM$ via \mbox{$\phi \mapsto \phi(1)$}.  A
homogeneous element in this $\cQ$-submodule that maps to a nonzero
element of~$\socct\cM$ is a \emph{(global) closed cogenerator}
of~$\cM$ along~$\tau$.  If $D \subseteq \cQ$ is a downset, then a
\emph{closed cogenerator} of~$D$ is the degree in~$\cQ$ of a closed
cogenerator of~$\kk[D]$.

\item\label{i:local-closed-cogen}%
Regard $\socc(\cmt)$ naturally as contained in~$\cmt$ via $\phi
\mapsto \phi(1)$.  A nonzero homogeneous element in $\socc(\cmt)$ is a
\emph{local closed cogenerator} of~$\cM$ along~$\tau$.
\end{enumerate}
\end{defn}

\begin{remark}\label{r:notation-soc-tau}
Notationally, a subscript on ``$\soc$'' serves as a visual cue that
the functor is over a partially ordered group, as faces of posets are
not defined in more generality.  This visual cue persists throughout
the more general notions of~socle.
\end{remark}

\begin{remark}\label{r:socc0}
The closed cogenerator functor over a partially ordered group is the
global closed cogenerator functor along the trivial face: $\socc =
\socc[\{\0\}]$ and it equals the local cogenerator functor
along~$\{\0\}$.
\end{remark}

\begin{remark}\label{r:witness}
In looser language, a closed cogenerator of~$\cM$ along~$\tau$ is an
element
\begin{itemize}
\item%
annihilated by moving up in any direction outside of~$\tau$ but that
\item%
remains nonzero when pushed up arbitrarily along~$\tau$.
\end{itemize}
Equivalently, a closed cogenerator along~$\tau$ is an element whose
annihilator under the action of~$\cQ_+$ on~$\cM$ equals the prime
ideal $\mm_\tau = \cQ_+ \minus \tau$ of the positive cone~$\cQ_+$.
\end{remark}

\begin{example}\label{e:socct}
The closed socle along a face~$\tau$ of the indicator quotient
$\kk[D]$ for any downset~$D$ in a partially ordered group~$\cQ$ with
partially ordered quotient $\qzt$ is
$$%
  \socct\, \kk[D] = \kk[\mtd],
$$
where $\mtd$ is the image in~$\qzt$ of the set of closed cogenerators
of~$D$ along~$\tau$:
$$%
  \mtd =
  \big\{\aa \in D \mid (\aa+\cQ_+) \cap D = \aa + \tau\big\}/\,\ZZ\tau.
$$
The set of closed cogenerators of~$D$ along~$\tau$ can also be
characterized as the elements of~$D$ that become maximal in the
localization~$D_\tau$ of~$D$
(Definition~\ref{d:PF}.\ref{i:localization}).
\end{example}

Every global closed cogenerator yields a local one.

\begin{prop}\label{p:local-vs-global-closed}
Fix a partially ordered group~$\cQ$.  There is a natural injection
$$%
  \socct\cM \into \socc(\cmt)
$$
for any $\cQ$-module~$\cM$ if $\tau$ is a face with partially ordered
quotient~$\qzt$.
\end{prop}
\begin{proof}
Localizing any homomorphism $\kk[\aa+\tau] \to \cM$ along~$\tau$
yields a homomorphism $\kk[\aa+\ZZ\tau] \to \cM_\tau$, so
$\hhom_\cQ\bigl(\kk[\tau],\cM\bigr){}_\tau$ is naturally a submodule
of $\hhom_\cQ\bigl(\kk[\ZZ\tau],\cM_\tau\bigr)$.  The claim now
follows from Lemma~\ref{l:exact-qr} and the next result.
\end{proof}

\begin{lemma}\label{l:socct}
If $\cQ$ and $\qzt$ are partially ordered, there is a canonical
\mbox{isomorphism}
$$%
  \hhom_\cQ\bigl(\kk[\ZZ\tau],\cM_\tau\bigr)/\tau
  \cong
  \hhom_{\qzt}(\kk,\cM/\tau).
$$
\end{lemma}
\begin{proof}
Follows from the definitions, using that $\kk[\ZZ\tau]/\tau = \kk$ in
$(\qzt)$-degree~$\0$.
\end{proof}

The following crucial remark highlights the difference between
real-graded algebra and integer-graded algebra.  It is the source of
much of the subtlety in the theory developed in this paper,
particularly Sections~\ref{s:socle}--\ref{s:hulls}.

\begin{remark}\label{r:soc-vs-supp}
In contrast with taking support on a face
(Proposition~\ref{p:support-localizes}) and also with socles in
commutative algebra over noetherian local or graded rings (e.g.,
Definition~\ref{d:cogenerator}), localization need not commute with
taking closed socles along faces of positive dimension in real
polyhedral groups.  In other words, the injection in
Proposition~\ref{p:local-vs-global-closed} need not be surjective:
there can be local closed cogenerators that do not lift to global
ones.  The problem comes down to the homogeneous prime ideals of the
monoid algebra $\kk[\cQ_+]$ not being finitely generated, so the
quotient $\kk[\tau]$ fails to be finitely presented; it means that
$\Hom_{\kk[\cQ_+]}\bigl(\kk[\tau],-\bigr)$ need not commute with $A
\otimes_{\kk[\cQ_+]}-$, even when $A$ is a flat $\kk[\cQ_+]$-algebra
such as a localization of~$\kk[\cQ_+]$.  The context of
$\RR^n$-modules complicates the relation between support on~$\tau$ and
closed cogenerators along~$\tau$ because the ``thickness'' of the
support can approach~$0$ without ever quantum jumping all the way
there and, importantly, remaining there along an entire translate
of~$\tau$, as it would be forced to for a discrete group like~$\ZZ^n$.
See Examples~\ref{e:PF} and~\ref{e:PF'}, for instance, where the
support on the $x$-axis has no closed socle along the $x$-axis.  This
issue is independent of the density phenomenon explored in
Section~\ref{s:minimality}; indeed, the downset in Example~\ref{e:PF}
is closed, so its socle equals its closed socle and is closed.
\end{remark}

\begin{prop}\label{p:left-exact-tau-closed}
The global closed cogenerator functor $\socct\!$ along any face~$\tau$
of a partially ordered group is left-exact, as is the local version if
$\qzt$ is partially ordered.
\end{prop}
\begin{proof}
For the global case, $\hhom_Q(\kk[\tau],-)$ is exact because it occurs
in the category of graded modules over the monoid
algebra~$\kk[\cQ_+]$, and quotient-restriction is exact by
Lemma~\ref{l:exact-qr}.  For the local case, use exactness of $\cM
\mapsto \cmt$ again (Lemma~\ref{l:exact-qr}) and left-exactness of
closed socles (Lemma~\ref{l:left-exact-socc}), the latter applied
over~$\qzt$.
\end{proof}

\begin{lemma}\label{l:semialgebraic'}
If $D$ is a semialgebraic downset in a real polyhedral group~$\cQ$
then $\mtd$ is semialgebraic in~$\qrt$ for any face~$\tau$
of~$\cQ_+$.
\end{lemma}
\begin{proof}
The projection of a semialgebraic set is semialgebraic, so by
Example~\ref{e:socct} it suffices to prove that the set of degrees of
closed cogenerators of~$\kk[D]$ along~$\tau$ is semialgebraic.  The
argument comes in two halves, both following the framework of the
proof of Lemma~\ref{l:semialgebraic}.  For the first half, simply
replace $\mm$ by $\mm_\tau = \cQ_+ \minus \tau$ to find that
$\big\{\aa \in D \mid (\aa + \cQ_+) \cap D \subseteq \aa + \tau\big\}$
is semialgebraic.  The second half uses $\tau$ instead of~$\mm$, and
this time it intersects the subset~$X$ with $D \times (\cQ \minus D)$
to find that $\big\{\aa \in D \mid (\aa + \cQ_+) \cap D \supseteq \aa
+ \tau\big\}$ is semialgebraic.  The desired set of degrees is the
intersection of these two semialgebraic sets.
\end{proof}

\begin{prop}\label{p:socct-encoding}
If a module $\cM$ over a partially ordered group is finitely encoded
then so is its closed socle $\socct\cM$ along any face~$\tau$.  If
$\cM$ is semialgebraic over a real polyhedral group then so
is~$\socct\cM$.
\end{prop}
\begin{proof}
If $\cM$ is finitely encoded then $\cM = \ker(E^0 \to E^1)$ with $E^i$
a finite direct sum of indicator quotients.  Thus $\socct\cM =
\ker(\socct E^0 \to \socct E^1)$ by
Proposition~\ref{p:left-exact-tau-closed}.  For the finitely encoded
claim it therefore suffices, by Lemma~\ref{l:abelian-category}, to
prove that $\socct\,\kk[D]$ is finitely encoded if~$\kk[D]$ is.
Writing $\dzt$ for the image of~$D$ in $\qzt$, Example~\ref{e:socct}
implies that $\socct\,\kk[D] = \ker\bigl(\kk[\dzt] \to \kk[\dzt \minus
\mtd]\bigr)$ is an indicator copresentation, where the set $\dzt \minus
\mtd$ is downset because $\mtd$ is contained in the set $\max(\dzt)$
of maximal elements of~$\dzt$.

The same argument shows the semialgebraic claim by
Lemma~\ref{l:semialgebraic'}.
\end{proof}

\begin{remark}\label{r:along}
Closed socles, without reference to faces, work over arbitrary posets
and are actually used that way in this work (over $\fqo$, for
instance, in Section~\ref{sub:soc}).  That explains why this separate
section on closed socles along faces of positive dimension is
required, instead of simply doing Section~\ref{sub:socc} in this
specificity in the first~place.
\end{remark}

%%%%%%%%%%%%%%%%%%%%%%%%%%%%%%%%%%%%%%%%%%%%%%%%%%%%%%%%%%%%%%%%%%%%%%%%%
\subsection{Socles along faces of positive dimension}\label{sub:soc-along}

\begin{lemma}\label{l:nabt}
If $\tau$ is a face of a real polyhedral group~$\cQ$ then the face
poset of the quotient real polyhedral group $\qrt$ is isomorphic to
the open star $\nabt$ from Example~\ref{e:nabla} by the map $\nabt \to
(\qrt)_+$ sending $\sigma \in \nabt$ to its image $\sigma/\tau$
in~$\qrt$.
\end{lemma}
\begin{proof}
See Remark~\ref{r:qrt}.
\end{proof}

\begin{defn}\label{d:upper-boundary-tau}
In the situation of Lemma~\ref{l:nabt}, endow $\nabt$ with the monoid
and poset structures from Remark~\ref{r:semilattice=monoid}, so
$\sigma \preceq \sigma'$ in~$\nabt$ if $\sigma \supseteq \sigma'$.
The \emph{upper boundary functor along~$\tau$} takes~$\cM$ to the $\cQ
\times \nabt$-module $\dt\cM = \bigoplus_{\sigma\in\nabt}\ds[\tau]\cM
= \delta\cM/\bigoplus_{\sigma\not\in\nabt}\ds\cM$.
\end{defn}

The notation is such that $\ds[\tau] \neq 0 \iff \sigma \supseteq
\tau$.

\begin{defn}\label{d:kats}
Fix a partially ordered group~$\cQ$, a face~$\tau$, and an arbitrary
commutative monoid~$\cP$.  The \emph{skyscraper} $(\cQ \times
\cP)$-module at $(\aa,\sigma) \in \cQ \times \cP$ along~$\tau$ is
$$%
  \kats = \kk[\aa+\tau] \otimes_\kk \kk_\sigma,
$$
the right-hand side being a module over the ring $\kk[\cQ_+]
\otimes_\kk \kk[\cP] = \kk[\cQ_+ \times \cP]$ with tensor factors as
in Definitions~\ref{d:socc} and~\ref{d:hhom}.
Set
$$%
  \hhom_{\cQ\times\cP}\bigl(\kk[\tau],-\bigr) =
  \bigoplus_{(\aa,\sigma) \in
  \cQ\times\cP} \Hom_{\cQ\times\cP}\bigl(\kats,-\bigr).
$$
\end{defn}

\begin{remark}\label{r:kats}
When $\cP$ is trivial, this notation agrees with
Definition~\ref{d:hhom}, because $\cQ\times\{\0\} \cong \cQ$
canonically, so $\hhom_{\cQ\times\{\0\}}\bigl(\kk[\tau],-\bigr) =
\hhom_\cQ\bigl(\kk[\tau],-\bigr)$.
\end{remark}

%pagebreak
\begin{defn}\label{d:soct}
Fix a real polyhedral group~$\cQ$, a face~$\tau$, and a
$\cQ$-module~$\cM$.
\begin{enumerate}
\item\label{i:global-soc-tau}%
The \emph{(global) cogenerator functor along~$\tau$} takes $\cM$ to
its \emph{(global) socle along~$\tau$}:
$$%
  \soct\cM = \hhom_{\cQ\times\nabt}\bigl(\kk[\tau],\dt\cM\bigr)/\tau.
$$
The $\nabt$-graded components of $\soct\cM$ are denoted by
$\soct[\sigma]\cM$ for $\sigma \in \nabt$.

\item\label{i:local-soc-tau}%
The \emph{local cogenerator functor along~$\tau$} takes $\cM$ to its
\emph{local socle along~$\tau$}:
$$%
  \soc(\cmt)
  =
  \socc\delta(\cmt)
  =
  \hhom_{\qrt\times\nabt}\bigl(\kk,\delta(\cM/\tau)\bigr),
$$
where the upper boundary is over~$\qrt$ and the closed socle is over
$\qrt \times \nabt$.  Elements of $\soc(\cmt)$ are identified with
elements of~$\delta(\cmt)$ via $\phi \mapsto \phi(1)$.

\item\label{i:global-cogen}%
Regard $\hhom_{\cQ\times\nabt}\bigl(\kk[\tau],\dt\cM\bigr)$ as a
$(\cQ\times\nabt)$-submodule of~$\dt\cM$ via $\phi \mapsto \phi(1)$.
A homogeneous element $s$ in this submodule that maps to a nonzero
element of $\soct\cM$ is a \emph{(global) cogenerator} of~$\cM$
along~$\tau$, and if $s \in \ds[\tau]\cM$ then it has
\emph{nadir}~$\sigma$.  If $D \subseteq \cQ$ is a downset, then a
\emph{cogenerator} of~$D$ along~$\tau$ with nadir~$\sigma$ is the
degree in~$\cQ$ of a cogenerator of~$\kk[D]$ with nadir~$\sigma$
along~$\tau$.

\item\label{i:local-cogen}%
Regard $\soc(\cmt)$ as contained in~$\delta(\cmt)$ via $\phi \mapsto
\phi(1)$.  A nonzero homogeneous element in $\soc(\cmt)$ is a
\emph{local cogenerator} of~$\cM$ along~$\tau$.
\end{enumerate}
\end{defn}

\begin{remark}\label{r:soct/tau}
The reason to quotient by~$\tau$ in
Definition~\ref{d:soct}.\ref{i:global-soc-tau} is to lump together all
cogenerators with nadir~$\sigma$ along the same translate
of~$\RR\tau$.  This lumping makes it possible for a socle basis to
produces a downset hull that is (i)~as minimal as possible and
(ii)~finite.  The lumping also creates a difference between the notion
of socle element and that of cogenerator: a socle element is a class
of cogenerators, these classes being indexed by elements in the
quotient-restriction.  In contrast, a local cogenerator is a
cogenerator of the quotient-restriction itself, so a local cogenerator
is already an element in the socle of the quotient-restriction.  This
difference between socle element and cogenerator already arises for
closed socles along faces (Definition~\ref{d:socct}) but disappears in
the context of socles not along faces (see Remark~\ref{r:socc0}), be
they over real polyhedral groups (Definition~\ref{d:soc}) or closed
over posets (Definition~\ref{d:socc}).
\end{remark}

\begin{remark}\label{r:soct-nabt}
If localization commuted with cogenerator functors, then the
restriction from $\fqo$ to~$\nabt$ in
Definition~\ref{d:soct}.\ref{i:global-soc-tau} would happen
automatically, because localizing $\cM$ along~$\tau$ would yield a
module over $\cQ_+\! + \RR\tau$, whose face poset is
naturally~$\nabt$.  But in this real polyhedral setting, the
restriction from $\fqo$ to~$\nabt$ must be imposed manually because
the $\Hom$ must be taken before localizing
(Remark~\ref{r:soc-vs-supp}), when the default face poset is
still~$\fqo$.
% This is crucial: the socle of~$\cM$ along~$\tau$ does not witness
% injections from the localization $\cM_\tau$ into anything; it
% witnesses the nonvanishing of maps to downset modules (defn: direct
% sums of indicator quotients) directly from subsets of~$\cM$, namely
% the \emph{basins} of the cogenerators in question.
\end{remark}

\begin{remark}\label{r:nabt}
If $\aa$ is a cogenerator of~$D$ along~$\tau$, then the topology
of~$D$ at~$\aa$ is induced by downsets of the form $\aa' -
\sigma^\circ$ for faces $\sigma \in \nabt$ and elements $\aa' \in \aa
+ \tau^\circ$.  This subtle issue regarding shapes of cogenerators
along~$\tau$ is a vital reason for using $\nabt$ instead of~$\fqo$.
It is tempting to expect that if a face~$\sigma$ is minimal in the
shape~$\nda$, then any expression of~$D$ as an intersection of
downsets must induce the topology of~$D$ at~$\aa$ by explicitly taking
$\aa - \sigma^\circ$ into account in one of the intersectands.  One
way to accomplish that would be for an intersectand to be a union of
downsets of the form $\bb - \cQ_{\nda}$ (see
Definition~\ref{d:cocomplex}) in which one of the elements~$\bb$
is~$\aa$.  But if $\sigma \in \nd[\aa']$ for all $\aa' \in \aa +
\tau$, or even merely for a single element $\aa' \in \aa +
\tau^\circ$, then
$$%
  \aa - \sigma^\circ
  =
  \aa' - (\aa'-\aa - \sigma^\circ)
  \in
  \aa' - (\tau^\circ + \sigma^\circ)
  \subseteq
  \aa' - (\tau \vee \sigma)^\circ.
$$
As the purpose of cogenerators is to construct downset
decompositions as minimally as possible, it is counterproductive to
think of~$\sigma$ as being a valid $\fqo$-socle degree unless $\sigma
\in \nabt$, because otherwise it fails to give rise to an essential
cogenerator.  See Theorem~\ref{t:downset=union} for the most general
possible view of considerations in this Remark.
\end{remark}

\begin{remark}\label{r:soc-as-k-vect}
Although $\soct\cM$ is a module over $\qrt \times \nabt$ by
construction, the actions of $\qrt$ and $\nabt$ on it are trivial, in
the sense that attempting to move a nonzero homogeneous element up in
one of the posets either takes the element to~$0$ or leaves it
unchanged.  (The latter only happens if the degree is unchanged, which
occurs only when acting by the identity~$\0 \in \qrt$ or when acting
by $\sigma \in \nabt$ on an element of $\nabt$-degree $\sigma'
\subseteq \sigma$.)  That is what it means to be a direct sum of
skyscraper modules.  It implies that any direct sum decomposition of
$\soct\cM$ as a vector space graded by $\qrt \times \nabt$ is also a
decomposition of $\soct\cM$ as a $\qrt$-module or as a $\nabt$-module.
\end{remark}

\begin{lemma}\label{l:intersection-of-kernels}
If~$\tau$ is a face of a is a real polyhedral group~$\cQ$ and $\cN =
\bigoplus_{\sigma\in\nabt} \cN_\sigma$ is a module over $\cQ \times
\nabt$, then $\Hom_{\nabt}(\kk_\sigma,\cN)/\tau \cong
\Hom_{\nabt}(\kk_\sigma,\cnt)$, and hence
$$%
  (\ntsoc\cN)/\tau \cong \ntsoc(\cnt).
$$
\end{lemma}
\begin{proof}
$\Hom_{\nabt}(\kk_\sigma,\cN)$ is the intersection of the kernels of
the $\cQ$-module homomorphisms $\cN_\sigma \to \cN_{\sigma'}$ for
faces $\sigma \supset \sigma'$, so the isomorphism of $\Hom$ modules
follows from Lemma~\ref{l:exact-qr}.  The socle isomorphism follows by
taking the direct sum over $\sigma \in \nabt$.
\end{proof}

\begin{prop}\label{p:either-order}
The functors $\socct\!$ and $\ntsoc$ commute.  In particular,
$$%displaystyle
  \ntsoc(\socct\dt\cM)
  \cong
  \soct\cM
  \cong
  \socct(\ntsoc\dt\cM).
$$
\end{prop}
\begin{proof}
By taking direct sums over $\aa$ and~$\sigma$, this is mostly the
natural isomorphisms
\begin{align*}
\Hom_{\nabt}\bigl(\kk_\sigma,\Hom_\cQ\bigl(\kk[\aa+\tau],-\bigr)\bigr)
  &\cong\Hom_{\cQ\times\nabt}\bigl(\kats,-\bigr)
\\
  &\cong\Hom_\cQ\bigl(\kk[\aa+\tau],\Hom_{\nabt}(\kk_\sigma,-)\bigr)
\end{align*}
that result from the adjunction between Hom and~$\otimes$.  Taking the
quotient-restriction along~$\tau$
(Definition~\ref{d:quotient-restriction}) almost yields the desired
result; the only issue is that the left-hand side requires
Lemma~\ref{l:intersection-of-kernels}.
\end{proof}

\begin{example}\label{e:soct}
If $\aa$ is a cogenerator of a downset $D \subseteq \cQ$ along~$\tau$
with nadir~$\sigma$, then reasoning as in
Example~\ref{e:soc-Rn-downset} and using Definition~\ref{d:del-nabla},
computing $\ntsoc$ first in Proposition~\ref{p:either-order} shows
that $\sigma \in \del(\nda \cap \nabt)$.  What $\socct$ then does is
verify that the image $\wt\aa$ of~$\aa$ in~$\qrt$ is maximal with this
property, by Example~\ref{e:socct}.
\end{example}

\begin{prop}\label{p:local-vs-global}
There is a natural injection
$$%
  \soct\cM \into \soc(\cmt)
$$
for any module~$\cM$ over a real polyhedral group~$\cQ$ and any face
$\tau$ of~$\cQ$.
\end{prop}
\begin{proof}
By Proposition~\ref{p:local-vs-global-closed} $\socct\cN \into
\socc(\cnt)$ for $\cN = \ntsoc\dt\cM$ viewed as a $\qrt$-module.
Proposition~\ref{p:either-order} yields $\socct\cN = \soct\cM$.  It
remains to show that $\qtsoc(\cnt) = \soc(\cmt)$.  To that end, first
note that
$$%
  (\ntsoc\dt\cM)/\tau
  \cong
  \ntsoc\bigl((\dt\cM)/\tau\bigr)
  \cong
  \ntsoc\delta(\cM/\tau),
$$
the first isomorphism by Lemma~\ref{l:intersection-of-kernels} and the
second by Lemma~\ref{l:ds-vs-qr}, which shows that the modules acted
on by $\ntsoc$ are isomorphic.  Now apply the last isomorphism in
Lemma~\ref{l:either-order}, with $\cQ$ replaced by~$\qrt$ so that
automatically $\fqo$ must be replaced by~$\nabt$ via
Lemma~\ref{l:nabt}.
\end{proof}

\begin{lemma}\label{l:ds-vs-qr}
If $\sigma \supseteq \tau$ then $(\ds\cM)/\tau \cong \dst(\cmt)$.
\end{lemma}
\begin{proof}
\hspace{-1.205pt}Explicit calculations from the definitions show that in
degree~$\aa/\hspace{-1pt}\tau$ both sides~equal
$$%displaystyle
  \dirlim_{\substack{\aa'\in\aa-\sigma^\circ\\\vv\in\tau}}\cM_{\aa'+\vv},
$$
although they take the colimits in different orders: $\vv$ first or
$\aa'$ first.  The hypothesis that $\sigma \supseteq \tau$ enters to
show that any direct limit over $\{\aa' \in \cQ \mid
\aa'/\tau\in\aa/\tau-(\sigma/\tau)^\circ\}$ can equivalently be
expressed as a direct limit over $\aa'\in\aa-\sigma^\circ$.
\begin{excise}{%
  And here are the gory details:
  % Ugh: what is each side in degree $\aa/\tau$?
  $$%
    \bigl((\ds\cM)/\tau\bigr){}_{\aa/\tau}
    =
    \bigl((\ds\cM)_\tau\bigr){}_\aa
    =
    \dirlim_{\vv\in\tau}\bigl(\dirlim_{\aa'\in\aa+\vv-\sigma^\circ}\cM_{\aa'}\bigr)
    =
    \dirlim_{\vv\in\tau}\bigl(\dirlim_{\aa'\in\aa-\sigma^\circ}\cM_{\aa'+\vv}\bigr).
  $$
  and
  \begin{align*}
  \bigl(\dst(\cM/\tau)\bigr){}_{\aa/\tau}
  & =
    \dirlim_{\aa'/\tau\in\aa/\tau-(\sigma/\tau)^\circ}(\cmt)_{\aa'/\tau}
  \\
  & =
    \dirlim_{\aa'/\tau\in\aa/\tau-(\sigma/\tau)^\circ}(\cM_\tau)_{\aa'}
  \\
  & =
    \dirlim_{\aa'/\tau\in\aa/\tau-(\sigma/\tau)^\circ}\bigl(
    \dirlim_{\vv\in\tau}\cM_{\aa'+\vv}\bigr)
  \\
  & =
    \dirlim_{\aa'\in\aa-\sigma^\circ}\bigl(
    \dirlim_{\vv\in\tau}\cM_{\aa'+\vv}\bigr),
  \end{align*}
  the final line being the only one not completely obviously directly
  from the definitions, since it uses that $\sigma \supseteq \tau$.
}\end{excise}%
\end{proof}

\begin{cor}\label{c:at-most-one}
An indicator quotient for a downset in a real polyhedral group has at
most one linearly independent socle element along each face with given
nadir and~degree.  In fact, the degrees of independent socle elements
along~$\tau$ with fixed nadir are incomparable in~$\qrt$, and nadirs
of socle elements with fixed degree are
\mbox{incomparable}~in~$\nabt$.
\end{cor}
\begin{proof}
A socle element of an indicator quotient~$E$ along a face~$\tau$
of~$\cQ$ is a local socle element of~$E$ along~$\tau$ by
Proposition~\ref{p:local-vs-global}.  Local socle elements
along~$\tau$ are socle elements (along the minimal face~$\{\0\}$) of
the quotient-restriction along~$\tau$ by
Definition~\ref{d:soct}.\ref{i:local-soc-tau}.  But $E/\tau$ is an
indicator quotient of~$\kk[\qrt]$, so its socle degrees with fixed
nadir~$\sigma$ are incomparable, as are its nadirs with fixed socle
degrees, by~Example~\ref{e:soc-Rn-downset}.
\end{proof}

\begin{example}\label{e:soct-k[tau]}
Propositions~\ref{p:either-order} and~\ref{p:local-vs-global} ease
some socle computations.  To see how, consider the indicator
$\cQ$-module~$\kk[\rho]$ for a face~$\rho$ of~$\cQ$.
Proposition~\ref{p:local-vs-global} immediately implies that
$\soct\kk[\rho] = 0$ unless $\rho \supseteq \tau$, because localizing
along~$\tau$ yields $\kk[\rho]_\tau = 0$ unless $\rho \supseteq \tau$.

Next compute $\ds\kk[\rho]$.  When either $\aa \not\in \rho$ or
$\sigma \not\subseteq \rho$, the direct limit in
Definition~\ref{d:atop-sigma} is over a set $\aa - \sigma^\circ$ of
degrees in which $\kk[\rho] = 0$ in a neighborhood of~$\aa$.  Hence
the only faces that can appear in $\dt\kk[\rho]$ lie in the interval
between~$\tau$ and~$\rho$, so assume $\tau \subseteq \sigma \subseteq
\rho$.  If $\bigl(\ds\kk[\rho]\bigr){}_\aa \neq 0$ then it
equals~$\kk$ because $\kk[\rho]$ is an indicator module for a subset
of~$\cQ$.  Moreover, if $(\ds\kk[\rho])_\aa = \kk$ then the same is
true in any degree $\bb \in \aa + \rho$ because $(\bb -\nolinebreak
\aa) +\nolinebreak (\aa - \sigma^\circ) \cap \rho \subseteq (\bb -
\sigma^\circ) \cap \rho$.  Thus $\ds\kk[\rho]$ is torsion-free as a
$\kk[\rho]$-module.

The $\socct\!$ on the left side of Proposition~\ref{p:either-order},
which by Definition~\ref{d:socct}.\ref{i:global-socc-tau} is a
quotient-restriction of a module
$\hhom_\cQ\bigl(\kk[\tau],\dt\kk[\rho]\bigr)^{\!}$, can only be
nonzero if $\tau = \rho$, as any nonzero image of~$\kk[\tau]$ is a
torsion $\kk[\rho]$-module.  Hence the socle of~$\kk[\rho]$
along~$\tau$ equals the closed socle along $\tau = \rho$, which is
computed directly from Definition~\ref{d:soct}.\ref{i:global-soc-tau}
and Definition~\ref{d:quotient-restriction} to be
$\hhom_\cQ\bigl(\kk[\tau],\kk[\tau]\bigr)/\tau = \kk[\tau]/\tau$.  In
summary,
$$%
  \soct\kk[\rho] =
  \begin{cases}
  \kk_\0\text{ for } \0 \in \qrt & \text{if } \tau = \rho
\\               0               & \text{otherwise.}
\end{cases}
$$
\end{example}

\begin{prop}\label{p:left-exact-tau}
The global cogenerator functor $\soct$ along any face~$\tau$ of a
real polyhedral group is left-exact, as is the local cogenerator
functor along~$\tau$.
\end{prop}
\begin{proof}
Proposition~\ref{p:left-exact-tau-closed} and
Lemma~\ref{l:exact-delta}.
\end{proof}

Here is the final version of the statement that a module in the
category of finitely encoded modules or semialgebraic algebraic
modules remains there upon taking socles.  Previous versions are used
in the proof, but for modules over real polyhedral group this is the
only statement worth remembering, as all of the others are special
cases.  That said, it is also worth noting
Proposition~\ref{p:socc-encoding}, which treats the finitely encoded
category over an arbitrary poset.

\begin{thm}\label{t:soct-encoding}
If a module $\cM$ over a real polyhedral group is finitely encoded
then so is its socle $\soct\cM$ along any face~$\tau$.  If $\cM$ is
semialgebraic then so is~$\soct\cM$.  For any face~$\sigma \supseteq
\tau$, these statements remain true for the socle along~$\tau$ with
nadir~$\sigma$.
\end{thm}
\begin{proof}
By Proposition~\ref{p:either-order}, $\soct\cM$ is a composite of the
functors $\delta$, $\socct$, and~$\ntsoc$ in some order.  For $\delta$
use Proposition~\ref{p:frontier-encoding}.  For $\socct$ use
Proposition~\ref{p:socct-encoding}.  For $\ntsoc$ use
Lemma~\ref{l:abelian-category}: the $\nabt$-graded component of
$\ntsoc\cN$ is the intersection of the kernels of the homomorphisms
$\cN_\sigma \to \cN_{\sigma'}$ for $\sigma \supset \sigma'$.  This
argument for~$\ntsoc$ already proves the claim concerning a fixed
nadir~$\sigma$ (see also Remark~\ref{r:soc-as-k-vect}).
\end{proof}

%%%%%%%%%%%%%%%%%%%%%%%%%%%%%%%%%%%%%%%%%%%%%%%%%%%%%%%%%%%%%%%%%%%%%%%%%
\section{Essential property of socles}\label{s:essential}%%%%%%%%%%%%%%%%
%%%%%%%%%%%%%%%%%%%%%%%%%%%%%%%%%%%%%%%%%%%%%%%%%%%%%%%%%%%%%%%%%%%%%%%%%

In this section, $\cQ$ is a real polyhedral group unless otherwise
stated.

The culmination of the foundations developed in Section~\ref{s:socle}
says that socles and cogenerators detect injectivity of homomorphisms
between finitely encoded modules over real polyhedral groups
(Theorem~\ref{t:injection}), as they do for noetherian rings in
ordinary commutative algebra.  The theory is complicated by there
being no actual submodule containing a given non-closed socle element;
that is why socles are functors that yield submodules of localizations
of auxiliary modules rather than submodules of localizations of the
given module itself.  Nonetheless, it comes down to the fact that,
when $D \subseteq \cQ$ is a downset, every element can be pushed up to
a cogenerator.  Theorem~\ref{t:divides} contains a precise statement
that suffices for the purpose of Theorem~\ref{t:injection}, although
the definitive version of Theorem~\ref{t:divides} occurs in
Section~\ref{s:minimality}, namely Theorem~\ref{t:downset=union}.

The proof of Theorem~\ref{t:divides} requires a
definition---essentially the notion dual to that of shape
(Proposition~\ref{p:shape}).  Informally, it is the set of faces
$\sigma$ such that a neighborhood of~$\aa$ in $\aa + \sigma^\circ$ is
contained in the downset~$D$.  The formal definition reduces by
negation to the discussion surrounding tangent cones of downsets
(Section~\ref{sub:tangent}), noting that the negative of an upset is a
downset.

\begin{defn}\label{d:upshape}
The \emph{upshape} of a downset $D$ in a real polyhedral group~$\cQ$
at~$\aa$ is
$$%
  \dda = \cfq \minus \nabla_{\!-U}^{-\aa\,},
$$
where $U = \cQ \minus D$ is the upset complementary to~$D$.
\end{defn}

\begin{lemma}\label{l:upshape}
The upshape $\dda$ is a polyhedral complex (a downset) in~$\cfq$.  As
a function of\/~$\aa$, for fixed~$D$ the upshape $\dda$ is decreasing,
meaning $\aa \preceq \bb \implies \dda \supseteq \ddb$.
\end{lemma}
\begin{proof}
These claims are immediate from the discussion in
Section~\ref{sub:tangent}.
\end{proof}

\begin{remark}\label{r:upshape}
The upshape $\dda$ is a rather tight analogue of the Stanley--Reisner
complex of a simplicial complex, or more generally the lower Koszul
simplicial complex \cite[Definition~5.9]{cca} of a monomial ideal in a
degree from~$\ZZ^n$.  (The complex $K_\bb(I)$ would need to be indexed
by $\bb - \mathrm{supp}(\bb)$ to make the analogy even tighter.)
Similarly, the shape of a downset at an element of~$\cQ$ is analogous
to the upper Koszul simplicial complex of a monomial ideal
\cite[Definition~1.33]{cca}.
\end{remark}

The general statement about pushing up to cogenerators relies on the
special case of closed cogenerators for closed downsets.

\begin{lemma}\label{l:cogenerator}
If $D \subseteq \cQ$ is a downset and the part of~$D$ above $\bb \in
D$ is closed, so $(\bb + \cQ_+) \cap D = (\bb + \cQ_+) \cap \oD$, then
$\bb \preceq \aa$ for some closed cogenerator~$\aa$ of~$D$.
\end{lemma}
\begin{proof}
It is possible that $\bb + \cQ_+ \subseteq D$, in which case $D = \cQ$
and $\bb$ is by definition a closed cogenerator along $\tau = \cQ_+$.
Barring that case, the intersection $(\bb + \cQ_+) \cap \del D$ of the
principal upset at~$\bb$ with the boundary of~$D$ is nonempty.  Among
the points in this intersection, choose~$\aa$ with minimal upshape.
Observe that $\{\0\} \in \dda$ because $\aa \in D$, so~$\dda$ is
nonempty.

Let $\tau \in \dda$ be a facet.  The goal is to conclude that $\dda =
\cF_\tau$ has no facet other than~$\tau$, for then $\dd{\aa'} =
\cF_\tau$ for all $\aa' \succeq \aa$ in~$D$ by upshape minimality and
Lemma~\ref{l:upshape},
% %
% \begin{excise}{%
% %
%   \comment{it might be kind to the reader to include a slight bit more
%   justification here:} any ray contained in~$\tau$---not necessarily an
%   extreme ray---that exits~$D$ has a final point in~$D$, because the
%   relevant part of the boundary of~$D$ is contained in~$D$.  That
%   final point has visibly smaller upshape, because no vector along the
%   ray itself lies in the upshape there.
% %
% }\end{excise}%
% %
and hence $\aa$ is a cogenerator of~$D$ along~$\tau$ by
Definition~\ref{d:socct} (see also Remark~\ref{r:witness}).

Suppose that $\rho \in \cfq$ is a ray that lies outside of~$\tau$.  If
$\rho \in \dda$ then upshape minimality implies $\rho \in \dd{\aa'}$
for any $\aa' \in (\aa + \tau^\circ) \cap D$, and such an $\aa'$
exists by definition of upshape.  Consequently, some face containing
both~$\rho$ and~$\tau$ lies in~$\dda$: if $\vv$ is any sufficiently
small vector along~$\rho$, then $\aa' + \vv = \aa + (\aa' - \aa) + \vv
\in D$, and the smallest face containing $(\aa' - \aa) + \vv$ contains
both the interior of~$\tau$ (because it contains $\aa' - \aa$)
and~$\rho$ (because it contains~$\vv$).  But this is impossible, so in
fact $\dda = \cF_\tau$.
\end{proof}

\begin{thm}\label{t:divides}
If $D$ is a downset in a real polyhedral group~$\cQ$, then there are
faces $\tau \subseteq \sigma$ of~$\cQ_+\!$ and a cogenerator~$\aa$
of~$D$ along~$\tau$ with nadir~$\sigma$ such that $\bb \preceq \aa$.
\end{thm}
\begin{proof}
It is possible that $\bb + \cQ_+ \subseteq D$, in which case $D = \cQ$
and $\bb$ is by definition a closed cogenerator along $\tau =
\cQ_+\!$, which is the same as a cogenerator along~$\cQ_+\!$ with
nadir~$\cQ_+\!$.  Barring that case, the intersection $(\bb + \cQ_+)
\cap \del D$ of the principal upset at~$\bb$ with the boundary of~$D$
is nonempty.  Among the points in this intersection, there is one with
minimal shape, and it suffices to treat the case where this point
is~$\bb$ itself.

Minimality of~$\nd[\bb]$ implies that the shape does not change upon
going up from~$\bb$ while staying in the closure~$\oD$.
% (Note: it is crucial that what's written here is $\oD$, not~$D$ or
% $\ds D$ for some~$\sigma$.)
Consequently, given any face $\sigma \in \nda$, the shape of~$D$ at
every point in $\bb + \cQ_+\!$ that lies in~$\oD$ also
contains~$\sigma$.  Equivalently by
Proposition~\ref{p:downset-upper-boundary}.\ref{i:nabla}, $(\bb +
\cQ_+) \cap \ds D = (\bb + \cQ_+) \cap \oD$.
Lemma~\ref{l:cogenerator} produces a closed cogenerator~$\aa$ of~$\ds
D$, along some face~$\tau$, satisfying $\bb \preceq \aa$.  Since
$\nda$ is a nonempty cocomplex, its intersection with $\nabt$ is
nonempty, so assume $\sigma \in \nda \cap \nabt$.  The closed
cogenerator~$\aa$ of~$\ds D$ need not be a cogenerator of~$D$, but if
$\sigma$ is minimal under inclusion in~$\nda \cap \nabt$, then $\aa$
is indeed a cogenerator of~$D$ along~$\tau$ with nadir~$\sigma$ by
Proposition~\ref{p:either-order}---specifically the first displayed
isomorphism---applied to Example~\ref{e:socc}.
\end{proof}

\begin{remark}\label{r:essential}
The arguments in the preceding two proofs are essential to the whole
theory of socles, which hinges upon them.  The structure of the
arguments dictate the forms of all of the notions of socle,
particularly those involving cogenerators along~faces.
\end{remark}

Theorem~\ref{t:injection} is intended for finitely encoded modules,
but because it has no cause to deal with generators, in actuality it
only requires half of a fringe presentation (or a little less; see
Remark~\ref{r:augmentation}).  The statement uses divisibility
(Definition~\ref{d:divides}), which works verbatim for~$\dt\cM$, by
Definition~\ref{d:upper-boundary-tau}, because it refers only to upper
boundaries atop a single~face~$\sigma$.

\begin{thm}[Essentiality of socles]\label{t:injection}
Fix a homomorphism $\phi: \cM \to \cN$ of modules over a real
polyhedral group~$\cQ$.
\begin{enumerate}
\item\label{i:phi=>soct}%
If $\phi$ is injective then $\soct\phi: \soct\cM \to \soct\cN$ is
injective for all faces~$\tau$ of~$\cQ_+$.
\item\label{i:soct=>phi}%
If $\soct\phi: \soct\cM \to \soct\cN$ is injective for all
faces~$\tau$ of~$\cQ_+\!$ and $\cM$ is downset-finite, then $\phi$
is~injective.
\end{enumerate}
If $\cM$ is downset-finite then each homogeneous element of~$\cM\!$
divides a cogenerator of~$\cM$.
\end{thm}
\begin{proof}
Item~\ref{i:phi=>soct} is a special case of
Proposition~\ref{p:left-exact-tau}.  Item~\ref{i:soct=>phi} follows
from the divisibility claim, for if $y$ divides a cogenerator~$s$
along~$\tau$ then $\phi(y) \neq 0$ whenever $\soct\phi(\wt s) \neq 0$,
where $\wt s$ is the image of~$s$ in~$\soct\cM$.

For the divisibility claim, fix a downset hull $\cM \into
\bigoplus_{j=1}^k E_j$ and a nonzero $y \in \cM_\bb$.  For some~$j$
the projection $y_j \in E_j$ of~$y$ divides a cogenerator of~$E_j$
along some face~$\tau$ with some nadir~$\sigma$ by
Theorem~\ref{t:divides}.  Choose one such cogenerator~$s_j$, and
suppose it has degree $\aa \in \cQ$.  There can be other indices~$i$
such that $(\soct[\sigma]E_i)_\wt\aa \neq 0$, where $\wt\aa$ is the
image of~$\aa$ in~$\qrt$.  For any such index~$i$, as long as $y_i
\neq 0$ it divides a unique cogenerator in $s_i \in \ds[\tau]E_i$ by
Corollary~\ref{c:at-most-one}.  Therefore the image of~$y$ in $E =
\bigoplus_{j=1}^k E_j$ divides the sum of these cogenerators~$s_j$.
But that sum is itself another cogenerator of~$E$ along~$\tau$ with
nadir~$\sigma$ in degree~$\aa$, and the fact that $y$ divides it
places the sum in the image of the injection
(Lemma~\ref{l:exact-delta}) $\ds[\tau]\cM \into \ds[\tau]E$.
\begin{excise}{%
  To wit: $y$ divides some element $x \in \dt\cM$ that maps to $\soct
  E$ by the injection (Lemma~\ref{l:exact-delta}) $\dt\cM \into \dt
  E$.  The image of this element is a closed cogenerator of~$\dt E$,
  meaning that it is annihilated by the right stuff.  If $x$ itself
  wasn't annihilated by the same stuff, then multiplying~$x$ by the
  stuff would yield nonzero elements of~$\dt\cM$ that go to zero in
  $\dt E$.
}\end{excise}%
\end{proof}

\begin{cor}\label{c:essential-submodule}
Fix a downset-finite module $\cM$ over a real polyhedral~group.
\begin{enumerate}
\item\label{i:0}%
$\cM = 0$ if and only if $\soct\cM = 0$ for all faces~$\tau$.

\item\label{i:cap}%
$\soct\cM' \cap \soct\cM'' = \soct(\cM' \cap \cM'')$ in~$\soct\cM$ for
submodules $\cM'$ and~$\cM''$~of~$\cM$.
\end{enumerate}
\end{cor}
\begin{proof}
That $\cM = 0 \implies \cM = 0$ is trivial.  On the other hand, if
$\soct\cM = 0$ for all~$\tau$ then $\cM$ is a submodule of~$0$ by
Theorem~\ref{t:injection}.\ref{i:soct=>phi}.

The second item is basically left-exactness
(Proposition~\ref{p:left-exact-tau}):
\begin{align*}
  \soct(\cM' \cap \cM'')
& =
  \soct\ker(\cM' \to \cM/\cM'')
\\*
& =
  \ker\bigl(\soct\cM' \to \soct(\cM/\cM'')\bigr)
\\*
& =
  \ker(\soct\cM' \to \soct\cM/\soct\cM'')
\\*
& =
  \soct\cM' \cap \soct\cM'',
\end{align*}
where the penultimate equality is because $\soct\cM''$ is the kernel
of the homomorphism $\soct\cM \to \soct(\cM/\cM'')$, so that
$\soct\cM/\soct\cM'' \into \soct(\cM/\cM'')$.
\end{proof}

There is a much stronger statement connecting socles to essential
submodules (Theorem~\ref{t:essential-submodule}), but it requires
language to speak of density in socles as well as tools to produce
submodules from socle elements, which are the main themes of
Section~\ref{s:minimality}.

%%%%%%%%%%%%%%%%%%%%%%%%%%%%%%%%%%%%%%%%%%%%%%%%%%%%%%%%%%%%%%%%%%%%%%%%%
\section{Minimality of socle functors}\label{s:minimality}%%%%%%%%%%%%%%%
%%%%%%%%%%%%%%%%%%%%%%%%%%%%%%%%%%%%%%%%%%%%%%%%%%%%%%%%%%%%%%%%%%%%%%%%%

Socles capture the entirety of a downset by maximal elements in
closures along faces; that is in some sense the main content of socle
essentiality (Theorem~\ref{t:injection}), or more precisely
Theorem~\ref{t:divides}.  But since closures are involved, it is
reasonable to ask if anything smaller still captures the entirety of
every downset.  Algebraically, for arbitrary modules, this asks for
subfunctors of cogenerator functors.  The particular subfunctors here
concern the graded degrees of socle elements, for which
notation~is~needed.

\begin{defn}\label{d:degree}
The \emph{degree set} of any module~$\cN$ over a poset~$\cP$ is
$$%
  \deg\cN  = \{\aa \in \cP \mid \cN_\aa \neq 0\}.
$$
Write $\deg_\cP = \deg$ if more than one poset could be intended.
\end{defn}

%%%%%%%%%%%%%%%%%%%%%%%%%%%%%%%%%%%%%%%%%%%%%%%%%%%%%%%%%%%%%%%%%%%%%%%%%
\subsection{Neighborhoods of group elements}\label{sub:nbds}\mbox{}%%%%%%

\noindent
The topological condition characterizing when enough cogenerators are
present is a sort of density in the set of all cogenerators.
Lemma~\ref{l:downset-upper-boundary} has a related closure notion.

\begin{defn}\label{d:sigma-neighborhood}
Fix faces~$\sigma \supseteq \tau$ of a real polyhedral group~$\cQ$.
\begin{enumerate}
\item\label{i:neighborhood}%
A \emph{$\sigma$-neighborhood} of a point $\wt\aa \in \qrt$ in a
subset $X \subseteq \qrt$ is the intersection of~$X$ with a subset
of the form $(\aa - \vv + \cQ_+)/\hspace{.2ex}\RR\tau$ with $\vv \in
\sigma^\circ$ and $\wt\aa = \aa + \RR\tau$.

\item\label{i:limit-point}%
A \emph{$\sigma$-limit point} of a subset $X \subseteq \qrt$ is a
point $\wt\aa \in \qrt$ that is a limit (in the usual topology) of
points in~$X$ each of which lies in a $\sigma$-neighborhood
of~$\wt\aa$.

\item\label{i:closure}%
The \emph{$\sigma$-closure} of $X \subseteq \qrt$ is the set of points
$\wt\aa \in \qrt$ such that $X$ has at least one point in every
$\sigma$-neighborhood of~$\wt\aa$.
\end{enumerate}
\end{defn}

\begin{remark}\label{r:usual-closure}
The sets $X$ to which Definition~\ref{d:sigma-neighborhood} is applied
are typically decomposed as finite unions of antichains (but see
Proposition~\ref{p:sigma-nbd-cogen} for an instance where this is not
the case).  Such sets ``cut across'' subsets of the form $(\aa - \vv +
\cQ_+)/\RR\tau$, rather than being swallowed by them, so
$\sigma$-neighborhoods have a fighting chance of reflecting some
concept of closeness in antichains.  If $\sigma = \cQ_+$ and $\tau =
\{\0\}$, for example, and $X$ is an antichain in~$\cQ$, then a
$\sigma$-neighborhood of a point $\aa \in X$ is the same thing as a
usual open neighborhood of~$\aa$ in~$X$, so $\sigma$-closure is the
usual topological closure.  If, at the other extreme, $\sigma = \tau$,
then every antichain in~$\qrt$ is $\tau$-closed.
\end{remark}

\begin{example}\label{e:sigma-limit}
Let $\cQ = \RR^2$ and $\tau = \{\0\}$.  Take for $X \subset \RR^2$ the
convex hull of $\0,\ee_1,\ee_2$ but with the first standard basis
vector $\ee_1$ removed.  If $\sigma$ is the $x$-axis of~$\RR^2_+$,
then in
$$%
  X =
  \begin{array}{@{}c@{}}\includegraphics[height=20mm]{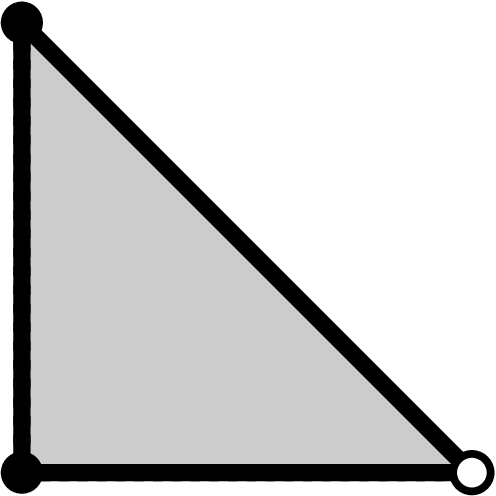}\end{array}
  \text{the blue points of}\quad
  \begin{array}{@{}c@{}}\includegraphics[height=20mm]{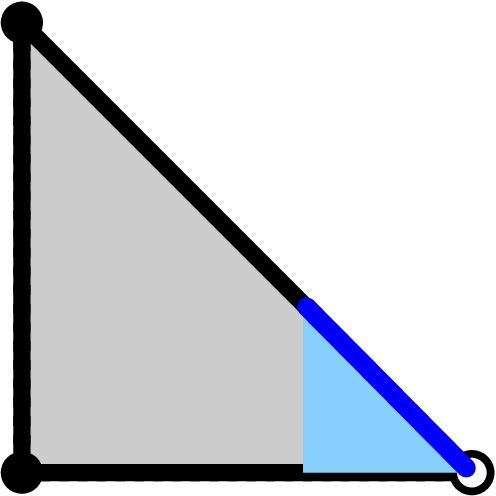}\end{array}
$$
constitute $\sigma$-neighborhoods of~$\ee_1$: in $X$ take all of the
blue points and in the half-open hypotenuse take only the bold blue
segment.  The point~$\ee_1$ is a $\sigma$-limit point of the half-open
hypotenuse.
\end{example}

The next result is applied in the proof of Theorem~\ref{t:dense}.  It
is close to---but not quite---what is needed for the
$\sigma$-neighborhoods to form a base for a topology on~$X$; to be a
base, the intersections would have to allow $\sigma$-neighborhoods of
distinct points.
% https://math.stackexchange.com/questions/241210/base-of-a-topology

\begin{lemma}\label{l:sigma-neighborhood}
Any finite intersection of $\sigma$-neighborhoods of a point $\wt\aa
\in \qrt$ in a subset $X \subseteq \qrt$ contains a
$\sigma$-neighborhood of\/ $\wt\aa \in \qrt$ in~$X$.
\end{lemma}
\begin{proof}
If $\vv \in \sigma^\circ$ then $\aa - \vv + \sigma$ contains an open
neighborhood of~$\aa$ in $\aa - \sigma$.  Therefore so does a finite
intersection $U$ of sets of the form $\aa - \vv + \sigma$.  Any point
$\vv' \in U \cap (\aa - \sigma^\circ)$ yields the desired
$\sigma$-neighborhood $\aa - \vv' + \cQ_+$.  That proves the case
$\tau = \{\0\}$.  Reducing modulo~$\RR\tau$ proves the general case.
\end{proof}

The concept of $\sigma$-neighborhood provides a means to connect
socles (Definition~\ref{d:soct}) with support
(Definition~\ref{d:support}) and primary decomposition
(Definition~\ref{d:primDecomp'}).

\begin{prop}\label{p:sigma-nbd-cogen}
In a real polyhedral group~$\cQ$, every cogenerator of a downset~$D$
along~$\tau$ with nadir~$\sigma$ has a $\sigma$-neighborhood $\OO$
in~$D \subseteq \cQ$ (so $\sigma \supseteq \{\0\}$ are the faces in
Definition~\ref{d:sigma-neighborhood}) such that $\kk[\OO] \subseteq
\kk[D]$ is $\tau$-coprimary and globally supported~on~$\tau$.
\end{prop}
\begin{proof}
Let $\aa$ be such a cogenerator of~$D$.  Suppose $\{\aa_k\}_{k\in\NN}
\subseteq \aa - \sigma^\circ \subseteq D$ is any sequence converging
to~$\aa$.  If $\aa_k$ is supported on a face~$\tau'$, then $\tau'
\supseteq \tau$ because $\aa \succeq \aa_k$ and $\aa$ remains a
cogenerator of the localization of~$D$ along~$\tau$ by
Proposition~\ref{p:local-vs-global}.  The same argument shows that the
$\sigma$-neighborhood $\OO = \aa_k + \cQ_+$ yields a submodule
$\kk[\OO] \subseteq \kk[D]$ such that $\kk[\OO] \into \kk[\OO]_\tau$.
The goal is therefore to show that some $\aa_k$ is supported
on~$\tau$, for then all of~$\OO$ is supported on~$\tau$, as support
can only decrease~upon~going up~in~$\cQ$.

If each $\aa_k$ is supported on a face properly containing~$\tau$,
then, restricting to a subsequence if necessary, assume that it is the
same face~$\tau'$ for all~$k$.  (This uses the finiteness of the
number of faces.)  But then $\aa + \tau' = \lim_k(\aa_k + \tau')$ is
contained in $\ds D$, contradicting the fact that $\aa$ is supported
on~$\tau$ in~$\ds D$.
\end{proof}

\begin{example}\label{e:sigma-nbd-cogen}
All three of the downsets
$$%
\begin{array}{@{}*3{c@{\qquad\qquad}}c}
 \begin{array}{@{}c@{}}\includegraphics[height=30mm]{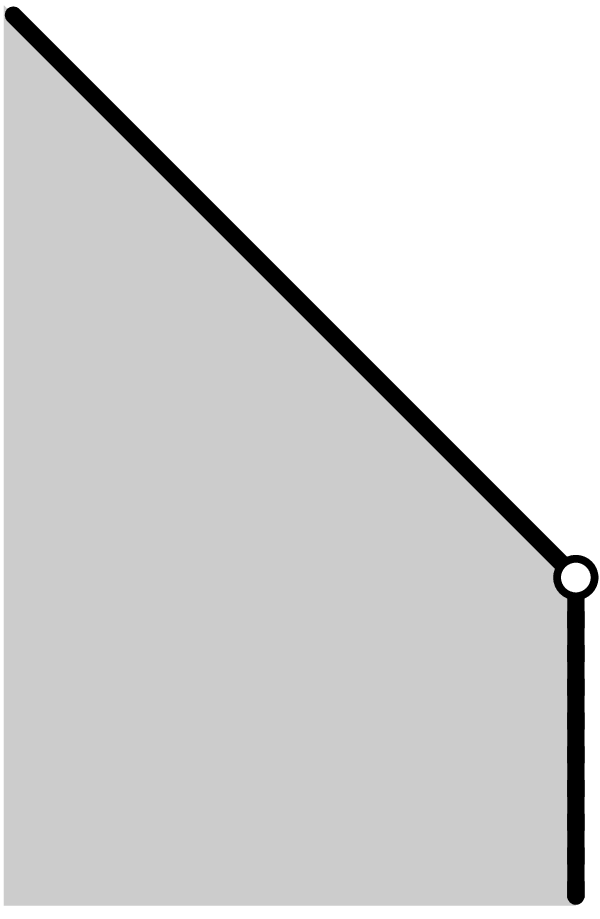}\end{array}
&\begin{array}{@{}c@{}}\includegraphics[height=30mm]{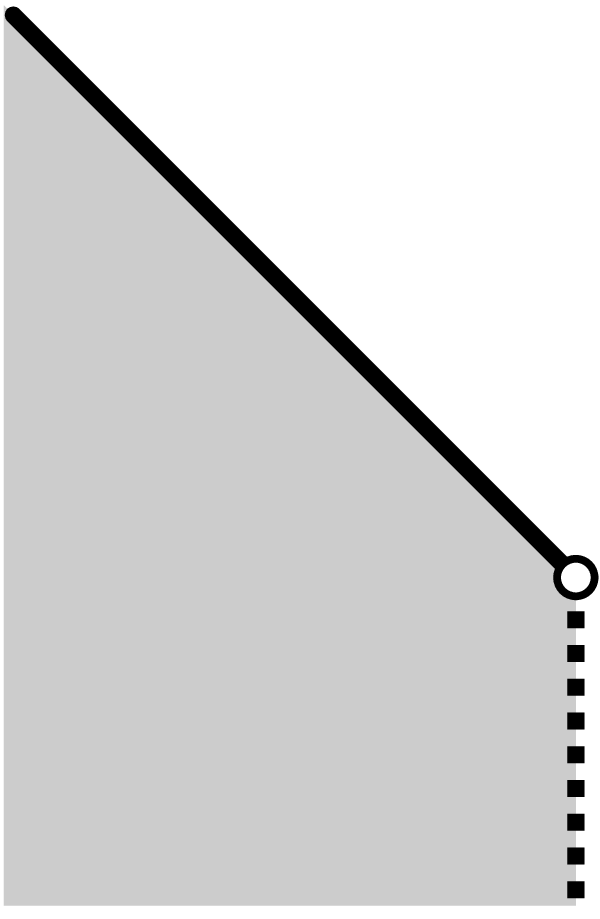}\end{array}
%  = \bigcup
%  \begin{array}{@{}c@{}}\includegraphics[height=30mm]{union2}\end{array}
&\begin{array}{@{}c@{}}\includegraphics[height=30mm]{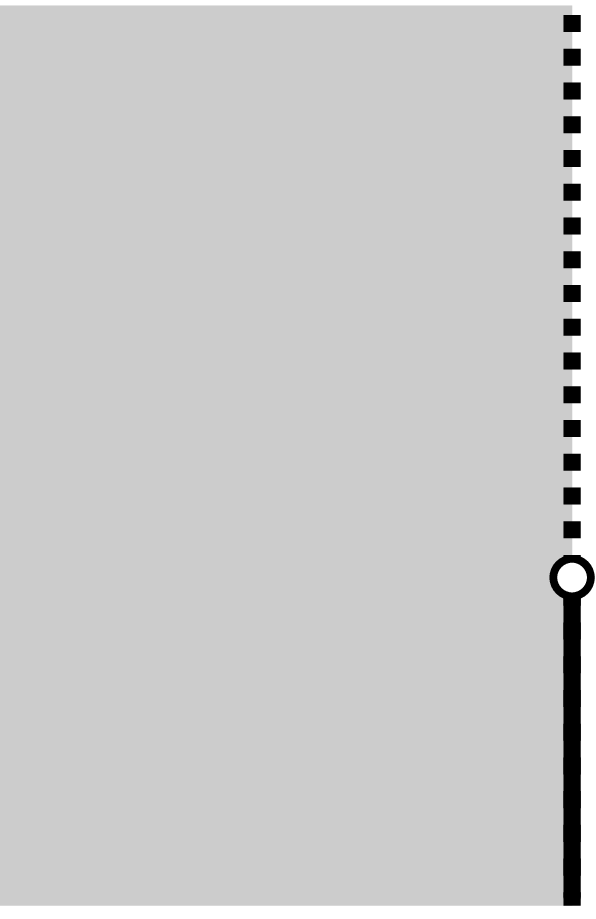}\end{array}
\\
D_1 & D_2 & D_3
\\[-.5ex]
\end{array}
$$
in~$\RR^2$ have a cogenerator at the open corner~$\aa$ along the face
$\tau = \{\0\}$, but their behaviors near~$\aa$ differ in character.
Write $\sigma_x$ and~$\sigma_y$ for the faces of~$\RR^2_+$ that are
its horizontal and vertical axes, respectively.
\begin{enumerate}
\item%
Here $\aa$ has two nadirs: it is a cogenerator along~$\tau = \{\0\}$
for both~$\sigma_x$ and~$\sigma_y$ by
Proposition~\ref{p:downset-upper-boundary} and Example~\ref{e:socc}.
The blue set in Example~\ref{e:sigma-limit} constitutes a
$\tau$-coprimary $\sigma_x$-neighborhood of~$\aa$ globally supported
on~$\tau$, as in Proposition~\ref{p:sigma-nbd-cogen}, if the open
point there is also~$\aa$.

\item%
Here $\aa$ has only the nadir~$\sigma_x$, because the downset has no
points in $\aa + \RR\sigma_y$ to take the closure of in
Lemma~\ref{l:downset-upper-boundary}.  Again, the blue set in
Example~\ref{e:sigma-limit} constitutes the desired
$\sigma_x$-neighborhood of~$\aa$.

\item%
Here $\aa$ has only the nadir~$\sigma_y$.  It is possible to compute
this directly, but it is more apropos to note that
Proposition~\ref{p:sigma-nbd-cogen} rules out $\sigma_x$ as a nadir.
Indeed, every $\sigma_x$-neighbor\-hood of~$\aa$ in~$D_3$ is an
infinite vertical strip.  None of these $\sigma_x$-neighbor\-hoods are
supported on $\tau = \{\0\}$, since elements therein persist forever
along~$\sigma_y$.  In contrast, every choice of $\vv \in
\sigma_y^\circ$ yields a $\sigma_y$-neighborhood $-\vv +\nolinebreak
\RR^2_+ \cap\nolinebreak D_3$ supported on~$\{\0\}$; that is, the
entire negative $y$-axis is supported on~$\{\0\}$.
\end{enumerate}
\end{example}

Compare the following with Example~\ref{e:soct-k[tau]}; it is the
decisive more or less explicit calculation that justifies the general
theory of socles and provides its foundation.  Note that being
$\tau$-coprimary only requires an essential submodule to be globally
supported on~$\tau$; see the under-hyperbola in Example~\ref{e:PF}.
Dually, being globally supported on~$\tau$ allows for elements with
support strictly contained in~$\tau$.

\begin{cor}\label{c:soc(coprimary)}
Fix a face~$\tau$ of a real polyhedral group~$\cQ$ and a subquotient
$\cM$ of\/~$\kk[\cQ]$ that is $\tau$-coprimary and globally supported
on~$\tau$.  Then $\socp\cM = 0$ unless $\tau' = \tau$.
\end{cor}
\begin{proof}
Proposition~\ref{p:local-vs-global} implies that $\socp\cM = 0$ unless
$\tau' \supseteq \tau$ by definition of global support: localizing
along~$\tau'$ yields $\cM_{\tau'} = 0$ unless $\tau' \subseteq \tau$.
On the other hand, $\cM$ being a subquotient of~$\kk[\cQ]$ means that
$\cM \subseteq \kk[D]$ for some downset~$D$.  By left-exactness of
socles (Proposition~\ref{p:left-exact-tau}), every cogenerator
of~$\cM$ is a cogenerator of~$\kk[D]$.  Applying
Proposition~\ref{p:sigma-nbd-cogen} to any such cogenerator
along~$\tau'$ implies that $\tau = \tau'$, because no $\tau$-coprimary
module has a submodule supported on a face strictly contained
in~$\tau$.
\end{proof}

Definition~\ref{d:sigma-neighborhood}.\ref{i:limit-point} stipulates
no condition the generators of the relevant
$\sigma$-neighbor\-hoods---the vectors $\aa - \vv$ in
Definition~\ref{d:sigma-neighborhood}.\ref{i:neighborhood}.  The
difference between being a $\sigma$-limit point and lying in the
$\sigma$-closure is hence that for $\sigma$-closure, the convergence
is stipulated on the generators of the $\sigma$-neighborhoods rather
than on the points of~$X$.  That said, the a~priori weaker (that is,
more inclusive) notion of $\sigma$-limit point is equivalent: the
generators can be forced to converge.

\begin{prop}\label{p:aa_k}
If a sequence $\{\aa'_k\}_{k\in\NN}^{\,}$ in a real polyhedral
group~$\cQ$ has $\aa'_k \to \aa$ and $\aa'_k \in \aa_k + \cQ_+$ for
some $\aa_k \in \aa - \sigma^\circ$, where $\sigma$ is a fixed face,
then it is possible to choose the elements~$\aa_k$ so that $\aa_k \to
\aa$.  Consequently, if~$\sigma \supseteq \tau$ then the
$\sigma$-closure of any set $X \subseteq \qrt$ equals the set of its
$\sigma$-limit points.
\end{prop}
\begin{proof}
Writing $\aa'_k = \aa - \vv_k + \zz_k$ with $\vv_k \in \RR\sigma$ and
$\zz_k \in \sigma^\perp$, the only relevant thing the hypothesis
$\aa'_k \in \aa_k + \cQ_+$ does is force~$\zz_k$ to land in~$\qrsp$
when projected to~$\qrs$.
% (Note: this already uses the equivalence of adding $\qns$
% instead of~$\cQ_+$ in
% Definition~\ref{d:sigma-neighborhood}.\ref{i:neighborhood}.)
This leads us to consider the set $Z \subseteq \sigma^\perp$ of
vectors in~$\sigma^\perp$ whose images in~$\qrs$ lie in~$\qrsp$ and
have magnitude $\leq 1$.  Let $V \subseteq \RR\sigma$ be the ball of
radius~$1$.  Find $\sss \in \sigma^\circ$ so that $\sss + V + Z
\subseteq \cQ_+$.  To see that such an~$\sss$ exists, first find
$\sss$ so that $\sss + Z \subseteq \cQ_+$, which exists by rescaling
any element $\sss' \in \sigma^\circ$ because the projection of $(\sss'
+ \sigma^\perp) \cap \cQ_+$ to $\qrs$ contains a neighborhood of~$\0$
in~$\qrsp$.  Then observe that the condition $\sss + Z \subseteq
\cQ_+$ remains true after adding any element of~$\sigma$ to~$\sss$.
In particular, add the center of any ball in~$\sigma^\circ$ of
radius~$1$, which exists because $\sigma^\circ$ is nonempty, open
in~$\RR\sigma$, and closed under positive scaling.

Having fixed~$\sss$ with $\sss + V + Z \subseteq \cQ_+$, note that
$\epsilon\sss + \epsilon V + \epsilon Z \subseteq \cQ_+$.  On the other hand, 
the magnitudes of~$\vv_k$ and~$\zz_k$
% (under any choice of norm)
are bounded above by $\epsilon_k = |\aa'_k - \aa|$.  Therefore, $\aa'_k \in
\aa_k + \cQ_+$, where $\aa_k = \aa - \epsilon_k\sss \to \aa$, because
$\aa'_k = \aa - \vv_k + \zz_k \in \aa + \epsilon V + \epsilon Z$, and $\aa + \epsilon
V + \epsilon Z = (\aa - \epsilon\sss) + (\epsilon\sss + \epsilon V + \epsilon Z) \subseteq \aa
- \epsilon\sss + \cQ_+$.

The claim involving~$\tau$ follows, when $\tau = \{\0\}$, from
Proposition~\ref{p:<<}: it implies that each element of $\aa -
\sigma^\circ$ precedes some~$\aa_k$, and hence the
$\sigma$-neighborhood it generates contains~$\aa'_k$.  The case of
arbitrary~$\tau$ reduces to $\tau = \{\0\}$ by working
modulo~$\RR\tau$.
\end{proof}

%%%%%%%%%%%%%%%%%%%%%%%%%%%%%%%%%%%%%%%%%%%%%%%%%%%%%%%%%%%%%%%%%%%%%%%%%
\subsection{Dense cogeneration of downsets}\label{sub:dense-downsets}\mbox{}

\noindent
The subfunctor version of density in socles for modules requires first
a geometric version for downsets.  For geometric intuition, it is
useful to recall Lemma~\ref{l:<<}, which says that $\qns =
\sigma^\circ + \cQ_+$.  Thus $\aa - \qns$ is the
``coprincipal'' downset with apex~$\aa$ and shape~$\nabs$.  Adding
$\tau$ to get $\aa+\tau-\qns$ takes the union of these downsets
along $\aa + \tau$.

\begin{thm}\label{t:downset=union}
Let $\ats \subseteq \cQ$ be a set of cogenerators of a downset~$D$ in
a real polyhedral group~$\cQ$ along a face~$\tau$ with nadir~$\sigma$
for each $\sigma \in \nabt$.  If every cogenerator of~$D$ along~$\tau$
with nadir~$\sigma$ maps to a $\sigma$-limit point of the image
of~$\ats[\,]\!  =\nolinebreak \bigcup_{\sigma \supseteq \tau}
\!\ats$~in~$\qrt$,~%
then
\begin{equation*}%\label{eq:downset=union}
  D
  =
\bigcup_{\substack{\text{\rm faces }\sigma,\tau\\\text{\rm with }\sigma\supseteq\tau}}
\bigcup_{\raisebox{-.7ex}{$\scriptstyle\ \aa\in\ats$}} \aa + \tau - \qns.
\end{equation*}
\end{thm}
\begin{proof}
Theorem~\ref{t:divides} is equivalent to the desired result in the
case that every $\ats$ is the set of all cogenerators of~$D$
along~$\tau$ with nadir~$\sigma$, by Example~\ref{e:soct} and
Remark~\ref{r:nabt}.  Hence it suffices to show~that
$$%
  \bigcup_{\raisebox{-.6ex}{$\scriptstyle\sigma'\supseteq\tau$}}
  \bigcup_{\ \aa'\in\ats[\sigma']} \aa' + \tau - \qnp
  \supseteq
  \aa + \tau - \qns
$$
for any fixed cogenerator~$\aa$ of~$D$ along~$\tau$ with
nadir~$\sigma$.  In fact, by definition of $\sigma$-limit point, it is
enough to show that
$$%
  \bigcup_{k=1}^\infty \aa'_k + \tau - \qnk
  \supseteq
  \aa + \tau - \qns,
$$
where $\{\aa'_k\}_{k\in\NN}^{\,}$ is a sequence of elements 
of~$\ats[\,]$ such that
\begin{itemize}
\item%
$\aa_k$ lands in a $\sigma$-neighborhood of the image~$\wt\aa$
of~$\aa$ when projected to~$\qrt$, and
\item%
these images $\wt\aa'_k$ converge to~$\wt\aa$ in~$\qrt$
\end{itemize}
and $\sigma_k$ is a nadir of the cogenerator $\aa'_k$ along~$\tau$.

Note that there is something to prove even when $\sigma = \tau$ (see
the end of Remark~\ref{r:usual-closure}) because $\ats[\tau]$ only
needs to have at least one closed cogenerator in~$\cQ$ for each closed
socle degree in~$\qrt$, whereas the set of all closed cogenerators
along~$\tau$ mapping to a given socle degree might not be a single
translate of~$\tau$.  On the other hand, $\tau - \qnt = \tau -
\tau^\circ - \cQ_+\!$ by Lemma~\ref{l:<<}, and this is just $\RR\tau -
\cQ_+$.  Therefore $\aa + \tau - \qnt$ contains the translate of the
negative cone $-\cQ_+\!$ at every point mapping to $\wt\aa$,
cogenerator or otherwise, completing the case $\sigma = \tau$.

For general $\sigma \supseteq \tau$, again Lemma~\ref{l:<<} says that
$\aa + \tau - \qns = \aa + \tau - \sigma^\circ - \cQ_+$, and as before
this set is preserved under translation by~$\RR\tau$ because $-\cQ_+$
contains~$-\tau$.  Therefore the question reduces to the
quotient~$\qrt$, where it becomes
$$%
  \bigcup_{k=1}^\infty \wt\aa'_k - \qnk/\tau
  \supseteq
  \wt\aa - \qns/\tau.
$$
But as $\qnk/\tau = \sigma_k^\circ/\tau + (\qrt)_+$ by
Lemma~\ref{l:<<}, it does no harm (and helps the notation) to assume
that $\tau = \{\0\}$.  The desired statement is now
$$%
  \bigcup_{k=1}^\infty \aa'_k - \qnk
  \supseteq
  \aa - \qns,
$$
the hypotheses being those of Proposition~\ref{p:aa_k}.  The proof is
completed by applying Proposition~\ref{p:<<} to the
sequence~$\{\aa_k\}_{k\in\NN}$ produced by Proposition~\ref{p:aa_k},
noting that $\aa_k - \cQ_+ \subseteq \aa'_k - \qnk$ as soon as
$\aa_k \in \aa'_k - \qnk$, because $\aa'_k - \qnk$ is a
downset.
\end{proof}

\begin{example}\label{e:pictures}
Consider the downsets in Example~\ref{e:sigma-nbd-cogen}.  The
question is whether~$\aa$ is forced to appear in the union from
Theorem~\ref{t:downset=union} or not.
\begin{enumerate}
\item\label{i:pinch-0}%
The point~$\aa$ is a $\sigma_x$-limit point of~$D_1$ by
Example~\ref{e:sigma-limit}, which shares its geometry with~$D_1$ on
the relevant set, namely the $x$-axis and above.
Theorem~\ref{t:downset=union} therefore wants to force~$\aa$ to
appear.  However, every $\sigma_y$-neighborhood of~$\aa$ in~$D_1$
contains exactly one cogenerator, namely~$\aa$ itself.  Therefore
$\aa$ is indeed forced to appear.

\item\label{i:pinch-negative-y}%
In contrast, $\aa$ is not needed for~$D_2$.  Abstractly, this is
because $D_2$ is missing precisely the negative $y$-axis that caused
$\aa$ to be forced in~$D_1$.  But geometrically it is evident that
$D_2$ equals the union of the closed negative quadrants hanging from
the open diagonal ray.

\item\label{i:half-plane}%
Here $\aa$ is the sole cogenerator of~$D_3$ along $\tau = \{\0\}$, so
it is forced to appear.
\end{enumerate}
\end{example}

\begin{remark}\label{r:irred-decomp}
Theorem~\ref{t:downset=union} is the analogue for real polyhedral
groups of the fact that monomial ideals in affine semigroup rings
admit unique irredundant irreducible decompositions
\cite[Corollary~11.5]{cca}.  To see the analogy, note that expressing
a downset as a union is the same as expressing its complementary upset
as an intersection.  In Theorem~\ref{t:downset=union} the union is
neither unique nor irredundant, but only in the sense that a
topological space can have many dense subsets, each of which can
usually be made smaller by omitting some points.  The union in which
$\ats$ maps onto $\deg_{\qrt}\soct[\sigma]\kk[D]$ modulo~$\RR\tau$ is
still canonical, though redundant in a predictable manner.
\end{remark}

\begin{cor}\label{c:sigma-neighborhood}
Fix a cogenerator~$\aa$ of a downset~$D$ along a face~$\tau$ with
nadir~$\sigma$ in a real polyhedral group.  If\/ $\bb \in D$ and $\bb
\preceq \aa$, then the image $\wt\aa$ of~$\aa$ in~$\qrt$ has a
$\sigma$-neighborhood~$\OO$ in $\deg_{\qrt}\soct\kk[D]$ such that
$\wt\bb \preceq \wt\aa'$ for all\/ $\wt\aa' \in \OO$.
\end{cor}
\begin{proof}
Assume $\bb \in D$ and $\bb \preceq \aa$.
Theorem~\ref{t:downset=union} implies that $\bb \in \aa + \tau - \qns
= \aa + \RR\tau - \sigma^\circ - \cQ_+$.  Therefore $\aa + \RR\tau =
\bb + \RR\tau + \sss + \qq$ for some $\sss \in \sigma^\circ$ and $\qq
\in \cQ_+$.  The $\sigma$-neighborhood in question is $(\wt\bb +
\wt\qq + \cQ_+) \cap \deg_{\qrt}\soct\kk[D]$.
\end{proof}

%%%%%%%%%%%%%%%%%%%%%%%%%%%%%%%%%%%%%%%%%%%%%%%%%%%%%%%%%%%%%%%%%%%%%%%%%
\subsection{Dense subfunctors of socles}\label{sub:subfunctors}\mbox{}%%%

\noindent
In general, a subfunctor $\Phi: \cA \to \cB$ of a covariant functor
$\Psi: \cA \to \cB$ is a natural transformation $\Phi \to \Psi$ such
that $\Phi(A) \subseteq \Psi(A)$ for all objects $A \in \cA$
\cite[Chapter~III]{eilenberg-maclane1945};
% https://ncatlab.org/nlab/show/subfunctor
denote this by $\Phi \subseteq \Psi$.  (This notation assumes that the
objects of~$\cB$ are sets, which they are here; in general, $\Phi(A)
\to \Psi(A)$ should be monic.)

\begin{defn}\label{d:dense-subfunctor}
A subfunctor $\cst = \bigoplus_{\sigma \in \nabt} \cst[\sigma]
\subseteq \soct$ from modules over~$\cQ$ to modules over $\qrt
\times\nolinebreak \nabt$ is \emph{dense} if the $\sigma$-closure of
$\deg_{\qrt}\cst\kk[D]$ contains $\deg\soct[\sigma]\kk[D]$ for all
faces $\sigma \supseteq \tau$ and downsets $D \subseteq \cQ$.  An
\emph{$\cS$-cogenerator} of a $\cQ$-module~$\cM$ is a cogenerator
of~$\cM$ along some face~$\tau$ whose image in $\soct\cM$ lies
in~$\cst\cM$.
\end{defn}

\begin{thm}\label{t:dense}
Fix subfunctors $\cS_\tau \subseteq \soct$ for all faces~$\tau$ of a
real polyhedral group.  Theorem~\ref{t:injection} holds with $\cS$ in
place of $\soc$ if and only if $\cst$ is dense in~$\soct$ for
all~$\tau$.
\end{thm}
\begin{proof}
% Is it true that every subfunctor of $\soct$ is necessarily
% left-exact?  No.
% % https://mathoverflow.net/questions/126191/are-subfunctors-of-left-exact-functors-also-left-exact
% Well, at least not a~priori.  But it is true that
Every subfunctor of any left-exact functor takes injections to
injections; therefore Theorem~\ref{t:injection}.\ref{i:phi=>soct}
holds for any subfunctor of~$\soct$ by
Proposition~\ref{p:left-exact-tau}.  The content is that
Theorem~\ref{t:injection}.\ref{i:soct=>phi} is equivalent to density
of $\cst$ in $\soct$ for all~$\tau$.

First suppose that $\cst$ is dense in $\soct$ for all~$\tau$.  It
suffices to show that each homogeneous element $y \in \cM$ divides
some $\cS$-cogenerator~$s$, for then $\phi(y) \neq 0$ whenever
$\cst\phi(\wt s) \neq 0$, where $\wt s$ is the image of~$s$
in~$\cst\cM \subseteq \soct\cM$.  There is no harm in assuming that
$\cM$ is a submodule of its downset hull: $\cM \subseteq E =
\bigoplus_{j=1}^k E_j$.  Theorem~\ref{t:injection} produces a
cogenerator~$x$ of~$E$ that is divisible by~$y$, and $x$ is
automatically a cogenerator of~$\cM$---say $x \in \ds[\tau]\cM
\subseteq \ds[\tau]E$---because $y$~divides~$x$.  Write $x =
\sum_{j=1}^k x_j \in \ds[\tau]E = \bigoplus_{j=1}^k \ds[\tau]E_j$.
For any index~$j$ such that $x_j \neq 0$,
Corollary~\ref{c:sigma-neighborhood} and the density hypothesis yields
a $\sigma$-neighborhood of~$\wt\aa$ containing a socle element~$\wt
s_j$ mapped to by an $\cS$-cogenerator~$s_j$ that is divisible
by~$y_j$.  An $\cS$-cogenerator~$s$ of~$\cM$ divisible by~$y$ is
constructed from~$s_j$ just as an ordinary cogenerator is constructed
from~$s_j$ in the second paragraph of the proof of
Theorem~\ref{t:injection}.

Now suppose that $\cst[\,]$ is not dense in~$\soct$ for some
face~$\tau$, so some downset $D \subseteq\nolinebreak \cQ$ has a
cogenerator $\aa \in \cQ$ whose image $\wt\aa \in
\deg\soct[\sigma]\kk[D] \subseteq \qrt$ has a $\sigma$-neighborhood
$\deg_{\qrt}\soct\kk[D] \cap (\aa - \vv + \cQ_+)/\hspace{.2ex}\RR\tau$
devoid of images of $\cS$-cogenerators along~$\tau$.  Appealing to
Lemma~\ref{l:sigma-neighborhood}, the intersection of $\aa - \vv +
\cQ_+$ with a $\sigma$-neighborhood~$\OO$ of~$\aa$ in~$D$ from
Proposition~\ref{p:sigma-nbd-cogen} contains another
$\sigma$-neighborhood~$\OO'$ of~$\aa$ that still satisfies the
conclusion of Proposition~\ref{p:sigma-nbd-cogen} because every
submodule of any $\tau$-coprimary module globally supported on~$\tau$
is also $\tau$-coprimary and globally supported on~$\tau$.
% (this requires proof, but it is not hard from the definitions plus
% exactness of localization)
The injection $\kk[\OO'] \into \kk[D]$ yields an injection
$\cst\kk[\OO'] \into \cst\kk[D]$, but by construction $\cst\kk[D]$
vanishes in all degrees from $\deg_{\qrt}\soct\kk[\OO']$, so
$\cst\kk[\OO'] = 0$.  On the other hand, $\socp\kk[\OO'] = 0$ for
$\tau' \neq \tau$ by Corollary~\ref{c:soc(coprimary)}, so the
subfunctor $\csp$ vanishes on $\kk[\OO']$ for all faces~$\tau'$.
Consequently, applying~$\csp$ to the homomorphism $\phi: \kk[\OO'] \to
0$ yields an injection $0 \into 0$ for all faces~$\tau'$ even though
$\phi$ is not injective.
\end{proof}

%%%%%%%%%%%%%%%%%%%%%%%%%%%%%%%%%%%%%%%%%%%%%%%%%%%%%%%%%%%%%%%%%%%%%%%%%
\section{Essential submodules via density in socles}\label{s:density}%%%%
%%%%%%%%%%%%%%%%%%%%%%%%%%%%%%%%%%%%%%%%%%%%%%%%%%%%%%%%%%%%%%%%%%%%%%%%%

%begin{remark}\label{r:socle-submodules}
The $\sigma$-neighborhoods in Proposition~\ref{p:sigma-nbd-cogen}
transfer cogenerators back into honest submodules; they are, in that
sense, the reverse of Definition~\ref{d:atop-sigma}.  In fact this
transference of cogenerators into submodules works not merely for
indicator quotients but for arbitrary modules with finite downset
hulls, as in Theorem~\ref{t:essential-submodule}.
%end{remark}
The key is the generalization of $\sigma$-neighborhoods to arbitrary
downset-finite modules.

\begin{defn}\label{d:nearby}
Fix a module~$\cM$ over a real polyhedral group~$\cQ$ and a face
$\tau$.
\begin{enumerate}
\item\label{i:nearby}%
A $\sigma$-divisor (Definition~\ref{d:divides}) $y \in \cM$ of a
cogenerator of~$\cM$ along~$\tau$ with nadir~$\sigma$
(Definition~\ref{d:soct}) is \emph{nearby} if $y$ is globally
supported on~$\tau$ (Definition~\ref{d:support}).

\item\label{i:neighborhood-in-M}%
A \emph{$\sigma$-neighborhood in~$\cM$} of a cogenerator $s \in
\ds[\tau]\cM$ is a submodule of~$\cM$ generated by a nearby
$\sigma$-divisor of~$s$.

\item\label{i:neighborhood-in-soc}%
A \emph{neighborhood in~$\soct\cM$} of a homogeneous socle element
$\wt s \in \soct[\sigma]\cM$ is $\soct\cN$ for a
$\sigma$-neigh\-borhood $\cN$ in~$\cM$ of a cogenerator in
$\ds[\tau]\cM$ that maps to~$\wt s$.

\item\label{i:dense}%
A submodule $\cst \subseteq \soct\cM$ is \emph{dense} if for all
$\sigma \supseteq \tau$, every neighborhood of every homogeneous
element of~$\soct[\sigma]\cM$ contains a nonzero element of~$\cst$.
\end{enumerate}
\end{defn}

\begin{lemma}\label{l:nearby}
Every neighborhood in~$\cM$ of every homogeneous element
in~$\soct[\sigma]\cM$ is a $\tau$-coprimary submodule of~$\cM\!$
globally supported on~$\tau$.
\end{lemma}
\begin{proof}
Let $y$ be a nearby $\sigma$-divisor of a cogenerator $s \in
\ds[\tau]\cM$.  Let $x$ be a homogeneous multiple of~$y$.  That~$x$ is
supported on~$\tau$ is automatic from the hypothesis that $y$ is
supported on~$\tau$.  To say that~$\<y\>$ is $\tau$-coprimary means,
given that it is supported on~$\tau$, that $\<y\>$ is a submodule of
its localization along~$\tau$.  But $s$ remains a cogenerator after
localizing along~$\tau$ by Proposition~\ref{p:local-vs-global}, so~$x$
must remain nonzero because it still divides~$s$ after localizing.
\end{proof}

\begin{prop}\label{p:nearby}
Fix a downset-finite module $\cM$ over a real polyhedral group with
faces $\sigma \supseteq \tau$.  Every cogenerator in $\ds[\tau]\cM$
has a $\sigma$-neighborhood in~$\cM$.
\end{prop}
\begin{proof}
Let $s \in \ds[\tau]\cM$ be the cogenerator, and let its degree be
$\deg_\cQ(s) = \aa \in \cQ$.  Choose a downset hull $\cM \into E =
\bigoplus_{j=1}^k E_j$, so $E_j = \kk[D_j]$ for a downset~$D_j$.
Express $s = s_1 + \dots + s_k \in \ds[\tau]E = \bigoplus_{j=1}^k
\ds[\tau]E_j$.  \mbox{Proposition}~\ref{p:sigma-nbd-cogen} produces~a
$\sigma$-neighborhood $\OO_j$ of~$\aa$ in~$\cQ$, for each index~$j$,
such that $\kk[\OO_j \cap D_j]$ is a $\sigma$-neighborhood in~$E_j$ of
the image $\wt s_j \in \soct[\sigma]E_j$.
Lemma~\ref{l:sigma-neighborhood} then yields a single
$\sigma$-neighborhood $\OO = \aa - \vv + \cQ_+$ of~$\aa$ in~$\cQ$ that
lies in the intersection $\bigcap_{j=1}^k \OO_j$.  The cogenerator $s
\in \ds[\tau]$ is a direct limit over $\aa - \sigma^\circ$; since
$\OO$ contains a neighborhood (in the usual topology) of~$\aa$
in~$\sigma^\circ$, some element $y \in \cM$ with degree in~$\OO$ is a
$\sigma$-divisor of~$s$.  This element~$y$ is
nearby~$s$~by~construction.
\end{proof}

The following generalization of Corollary~\ref{c:soc(coprimary)} to
modules with finite downset hulls is again the decisive computation.

\begin{cor}\label{c:soc(coprimary)'}
Fix a downset-finite $\tau$-coprimary $\cQ$-module~$\cM$ globally
supported on a face~$\tau$ of a real polyhedral group~$\cQ$.  Then
$\socp\cM = 0$ unless~$\tau' = \tau$.
\end{cor}
\begin{proof}
Proposition~\ref{p:local-vs-global} implies that $\socp\cM = 0$ unless
$\tau' \supseteq \tau$ by definition of global support: localizing
along~$\tau'$ yields $\cM_{\tau'} = 0$ unless $\tau' \subseteq \tau$.
On the other hand, applying Proposition~\ref{p:nearby} to any
cogenerator of~$\cM$ along a face~$\tau'$ implies that $\tau = \tau'$,
because no $\tau$-coprimary module has a submodule supported on a face
strictly contained in~$\tau$.
\end{proof}

\begin{thm}\label{t:essential-submodule}
In a downset-finite module $\cM$ over a real polyhedral~group, $\cM'$
is an essential submodule if and only if $\soct\cM'$ is dense
in~$\soct\cM$ for all faces~$\tau$.
\end{thm}
\begin{proof}
First assume that $\cM'$ is not an essential submodule, so $\cN \cap
\cM' = 0$ for some nonzero submodule $\cN \subseteq \cM$.  Let $s \in
\ds[\tau]\cN$ be a cogenerator.  Any $\sigma$-neighborhood of~$s$
in~$\cN$, afforded by Proposition~\ref{p:nearby}, has a socle
along~$\tau$ that is a neighborhood of~$\wt s$ in $\soct\cM$ whose
intersection with $\soct\cM'$ is~$0$.  Therefore $\soct\cM'$ is not
dense in~$\soct\cM$.

Now assume that $\soct\cM'$ is not dense in~$\soct\cM$ for
some~$\tau$.  That means $\soct[\sigma]\cM$ for some nadir~$\sigma$
has an element $\wt s$ with a neighborhood $\soct\cN$ that intersects
$\soct\cM'$ in~$0$.  But $\soct\cN \cap \soct\cM' = \soct(\cN \cap
\cM')$ by Corollary~\ref{c:essential-submodule}.\ref{i:cap}.  The
vanishing of this socle along~$\tau$ means that $\socp(\cN \cap \cM')
= 0$ for all faces~$\tau'$ by Corollary~\ref{c:soc(coprimary)'}, and
thus $\cN \cap \cM' = 0$ by
Corollary~\ref{c:essential-submodule}.\ref{i:0}.  Therefore $\cM'$ is
not an essential submodule of~$\cM$.
\end{proof}

\begin{example}\label{e:essential-submodule}
The convex hull of $\0,\ee_1,\ee_2$ in $\RR^2$ but with the first
standard basis vector $\ee_1$ removed defines a subquotient~$\cM$
of~$\kk[\RR^2]$.  It has submodule~$\cM'$ that is the indicator
function for the same triangle but with the entire $x$-axis removed.
All of the cogenerators of both modules occur along the face $\tau =
\{\0\}$ because both modules are globally supported on~$\{\0\}$.
However the ambient module---but not the submodule---has a cogenerator
$y \in \ds[\tau]\cM$ with nadir $\sigma = x$-axis of degree~$\ee_1$:
$$%
\begin{array}{@{}*6{c@{}}c}
\\[-3.8ex]
 \begin{array}{@{}c@{}}\includegraphics[height=20mm]{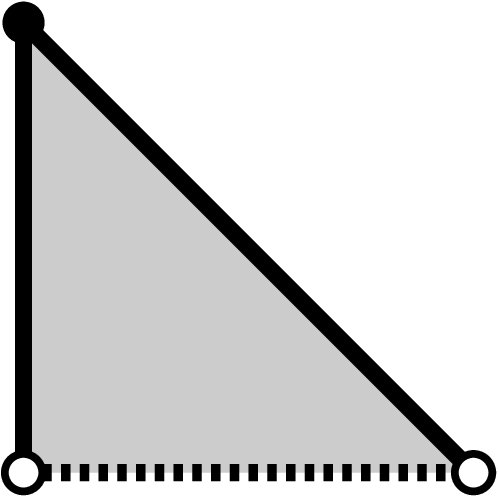}\end{array}
&\quad \subseteq \quad\ \
&\begin{array}{@{}c@{}}\includegraphics[height=20mm]{essential-ambient}\end{array}
&\quad \text{with $\sigma$-neighborhoods} \quad\ \
&\begin{array}{@{}c@{}}\includegraphics[height=20mm]{sigma-nbds}\end{array}
&\quad\ \ \text{of~$\ee_1$.}
\end{array}
$$
A typical $\sigma$-neighborhood of~$y$ in~$\cM$ is shaded in light
blue (in fact, $\cM$ itself is also a $\sigma$-neighborhood of~$y$),
with the corresponding neighborhood in $\soct[\sigma]\cM$ in bold
blue.  Every such neighborhood contains socle elements
in~$\soct[\sigma]\cM'$, so $\cM' \subseteq \cM$ is an essential
submodule by Theorem~\ref{t:essential-submodule}.  Trying to mimic
this example in a finitely generated context is instructive: pixelated
rastering of the horizontal lines either isolates the socle element at
the right-hand endpoint of the bottom edge or prevents it from
existing in the first place by aligning with the right-hand end of the
line above it.
\end{example}

%%%%%%%%%%%%%%%%%%%%%%%%%%%%%%%%%%%%%%%%%%%%%%%%%%%%%%%%%%%%%%%%%%%%%%%%%
\section{Primary decomposition over real polyhedral groups}\label{s:hulls}
%%%%%%%%%%%%%%%%%%%%%%%%%%%%%%%%%%%%%%%%%%%%%%%%%%%%%%%%%%%%%%%%%%%%%%%%%

This section takes the join of Section~\ref{s:decomp}, which develops
primary decomposition as far as possible over arbitrary polyhedral
partially ordered groups, and Section~\ref{s:socle}, which develops
socles over real polyhedral groups.  That is, it investigates how
socles interact with primary decomposition in real polyhedral groups.

%%%%%%%%%%%%%%%%%%%%%%%%%%%%%%%%%%%%%%%%%%%%%%%%%%%%%%%%%%%%%%%%%%%%%%%%%
\subsection{Associated faces}\label{sub:ass}\mbox{}%%%%%%%%%%%%%%%%%%%%%%

\noindent
What makes the theory for real polyhedral groups stronger than for
arbitrary polyhedral partially ordered groups is the following notion
familiar from commutative algebra, except that (as noted in
Section~\ref{s:density}) socle elements do not lie in the original
module.

\begin{defn}\label{d:associated}
A face $\tau$ of a real polyhedral group~$\cQ$ is \emph{associated} to
a downset-finite $\cQ$-module~$\cM$ if $\soct\cM \neq 0$.  If $\cM =
\kk[D]$ for a downset~$D$ then $\tau$ is \emph{associated}~to~$D$.
The set of associated faces of~$\cM$ or~$D$ is denoted by $\ass\cM$
or~$\ass D$.
\end{defn}

\begin{thm}\label{t:coprimary}
A downset-finite module~$\cM$ over a real polyhedral group is
$\tau$-coprimary if and only if $\socp\cM = 0$ whenever $\tau' \neq
\tau$, or equivalently, $\ass(\cM) = \{\tau\}$.
\end{thm}
\begin{proof}
If $\cM$ is not $\tau$-coprimary then either $\cM \to \cM_\tau$ has
nonzero kernel~$\cN$, or $\cM \to\nolinebreak \cM_\tau$ is injective
while $\cM_\tau$ has a submodule~$\cN_\tau$ supported on a face
strictly containing~$\tau$.  In the latter case, moving up by an
element of~$\tau$ shows that $\cN = \cN_\tau \cap \cM$ is nonzero.  In
either case, any cogenerator of~$\cN$ lies along a face $\tau' \neq
\tau$, so $0 \neq\nolinebreak \socp\cN \subseteq\nolinebreak \socp
\cM$.

On the other hand, if $\cM$ is $\tau$-coprimary then
$\Gamma_{\!\tau}\cM$ is an essential submodule of~$\cM$ because every
nonzero submodule of $\cM \subseteq \cM_\tau$ has nonzero intersection
with $\Gamma_{\!\tau}\cM_\tau$, and hence with $M \cap
\Gamma_{\!\tau}\cM_\tau = \Gamma_{\!\tau}\cM$, inside of the ambient
module~$\cM_\tau$ by Definition~\ref{d:primDecomp'}.\ref{i:coprimary}.
Theorem~\ref{t:essential-submodule} says that
$\socp\Gamma_{\!\tau}\cM$ is dense in $\socp\cM$ for all~$\tau'$.  But
$\socp\Gamma_{\!\tau}\cM = 0$ for $\tau' \neq \tau$ by
Corollary~\ref{c:soc(coprimary)'}, so density implies $\socp\cM = 0$
for $\tau' \neq \tau$.
\end{proof}

%%%%%%%%%%%%%%%%%%%%%%%%%%%%%%%%%%%%%%%%%%%%%%%%%%%%%%%%%%%%%%%%%%%%%%%%%
\subsection{Canonical primary decompositions of downsets}\label{sub:canonical}\mbox{}

\begin{lemma}\label{l:antichain}
A downset $D$ in a real polyhedral group is $\tau$-coprimary if and
only if
$$%
  D
  =
  \bigcup_{\substack{\text{\rm faces }\sigma\\\text{\rm with }\sigma\supseteq\tau}}
  \bigcup_{\raisebox{-.7ex}{$\scriptstyle\ \aa\in\ats$}} \aa + \tau - \qns
$$
for sets $\ats \subseteq \cQ$ such that the image in $\qrt \times
\nabt$ of\/ $\bigcup_{\sigma\supseteq\tau} \ats \times \{\sigma\}$ is
an antichain, and in that case $\ats$ projects to a subset of
$\deg_{\qrt}\soct[\sigma]\kk[D] \subseteq \qrt$ for each~$\sigma$.
\end{lemma}
\begin{proof}
If $D$ is $\tau$-coprimary, then it is such a union by
Theorem~\ref{t:coprimary} and Theorem~\ref{t:downset=union}, keeping
in mind the antichain consequences of Example~\ref{e:soc-Rn-downset}.

On the other hand, if $D$ is such a union, then first of all it is
stable under translation by~$\RR\tau$ because every member of the
union is.  Working in~$\qrt$, therefore, assume that $\tau = \{\0\}$.
Example~\ref{e:soc-Rn-downset} implies that every element of~$\ats$ is
a cogenerator of~$D$ with nadir~$\sigma$.
Proposition~\ref{p:sigma-nbd-cogen} produces a
$\sigma$-neighborhood~$\oas$ of~$\aa$ in~$D$ that is globally
supported on~$\{\0\}$ (and hence $\{\0\}$-coprimary).  But every
element $\bb \in D$ that precedes~$\aa$ also precedes some element
in~$\oas$; that is, $\bb \preceq \aa \implies (\bb + \cQ_+) \cap \oas
\neq \nothing$.  The union of the $\sigma$-neighborhoods $\oas$ over
all faces~$\sigma$ and elements $\aa \in \ats$ therefore
cogenerates~$D$, so $D$ is coprimary by Definition~\ref{d:primDecomp}.
\end{proof}

\begin{remark}\label{r:antichain}
The antichain condition in Lemma~\ref{l:antichain} is necessary: $\cQ$
itself is the union of all translates of~$-\cQ$, but $\cQ$ is
$\cQ_+$-coprimary, whereas $-\cQ$ is $\{\0\}$-coprimary.  Moreover,
the $\nabt$ component of the antichain condition is important; that
is, the nadirs also come into play.  For a specific example, take $D
\subseteq \RR^2$ to be the union of the open
\end{remark}\vskip -1.6ex
\begin{wrapfigure}{R}{0.15\textwidth}
  \vspace{-1.5ex}
  \includegraphics[height=11ex]{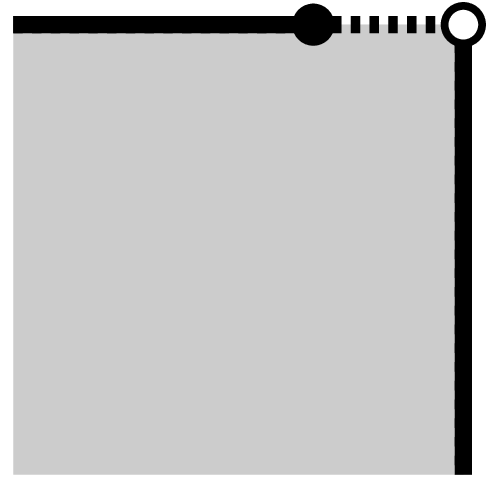}
  \vspace{-2ex}
\end{wrapfigure}
\noindent
negative quadrant cogenerated by~$\0$ and the closed negative quadrant
cogenerated by any point on the strictly negative $x$-axis.  The
$\cQ$-components of the two cogenerators are comparable in~$\cQ$, but
the nadirs are comparable the other way (it is crucial to remember
that the ordering on the nadirs is by~$\fqo$, not~$\cfq$, so smaller
faces are higher in the poset).  Of course, no claim can be made that
$\deg_{\qrt}\soct[\sigma]\kk[D]$ equals the image in~$\qrt$ of~$\ats$;
only the density claim in Theorem~\ref{t:downset=union} can be made.
%end{remark}

\begin{defn}\label{d:minimal-primDecomp-downset}
A primary decomposition (Definition~\ref{d:primDecomp}) $D =
\bigcup_{j=1\!}^k D_j$ of a downset in a real polyhedral group is
\emph{minimal} if
\begin{enumerate}
\item%
each face associated to~$D$ is associated to precisely one of the
downsets~$D_j$, and
\item%
the natural map $\soct\kk[D] \to \soct\bigoplus_{j=1}^k \kk[D_j]$ is
an isomorphism for all~faces~$\tau$.
\end{enumerate}
\end{defn}

\begin{thm}\label{t:hull-D}
Every downset $D$ in a real polyhedral group has a canonical minimal
primary decomposition
$$%
  D
  =
  \bigcup_{\raisebox{-.7ex}{$\scriptstyle\tau\in\ass D$}}
  \bigcup_{\substack{\sigma\supseteq\tau\\\,\aa\in\ats}}
  \aa + \tau - \qns,
$$
where $\ats \subseteq \cQ$ is the set of cogenerators of~$D$
along~$\tau$ with nadir~$\sigma$.
\end{thm}
\begin{proof}
That $D$ equals the union is a special case of
Theorem~\ref{t:downset=union}.  The new content is that for
fixed~$\tau$, the inner union is $\tau$-coprimary (\hspace{-.2pt}which
follows from Lemma~\ref{l:antichain})~and~the socle maps are
isomorphisms (which follows from Theorem~\ref{t:coprimary} and
Lemma~\ref{l:antichain}).%
\end{proof}

\begin{remark}\label{r:hull-D}
The $\tau$-primary component in Theorem~\ref{t:hull-D} is canonical
even though the set of cogenerators used to express the inner union
need not be; see Theorem~\ref{t:downset=union}.
\end{remark}

\begin{example}\label{e:not-Gamma_tau}
The canonical $\tau$-primary component in Theorem~\ref{t:hull-D} can
differ from the $\tau$-primary component $P_\tau(D)$ in
Definition~\ref{d:PF}.\ref{i:primary-component} and especially
Corollary~\ref{c:PF}, although it takes dimension at least~$3$ to do
it.  For a specific case, let $\tau$ be the $z$-axis in~$\RR^3$, and
let~$D_1$ be the $\{\0\}$-coprimary (Lemma~\ref{l:antichain}) downset
in~$\RR^3$ cogenerated by the nonnegative points on the surface $z =
1/(x^2 + y^2)$.  Then every point on the positive $z$-axis is
supported on~$\tau$ in~$D_1$.  That would suffice, for the present
purpose, but for the fact that $\tau$ fails to be associated to~$D_1$.
The remedy is to force $\tau$ to be associated by taking the union
of~$D_1$ with any downset $D_2 = \aa + \tau - \RR^3_+$ with $\aa =
(x,y,z)$ satisfying $xy < 0$, the point being for $D_2 \not\subseteq
D_1$ to be $\tau$-coprimary but not contain the $z$-axis itself.  The
canonical $\tau$-primary component of $D = D_1 \cup D_2$ is just $D_2$
itself, but by construction $\Gamma_{\!\tau} D$ also contains the
positive $z$-axis.  (Note: $D = D_1 \cup D_2$ is not the canonical
primary decomposition of~$D$ because $D_2$ swallows an open set of
cogenerators of~$D_1$, so these cogenerators must be omitted from the
$\{\0\}$-primary component to induce an isomorphism on socles.)  The
reason why three dimensions are needed is that $\tau$ must have
positive dimension, because elements supported on~$\tau$ must be
cloaked by those supported on a smaller face; but $\tau$ must have
codimension more than~$1$, because there must be enough room
modulo~$\RR\tau$ to have incomparable~elements.
\end{example}

%%%%%%%%%%%%%%%%%%%%%%%%%%%%%%%%%%%%%%%%%%%%%%%%%%%%%%%%%%%%%%%%%%%%%%%%%
\subsection{Minimal downset hulls of modules}\label{sub:minimal-hull}%%%%

\begin{defn}\label{d:minimal-hull}
A downset hull $\cM \to E = \bigoplus_{j=1}^k E_j$
(Definition~\ref{d:downset-hull}) of a module over a real polyhedral
group is
\begin{enumerate}
\item%
\emph{coprimary} if $E_j = \kk[D_j]$ is coprimary for all~$j$, so
$D_j$ is a coprimary downset,~and
\item%
\emph{minimal} if the induced map $\soct\cM \to \soct E$ is an
isomorphism for all faces~$\tau$.
\end{enumerate}
\end{defn}

\begin{thm}\label{t:hull-M}
Every downset-finite module~$\cM$ over a real polyhedral group admits
a minimal coprimary downset hull.
\end{thm}
\begin{proof}
Suppose that $\cM \to \bigoplus_{j=1}^k E_j$ is any finite downset
hull.  Replacing each $E_j$ by a primary decomposition of~$E_j$, using
Theorem~\ref{t:hull-D}, assume that this downset hull is coprimary.
Let $E^\tau$ be the direct sum of the $\tau$-coprimary summands
of~$E$.  Then $\soct E = \soct E^\tau$ by Theorem~\ref{t:coprimary}.
Replacing~$\cM$ with its image in~$E^\tau$, it therefore suffices to
treat the case where $\cM$ is $\tau$-coprimary and $E = E^\tau$.

The proof is by induction on the number~$k$ of summands of~$E$.  If $k
= 1$ then $M \subseteq \kk[D]$ is a submodule of a $\tau$-coprimary
downset module.  Let $D'$ be the union of the coprincipal downsets
$\aa + \tau - \qns$ over all $\aa \in \cQ$ and faces~$\sigma$ such
that the projection of~$\aa$ to~$\qrt$ is the degree of an element in
the image of the natural map $\soct[\sigma]\cM \to
\soct[\sigma]\kk[D]$.  Since that natural map is injective for
all~$\sigma$ by Theorem~\ref{t:injection}.\ref{i:phi=>soct}, it is a
consequence of Theorem~\ref{t:downset=union} that there is a
surjection $\kk[D] \onto \kk[D']$.  But Lemma~\ref{l:antichain}
implies that $\soct\kk[D'] \subseteq \soct\cM$, so $\soct\cM \to
\soct\kk[D']$ is an isomorphism.

When $k > 1$, let $\cM' = \ker(\cM \to E_k)$.  Then $\cM' \into
\bigoplus_{j=1}^{k-1} E_j$, so it has a minimal coprimary hull $\cM'
\into E'$ by induction.  The $k = 1$ case proves that $\cM'' =
\cM/\cM'$ has a minimal coprimary hull $\cM'' \into E''$.  The exact
sequence $0 \to \cM' \to \cM \to \cM'' \to 0$ yields an exact sequence
$$%
  0 \to \soct\cM' \to \soct\cM \to \soct\cM''
$$
which, if exact, automatically splits by Remark~\ref{r:soc-as-k-vect}.
Hence it suffices to prove that $\soct\cM \to \soct\cM''$ is
surjective.  For that, note that the image of~$\soct\cM$ in~$\soct E$
surjects onto its projection to~$\soct E_k$, but the image of
$\soct\cM \to \soct E_k$ is the image of the injection $\soct\cM''
\into \soct$ by construction.
\end{proof}

\begin{remark}\label{r:filtration}
The proof of the theorem shows more than the statement: any coprimary
downset hull $\cM \into E = E_1 \oplus \dots \oplus E_k$ of a
coprimary module~$\cM$ induces a filtration $0 = \cM_0 \subset \cM_1
\subset \dots \subset \cM_k = \cM$ such that $\soct\cM =
\bigoplus_{j=1}^k \soct(\cM_j/\cM_{j-1})$, and furthermore $\cM \into
E$ can be ``minimalized'', in the sense that a minimal hull~$E'$ can
be constructed from~$E$ so that $\soct\cM \cong \soct E'$ decomposes
as direct sum of factors $\soct(\cM_j/\cM_{j-1}) \cong \soct E'_j$.
Reordering the summands~$E_j$ yields another filtration of~$\cM$ with
the same property.  That $\soct\cM$ breaks up as a direct sum in so
many ways should not be shocking, in view of
Remark~\ref{r:soc-as-k-vect}.  The main content is that all of the
socle elements of $\cM/\cM_{k-1}$ are inherited from~$\cM$,
essentially because $\cM_{k-1}$ is the kernel of a homomorphism to a
direct sum of downset modules, so $\cM_{k-1}$ has no generators that
are not inherited from~$\cM$.
\end{remark}

\begin{remark}\label{r:analogue-of-injres}
Theorem~\ref{t:hull-M} is the analogue of existence of minimal
injective hulls for finitely generated modules over noetherian rings
\cite[Section~3.2]{bruns-herzog} (see also
Proposition~\ref{p:determined}).  The difference here is that a direct
sum---as opposed to direct product---can only be attained by gathering
cogenerators into finitely many~bunches.
\end{remark}

\begin{example}\label{e:downset-indecomposable}
The indicator module for the disjoint union of the strictly negative
axes in the plane injects in an appropriate way into one downset
module (the punctured negative quadrant) or a direct sum of two
(negative quadrants missing one boundary axis each).  Thus the
``required number'' of downsets for a downset hull of a given module
is not necessarily obvious and might not be a functorial invariant.
This may sound bad, but it should not be unexpected: the quotient by
an artinian monomial ideal in an ordinary polynomial ring can have
socle of arbitrary finite dimension, so the number of coprincipal
downsets required is well defined, but if downsets that are not
necessarily coprincipal are desired, then any number between $1$ and
the socle dimension would suffice.  This phenomenon is related to
Remark~\ref{r:soc-as-k-vect}: breaking the socle of a downset into two
reasonable pieces expresses the original downset as a union of the two
downsets cogenerated by the pieces.
\end{example}

%%%%%%%%%%%%%%%%%%%%%%%%%%%%%%%%%%%%%%%%%%%%%%%%%%%%%%%%%%%%%%%%%%%%%%%%%
\subsection{Minimal primary decomposition of modules}\label{sub:minimal-primary}

\begin{defn}\label{d:minimal-primary}
A primary decomposition $\cM \into \bigoplus_{i=1}^r \cM/\cM_i$
(Definition~\ref{d:primDecomp'}.\ref{i:primdecomp}) of a module over a
real polyhedral group is \emph{minimal} if $\soct\cM \to
\soct\bigoplus_{i=1}^r \cM/\cM_i$ is an isomorphism for all
faces~$\tau$.
\end{defn}

\begin{defn}\label{d:primary-component}
Given a coprimary downset hull $\cM \into E$ of an arbitrary
downset-finite module~$\cM$ over a real polyhedral group, write
$E^\tau$ for the direct sum of all summands of~$E$ that are
$\tau$-coprimary.  The kernel $\cM^\tau$ of the composite homomorphism
$\cM \to E \to E^\tau$ is the \emph{$\tau$-primary component of~$0$}
for this particular downset hull of~$\cM$.
\end{defn}

\begin{thm}\label{t:minimal-primary}
Every downset-finite module~$\cM$ over a real polyhedral group admits
a minimal primary decomposition.  In fact, if $\cM \into E$ is a
coprimary downset hull then $\cM \into \bigoplus_\tau \cM/\cM^\tau$ is
a primary decomposition that is minimal if $\cM \into E$ is minimal.
\end{thm}
\begin{proof}
Fix a coprimary downset hull $\cM \into E$.  The quotient
$\cM/\cM^\tau$ is $\tau$-coprimary since it is a submodule of the
coprimary module $E^\tau$ by construction, and $\cM \to \bigoplus_\tau
\cM/\cM^\tau$ is injective because the injection $\cM \into
\bigoplus_\tau E^\tau = E$ factors through $\bigoplus_\tau
\cM/\cM^\tau \subseteq E$.

\pagebreak
Theorem~\ref{t:coprimary} implies that $\socp(\cM/\cM^\tau) = 0$
unless $\tau = \tau'$, regardless of whether $\cM \into E$ is minimal.
And if the hull is minimal, then $\soct\cM \to \soct E^\tau$ is an
isomorphism (by hypothesis) that factors through the injection
$\soct(\cM/\cM^\tau) \into \soct E^\tau$ (by construction), forcing
$\soct\cM \cong \soct(\cM/\cM^\tau)$ to be an isomorphism for
all~$\tau$.
\end{proof}

\begin{cor}\label{c:hull-D}
Every indicator subquotient\/ $\kk[S]$ of\/~$\kk[\cQ]$ over a real
polyhedral group $\cQ$ has a canonical minimal primary decomposition.
(That $\kk[S]$ is a subquotient means that $S \subseteq \cQ$ is the
intersection of an upset and a downset in~$\cQ$.)
\end{cor}
\begin{proof}
Let $D$ be the downset cogenerated by~$S$.  Let $\kk[D] \into E$ be
the minimal coprimary downset hull of~$\kk[D]$ resulting from the
canonical minimal primary decomposition of~$D$ in
Theorem~\ref{t:hull-D}.  Then the composite map $\kk[S] \into \kk[D]
\into E$ is a minimal downset hull of~$\kk[S]$ to which 
Theorem~\ref{t:minimal-primary} applies.
\end{proof}

%%%%%%%%%%%%%%%%%%%%%%%%%%%%%%%%%%%%%%%%%%%%%%%%%%%%%%%%%%%%%%%%%%%%%%%%%
\section{Socles and essentiality over discrete polyhedral groups}\label{s:discrete}
%%%%%%%%%%%%%%%%%%%%%%%%%%%%%%%%%%%%%%%%%%%%%%%%%%%%%%%%%%%%%%%%%%%%%%%%%

The theory developed for real polyhedral groups in
Sections~\ref{s:socle}--\ref{s:gen-functors} applies as well to
discrete polyhedral groups (Definition~\ref{d:discrete-polyhedral}).
The theory is easier in the discrete case, in the sense that only
closed cogenerator functors are needed, and none of the density
considerations in Sections~\ref{s:minimality}--\ref{s:density} are
relevant.  The deduction of the discrete case is elementary, but it is
worthwhile to record the results, both because they are useful and for
comparison with the real polyhedral case.

For the analogue of Theorem~\ref{t:injection}, the notion of
divisibility in Definition~\ref{d:divides} makes sense, when $\sigma =
\{\0\}$, verbatim in the discrete polyhedral setting: an element $y
\in \cM_\bb$ \emph{divides} $x \in \cM_\aa$ if $\bb \in \aa - \cQ_+$
and $y \mapsto x$ under the natural map $\cM_\bb \to \cM_\aa$.

\begin{thm}[Discrete essentiality of socles]\label{t:discrete-injection}
Fix a homomorphism $\phi: \cM \to \cN$ of modules over a discrete
polyhedral group~$\cQ$.
\begin{enumerate}
\item\label{i:discrete-phi=>soct}%
If $\phi$ is injective then $\socct\phi: \socct\cM \to \socct\cN$ is
injective for all faces~$\tau$ of~$\cQ_+$.
\item\label{i:discrete-soct=>phi}%
If $\socct\phi: \socct\cM \to \socct\cN$ is injective for all
faces~$\tau$ of~$\cQ_+\!$ and $\cM$ is downset-finite, then $\phi$
is~injective.
\end{enumerate}
In fact, each homogeneous element of~$\cM\!$ divides some closed
cogenerator of~$\cM$.
\end{thm}
\begin{proof}
Item~\ref{i:discrete-phi=>soct} is a special case of
Proposition~\ref{p:left-exact-tau-closed}.  Item~\ref{i:soct=>phi}
follows from the divisibility claim, for if $y$ divides a closed
cogenerator~$s$ along~$\tau$ then $\phi(y) \neq 0$ whenever
$\socct\phi(\wt s) \neq\nolinebreak 0$, where $\wt s$ is the image
of~$s$ in~$\socct\cM$.  The divisibility claim follows from the case
where $\cM$ is generated by $y \in \cM_\bb$.  But $\<y\>$ is a
noetherian $\kk[\cQ_+]$-module and hence has an associated prime.
This prime equals the annihilator of some homogeneous element
of~$\<y\>$, and the quotient of~$\kk[\cQ_+]$ modulo this prime
is~$\kk[\tau]$ for some face~$\tau$ \cite[Section~7.2]{cca}.  That
means, by definition, that the homogeneous element is a closed
cogenerator along~$\tau$ divisible by~$y$.
\end{proof}

\pagebreak
The discrete analogue of Theorem~\ref{t:dense} is simpler in both
statement and proof.

\begin{thm}\label{t:subfunctor-discrete}
Fix subfunctors $\oS_\tau \subseteq \socct$ for all faces~$\tau$ of a
discrete polyhedral group.  Theorem~\ref{t:discrete-injection} holds
with $\oS$ in place of\/ $\socc$ if and only if\/ $\oS_\tau = \socct$
for all~$\tau$.
\end{thm}
\begin{proof}
Every subfunctor of any left-exact functor takes injections to
injections; therefore
Theorem~\ref{t:discrete-injection}.\ref{i:phi=>soct} holds for any
subfunctor of~$\socct$ by Proposition~\ref{p:left-exact-tau-closed}.
The content is that
Theorem~\ref{t:discrete-injection}.\ref{i:soct=>phi} fails as soon as
$\oS_\tau\cM \subsetneq \socct\cM$ for some module~$\cM$ and some
face~$\tau$.  To prove that failure, suppose $\wt s \in \socct\cM
\minus \oS_\tau\cM$ for some closed cogenerator~$s$ of~$\cM$
along~$\tau$.  Then $\<s\> \subseteq \cM$ induces an injection
$\oS_\tau\<s\> \into \oS_\tau\cM$, but by construction the image of
this homomorphism is~$0$, so~$\oS_\tau\<s\> = 0$.  But
$\socct[\tau']\<s\> = 0$ for all $\tau' \neq \tau$ because $\<s\>$ is
abstractly isomorphic to~$\kk[\tau]$, which has no associated primes
other than the kernel of~$\kk[\cQ_+] \onto \kk[\tau]$.  Consequently,
applying~$\oS_{\tau'}$ to the homomorphism $\phi: \<s\> \to 0$ yields
an injection $0 \into 0$ for all faces~$\tau'$ even though $\phi$ is
not injective.
\end{proof}

The analogue of Theorem~\ref{t:essential-submodule} is similarly
simpler.

\begin{thm}\label{t:discrete-essential-submodule}
In any module $\cM$ over a discrete polyhedral group, $\cM'$ is an
essential submodule if and only if\/ $\socct\cM' = \socct\cM$ for all
faces~$\tau$.
\end{thm}
\begin{proof}
First assume that $\cM'$ is not an essential submodule, so $\cN \cap
\cM' = 0$ for some nonzero submodule $\cN \subseteq \cM$.  Any closed
cogenerator~$s$ of~$\cN$ along any face~$\tau$ maps to a nonzero
element of $\socct\cM$ that lies outside of~$\socct\cM'$.  Conversely,
if $\socct\cM' \neq \socct\cM$, then any closed cogenerator of~$\cM$
that maps to an element $\socct\cM \minus \socct\cM'$ generates a
nonzero submodule of~$\cM$ whose intersection with~$\cM'$ is~$0$.
\end{proof}

The analogue of Theorem~\ref{t:hull-D} uses slightly modified
definitions but its proof~is~\mbox{easier}.

\begin{defn}\label{d:discrete-minimal-primDecomp-downset}
A primary decomposition (Definition~\ref{d:primDecomp}) $D =
\bigcup_{j=1\!}^k D_j$ of a downset in a discrete polyhedral group is
\emph{minimal} if
\begin{enumerate}
\item%
the downsets~$D_j$ are coprimary for distinct associated faces of~$D$,
and
\item%
the natural map $\socct\kk[D] \to \socct\bigoplus_{j=1}^k \kk[D_j]$ is
an isomorphism for all~faces~$\tau$,
\end{enumerate}
where $\tau$ is \emph{associated} if some element generates an upset
in~$D$ that is a translate of~$\tau$.
\end{defn}

\begin{thm}\label{t:discrete-hull-D}
Every downset $D$ in a discrete polyhedral group has a canonical
minimal primary decomposition $D = \bigcup_\tau D^\tau$ as a union of
coprimary downsets
$$%
  D^\tau
  =
  \bigcup_{\aa_\tau \in \deg\socct\kk[D]} \aa_\tau - \cQ_+,
$$
where $\aa_\tau$ is viewed as an element in~$\qzt$ to write $\aa_\tau
\in \deg\socct\kk[D]$ but $\aa_\tau \subseteq \cQ$ is viewed as a
coset of\/~$\ZZ\tau$ to write $\aa_\tau - \cQ_+$.
\end{thm}
\begin{proof}
The downset~$D$ is contained in the union by the final line of
Theorem~\ref{t:discrete-injection}, but the union is contained in~$D$
because every closed cogenerator of~$D$ is an element of~$D$.  It
remains to show that $D^\tau$ is coprimary and that the socle maps are
isomorphisms.

Each nonzero homogeneous element $y \in \kk[D^\tau]$ divides an
element~$s_y$ whose degree lies in some coset~$\aa_\tau$ by
construction.  As $D^\tau \subseteq D$, each such element~$s_y$ is a
closed cogenerator of~$D^\tau$ along~$\tau$.  Therefore $\kk[D^\tau]$
is coprimary, inasmuch as no prime other than the one associated
to~$\kk[\tau]$ can be associated to~$D^\tau$.  The same argument shows
that these elements~$s_y$ generate an essential submodule of~$D^\tau$,
and then Theorem~\ref{t:discrete-essential-submodule} yields the
isomorphism on socles.
\end{proof}

\begin{remark}\label{r:discrete-irred-decomp}
Each of the individual coprincipal downsets in
Theorem~\ref{t:discrete-hull-D} is an \emph{irreducible component}
of~$D$.  Thus the theorem can also be interpreted as saying that every
downset in a discrete polyhedral group has a unique irredundant
irreducible decomposition; the irredundant condition stems from the
socle isomorphisms.
\end{remark}

\begin{defn}\label{d:discrete-minimal-hull}
A downset hull $\cM \to E = \bigoplus_{j=1}^k E_j$
(Definition~\ref{d:downset-hull}) of a module over a discrete
polyhedral group is
\begin{enumerate}
\item%
\emph{coprimary} if $E_j = \kk[D_j]$ is coprimary for all~$j$, so
$D_j$ is a coprimary downset,~and
\item%
\emph{minimal} if the induced map $\socct\cM \to \socct E$ is an
isomorphism for all faces~$\tau$.
\end{enumerate}
\end{defn}

The discrete analogue of Theorem~\ref{t:hull-M} appears to be new.

\begin{thm}\label{t:discrete-hull-M}
Every downset-finite module~$\cM$ over a discrete polyhedral group
admits a minimal coprimary downset hull.
\end{thm}
\begin{proof}
The argument follows that of Theorem~\ref{t:hull-M}.  In the course of
the proof, note that the discrete analogue of
Theorem~\ref{t:coprimary} is the definition of associated prime and
that the analogue of Remark~\ref{r:soc-as-k-vect} holds (more easily)
in the discrete polyhedral setting.
\end{proof}

\begin{remark}\label{r:discrete-filtration}
Remark~\ref{r:filtration} holds verbatim over discrete polyhedral
groups.
\end{remark}

Finally, here is the discrete version of minimal primary
decomposition.

\begin{defn}\label{d:discrete-minimal-primary}
A primary decomposition $\cM \into \bigoplus_{i=1}^r \cM/\cM_i$
(Definition~\ref{d:primDecomp'}.\ref{i:primdecomp}) of a module over a
discrete polyhedral group is \emph{minimal} if $\socct\cM \to
\socct\bigoplus_{i=1}^r \cM/\cM_i$ is an isomorphism for all
faces~$\tau$.
\end{defn}

\begin{defn}\label{d:discrete-primary-component}
Given a coprimary downset hull $\cM \into E$ of an arbitrary
downset-finite module~$\cM$ over a discrete polyhedral group, write
$E^\tau$ for the direct sum of all summands of~$E$ that are
$\tau$-coprimary.  The kernel $\cM^\tau$ of the composite homomorphism
$\cM \to E \to E^\tau$ is the \emph{$\tau$-primary component of~$0$}
for this particular downset hull of~$\cM$.
\end{defn}

\begin{thm}\label{t:discrete-minimal-primary}
Every downset-finite module~$\cM$ over a discrete polyhedral group
admits a minimal primary decomposition.  In fact, if $\cM \into E$ is
a coprimary downset hull then $\cM \into \bigoplus_{\tau\!}
\cM/\cM^\tau\!$ is a primary decomposition that is minimal if $\cM
\hspace{-.45pt}\into\hspace{-.45pt}\nolinebreak E$~is~\mbox{minimal}.
\end{thm}
\begin{proof}
Follow the proof of Theorem~\ref{t:minimal-primary}.
\end{proof}

\begin{cor}\label{c:discrete-hull-D}
Every indicator subquotient\/ $\kk[S]$ of\/~$\kk[\cQ]$ over a discrete
polyhedral group $\cQ$ has a canonical minimal primary decomposition.
% (That $\kk[S]$ is a subquotient means that $S \subseteq \cQ$ is the
% intersection of an upset and a downset in~$\cQ$.)
\end{cor}
\begin{proof}
The proof of Corollary~\ref{c:hull-D} works verbatim, citing the
discrete analogues of Theorems~\ref{t:hull-D}
and~\ref{t:minimal-primary}, namely Theorems~\ref{t:discrete-hull-D}
and~\ref{t:discrete-minimal-primary}.
\end{proof}

%%%%%%%%%%%%%%%%%%%%%%%%%%%%%%%%%%%%%%%%%%%%%%%%%%%%%%%%%%%%%%%%%%%%%%%%%
\section{Generator functors and tops}\label{s:gen-functors}%%%%%%%%%%%%%%
%%%%%%%%%%%%%%%%%%%%%%%%%%%%%%%%%%%%%%%%%%%%%%%%%%%%%%%%%%%%%%%%%%%%%%%%%

The theory of generators is Matlis dual (Section~\ref{sub:matlis}) to
the theory of cogenerators.  Every result for socles, downsets, and
cogenerators therefore has a dual.  All of these dual statements can
be formulated so as to be straightforward, but sometimes they are less
natural (see Remarks~\ref{r:global-closed-gen}
and~\ref{r:no-local-top}, for example), sometimes they are weaker (see
Remark~\ref{r:dir-vs-inv}), and sometimes there are natural
formulations that must be proved equivalent to the straightforward
dual (see Definition~\ref{d:topr} and Theorem~\ref{t:top=socvee}, for
example).  This section presents those Matlis dual notions that are
used in later sections, along with a few notions or results for which
no direct dual is present in earlier sections, such as
Proposition~\ref{p:surjection-ZZ}, Proposition~\ref{p:surjection-RR},
and Lemma~\ref{l:mrx}.

%%%%%%%%%%%%%%%%%%%%%%%%%%%%%%%%%%%%%%%%%%%%%%%%%%%%%%%%%%%%%%%%%%%%%%%%%
\subsection{Lower boundary functors}\label{sub:lower-boundary}\mbox{}%%%%

\noindent
The following are Matlis dual to Definition~\ref{d:atop-sigma}.
Lemma~\ref{l:natural}, and Definition~\ref{d:upper-boundary}.

\begin{defn}\label{d:beneath-xi}
For a module~$\cM$ over a real polyhedral group~$\cQ$, a face~$\xi$
of~$\cQ_+$, and a degree $\bb \in \cQ$, the \emph{lower boundary
beneath~$\xi$} at~$\bb$ in~$\cM$ is the vector space
$$%
  (\lx\cM)_\bb
  =
  \cM_{\bb+\xi}
  =
  \invlim_{\bb' \in \bb + \xi^\circ} \cM_{\bb'}.
$$
\end{defn}

\begin{lemma}\label{l:natural-dual}
The structure homomorphisms of~$\cM$ as a $\cQ$-module induce natural
homomorphisms $\cM_{\bb+\xi} \to \cM_{\cc+\eta}$ for $\bb \preceq
\cc$ in~$\cQ$ and faces $\xi \subseteq \eta$
of~$\cQ_+$.\hfill$\square$
\end{lemma}

\begin{remark}\label{r:semilattice=monoid'}
In contrast with Remark~\ref{r:semilattice=monoid}, the relevant
monoid structure here on the face poset~$\cfq$ of the positive
cone~$\cQ_+$ is dual to the monoid denoted~$\fqo$.  In this case the
monoid axioms use that $\cfq$ is a bounded join semilattice, the
monoid unit being $\{\0\}$.  The induced partial order on~$\cfq$ is
the usual one, with $\xi \preceq \eta$ if~$\xi \subseteq \eta$.
\end{remark}

\begin{defn}\label{d:lower-boundary}
Fix a module~$\cM$ over a real polyhedral group~$\cQ$ and a degree
$\bb \in \cQ$.  The \emph{lower boundary functor} takes~$\cM$ to the
$\cQ \times \cfq$-module $\del\cM$ whose fiber over $\bb \in \cQ$ is
the $\cfq$-module
$$%
  (\del\cM)_\bb
  =
  \bigoplus_{\xi \in \cfq} \cM_{\bb+\xi}
  =
  \bigoplus_{\xi \in \cfq} (\lx\cM)_\bb.
$$
The fiber of $\del\cM$ over $\xi \in \cfq$ is the \emph{lower
boundary~$\lx\cM\!$ of~$\cM$ beneath~$\xi$}.
\end{defn}

\begin{remark}\label{r:dir-vs-inv}
Direct and inverse limits play differently with vector space duality.
Consequently, although the notion of lower boundary functor is
categorically dual to the notion of upper boundary functor, the
duality only coincides unfettered with vector space duality in one
direction, and some results involving tops are necessarily weaker than
the corresponding results for socles; compare
Theorem~\ref{t:injection} with~\ref{t:surjection-RR} and
Example~\ref{e:surjection}, for instance.  To make precise statements
throughout this section on generator functors, starting with
Lemma~\ref{l:lx-vee}, it is necessary to impose a finiteness condition
that is somewhat stronger than $\cQ$-finiteness
(Definition~\ref{d:Q-finite}).
\end{remark}

\begin{defn}\label{d:infinitesimally-Q-finite}
A module~$\cM$ over a real polyhedral group~$\cQ$ is
\emph{infinitesimally $\cQ$-finite} if its lower boundary module
$\del\cM$ is $\cQ$-finite.
\end{defn}

\begin{lemma}\label{l:lx-vee}
If $\xi$ is a face of a real polyhedral group~$\cQ$, then
\begin{enumerate}
\item\label{i:all}%
$\lx(\cM^\vee) = (\dx\cM)^\vee$ for all $\cQ$-modules~$\cM$, and
\item\label{i:inf}%
$(\lx\cM)^\vee = \dx(\cM^\vee)$ if $\cM$ is infinitesimally
$\cQ$-finite
\end{enumerate}
\end{lemma}
\begin{proof}
Degree by degree $\bb \in \cQ$, the first of these
is because the vector space dual of a direct limit is the inverse
limit of the vector space duals.
% \noindent
% Proof: Given $\bigoplus V_i \onto \dirlim V_i$, taking vector space
% duals yields $\prod V_i^* \otni (\dirlim V_i)^*$.  But
% \begin{align*}
% \bigl(\prod_k\ell_k\bigr)(\dots,x_i,\dots,-f_{ij}x_i,\dots) = 0
%   &\iff \ell_i(x_i) - \ell_j(f_{ij}x_i) = 0\ \forall...
% \\[-2.2ex]
%   &\iff \ell_i = \ell_j \circ f_{ij}\ \forall...
% \\
%   &\iff \ell_i = f_{ij}^*\ell_j\ \forall...
% \\
%   &\iff \ell \in \invlim(V_i^*)
% \end{align*}
Swapping ``direct'' and ``inverse'' only works with additional
hypotheses, and one way to ensure these is to assume infinitesimal
$\cQ$-finiteness of~$\cM$.  Indeed, then $\cM = \del^{\{\0\}}\cM$ is
$\cQ$-finite, so replacing $\cM$ with~$\cM^\vee$ in the first item
yields $\lx\cM = \bigl(\dx(\cM^\vee)\bigr){}^\vee$ by
Lemma~\ref{l:vee-vee}.  Thus $(\lx\cM)^\vee = \dx(\cM^\vee)$, as
$\lx\cM$---and hence $\bigl(\dx(\cM^\vee)\bigr){}^\vee\!$ and
$\dx(\cM^\vee)$---is~also~\mbox{$\cQ$-finite}.
\end{proof}

\begin{example}\label{e:infinitesimally-Q-finite}
Any module~$\cM$ that is a quotient of a finite direct sum of upset
modules (``upset-finite'' in Definition~\ref{d:minimal-cover}) over a
real polyhedral group is infinitesimally $\cQ$-finite.  Indeed, the
Matlis dual of such a quotient is a downset hull demonstrating that
$\cM^\vee$ is downset-finite and hence $\cQ$-finite.
Proposition~\ref{p:downset-upper-boundary} and exactness of upper
boundary functors (Lemma~\ref{l:exact-delta}) implies that
$\delta(\cM^\vee)$ remains downset-finite and hence $\cQ$-finite.
Applying Lemma~\ref{l:lx-vee}.\ref{i:all} to~$\cM^\vee$ and using that
$(\cM^\vee)^\vee\! = \cM$ (Lemma~\ref{l:vee-vee}) on the left-hand
side yields that $\del\cM$ is $\cQ$-finite.  This example includes all
finitely encoded modules by
Theorem~\ref{t:syzygy}.\ref{i:upset-presentation}.
\end{example}

\begin{prop}\label{p:exact-lower-boundary}
The category of infinitesimally $\cQ$-finite modules over a real
polyhedral group~$\cQ$ is a full abelian subcategory of the category
of $\cQ$-modules.  Moreover, the lower boundary functor is exact on
this subcategory.
\end{prop}
\begin{proof}
Use Matlis duality, in the form of Lemma~\ref{l:lx-vee}, along with
Lemma~\ref{l:exact-delta}.
\end{proof}

%%%%%%%%%%%%%%%%%%%%%%%%%%%%%%%%%%%%%%%%%%%%%%%%%%%%%%%%%%%%%%%%%%%%%%%%% 
\subsection{Closed generator functors}\label{sub:closed-gen}\mbox{}%%%%%%

\noindent
Here is the Matlis dual of Definition~\ref{d:socc}.  Recall the
skyscraper $\cP$-module $\kk_p$ there.

\begin{defn}\label{d:topc}
Fix an arbitrary poset~$\cP$.  The \emph{closed generator functor}
$\kk \otimes_\cP -$ takes each $\cP$-module $\cN$ to its \emph{closed
top}: the quotient $\cP$-module
$$%
  \topc\cN = \kk \otimes_\cP \cN = \bigoplus_{p\in\cP} \kk_p\otimes_\cP\cN.
$$
When it is important to specify the poset, the notation $\cptop$ is
used instead of~$\topc\!$.  A \emph{closed generator} of
\emph{degree}~$p \in P$ is a nonzero element in $(\topc\cN)_p$.
\end{defn}

\begin{remark}\label{r:nttop}
$\cpsoc\cN \into \cN$ is the universal $\cP$-module monomorphism that
is~$0$ when composed with all nonidentity maps induced by going up
in~$\cP$.  The Matlis dual notion is $\cN \onto \cptop\cN$, the
universal $\cP$-module epimorphism that is~$0$ when composed with all
nonidentity maps induced by going up in~$\cP$.
\end{remark}

%%%%%%%%%%%%%%%%%%%%%%%%%%%%%%%%%%%%%%%%%%%%%%%%%%%%%%%%%%%%%%%%%%%%%%%%%
\subsection{Closed generator functors along faces}\label{sub:closed-gen-along}\mbox{}

\noindent
Generators along faces of partially ordered groups make sense just as
cogenerators along faces do; however, they are detected not by
localization but by the Matlis dual operation in
Example~\ref{e:dual-of-localization}, which is likely unfamiliar (and
is surely less familiar than localization).  An element in the
following can be thought of as an inverse limit of elements of~$\cM$
taken along the negative of the face~$\rho$.  This is Matlis dual to
the construction of the localization~$\cM_\rho$ as a direct limit.

\begin{defn}\label{d:cmr}
Fix a face~$\rho$ of a partially ordered group~$\cQ$ and a
$\cQ$-module~$\cM$.~~Set
$$%
  \cM^\rho = \hhom_\cQ\bigl(\kk[\cQ_+]_\rho,\cM\bigr).
$$
\end{defn}

The following is Matlis dual to
Definition~\ref{d:socct}.\ref{i:global-socc-tau}; see
Theorem~\ref{t:topc=socc^vee}.  Duals for the rest of
Definition~\ref{d:socct} are omitted for reasons detailed in
Remarks~\ref{r:global-closed-gen} and~\ref{r:no-local-top}.

\begin{defn}\label{d:topct}
Fix a partially ordered group~$\cQ$, a face~$\rho$, and a
$\cQ$-module~$\cM$.  The \emph{closed generator functor along~$\rho$}
takes $\cM$ to its \emph{closed top along~$\rho$}:
$$%
  \topcr\cM
  =
  \bigl(\kk[\rho]\otimes_\cQ\cM\bigr){}^{\rho\!}\big/\rho.
$$
\end{defn}

\begin{remark}\label{r:global-closed-gen}
The notion of (global) closed cogenerator has a Matlis dual, but since
the dual of an element---equivalently, a homomorphism
from~$\kk[\cQ_+]$---is not an element, the notion of closed generator
is not Matlis dual to a standard notion related to socles.  See
Remark~\ref{r:closed-gen}, which defines a generator along a
face~$\rho$ as an element of~$\cM^\rho$ that is not a multiple of any
generator of lesser degree.  Making this precise requires care
regarding what ``lesser'' means.
\end{remark}

\begin{remark}\label{r:no-local-top}
The notion of local socle has a Matlis dual, but it is not in any
sense a local top, because localization does not Matlis dualize to
localization (Example~\ref{e:dual-of-localization}).  Instead, Matlis
dualizing the local socle yields a functor $\kk \otimes_{\qzr}
\cM^\rho$ that surjects onto~$\topcr\cM$ by the Matlis dual of
Proposition~\ref{p:local-vs-global}.  Local socles found uses in
proofs here and there, such as Corollary~\ref{c:at-most-one},
Proposition~\ref{p:sigma-nbd-cogen}, Corollary~\ref{c:soc(coprimary)},
Lemma~\ref{l:nearby}, and Corollary~\ref{c:soc(coprimary)'}, via
Proposition~\ref{p:local-vs-global}.  But since Matlis duals of
statements hold regardless of their proofs, given appropriate
finiteness conditions (Definition~\ref{d:infinitesimally-Q-finite}),
local socles and their Matlis~duals have no further use in this paper.
\end{remark}

That said, a related but simpler functor that surjects onto~$\topcr$
plays a crucial role in the construction of birth modules: the next
result captures the general concept that the ``elements of~$\cM$ of
degree~$\bb_\rho \in \qzr$'' surject onto the top of~$\cM$
along~$\rho$~in~degree~$\bb_\rho$.

\begin{prop}\label{p:surjection-ZZ}
Over any partially ordered group~$\cQ$ there is a natural surjection
$$%
  \cM^\rho/\rho
  \onto
  \topcr\cM
$$
of modules over $\qzr$ for any face~$\rho$ of~$\cQ$.
\end{prop}
\begin{proof}
Tensor the surjection $\kk[\cQ_+] \onto \kk[\rho]$ with~$\cM$ to get a
surjection $\cM \onto \kk[\rho] \otimes_\cQ\nolinebreak \cM$.  Then
apply Lemma~\ref{l:matlis-pair} and Lemma~\ref{l:exact-qr}.
\end{proof}

In the next lemma, a prerequisite to the duality of closed socles and
tops, note that localization of~$\cN$ along~$\rho$ on the left-hand
side is hiding in the quotient-restriction.

\begin{lemma}\label{l:N/tau-dual}
For any module~$\cN$ over a partially ordered group and any
face~$\rho$,
$$%
  (\cnr)^\vee = (\cN^\vee)^\rho/\rho
  \qquad\text{and}\qquad
  (\cN^\vee)/\rho = (\cN^\rho/\rho)^\vee.
$$
\end{lemma}
\begin{proof}
This is Example~\ref{e:dual-of-localization} plus the observation that
quotient-restriction along~$\rho$ commutes with Matlis duality on
modules that are already localized along~$\rho$, namely
$$%
  (\cnr)^\vee = (\cN_\rho/\rho)^\vee = (\cN_\rho)^\vee\!/\rho,
$$
as can be seen directly from Definitions~\ref{d:matlis}
and~\ref{d:quotient-restriction}.
\end{proof}

\begin{thm}\label{t:topc=socc^vee}
If $\cQ$ is a partially ordered group and $\cM$ is any $\cQ$-module,
then
$$%
  (\topcr\cM)^\vee = \socct[\rho](\cM^\vee)
  \qquad\text{and}\qquad
  \topcr(\cM^\vee) = (\socct[\rho]\cM)^\vee.
$$
\end{thm}
\begin{proof}
The two are similar, but for the record both are written out:
\begin{align*}
(\topcr\cM)^\vee
  &=
  \bigl((\kk[\rho]\otimes_\cQ\cM)^\rho/\rho\bigr){}^\vee
\\&=
  \bigl(\kk[\rho]\otimes_\cQ\cM\bigr){}^\vee\!\big/\rho
  \text{ by Lemma~\ref{l:N/tau-dual}}
\\&=
  \hhom_\cQ\bigl(\kk[\rho],\cM^\vee\bigr)/\rho
  \text{ by Example~\ref{e:matlis}}
\\&=
  \socct[\rho](\cM^\vee),
\\[1ex]
\text{and}\qquad
(\socct[\rho]\cM)^\vee
  &=
  \bigl(\hhom_\cQ\bigl(\kk[\rho],\cM\bigr)/\rho\bigr){}^\vee
\\*&=
  \bigl(\hhom_\cQ\bigl(\kk[\rho],\cM\bigr){}^\vee\bigr){}^{\rho\!}\big/\rho
  \text{ by Lemma~\ref{l:N/tau-dual}}
\\*&=
  \bigl(\kk[\rho]\otimes_\cQ\cM^\vee\bigr){}^{\rho\!}\big/\rho
  \text{ by Example~\ref{e:matlis}}
\\*&=
  \topcr(\cM^\vee).
\qedhere
\end{align*}
\end{proof}

%%%%%%%%%%%%%%%%%%%%%%%%%%%%%%%%%%%%%%%%%%%%%%%%%%%%%%%%%%%%%%%%%%%%%%%%%
\subsection{Generator functors over real polyhedral groups}\label{sub:gen}\mbox{}

\noindent
Here is the Matlis dual to Definition~\ref{d:upper-boundary-tau},
using Definition~\ref{d:lower-boundary}.

\begin{defn}\label{d:lower-boundary-rho}%\label{d:Delta-rho}
For a face~$\rho$ of real polyhedral group, set $\nabro =
(\nabr)^\op$ the open star of~$\rho$ (Example~\ref{e:nabla}) with the
partial order opposite to Definition~\ref{d:upper-boundary-tau}, so
$$%
  \xi \preceq \eta \text{ in }\nabro\text{ if }\xi \subseteq \eta.
$$
The \emph{lower boundary functor
along~$\rho$} takes~$\cM$ to the $\cQ \times \nabro$-module $\lr\cM =
\bigoplus_{\xi\in\nabro}\lx[\rho]\cM$.
\end{defn}

\begin{prop}\label{p:nrtop}
Fix a face~$\rho$ of a real polyhedral group~$\cQ$.  The Matlis
dual~$\cN^\vee$ over~$\cQ$ of any module~$\cN$ over $\cQ\times\nabro$
is naturally a module over $\cQ\times\nabr$, and the same holds with
$\nabr$ and $\nabro$ swapped.  Moreover,
$$%
  (\nrtop\cN)^\vee = \nrsoc(\cN^\vee)
  \qquad\text{and}\qquad
  \nrtop(\cN^\vee) = (\nrsoc\cN)^\vee,
$$
where the Matlis duals are taken over~$\cQ$.
\end{prop}
\begin{proof}
Matlis duality over~$\cQ$ already reverses the arrows in the
$\nabro$-module structure on~$\cN$, making $\cN^\vee$ into a module
over~$\cQ\times\nabr$.  An adjointness calculation then yields
\begin{align*}
(\nrtop\cN)^\vee
  &=
  (\kk\otimes_{\nabro}\cN)^\vee
\\&=
  \hhom_{\nabr}(\kk,\cN^\vee)
\\&=
  \nrsoc(\cN^\vee),
\end{align*}
and the other adjointness calculation is similar.
\end{proof}

\begin{cor}\label{c:lr-vs-dr}
For a face~$\rho$ over a real polyhedral group~$\cQ$, the lower
boundary is Matlis dual over~$\cQ$ to the upper boundary: as modules
over $\cQ\times\nabro$,
\begin{enumerate}
\item%
$\lr(\cM^\vee) = (\dr\cM)^\vee$ for all $\cQ$-modules~$\cM$, and
\item%
$(\lr\cM)^\vee = \dr(\cM^\vee)$ if $\cM$ is infinitesimally
$\cQ$-finite.
\end{enumerate}
\end{cor}
\begin{proof}
Lemma~\ref{l:lx-vee} plus the first part of Proposition~\ref{p:nrtop}.
\end{proof}

Next is the Matlis dual of Definition~\ref{d:kats}, using the
skyscraper modules~$\kbrx$ there, followed by the Matlis dual of
Definition~\ref{d:soct}.

\begin{defn}\label{d:kbrx}
Fix a partially ordered group~$\cQ$, a face~$\rho$, and an arbitrary
commutative monoid~$\cP$.  Define a functor
$\kk[\rho]\otimes_{\cQ\times\cP}\ $ on modules~$\cN$ over $\cQ \times
\cP$ by
$$%
  \kk[\rho]\otimes_{\cQ\times\cP}\cN =
  \bigoplus_{(\bb,\xi) \in \cQ\times\cP} \kbrx\otimes_{\cQ\times\cP}\cN.
$$
\end{defn}

\begin{defn}\label{d:topr}
Fix a real polyhedral group~$\cQ$, a face~$\rho$, and a
$\cQ$-module~$\cM$.  The \emph{generator functor along~$\rho$} takes
$\cM$ to its \emph{top along~$\rho$}: the $(\qrt\times\nabro)$-module
$$%
  \topr\cM
  =
  \bigl(\kk[\rho]\otimes_{\cQ\times\nabro}\lr\cM\bigr){}^{\rho\!}\big/\rho.
$$
The $\nabro$-graded components of $\topr\cM$ are denoted by
$\topr[\xi]\cM$ for $\xi \in \nabro$.
\end{defn}

\begin{prop}\label{p:order-either}
The functors $\topcr\!$ and $\nrtop$ commute.  In particular,
$$%displaystyle
  \nrtop(\topcr\lr\cM)
  \cong
  \topr\cM
  \cong
  \topcr(\nrtop\lr\cM).
$$
\end{prop}
\begin{proof}
This is Matlis dual to Proposition~\ref{p:either-order}, but it holds
without any infinitesimal $\cQ$-finiteness restriction by proving it
directly:
% for comparison:
% $\nrtop\cN = \kk\otimes_{\nabro}\cN$
% $\topcr\cN = \bigl(\kk[\rho]\otimes_\cQ\cN\bigr){}^{\rho\!}\big/\rho$
\begin{align*}
\nrtop(\topcr\cN)
  &=
  \kk\otimes_{\nabro}\bigl(\kk[\rho]\otimes_\cQ\cN\bigr){}^{\rho\!}\big/\rho
  \text{ by Definitions~\ref{d:topc} and~\ref{d:topr}}
\\
  &=
  \kk\otimes_{\nabro}\bigl(\hhom\bigl(\kk[\cQ_+]_\rho,\kk[\rho]\otimes_\cQ\cN\bigr)\big/\rho\bigr)
  \text{ by Definition~\ref{d:cmr}}
\\
  &=
  \bigl(\kk\otimes_{\nabro}\hhom\bigl(\kk[\cQ_+]_\rho,\kk[\rho]\otimes_\cQ\cN\bigr)\bigr)\big/\rho
  \text{ by Lemma~\ref{l:exact-qr}}
\\
  &=
  \bigl(\kk[\rho]\otimes_\cQ\cN\otimes_{\nabro}\kk\bigr){}^{\rho\!}\big/\rho
  \text{ by Lemma~\ref{l:matlis-pair}}
\\
  &=
  \topcr(\nrtop\cN),
\end{align*}
whose penultimate line is equal to
$\bigl(\kk[\rho]\otimes_{\cQ\times\nabro}\cN\bigr){}^{\rho\!}\big/\rho$.
Now $\cN\hspace{-.2ex} = \lr\cM$ yields~$\topr\cM$.
\end{proof}

\begin{thm}\label{t:top=socvee}
Over a real polyhedral group, the generator functor along a
face~$\rho$ is Matlis dual to the cogenerator functor along~$\rho$:
\begin{enumerate}
\item%
$\topr(\cM^\vee) = (\socr\cM)^\vee$ for all $\cQ$-modules~$\cM$, and
\item%
$(\topr\cM)^\vee = \socr(\cM^\vee)$ if $\cM$ is infinitesimally
$\cQ$-finite.
\end{enumerate}
\end{thm}
\begin{proof}
The two calculations are similar, as in Theorem~\ref{t:topc=socc^vee},
so only one is included---the one that requires $\cM$ to be
infinitesimally $\cQ$-finite, to see exactly where that~enters:
\begin{align*}
(\topr\cM)^\vee
  &=
  \bigl(\nrtop(\topcr\lr\cM)\bigr){}^\vee
  \text{ by Proposition~\ref{p:order-either}}
\\&=
  \nrsoc\bigl((\topcr\lr\cM)^\vee\bigr)
  \text{ by Proposition~\ref{p:nrtop}}
\\&=
  \nrsoc\bigl(\socct[\rho]\bigl((\lr\cM)^\vee\bigr)\bigr)
  \text{ by Theorem~\ref{t:topc=socc^vee}}
\\&=
  \nrsoc\bigl(\socct[\rho]\dr(\cM^\vee)\bigr)
  \text{ by Corollary~\ref{c:lr-vs-dr}.\ref{i:inf}}
\\&=
  \socr(\cM^\vee)
  \text{ by Proposition~\ref{p:either-order}.}\qedhere
\end{align*}
\end{proof}

\begin{defn}\label{d:mrx}
For any faces $\rho \subseteq \xi$ of a real polyhedral group, set
$$%
  \mrx
  =
  \hhom_\cQ\bigl(\kk[\qnx]_\rho,\cM\bigr),
$$
where $\qnx = \xi^\circ + \cQ_+$ as in Lemma~\ref{l:<<} and
localization along~$\rho$ is as in~Definition~\ref{d:support}.
\end{defn}
\begin{lemma}\label{l:mrx}
Any faces $\rho \subseteq \xi$ of a real polyhedral group yield a
natural isomorphism
$$%
  (\lx\cM)^\rho
  \cong
  \mrx.
$$
\end{lemma}
\begin{proof}
A homomorphism $\kk[\bb_\rho + \cQ_+] \to \lx\cM$ for $\bb_\rho = \bb
+ \RR\rho$ is a family of elements $x_{\bb+\rr} \in
(\lx\cM)_{\bb+\rr}$ for $\rr \in \RR\rho$.  Each of these is an
inverse limit of elements in $\cM_{\bb'}$ for $\bb' \in \bb + \rr +
\xi^\circ$.  The union of these sets $\bb + \rr + \xi^\circ$ is
$\bb_\rho + \qnx$.  The universal property of inverse limits produces
a family of elements $x_{\bb'} \in \cM_{\bb'}$ for $\bb' \in \bb_\rho
+ \qnx$.  These elements are the images of $1 \in \kk_{\bb'} =
\kk[\bb_\rho + \qnx]_{\bb'}$ under a homomorphism~$\kk[\bb_\rho +
\qnx] \to \cM$.  And any such homomorphism yields a coherent family of
elements indexed by $\bb_\rho + \qnx$ whose inverse limits specify a
homomorphism $\kk[\bb_\rho + \cQ_+] \to \lx\cM$.
\end{proof}

\begin{prop}\label{p:exact-mrx}
For any faces $\rho \subseteq \xi$ of a real polyhedral group~$\cQ$, the
functor $\cM \mapsto \mrx$ is exact on the category of infinitesimally
$\cQ$-finite modules.
\end{prop}
\begin{proof}
Lemma~\ref{l:mrx}, Proposition~\ref{p:exact-lower-boundary},
Definition~\ref{d:cmr}, and Lemma~\ref{l:matlis-pair}.
\end{proof}

Finally, here is real analogue of the surjection in
Proposition~\ref{p:surjection-ZZ}.

\begin{prop}\label{p:surjection-RR}
Over any real polyhedral group~$\cQ$ there is a natural surjection
$$%
  \mrx/\rho
  \onto
  \topr[\xi]\cM
$$
of modules over $\qrr$ for any faces~$\rho$ and~$\xi$ with $\xi
\supseteq \rho$.
\end{prop}
\begin{proof}
Regard $\kk[\rho]\otimes_{\cQ\times\nabro}\lr\cM$ as a quotient
$(\cQ\times\nabro)$-module of~$\lr\cM$.  Then
$$%
  (\lr\cM)^\rho
  \onto
  \bigl(\kk[\rho]\otimes_{\cQ\times\nabro}\lr\cM\bigr){}^\rho
$$
by Lemma~\ref{l:matlis-pair}.  The result follows by applying
Lemma~\ref{l:mrx} to the component of the surjection in
$\nabro$-degree~$\xi$ and then taking the quotient-restriction
along~$\rho$.
\end{proof}

\begin{example}\label{e:beta}
Fix a face~$\rho$ and suppose faces $\eta$ and~$\xi$ both
contain~$\rho$.  Then the top of $\kk[\bb_\rho+\qnx]$ along any face
other than~$\rho$ vanishes, and $\topr[\eta]\kk[\bb_\rho+\qnx] =
\kk_{\bb_\rho}$ if $\eta = \xi$ and vanishes otherwise.  The vanishing
along faces other than~$\rho$ is because $\kk[\bb_\rho+\qnx]$ is a
secondary module by the Matlis dual to
Corollary~\ref{c:soc(coprimary)}.  The rest of the vanishing is
because the tensor product
$$%
  \kk[\rho]\otimes_{\cQ\times\nabro}\lr\cM =
  \kk[\rho]\otimes_\cQ\lr\cM \otimes_{\nabro}\kk
$$
is almost always~$0$: the $\kk[\rho]\otimes_\cQ$ takes care of
$\qrr$-degrees other than~$\bb_\rho$, while the $\otimes_{\nabro}\kk$
takes care $\nabto$-degrees other than~$\xi$.  Finally, the
calculation in degree~$(\bb_\rho,\xi)$ reduces by
Proposition~\ref{p:surjection-RR} and
Corollary~\ref{c:at-most-one}$^\vee$ to
$\Hom_\cQ\bigl(\kk[\bb_\rho+\qnx],\kk[\bb_\rho+\qnx]\bigr)
=\nolinebreak \kk$.
\end{example}

%%%%%%%%%%%%%%%%%%%%%%%%%%%%%%%%%%%%%%%%%%%%%%%%%%%%%%%%%%%%%%%%%%%%%%%%%
\section{Essential properties of tops}\label{s:tops}%%%%%%%%%%%%%%%%%%%%%
%%%%%%%%%%%%%%%%%%%%%%%%%%%%%%%%%%%%%%%%%%%%%%%%%%%%%%%%%%%%%%%%%%%%%%%%%

Begin with the dual to Definition~\ref{d:associated} and
Theorem~\ref{t:coprimary}.

\begin{defn}\label{d:attached}
Fix a face $\rho$ and a module~$\cM$ over a real or discrete
polyhedral~group.
\begin{enumerate}
\item%
The face~$\rho$ is \emph{attached} to~$\cM$ if
$\topr\cM \neq 0$.
\item%
If $\cM = \kk[U]$ for an upset~$U$ then $\rho$ is
\emph{attached}~to~$U$.
\item%
The set of attached faces of~$\cM$ or~$U$ is denoted by $\att\cM$
or~$\att D$.
\item%
The module~$\cM$ is \emph{$\rho$-secondary} if $\att(\cM) = \{\rho\}$.
\end{enumerate}
\end{defn}

\begin{defn}\label{d:minimal-cover}
An \emph{upset cover} of a module~$\cM$ over an arbitrary poset is a
surjection $\bigoplus_{j \in J} F^j \onto \cM$ with each $F^j$ an
upset module.  The cover is
\begin{enumerate}
\item%
\emph{finite} if $J$ is~finite.
\end{enumerate}\setcounter{separated}{\value{enumi}}%saves \enumi
The module~$\cM$ is \emph{upset-finite} if it admits a finite upset
cover.  If the poset is a real or discrete polyhedral group, then the
cover is
\begin{enumerate}\setcounter{enumi}{\value{separated}}%restores \enumi
\item\label{i:secondary}%
\emph{secondary} if $F^j = \kk[U_j]$ is secondary for all~$j$, so
$U_j$ is a secondary upset,~and
\item%
\emph{minimal} if the induced map $\topr E \to \topr\cM$ is an
isomorphism for all faces~$\rho$.
\end{enumerate}
\end{defn}

That was Matlis dual to Definition~\ref{d:minimal-hull}.  Next are the
duals to Theorems~\ref{t:injection} and~\ref{t:discrete-injection}.

\begin{thm}[Essentiality of real tops]\label{t:surjection-RR}
Fix a homomorphism $\phi: \cN \to \cM$ of~modules over a real
polyhedral group~$\cQ$.
\begin{enumerate}
\item\label{i:phi=>topr}%
If $\phi$ is surjective with $\cM$ and~$\cN$ both being
infinitesimally $\cQ$-finite modules, then $\topr\phi: \topr\cN \to
\topr\cM$ is surjective for all faces~$\rho$~of~$\cQ_+$.
\item\label{i:topr=>phi}%
If $\topr\phi: \topr\cN \to \topr\cM$ is surjective for all
faces~$\rho$ of~$\cQ_+\!$ and $\cM$ is upset-finite, then $\phi$
is~surjective.\hfill$\square$
\end{enumerate}
\end{thm}

\begin{thm}[Essentiality of discrete tops]\label{t:surjection-ZZ}
Fix a homomorphism $\phi: \cN \to \cM$ of modules over a real or
discrete polyhedral group~$\cQ$.
\begin{enumerate}
\item
If $\phi$ is surjective then $\topr\phi: \topr\cN \to \topr\cM$ is
surjective for all faces~$\rho$~of~\hspace{.3ex}$\cQ_+$.
\item
If $\topr\phi: \topr\cN \to \topr\cM$ is surjective for all
faces~$\rho$ of~$\cQ_+\!$ and $\cM$ is upset-finite, then $\phi$
is~surjective.\hfill$\square$
\end{enumerate}
\end{thm}

\begin{example}\label{e:surjection}
Some hypothesis is needed in
Theorem~\ref{t:surjection-RR}.\ref{i:phi=>topr}, in contrast to
Theorem~\ref{t:injection}.\ref{i:phi=>soct} or indeed
Theorem~\ref{t:surjection-ZZ}.\ref{i:phi=>topr}.  Let $\cM = \kk[U]$
for the open half-plane $U \subset \RR^2$ above the antidiagonal line
$y = -x$.  Then $\cM$ is $\{\0\}$-secondary, with
$(\top_{\{\0\}}^\xi\!\cM)_\bb \neq 0$ precisely when $\bb$ lies on the
antidiagonal and $\xi$ is the $x$-axis or $y$-axis.  The direct sum
$\bigoplus_{\bb \neq \0} \bigl(\kk[\bb + \qny[x]] \oplus \kk[\bb +
\qny[y]]\bigr)$ surjects onto~$\cM$, but the map on tops fails to hit
any element in $\RR^2$-degree~$\0$.  This kind of behavior might lead
one to wonder: why is the Matlis dual not a counterexample to
Theorem~\ref{t:injection}.\ref{i:phi=>soct}?  Because $\cM^\vee$ does
not possess a well defined map to a direct sum indexed by $\aa \neq
\0$ along the antidiagonal line, only to a direct product.  Any
sequence of points $\vv_k \in -U$ converging to~$\0$ yields a sequence
of elements $z_k \in \cM^\vee$.  The image of the sequence
$\{z_k\}_{k=1}^\infty$ in any particular one (or finite direct sum) of
the downset modules of the form $\kk[\aa - \qny[x]]$ with $\aa \neq
\0$ is eventually~$0$, but in the direct product the sequence
$\{z_k\}_{k=1}^\infty$ survives forever.  The direct limit of the
image sequence witnesses the nonzero socle of the direct product at
the missing~point~$\0$.
\end{example}

\begin{thm}\label{t:cover-M}
Every upset-finite module~$\cM$ over a real or discrete polyhedral
group admits a minimal secondary upset hull.
\end{thm}
\begin{proof}
This is the Matlis dual of Theorems~\ref{t:hull-M}
and~\ref{t:discrete-hull-M}, using
Example~\ref{e:infinitesimally-Q-finite} to allow the results of
Section~\ref{s:gen-functors} to be applied at will.
\end{proof}

The straightforward dualization of primary decomposition in
Sections~\ref{sub:minimal-primary} and~\ref{s:discrete} to secondary
decomposition is omitted.

%%%%%%%%%%%%%%%%%%%%%%%%%%%%%%%%%%%%%%%%%%%%%%%%%%%%%%%%%%%%%%%%%%%%%%%%%
\section{Minimal presentations over discrete or real polyhedral groups}\label{s:min}
%%%%%%%%%%%%%%%%%%%%%%%%%%%%%%%%%%%%%%%%%%%%%%%%%%%%%%%%%%%%%%%%%%%%%%%%%

\begin{defn}\label{d:minimal-presentations}
Fix a module~$\cM$ over a real or discrete polyhedral group.
\begin{enumerate}
\item%
An upset presentation $F_1 \to F_0$ of~$\cM$ is \emph{minimal} if $F_1
\onto \ker(F_0 \to \cM)$ and $F_0 \onto \cM$ are minimal upset covers
(Definition~\ref{d:minimal-cover}).
\item%
An upset resolution $F_\spot$ of~$\cM$ is \emph{minimal} if the upset
presentation $F_1 \to F_0$ and upset covers $F_{i+1} \onto \ker(F_i
\to F_{i-1})$ for all $i \geq 1$ are minimal.
\item%
A downset copresentation $E^0 \to E^1$ of~$\cM$ is \emph{minimal} if
$\coker(\cM \to\nolinebreak E^0) \into E^1$ and $\cM \into E^0$ are
minimal downset hulls (Definition~\ref{d:minimal-hull}).
\item%
A downset resolution $E^\spot$ of~$\cM$ is \emph{minimal} if the
downset copresentation $E^0 \to E^1$ and downset hulls $\coker(E^{i-1}
\to E^i) \into E^{i+1}$ for all $i \geq 1$ are minimal.
\item%
A fringe presentation $F \to E$ of~$\cM$ is \emph{minimal} if it is
the composite of a minimal upset cover of~$\cM$ and a minimal downset
hull of~$\cM$.
\end{enumerate}
\end{defn}

\begin{thm}\label{t:presentations-minimal}
A module~$\cM$ over a real or discrete polyhedral group~$\cQ$ is
finitely encoded if and only if $\cM$ admits
\begin{enumerate}
\item\label{i:fringe'}%
a minimal finite fringe presentation; or
\item\label{i:upset-presentation'}%
a minimal finite upset presentation; or
\item\label{i:downset-copresentation'}%
a minimal finite downset copresentation.
\end{enumerate}
When $\cM$ is semialgebraic over a real polyhedral group~$\cQ$, these
minimal presentations are all semialgebraic.
\end{thm}
\begin{proof}
In both the real and discrete cases, any one of these minimal
presentations is, in particular, finite, so the existence of any of
them implies that $\cM$ is finitely encoded by Theorem~\ref{t:syzygy}.
It is the other direction that requires the theory in
Sections~\ref{s:socle}--\ref{s:gen-functors}.

In the real polyhedral case, any finite downset hull can be
minimalized by Theorem~\ref{t:hull-M} and Remark~\ref{r:filtration}.
The Matlis dual of this statement says that any finite upset cover can
be minimalized, as well.  Composing these from a given finite fringe
presentations yields a minimal finite fringe presentation.  In
addition, the cokernel of any downset hull (minimal or otherwise) of a
finitely encoded module is finitely encoded by
Lemma~\ref{l:abelian-category}, so the cokernel has a minimal finite
downset hull by Theorem~\ref{t:hull-M} again.  That yields a minimal
finite downset copresentation.  The Matlis dual of a minimal finite
downset copresentation of the Matlis dual~$\cM^\vee$ is a minimal
upset presentation of~$\cM$ by Theorem~\ref{t:top=socvee} (which
applies unfettered to finitely encoded modules by
Example~\ref{e:infinitesimally-Q-finite}).

The discrete polyhedral case follows the parallel proof, using
Theorem~\ref{t:discrete-hull-M} and Remark~\ref{r:discrete-filtration}
instead of Theorem~\ref{t:hull-M} and Remark~\ref{r:filtration}.

If $\cM$ is semialgebraic, then the minimalization procedure in
Theorem~\ref{t:hull-M} and Remark~\ref{r:filtration} is semialgebraic
by induction on the number~$k$ of summands there, the base case being
the canonical primary decomposition of a semialgebraic downset in
Theorem~\ref{t:hull-D}, which is semialgebraic by
Theorem~\ref{t:soct-encoding}.
\end{proof}

\begin{remark}\label{r:missing-items}
Comparing Theorem~\ref{t:presentations-minimal} to
Theorem~\ref{t:syzygy}, various items are missing.
\begin{enumerate}
\item%
Theorem~\ref{t:presentations-minimal} makes no claim concerning
whether the presentations can minimalized if an encoding $\pi: \cQ \to
\cP$ has been specified beforehand.  It is a~priori possible that
deleting nonminimal generators of upsets and cogenerators of downsets
could prevent an indicator summand from being constant
on~fibers~of~$\pi$.

\item%
Theorem~\ref{t:presentations-minimal} makes no claim concerning finite
encodings dominating any one of the three presentations there, but as
each of these presentations is finite, existence is already implied by
Theorem~\ref{t:syzygy}, including semialgebraic~\mbox{considerations}.

\item\label{i:res}%
Theorem~\ref{t:presentations-minimal} makes no claim concerning
minimal finite resolutions.  Minimal resolutions of finitely encoded
modules over real or discrete polyhedral groups can be constructed
from scratch by Theorem~\ref{t:hull-M},
Theorem~\ref{t:discrete-hull-M}, and their Matlis duals (in the real
case that is Theorem~\ref{t:cover-M}), but there is no guarantee that
a minimal indicator resolution must terminate after finitely many
steps.
\end{enumerate}
\end{remark}

\begin{remark}\label{r:syzygy}
Remark~\ref{r:missing-items}.\ref{i:res} raises an intriguing point
about indicator resolutions: the bound on the length in
Theorem~\ref{t:syzygy} comes from the order dimension of the encoding
poset, which is more or less unrelated to the dimension of the real or
discrete polyhedral group.  It seems plausible that the geometry of
the polyhedral group asserts control to prevent the lengths from going
too high, just as it does to prevent the cohomological dimension of an
affine semigroup ring from going too high via Ishida complexes to
compute local cohomology \cite[Section~13.3.1]{cca}.  This points to
potential value of developing a derived functor side of the top-socle
/ birth-death / generator-cogenerator story for indicator resolutions
to solve Conjecture~\ref{conj:indicator-dimension-n}, which would be a
true indicator analogue of the Hilbert Syzygy Theorem.
\end{remark}

\begin{defn}\label{d:indicator-dimension}
Fix a poset~$\cQ$ and a $\cQ$-module~$\cM$.
\begin{enumerate}
\item%
The \emph{downset-dimension} of~$\cM$ is the smallest length of a
downset resolution of~$\cM$.

\item%
The \emph{upset-dimension} of~$\cM$ is the smallest length of an upset
resolution of~$\cM$.

\item%
The \emph{indicator-dimension} of~$\cM$ is maximum of its downset- and
upset-dimensions.

\item%
The \emph{indicator-dimension} of~$\cQ$ is the maximum of the
indicator-dimensions of its finitely encoded modules.
\end{enumerate}
\end{defn}

\begin{conj}\label{conj:indicator-dimension-n}
The indicator-dimension of any real or discrete polyhedral group~$\cQ$
equals the rank of~$\cQ$ as a free module (over the field~$\RR$ or
group~$\ZZ$, respectively).
\end{conj}

\begin{remark}\label{r:conj}
It is already open to find a module of indicator-dimension as high
as~$2$ over~$\RR^2$.  It would not be shocking if the rank
of~$\cQ$ were an upper bound instead of an equality for the
indicator-dimension in the conjecture: the use of upset modules
instead of free modules could prevent the final syzygies that, in
finitely generated situations, come from elements supported at the
origin by local duality.
\end{remark}

%%%%%%%%%%%%%%%%%%%%%%%%%%%%%%%%%%%%%%%%%%%%%%%%%%%%%%%%%%%%%%%%%%%%%%%%%
\section{Birth and death posets and modules}\label{s:birth-death}%%%%%%%%
%%%%%%%%%%%%%%%%%%%%%%%%%%%%%%%%%%%%%%%%%%%%%%%%%%%%%%%%%%%%%%%%%%%%%%%%%

%%%%%%%%%%%%%%%%%%%%%%%%%%%%%%%%%%%%%%%%%%%%%%%%%%%%%%%%%%%%%%%%%%%%%%%%%
\subsection{Discrete polyhedral case}\label{sub:birth-death-ZZ}\mbox{}%%%

\noindent
This subsection works in the generality of discrete polyhedral groups
(Definition~\ref{d:discrete-polyhedral}).  Notation and concepts from
Sections~\ref{s:socle} and~\ref{s:gen-functors}, particularly those
surrounding quotient-restriction and closed cogenerators along faces
of polyhedral groups (Section~\ref{sub:socc-along}) as well as closed
generator functors along faces (Section~\ref{sub:closed-gen-along})
are used freely here, without further cross-reference.  For example,
see Definition~\ref{d:cmr} for the meaning of~$\cmr$.  Basic poset
concepts from Section~\ref{s:encoding} are also required, such as
pullbacks in Definition~\ref{d:encoding}.

\begin{defn}\label{d:birth-ZZ}
Fix a module~$\cM$ over a discrete polyhedral group~$\cQ$.
\begin{enumerate}
\item\label{i:Q-poset-ZZ}%
The \emph{birth poset of~$\cQ$} is the disjoint union $\cbq =
\bigcupdot_\rho \qzr = \bigcupdot_\rho \cbq[\rho]$ with
$$%
  (\bb_\rho \in \qzr) \preceq (\brp \in \qzr')
  \iff
  \bb_\rho + \cQ_+ \supseteq \brp + \cQ_+,
$$
where $\bb_\rho \subseteq \cQ$ is viewed as a coset of~$\ZZ\rho$
in~$\cQ$ for the purpose of writing $\bb_\rho + \cQ_+$.
\item\label{i:M-poset-ZZ}%
The \emph{birth poset of~$\cM$} is the subposet $\cbm =
\bigcupdot_\rho \deg\topcr\cM \subseteq \cbq$.  Write
$$%
  \iota_\rho: \deg\topcr\cM \to \qzr
$$
for the $\rho$-component $\cbm[\rho] \into \cbq[\rho]$ of the
inclusion $\cbm \into \cbq$.

\item\label{i:module-ZZ}%
The \emph{birth module} of~$\cM$ is the $\cbm$-graded vector space
$$%
  \Gen\cM
  =
  \bigoplus_\rho \iota_\rho^*
  \bigl(\cmr/\rho\bigr),
$$
so the component in degree $\bb_\rho \!\in\! \cbm[\rho]$ is
$\Gen_{\bb_\rho}\!\cM = (\cmr/\rho)_{\bb_\rho}$.
\end{enumerate}
\end{defn}

\begin{lemma}\label{l:gen-ZZ}
The birth module $\Gen\cM$ is naturally a $\cbm$-module.
\end{lemma}
\begin{proof}
If $\bb_\rho \in \qzr$ and $\brp \in \qzr'$, then $\bb_\rho + \cQ_+
\supseteq \brp + \cQ_+$ implies that any homomorphism $\kk[\bb_\rho +
\cQ_+] \to \cM$ restricts to a homomorphism $\kk[\brp + \cQ_+] \to
\cM$.  Note that $\bb_\rho + \cQ_+$ equals its translate by an element
of~$\ZZ\rho$, so these all yield the same~restriction to $\brp +
\cQ_+$; hence the restriction descends to $\cmr/\rho =
\hhom_\cQ\bigl(\kk[\cQ_+]_\rho,\cM\bigr)\big/\rho$.
\end{proof}

\begin{defn}\label{d:death-ZZ}
Fix a module~$\cM$ over a discrete polyhedral group~$\cQ$.
\begin{enumerate}
\item%
The \emph{death poset of~$\cQ$} is the disjoint union $\cdq =
\bigcupdot_\tau \qzt = \bigcupdot_\tau \cdq[\tau]$ with
$$%
  (\aa_\tau \in \qzt) \preceq (\aa'_{\tau'} \in \qzt')
  \iff
  \aa_\tau - \cQ_+ \subseteq \aa'_{\tau'} - \cQ_+,
$$
where $\aa_\tau \subseteq \cQ$ is viewed as a coset of~$\ZZ\tau$
in~$\cQ$ for the purpose of writing $\aa_\tau + \cQ_+$.
\item%
The \emph{death poset of~$\cM$} is the subposet $\cdm =
\bigcupdot_\tau \cdm[\tau] = \bigcupdot_\tau \deg\socct\cM \subseteq
\cdq$.

\item%
The \emph{death module} of~$\cM$ is
$$%
  \Soc\cM = \prod_{\aa_\tau \in \cdm} (\socct\cM)_{\aa_\tau}
$$
whose factor in degree $\aa_\tau \in \cdq[\tau]$ is also denoted
$\Soc_{\aa_\tau}\!\cM = (\socct\cM)_{\aa_\tau}$.
\end{enumerate}
\end{defn}

\begin{remark}\label{r:poset-of-cogenerators}
The socles and tops are closed because the polyhedral group is
discrete.
\end{remark}

\begin{lemma}\label{l:death-ZZ}
The death module $\Soc\cM$ is naturally a $\cdm$-module and a
$\cdq$-module.
\end{lemma}
\begin{proof}
This follows from the (simpler) discrete polyhedral analogue of
Remark~\ref{r:soc-as-k-vect}.
\end{proof}

\begin{remark}\label{r:cq-vs-gq-ZZ}
$\cdq = \bnq^{\,\op}$, in the sense that $\aa_\tau \succeq
\aa'_{\tau'}$ in $\cdq \iff -\aa_\tau \preceq -\aa'_{\tau'}$ in
$\cbq$.
\end{remark}

\begin{remark}\label{r:why-cm}
Given that the poset-module structure on $\Soc\cM$ is trivial, why
bother with it?  The reason is to be able to compare degrees of
elements in~$\Soc\cM$ with degrees of elements in~$\Gen\cM$.  The
partial orders on~$\cbq$ and~$\cdq$ are defined so that if a
localization $F$ of~$\kk[\cQ_+]$ has a nonzero map to an
indecomposable injective $\kk[\aa_\tau - \cQ_+]$, then $F$ has a
nonzero map to all of the indecomposable injectives indexed by
elements above~$\aa_\tau$ in~$\cdq$.  This relation is formalized in
Proposition~\ref{p:birth-union-death-ZZ}.
\end{remark}

\begin{prop}\label{p:birth-union-death-ZZ}
The disjoint union $\cbq \cupdot \cdq$ is partially ordered by
$$%
  \bb_\rho \preceq \aa_\tau
  \iff
  (\bb_\rho + \cQ_+) \cap (\aa_\tau - \cQ_+) \neq \nothing
$$
along with the given partial orders on~$\cbq\!$ and~$\cdq$, if~$\cQ$ is
a discrete polyhedral group.~~\hfill$\square$
\end{prop}

\begin{remark}\label{r:partial-compactification-ZZ}
The birth and death posets of~$\cQ$ can be thought of as partial
compactifications of~$\cQ$, in essence the integer points on a toric
variety.  In the case $\cQ = \ZZ^n$, for example, with its standard
positive cone~$\NN^n$, the birth poset can be thought of as the set of
``lattice points'' in a space homeomorphic to the union of the faces
of the unit hypercube $[0,1]^n$ containing the origin in~$\RR^n$.  In
any neighborhood of the origin in~$\RR^n$ this conception is realized
by exponentiating elements of~$\ZZ^n$ and its quotients modulo
coordinate subspaces, setting $e^{-\infty} = 0$.  The death poset
of~$\ZZ^n$ is similarly described, but the relevant set of faces
consists of those containing the point $(1,\dots,1)$ instead of the
origin.  This geometric view of birth and death posets makes the
disjoint union in Proposition~\ref{p:birth-union-death-ZZ}
particularly clear, although it points out that really the union
should not be disjoint but rather should identify the copies of~$\cQ$
sitting inside of~$\cbq$ and~$\cdq$.
\end{remark}

%%%%%%%%%%%%%%%%%%%%%%%%%%%%%%%%%%%%%%%%%%%%%%%%%%%%%%%%%%%%%%%%%%%%%%%%%
\subsection{Real polyhedral case}\label{sub:birth-death}\mbox{}%%%%%%%%%%

\noindent
This subsection works in the generality of real polyhedral groups
(Definition~\ref{d:polyhedral}).  Notation and concepts from
Sections~\ref{s:socle} and~\ref{s:gen-functors}, particularly those
surrounding face posets and cogenerators along faces of real
polyhedral groups (Sections~\ref{sub:tangent} and~\ref{sub:soc-along})
as well as generator functors along faces (Section~\ref{sub:gen}) and
quotient-restriction (Section~\ref{sub:socc-along}) are used freely
here, without further cross-reference.  For example, see
Definition~\ref{d:mrx}, Lemma~\ref{l:mrx}, and
Proposition~\ref{p:surjection-RR} for the meaning of~$\mrx$.  Basic
poset concepts from Section~\ref{s:encoding} are also required, such
as pullbacks in Definition~\ref{d:encoding}.

\begin{defn}\label{d:birth-RR}
Fix a module~$\cM$ over a real polyhedral group~$\cQ$.
\begin{enumerate}
\item\label{i:Q-poset-RR}%
The \emph{birth poset of~$\cQ$} is the disjoint union $\cbq =
\bigcupdot_\rho \,\qrr \times \nabro$ with
$$%
  \brx \preceq \brxp
  \iff
  \bb_\rho + \qnx \supseteq \brp + \qnx['],
$$
where $\bb_\rho \subseteq \cQ$ is viewed as a coset of~$\RR\rho$
in~$\cQ$ for the purpose of writing $\bb_\rho + \qnx$.  Set
$\cbq[\rho] = \qrr \times \nabro$ and $\cbq[\rho,\xi] = \qrr$.
\item\label{i:M-poset-RR}%
The \emph{birth poset of~$\cM$} is the subposet $\cbm =
\bigcupdot_\rho \deg\topr\cM \subseteq \cbq$.  Write
$$%
  \iota_\rho: \deg\topr\cM \to \qrr \times \nabro
  \text{\quad and \quad}
  \iota_\rho^\xi: \deg\topr[\xi]\cM \to \qrr
$$
for the components $\cbm[\rho] \into \cbq[\rho]$ and $\cbm[\rho,\xi]
\into \cbq[\rho,\xi]$ of the inclusion $\cbm \into \cbq$.

\item\label{i:module-RR}%
The \emph{birth module} of~$\cM$ is the $\cbm$-graded vector space
$$%
  \Gen\cM
  =
  \bigoplus_\rho \bigoplus_{\xi\in\nabro} \,(\iota_\rho^\xi)^*
  \bigl(\mrx/\rho\bigr),
$$
so the component in degree $\brx \in \cbm[\rho]$ is
$\Gen_{\bb_\rho}^{\xi}\!\cM = (\mrx/\rho)_{\bb_\rho}$.
\end{enumerate}
\end{defn}

\begin{example}\label{e:birth-RR}
The birth poset~$\cbr$ of the real polyhedral group~$\RR$ is totally
ordered.  Heuristically, it feels like $\RR$ but with each point
doubled and a negative infinity thrown in.  More precisely, let
$-\infty = (\RR/\RR,\RR_+)$ be the unique coset of $\RR = \RR\RR_+$
in~$\RR$.  For each real number~$b$, let $b_\bullet$ be the
pair~$(b,0)$ and $b_\circ$ the pair~$(b,\RR_+)$.  Then $\cbr$ consists
of the point $-\infty$ along with $b_\bullet$ and $b_\circ$ for all $b
\in \RR$, with the total order that has $-\infty < b_\bullet < b_\circ
< b'_\bullet$ if $b < b'$.  More geometrically, $\cbr$ is the set of
positive-pointing rays in~$\RR$ totally ordered by inclusion.  The
rays come in three flavors, namely $(-\infty,\infty) < [b,\infty) <
(b,\infty)$, and the latter is $< [b',\infty)$ if $b < b'$.
\end{example}

\begin{lemma}\label{l:gen}
The birth module $\Gen\cM$ is naturally a $\cbm$-module.
\end{lemma}
\begin{proof}
If $\brx \preceq \brxp$ then any homomorphism $\kk[\bb_\rho + \qnx]
\to \cM$ restricts to a homomorphism $\kk[\brp + \qnx[']] \to \cM$.
Note that $\bb_\rho + \qnx$ equals its translate by an element
of~$\RR\rho$, so any of these translates yield the same restriction to
$\brp + \qnx[']$; hence the restriction descends to $\mrx/\rho =
\hhom_\cQ\bigl(\kk[\qnx]_\rho,\cM\bigr)\big/\rho$.
\end{proof}

\begin{defn}\label{d:death-RR}
Fix a module~$\cM$ over a real polyhedral group~$\cQ$.
\begin{enumerate}
\item%
The \emph{death poset of~$\cQ$} is the disjoint union $\cdq =
\bigcupdot_\tau \qrt \times \nabt$ with
$$%
  (\aa_\tau,\sigma) \preceq (\aa'_{\tau'},\sigma')
  \iff
  \aa_\tau - \qns \subseteq \aa'_{\tau'} - \qns['],
$$
where $\aa_\tau \subseteq \cQ$ is viewed as a coset of~$\RR\tau$
in~$\cQ$ for the purpose of writing $\aa_\tau + \qns$.
Set
$\cdq[\tau] = \qrt \times \nabt$ and $\cdq[\tau,\sigma] = \qrt$.
\item%
The \emph{death poset of~$\cM$} is the subposet $\cdm =
\bigcupdot_\tau \cdm[\tau] = \bigcupdot_\tau \deg\soct\cM \subseteq
\cdq$.

\item%
The \emph{death module} of~$\cM$ is
$$%
  \Soc\cM
  =
  \prod_{\aa_\tau \in \cdm[\tau]} (\soct\cM)_{\aa_\tau}
  =
  \prod_{(\aa_\tau,\sigma) \in \cdm[\tau,\sigma]} (\soct[\sigma]\cM)_{\aa_\tau}
$$
whose factor in degree $\aa_\tau \in \cdq[\tau]$ is also denoted
$\Soc_{\aa_\tau}\!\cM = (\soct\cM)_{\aa_\tau}$ and whose factor in
degree $(\aa_\tau,\sigma) \in \cdq[\tau,\sigma]$ is also denoted
$\Soc_{\aa_\tau}^\sigma\!\cM = (\soct[\sigma]\cM)_{\aa_\tau}$.
\end{enumerate}
\end{defn}

\begin{example}\label{e:death-RR}
The death poset of the real polyhedral group~$\RR$ is the negative
(or, if you like, Matlis dual) of the birth poset in
Example~\ref{e:birth-RR}, so $\cdr$ has $a'_\bullet <\nolinebreak
a_\circ <\nolinebreak a_\bullet < \infty$ for all $a' < a \in \RR$.
Geometrically, $\cdr$ is the set of negative-pointing rays in~$\RR$
totally ordered by inclusion, meaning $(-\infty,a'] < (-\infty,a) <
(-\infty,a] < (-\infty,\infty)$ when $a' < a$.
\end{example}

\begin{lemma}\label{l:Soc}
The death module $\Soc\cM$ is naturally a module over\/~$\cdm$
and~$\cdq$.
\end{lemma}
\begin{proof}
This follows from Remark~\ref{r:soc-as-k-vect}.
\end{proof}

\begin{prop}\label{p:birth-union-death}
The disjoint union $\cbq \cupdot \cdq$ is partially ordered by
$$%
  (\bb_\rho,\xi) \preceq (\aa_\tau,\sigma)
  \iff
  (\bb_\rho + \qnx) \cap (\aa_\tau - \qns) \neq \nothing
$$
along with the given partial orders on~$\cbq\!$ and~$\cdq$, if $\cQ$ is
a real polyhedral group.~~\hfill$\square$
\end{prop}

\begin{remark}\label{r:partial-compactification-RR}
The geometric picture in Remark~\ref{r:partial-compactification-ZZ}
remains roughly valid, but now the face-poset factors $\nabro$
and~$\nabt$ come into consideration.  Whereas passing from
$(\bb_\rho,\xi)$ to~$(\brp,\xi)$ for a fixed face~$\xi$ should be
thought of as a macroscopic motion, passing from $(\bb_\rho,\xi)$
to~$(\bb_\rho,\xi')$ should be thought of as an infinitesimal nudge.
Thus every point in the partial compactification becomes
``arithmetically thickened'' in a manner reminiscent of the way an
inertia group on an orbifold keeps track of automorphisms along a
stratum, noting that indeed the nature of the thickening is precisely
dependent on the stratum because along~$\rho$ only infinitesimal
motions to faces $\xi \supseteq \rho$ are allowed.
\end{remark}

%%%%%%%%%%%%%%%%%%%%%%%%%%%%%%%%%%%%%%%%%%%%%%%%%%%%%%%%%%%%%%%%%%%%%%%%%
\section{Death functors and QR codes}\label{s:qr}%%%%%%%%%%%%%%%%%%%%%%%%
%%%%%%%%%%%%%%%%%%%%%%%%%%%%%%%%%%%%%%%%%%%%%%%%%%%%%%%%%%%%%%%%%%%%%%%%%

%%%%%%%%%%%%%%%%%%%%%%%%%%%%%%%%%%%%%%%%%%%%%%%%%%%%%%%%%%%%%%%%%%%%%%%%%
\subsection{Death functors}\label{sub:death}\mbox{}%%%%%%%%%%%%%%%%%%%%%%

\begin{thm}\label{t:death-functor}
Fix a module~$\cM$ over a polyhedral partially ordered group~$\cQ$.
\begin{enumerate}
\item\label{i:death-functor-ZZ}%
If $\cQ$ is discrete and $\aa_\tau \in \cdm$ is a death degree, then
there is a \emph{death functor}
$$%
  \QR_{\aa_\tau}:
  \cmr_{\bb_\rho}
  \to
  \Soc_{\aa_\tau}\!\cM,
$$
natural in $\bb_\rho \in \cbq$, where $\cmr_{\bb_\rho} =
\Hom_\cQ\bigl(\kk[\bb_\rho + \cQ_+],\cM\bigr)$ as in
Definition~\ref{d:cmr}.

\item\label{i:death-functor-RR}%
If $\cQ$ is real and $(\aa_\tau,\sigma) \in \cdm$ is a death degree,
then there is a \emph{death functor}
$$%
  \QR_{\aa_\tau}^\sigma:
  \mrx_{\bb_\rho}
  \to
  \Soc_{\aa_\tau}^\sigma\!\cM,
$$
natural in $(\bb_\rho,\xi) \!\in\! \cbq$,
\hspace{-1pt}where\hspace{-1pt} $\mrx_{\bb_\rho}
\hspace{-1pt}=\hspace{-1pt} \Hom_\cQ\hspace{-.5pt}\bigl(\kk[\bb_\rho +
\qnx],\!\cM\bigr)\!$ as in \mbox{Definition}\,\ref{d:mrx}.
\end{enumerate}
\end{thm}
\begin{proof}
The proof for the discrete case is strictly easier, since it is little
else than the real case with $\xi = \rho$ and $\sigma = \tau$.  The
following real proof is written to work verbatim after changing every
$\qnx$ to~$\cQ_+$, placing a bar over every~$\soct$, and erasing every
$\xi$~and~$\sigma$.

The image of any given homogeneous homomorphism $\phi: \kk[\bb_\rho +
\qnx] \to \cM$ is a quotient~$\kk[S]$ of~$\kk[\bb_\rho + \qnx]$, where
$S \subseteq \bb_\rho + \qnx$ is a downset~of~$\bb_\rho + \qnx$.
Therefore $\phi$ canonically induces an inclusion $\phi_*: \kk[S]
\into\nolinebreak \cM$.  If $\aa_\tau \in \cdm[\tau,\sigma]$ then
either $\aa_\tau \not\in \cds[\tau,\sigma]$ or $(\bb_\rho,\xi) \preceq
(\aa_\tau,\sigma)$ and, for some $\bb \in \bb_\rho + \qnx$, the
element $1 \in \kk[S]_\bb$ divides a cogenerator~$s_\aa$ of~$\kk[S]$
along~$\tau$ with nadir~$\sigma$ in some degree~$\aa$ in the
coset~$\aa_\tau$.  All choices of~$\aa$, regardless of which $\bb \in
\bb_\rho + \qnx$ is used, result in cogenerators~$s_\aa^\sigma$ with
the same image $s_{\aa_\tau}^\sigma \in \soct[\sigma]\kk[S]$ because
of the quotient-restriction modulo~$\tau$ in
Definition~\ref{d:soct}.\ref{i:global-soc-tau}.  The socle
element~$s_{\aa_\tau}^\sigma$ is taken functorially
to~$\soct[\sigma]\cM$ by the map $\soct[\sigma]\phi_*:
\soct[\sigma]\kk[S] \into \soct[\sigma]\cM$ in
Theorem~\ref{t:discrete-injection}.  The desired functor takes $\phi$
to
$$%
\QR_{\aa_\tau}^\sigma\phi = 
\begin{cases}
\soct[\sigma]\phi_*(s_{\aa_\tau}^\sigma) &\text{if }
  (\bb_\rho,\xi) \preceq (\aa_\tau,\sigma)\text{ and }
  \aa_\tau \in \cds[\tau,\sigma]
\\
            0                  &\text{otherwise}.
\end{cases}
$$
The naturality of the death functor is because $(\bb_\rho,\xi) \preceq
(\brp) \preceq (\aa_\tau,\sigma)$ implies that the element $1 \in
\kk[S]_{\bb'}$ for some $\bb' \in \brp + \qnx[']$ also divides one of
the aforementioned closed cogenerators~$s_\aa^\sigma$ of~$\kk[S]$
along~$\tau$ with nadir~$\sigma$ in some degree $\aa \in \aa_\tau$.
\end{proof}

%%%%%%%%%%%%%%%%%%%%%%%%%%%%%%%%%%%%%%%%%%%%%%%%%%%%%%%%%%%%%%%%%%%%%%%%%
\subsection{QR codes}\label{sub:qr}\mbox{}%%%%%%%%%%%%%%%%%%%%%%%%%%%%%%%

\begin{defn}\label{d:qr-code}
The \emph{(functorial) \qrcode} of a $\cQ$-module~$\cM$ is the
homomorphism
$$%
  \QR = \prod_{\aa_\tau \in \cdm} \QR_{\aa_\tau}:
  \Gen\cM \to
  \Soc\cM
$$
induced by Theorem~\ref{t:death-functor}.\ref{i:death-functor-ZZ} if
$\cQ$ is a discrete polyhedral group and
$$%
  \QR = \prod_{(\aa_\tau,\sigma) \in \cdm} \QR_{\aa_\tau}^\sigma:
  \Gen\cM \to
  \Soc\cM
$$
induced by Theorem~\ref{t:death-functor}.\ref{i:death-functor-RR} if
$\cQ$ is a real polyhedral group.
\end{defn}

\begin{remark}\label{r:poset-module-homomorphism-ZZ}
Functoriality of the death map means that $\QR$ is a morphism of
modules over posets, in the sense that if $(\bb_\rho,\xi) \preceq
(\brp,\xi')$ then
the composite homomorphism
$$%
  \Gen_{\bb_\rho}^\xi\!\cM
  \stackrel{\ \ \ \ \ \,}\fillrightmap
  \Gen_{\bb'\!\!_{\rho'}\!}^{\xi'}\cM
  \stackrel{\textstyle\,\QR_{\bb'\!\!_{\rho'}}^{\xi'}\,}\fillrightmap
  \prod_{(\bb'\!\!_{\rho'},\xi') \preceq (\aa_\tau,\sigma)} \Soc_{\aa_\tau}^\sigma\!\cM
$$
is also expressible as the composite homomorphism
$$%
  \Gen_{\bb_\rho}^\xi\!\cM
  \stackrel{\textstyle\,\QR_{\bb_\rho}^\xi\ }\fillrightmap
  \prod_{(\bb_\rho,\xi) \preceq (\aa_\tau,\sigma)} \Soc_{\aa_\tau}^\sigma\!\cM
  \stackrel{\ \ \ \ \ \,}\fillonto
  \prod_{(\bb'\!\!_{\rho'},\xi') \preceq (\aa_\tau,\sigma)}\Soc_{\aa_\tau}^\sigma\!\cM.
$$
\end{remark}

\vbox{
\begin{wrapfigure}{R}{0.2\textwidth}
  \vspace{-2.3ex}
  \hspace{-1ex}
  \(
  \begin{array}{@{}c@{}}
  \includegraphics[height=9mm]{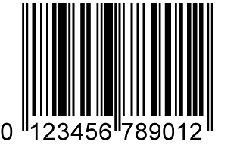}\\[-1ex]\scriptstyle n=1
  \end{array}
  \raisebox{2mm}{$\goesto$}\
  \begin{array}{@{}c@{}}
  \includegraphics[height=9mm]{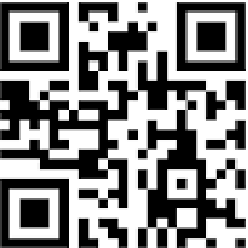}\\[-1ex]\scriptstyle n=2
  \end{array}
  \)
  \hspace{-2ex}
  \vspace{-3ex}
\end{wrapfigure}
\refstepcounter{thm}
\smallskip
\noindent
\textbf{Remark~\arabic{section}.\arabic{thm}.}\label{r:qr-etymology}
%begin{remark}\label{r:qr-etymology}
The letters ``QR'' stand for ``quotient-restriction'', but ``\qrcode''
also refers to the two-dimensional pixelated versions of bar codes
typically parsed by camera-enabled mobile phones.
%end{remark}
}
\smallskip

\begin{thm}\label{t:recover}
Every finitely encoded module over a discrete or real polyhedral group
is recovered functorially from its \qrcode; namely, $\cM$ is the image
of the natural~map
\begin{equation}\label{eq:recover-prod-ZZ}
  \bigoplus_{\bb_\rho\in\cbm} \kk[\bb_\rho + \cQ_+]\otimes_\kk\Gen_{\bb_\rho}\!\cM
  \stackrel{\textstyle\ \widehat\QR\ \,}\fillrightmap
  \prod_{\aa_\tau\in\cdm}\kk[\aa_\tau - \cQ_+]\otimes_\kk\Soc_{\aa_\tau}\!\cM
\end{equation}
if the underlying polyhedral group is discrete, and
\begin{equation}\label{eq:recover-prod-RR}
\bigoplus_{(\bb_\rho,\xi)\in\cbm}\kk[\bb_\rho+\qnx]\otimes_\kk\Gen_{\bb_\rho}^\xi\!\cM
  \stackrel{\textstyle\ \widehat\QR\ \,}\fillrightmap
\prod_{(\aa_\tau,\sigma)\in\cdm}\kk[\aa_\tau-\qns]\otimes_\kk\Soc_{\aa_\tau}^\sigma\!\cM
\end{equation}
if the underlying polyhedral group is real.
\end{thm}
\begin{proof}
As with Theorem~\ref{t:death-functor}, the proof is the same, mutatis
mutandis, for the discrete case instead of the real case, so only the
real case is presented explicitly.

The death functor
$\Gen_{\bb_\rho}^\xi\!\cM\to\Soc_{\aa_\tau}^\sigma\!\cM$ naturally
specifies a~%
\mbox{homomorphism} $\kk[\bb_\rho + \qnx] \otimes_\kk
\Gen_{\bb_\rho}^\xi\!\cM \to \kk[\aa_\tau - \qns] \otimes_\kk
\Soc_{\aa_\tau}^\sigma\!\cM$, and hence the natural map~$\wh\QR$, via
the homomorphism $\kk[\bb_\rho + \qnx] \to \kk[\aa_\tau - \qns]$
defined by $1 \in \kk$ if $(\bb_\rho,\xi) \preceq (\aa_\tau,\sigma)$
and~$0$ otherwise.  Using the inclusion $\cQ = \cQ/\{\0\} \times
\{\cQ_+\} \subseteq \cbq$ to make sense of the relevant direct limits,
$\Gen\cM$ pushes forward to a $\cQ$-module
$$%
  \Gen^\cQ\!\cM
  =
  \bigoplus_{\aa\in\cQ}\ \dirlim_{(\bb_\rho,\xi)\preceq\aa}\Gen_{\bb_\rho}^\xi\!\cM
$$
universal among $\cQ$-modules with a map from~$\Gen\cM$.  The
map~$\wh\QR$ factors through $\Gen^\cQ\!\cM$ by functoriality in
Theorem~\ref{t:death-functor}; write~$\wh\QR^{\,\cQ\!}$ for the
induced map on $\Gen^\cQ\!\cM$.

The direct sum in Eq.~(\ref{eq:recover-prod-RR}) maps surjectively
onto~$\cM$ by Theorem~\ref{t:surjection-RR}.\ref{i:topr=>phi} (or
Theorem~\ref{t:surjection-ZZ}.\ref{i:topr=>phi} in the discrete case).
Universality of the pushforward forces this surjection to factor
through $\Gen^\cQ\!\cM$.  On the other hand, $\wh\QR^{\,\cQ\!}$
factors through this surjection $\Gen^\cQ\!\cM \onto \cM$ by
construction of the death functor in Theorem~\ref{t:death-functor}:
every element of $\Gen\cM$ that maps to~$0$ in~$\cM$ certainly maps
to~$0$ in~$\Soc\cM$.  But the induced map from~$\cM$ to the product in
Eq.~(\ref{eq:recover-prod-RR}) is injective by
Theorem~\ref{t:injection} (or Theorem~\ref{t:discrete-injection} in
the discrete case) because it induces an injection on socles by
construction.  (Note: the socle of the product in
Eq.~(\ref{eq:recover-prod-RR}) along a face could a~priori be much
bigger than that of~$\cM$, but composing the map from~$\cM$ with the
projection to any factor of the product shows that the induced map on
socles must at least be injective.)
\end{proof}

\begin{remark}\label{r:qr-code}
Keeping track of entire graded components of~$\cM$
(technically:~$\mrx/\rho$) feels wasteful compared to keeping track
merely of vector spaces~$\topr[\xi]\cM$, whose graded components are
best viewed as vector spaces that count minimal generators.  This
waste feels unpleasant, but it is required to make the QR~morphism
functorial.  And perhaps it shouldn't feel so bad: writing down
$\topr[\xi]\cM$ as the quotient that it naturally is requires knowing
$\Gen_{\bb_\rho}^\xi\cM$ as well as the kernel of its surjection
to~$\topr[\xi]\cM$ (see Theorem~\ref{t:life-top} for more about that),
so it is not horrible to request that the filtered object serve as the
data structure instead of the associated graded object.

In finitely generated cases---or better, finitely determined ones from
Section~\ref{s:ZZn}---the QR~code is a functorial homomorphism between
vector spaces of finite dimension that encapsulates all of a given
module.  As such, it is similar to flange presentation
(Definition~\ref{d:flange}).  The latter, however, sacrifices
functoriality for efficiency by choosing~bases.
\end{remark}

\begin{remark}\label{r:unwind}
It is worth unwinding the definitions that lead to the reconstruction
homomorphism~$\wh\QR$ featured in Theorem~\ref{t:recover}.  Over a
discrete polyhedral group, $\Gen_{\bb_\rho}\!\cM =\nolinebreak
(\cmr/\rho)_{\bb_\rho}$ from
Definition~\ref{d:birth-ZZ}.\ref{i:module-ZZ} and $\cM^\rho =
\hhom_\cQ\bigl(\kk[\cQ_+]_\rho,\cM\bigr)$ from Definition~\ref{d:cmr},
so the direct summands in Eq.\,(\ref{eq:recover-prod-ZZ}) are
\begin{equation}\label{eq:life-ZZ}
  \kk[\bb_\rho+\cQ_+]\otimes_\kk\Hom_\cQ\bigl(\kk[\bb_\rho+\cQ_+],\cM\bigr).
\end{equation}
Similarly, over a real polyhedral group, $\Gen_{\bb_\rho}^{\xi}\!\cM =
(\mrx/\rho)_{\bb_\rho}$ from
Definition~\ref{d:birth-RR}.\ref{i:module-RR} and $\mrx =
\hhom_\cQ\hspace{-.1ex}\bigl(\kk[\qnx]_\rho,\hspace{-.1ex}\cM\bigr)$
from Definition\,\ref{d:mrx}, so the summands
in~Eq.\hspace{.2ex}(\ref{eq:recover-prod-RR})~are
\begin{equation}\label{eq:life-RR}
  \kk[\bb_\rho+\qnx]\otimes_\kk\Hom_\cQ\bigl(\kk[\bb_\rho+\qnx],\cM\bigr).
\end{equation}
Note the lack of underlines on the $\Hom$ functors: these are single
graded pieces of modules, not entire graded modules.
\end{remark}

\begin{defn}\label{d:life-module}
The \emph{life module} of~$\cM$ over a discrete or real polyhedral
group~$\cQ$ is
$$%
  \life\cM
  =
  \bigoplus_{\bb_\rho\in\cbm} \kk[\bb_\rho + \cQ_+]\otimes_\kk\Gen_{\bb_\rho}\!\cM
  \quad\text{or}\quad
  \life\cM
  =\hspace{-1ex}
  \bigoplus_{(\bb_\rho,\xi)\in\cbm}\hspace{-1ex}
  \kk[\bb_\rho+\qnx]\otimes_\kk\Gen_{\bb_\rho}^\xi\!\cM,
$$
according to whether $\cQ$ is discrete or real, respectively.
\end{defn}

\begin{remark}\label{r:life}
The image of the reconstruction homomorphism $\wh\QR$ in
Theorem~\ref{t:recover} is naturally isomorphic to~$\cM$, so using the
notion of life module, $\wh\QR$ induces a natural surjection $\life\cM
\onto \cM$, also denoted~$\wh\QR$ by abuse of notation, that is easy
to describe explicitly: the displayed expressions in
Remark~\ref{r:unwind} map to~$\cM$ by $y \otimes\nolinebreak \phi
\mapsto\nolinebreak \phi(y)$.
\end{remark}

%%%%%%%%%%%%%%%%%%%%%%%%%%%%%%%%%%%%%%%%%%%%%%%%%%%%%%%%%%%%%%%%%%%%%%%%%
\section{Elder morphisms}\label{s:elder}%%%%%%%%%%%%%%%%%%%%%%%%%%%%%%%%%
%%%%%%%%%%%%%%%%%%%%%%%%%%%%%%%%%%%%%%%%%%%%%%%%%%%%%%%%%%%%%%%%%%%%%%%%%

\begin{defn}\label{d:life-prec-beta}
If $\beta \in \cbq$ is any element of the birth poset of~$\cQ$, then
define $\life_{\prec\beta}\cM$ by taking the direct sum in
Definition~\ref{d:life-module} over birth poset elements strictly
preceding~$\beta$, and similarly define $\life_{\preceq\beta}\cM$ by
taking the direct sum over birth poset elements weakly
preceding~$\beta$.
\end{defn}

\begin{conv}\label{conv:top}
For any birth degree $\beta \in \cbq$ over a real or polyhedral
group~$\cQ$,~set
\begin{align*}
\top_\beta\cM
  &=
  \begin{cases}
  (\topcr\cM){}_{\bb_\rho}
    & \text{if }\cQ\text{ is discrete and }\beta = \bb_\rho
  \\[.5ex]
  (\topr[\xi]\cM)_{\bb_\rho}
    & \text{if }\cQ\text{ is real and }\beta = (\bb_\rho,\xi)
  \end{cases}
\\[1ex]
\kk[\beta + \cqb]
  &=
  \begin{cases}
  \kk[\bb_\rho + \cQ_+]\otimes_\kk\Gen_{\bb_\rho}\!\cM
    & \text{if }\cQ\text{ is discrete and }\beta = \bb_\rho
  \\[.5ex]
  \kk[\bb_\rho+\qnx]\otimes_\kk\Gen_{\bb_\rho}^\xi\!\cM
    & \text{if }\cQ\text{ is real and }\beta = (\bb_\rho,\xi).
  \end{cases}
\end{align*}
Parallel notations work for $\Gen_\beta\cM$ and $\life_\beta\cM =
\kk[\beta+\cqb]\otimes_\kk\Gen_\beta\cM$, as well as dually for
$\soc_\alpha\cM$ and $\Soc_\alpha\cM$ when $\alpha \in \cdq$.
\end{conv}

\begin{defn}\label{d:elder-module}
Fix a module~$\cM$ over a real or discrete polyhedral group~$\cQ$ and
a birth degree $\beta \in \cbq$.  (Thus $\beta$ has the
form~$\bb_\rho$ if $\cQ$ is discrete and $(\bb_\rho,\xi)$ if $\cQ$ is
real.)
\begin{enumerate}
\item%
The \emph{elder submodule} of~$\cM$ at~$\beta$ is the submodule
$$%
  \cM_{\prec\beta} = \wh\QR(\life_{\prec\beta}\!\cM) \subseteq \cM.
$$

\item%
The \emph{extant submodule} of~$\cM$ at~$\beta$ is the submodule
$$%
  \cM_{\preceq\beta} = \wh\QR(\life_{\preceq\beta}\!\cM) \subseteq \cM.
$$

\item%
The \emph{elder quotient} of~$\cM$ at~$\beta$ is
$\cM_{\preceq\beta}/\cM_{\prec\beta}$.
\end{enumerate}
\end{defn}

\begin{example}\label{e:RR-elder}
When $\cQ = \RR$ and~$\cM$ is expressed as a direct sum of
\emph{bars}---indicator modules for intervals that on each end may be
open, closed, or infinite---the elder submodule~$\cM_{\prec\beta}$ is
the submodule of~$\cM$ that is the direct sum of those bars whose left
endpoints occur strictly earlier than~$\beta$.  (See
Example~\ref{e:birth-RR} for the explicit description of~$\cbr$ as a
totally ordered set.)  The extant submodule of~$\cM$ is the direct sum
of those bars whose left endpoints occur at~$\beta$ or earlier.  The
elder quotient is the direct sum of those bars whose left endpoints
occur exactly at~$\beta$.  This description has been phrased in terms
of a direct sum decomposition of~$\cM$ specified by the bar code
of~$\cM$, but one of the main points of the theory is that the result
is functorial: the direct sum of bars with fixed left endpoint is
canonically a subquotient of~$\cM$.
\end{example}

\begin{thm}\label{t:life-top}
If~$\cM$ is upset-finite over a discrete or real polyhedral
group~$\cQ$ then
$$%
  \Gen(\cM_{\preceq\beta}/\cM_{\prec\beta})
  =
  \top_\beta(\cM_{\preceq\beta}/\cM_{\prec\beta})
  =
  \top_\beta\cM
$$
for any birth degree $\beta \in \cbq$, and
$\top_\gamma(\cM_{\preceq\beta}/\cM_{\prec\beta}) = 0$ if $\gamma \neq
\beta$.
\end{thm}
\begin{proof}
The surjection $\wh\QR: \life\cM \onto \cM$ induces a map
$\top_\gamma\wh\QR: \top_\gamma\life\cM \to \top_\gamma\cM$ that is an
isomorphism for all~$\gamma \in \cbq$ by construction, using
Example~\ref{e:beta} to calculate $\top_\gamma\life\cM$ (at least in
the real case; the easier discrete analogue of Example~\ref{e:beta} is
true and applied without further comment).  As the map $\life_\beta
\to \cM_{\preceq\beta}/\cM_{\prec\beta}$ is a finite upset cover, it
induces a surjection $\topr\life_\beta\cM \to
\topr(\cM_{\preceq\beta}/\cM_{\prec\beta})$ for all faces~$\rho$ by
Theorems~\ref{t:surjection-RR} and~\ref{t:surjection-ZZ}.
Example~\ref{e:beta} therefore implies all of the claimed vanishing as
well as $\top_\beta\cM = \top_\beta\life_\beta\cM \onto
\top_\beta(\cM_{\preceq\beta}/\cM_{\prec\beta})$.  The only thing that
could a~priori ruin the claimed right-hand equality is if the
composite $\life_{\prec\beta} \to \life_{\preceq\beta} \to
\cM_{\preceq\beta}/\cM_{\prec\beta}$ managed to affect~$\top_\beta$,
but $\top_\beta$ vanishes on $\life_{\prec\beta}\cM$ by
Example~\ref{e:beta} again.
% no exactness of top is used here except as cited in
% Theorems~\ref{t:surjection-RR} and~\ref{t:surjection-ZZ}.  This last
% bit, about the vanishing, works for any functor, regardless of
% exactness: if the map is zero then it's zero and that's it.
Using Proposition~\ref{p:exact-mrx} to produce a surjection
$\Gen_\beta\life_\beta\cM \onto
\Gen_\beta(\cM_{\preceq\beta}/\cM_{\prec\beta})$, the remaining
claimed equality follows by Proposition~\ref{p:surjection-RR} from
$\Gen_\beta\life_\beta\cM = \top_\beta\life_\beta\cM$, given that
$\top_\beta\life_\beta\cM \simto
\top_\beta(\cM_{\preceq\beta}/\cM_{\prec\beta})$ is an isomorphism.
\end{proof}

\begin{example}\label{e:RR-elder'}
In the situation of Example~\ref{e:RR-elder}, the theorem says what is
expected: the entire degree~$\beta$ piece of the elder quotient
$\cM_{\preceq\beta}/\cM_{\prec\beta}$ equals the vector space spanned
by the left endpoints of bars that sit at~$\beta$.
\end{example}

\begin{defn}\label{d:elder-morphism}
The \emph{elder morphism} of an upset-finite module over a real or
discrete polyhedral group is the functorial \qrcode
$$%
  \QR^\beta: \top_\beta\cM \to \Soc(\cM_{\preceq\beta}/\cM_{\prec\beta})
$$
of $\cM_{\preceq\beta}/\cM_{\prec\beta}$ afforded by
Theorem~\ref{t:life-top} and Definition~\ref{d:qr-code}.
\end{defn}

\begin{remark}\label{r:elder-upset-finite}
Defining an elder morphism only requires that the module~$\cM$ be
upset-finite, but no claims are made about whether the definition is
reasonable unless $\cM$ is finitely encoded, so that information about
$\cM$, as well as its elder and extant submodules, can be recovered
from their \qrcode s.
\end{remark}

\begin{remark}\label{r:elder-rule}
The elder morphism could have been defined instead to have the target
$\Soc(\cM/\cM_{\prec\beta})$, because of the natural inclusion
$\Soc(\cM_{\preceq\beta}/\cM_{\prec\beta}) \subseteq
\Soc(\cM/\cM_{\prec\beta})$ from Theorems~\ref{t:injection}
and~\ref{t:discrete-injection}.  And linguistically it is useful to
speak of classes or elements dying in a quotient of~$\cM$ rather than
in a quotient of~$\cM_{\preceq\beta}$, even though the class in
question of course persists in the submodule
$\cM_{\preceq\beta}/\cM_{\prec\beta}$ of the
module~$\cM/\cM_{\prec\beta}$ if it was born at~$\beta$.
\end{remark}

\begin{remark}\label{r:prior-work-functors}
Crawley-Boevey constructed a functor from $\RR$-indexed persistence
modules to the category of vector spaces which counts number of
intervals of a given type \cite{crawley-boevey}.  Cochoy and Oudot
extended this functorial construction to special \mbox{$2$-parameter}
persistence modules satisfying an ``exact diamond condition'', which
guarantees thin summands \cite{cochoy-oudot}.  Elder morphisms can be
thought of as multiparameter generalizations: given a birth parameter
$\beta \in \cbq$, the elder morphism functorially isolates all
possible \emph{elder-death} parameters for classes born at~$\beta$ in
the form of the death poset~$\cde$ of the elder quotient at~$\beta$.
The number of ``bars'' born at~$\beta$ that die an elder-death by
joining the elder submodule at $\alpha \in \cde$ is categorified by
functorial vector spaces, namely tops and socles.  The elder morphism
then associates a functorial linear birth-to-death map (a~\qrcode)
between those vector spaces.
\end{remark}

%%%%%%%%%%%%%%%%%%%%%%%%%%%%%%%%%%%%%%%%%%%%%%%%%%%%%%%%%%%%%%%%%%%%%%%%%
\section{Generators}\label{s:generators}%%%%%%%%%%%%%%%%%%%%%%%%%%%%%%%%%
%%%%%%%%%%%%%%%%%%%%%%%%%%%%%%%%%%%%%%%%%%%%%%%%%%%%%%%%%%%%%%%%%%%%%%%%%

\begin{remark}\label{r:what-is-a-generator?}
Generators of graded modules are homogeneous elements that do not lie
in the submodule generated in any lower degree.  That is, roughly
speaking, how generators are defined here, as well, but with two key
enhancements in the context of finitely encoded modules over
polyhedral groups~$\cQ$.  First, generators are only elements of
modules in the discrete case, although even then they are more
properly thought of as homomorphisms from localizations
$\kk[\cQ_+]_\rho$ along faces.  In the real case, generators are
homomorphisms from ``open localizations''~$\kk[\qnx]_\rho$ along
faces.  These should be thought of as rather slight generalizations of
the usual notion of element, which is a homomorphism from the free
module $\kk[\cQ_+]$ itself.  The second enhancement makes precise the
notion of submodule generated in lower degree, given that generators
need not be elements---and even in the discrete case should be thought
of as equivalence classes of elements under translation along a face.
The elder submodule fills that role.
\end{remark}

\begin{defn}\label{d:generator-RR}
Fix a module~$\cM$ over a polyhedral partially ordered group~$\cQ$.
\begin{enumerate}
\item%
If $\cQ$ is discrete then a \emph{generator} of~$\cM$ of \emph{degree
$\bb \in \cQ$} \emph{along a face~$\rho$} is an element in $\cmr_\bb$
that lies outside of the elder submodule $\cM_{\prec\bb_\rho}$, where
$\bb_\rho = \bb+\ZZ\rho \in \qzr$.

\item%
If $\cQ$ is real then a \emph{generator} of~$\cM$ of \emph{degree $\bb
\in \cQ$} \emph{along a face~$\rho$} with \emph{nadir $\xi \in
\nabro$} is an element in~$\mrx_\bb$ that lies outside of the elder
submodule $\cM_{\prec(\bb_\rho,\xi)}$.
\end{enumerate}
\end{defn}

\begin{remark}\label{r:generator-RR}
Some recall of the objects, including their notation and relations
between them, is in order.  In the discrete case, $\cM^\rho =
\hhom_\cQ\bigl(\kk[\cQ_+]_\rho,\cM\bigr)$ from Definition~\ref{d:cmr},
and the condition of lying outside of the elder submodule means that
the image of a degree~$\bb$ generator along~$\rho$ maps to a nonzero
element in $\topcr\cM$, necessarily of degree~$\bb_\rho$, under the
surjection in Proposition~\ref{p:surjection-ZZ}.

Similarly, in the real case, $\mrx =
\hhom_\cQ\bigl(\kk[\qnx]_\rho,\cM\bigr) =
\hhom_\cQ\bigl(\kk[\cQ_+]_\rho,\lx\cM\bigr)$ as in
Definition~\ref{d:mrx} and Lemma~\ref{l:mrx}.  The condition of lying
outside of the elder submodule means that the generator in question
maps to a nonzero element of~$\topr[\xi]\cM$, again necessarily of
degree~$\bb_\rho$, by Proposition~\ref{p:surjection-RR}.

This Remark is intended to lay bare a duality between generators and
cogenerators; compare
Definition~\ref{d:socct}.\ref{i:global-closed-cogen} for the discrete
case and Definition~\ref{d:soct}.\ref{i:global-cogen} for the real
case.
\end{remark}

\begin{remark}\label{r:closed-gen}
A \emph{closed generator} could be defined as an honest element
of~$\cM$, namely the image of $1 \in \kk[\cQ_+]_\rho$, assuming either
the discrete case or that the nadir of the generator in question
is~$\rho$ in the real case.  This is analogy with closed cogenerators
in Definition~\ref{d:socct}.\ref{i:global-closed-cogen}, which are
similarly honest elements of modules.
\end{remark}

%%%%%%%%%%%%%%%%%%%%%%%%%%%%%%%%%%%%%%%%%%%%%%%%%%%%%%%%%%%%%%%%%%%%%%%%%
\section{Functorial bar codes}\label{s:barcodes}%%%%%%%%%%%%%%%%%%%%%%%%%
%%%%%%%%%%%%%%%%%%%%%%%%%%%%%%%%%%%%%%%%%%%%%%%%%%%%%%%%%%%%%%%%%%%%%%%%%

The notation from Convention~\ref{conv:top} is in effect throughout
this section.

\begin{lemma}\label{l:gr}
If $\cM$ is a module over a real or discrete polyhedral group~$\cQ$
and $\alpha \in \cdq$, then $\soc_\alpha\cM$ has a canonical filtration
by the poset~$\cbq$:
$$%
  (\soc_\alpha\cM)_{\preceq\beta} = \soc_\alpha(\cM_{\preceq\beta}).
$$
\end{lemma}
\begin{proof}
If $\beta \preceq \beta'$ then $\cM_{\preceq\beta} \subseteq
\cM_{\preceq\beta'}$, so $\soc_\alpha(\cM_{\preceq\beta}) \subseteq
\soc_\alpha(\cM_{\preceq\beta'})$ by Theorem~\ref{t:injection}.
\end{proof}

\begin{remark}\label{r:gr}
The equality in the lemma allows the parentheses to be removed: both
sides are $\soc_\alpha\cM_{\preceq\beta}$.  It is perhaps simpler to
think of filtering $\soc_\alpha\cM$ by the subposet $\cbm \subseteq
\cbq$ instead, as this is a finite filtration (in every case of
computational interest) containing all of the interesting information
about the filtration by~$\cbq$.
\end{remark}

\begin{thm}\label{t:elder-projection}
Fix a module~$\cM$ over the partially ordered group $\cQ = \RR$
or~$\ZZ$.  For all birth parameters $\beta \in \cbq$ and death
parameters $\alpha \in \cdq$ there is a canonical isomorphism
$$%
  \soc_\alpha(\cM_{\preceq\beta}/\cM_{\prec\beta})
  \simto
  \soc_\alpha\cM_{\preceq\beta}/\soc_\alpha\cM_{\prec\beta}
  =
  \gr_\beta\soc_\alpha\cM.
$$
\end{thm}
\begin{proof}
When $\cQ = \RR$ there are two cases: in the notation of
Example~\ref{e:death-RR}, either $\alpha = \infty$ or not.  First
assume not, so $\alpha = (a,\sigma)$ with $\sigma = \{0\}$ if $\alpha
= a_\bullet$ and $\sigma = \RR_+$ if $\alpha = a_\circ$.  Suppose $s
\in \soc_\alpha(\cM_{\preceq\beta}/\cM_{\prec\beta})$.  Choose an
arbitrary lift $\wt s \in (\cM_{\preceq\beta})_{a-\sigma}$ of~$s$,
with notation as in Definition~\ref{d:atop-sigma}.  The homomorphisms
in Lemmas~\ref{l:natural} and~\ref{l:natural-dual} induce a natural
map $ \phi_\alpha: (\cM_{\preceq\beta})_{a-\sigma} \to
(\cM_{\preceq\beta})_{a+\xi}, $
where $\xi = \RR_+$ if $\sigma = \{0\}$ (i.e.,~if~$\alpha =
a_\bullet$) and $\xi = \{0\}$ if $\sigma = \RR_+$ (i.e., if $\alpha =
a_\circ$).  The hypothesis on~$s$ means that $\phi_\alpha(\wt s)$ lies
not merely in $(\cM_{\preceq\beta})_{a+\xi}$ but in
$(\cM_{\prec\beta})_{a+\xi}$.  Let $\psi_\alpha$ be the restriction
of~$\phi_\alpha$ to $(\cM_{\prec\beta})_{a-\sigma}$.  The preimage
of~$\phi_\alpha(\wt s)$ in $(\cM_{\prec\beta})_{a-\sigma}$
under~$\psi_\alpha$ is well defined modulo
$\soc_\alpha\cM_{\prec\beta}$ by definition of socle.  The desired
isomorphism is
$$%
 \omega_\alpha: s \mapsto \wt s-\psi_\alpha^{-1}\circ\phi_\alpha(\wt s).
$$
It is well defined as an element of
$(\cM_{\preceq\beta})_{a-\sigma}/\soc_\alpha\cM_{\prec\beta}$ because
if $y \in (\cM_{\prec\beta})_{a-\sigma}$ then
$\psi_\alpha^{-1}\circ\phi_\alpha(\wt s) = y$ modulo
$\soc_\alpha\cM_{\prec\beta}$.  And although $\omega_\alpha$ a~priori
takes~$\wt s$ to an arbitrary element of
$(\cM_{\preceq\beta})_{a-\sigma}/\soc_\alpha\cM_{\prec\beta}$, in fact
this element is annihilated by~$\phi_\alpha$ by construction and hence
lies in $\soc_\alpha\cM_{\preceq\beta}/\soc_\alpha\cM_{\prec\beta}$.
That $\omega_\alpha$ is surjective follows because an element
of~$\soc_\alpha\cM_{\preceq\beta}$ automatically provides an element
of $\soc_\alpha(\cM_{\preceq\beta}/\cM_{\prec\beta})$ and remains
unmodified by~$\omega_\alpha$.  That $\omega_\alpha$ is injective
follows because if $s$ is nonzero then $\wt s$ lies outside of
$(\cM_{\prec\beta})_{a-\sigma}$, whereas $\omega_\alpha$ modifies $\wt
s$ by an element of~$(\cM_{\prec\beta})_{a-\sigma}$.

The above argument for $\cQ = \ZZ$ is more elementary: every $\alpha =
(a,\sigma)$ is simply $\alpha = a$, every subscript $a - \sigma$
should simply be~$a$, and every subscript $a + \xi$ should be $a + 1$.
The natural map $\phi_a: (\cM_{\preceq\beta})_a \to
(\cM_{\preceq\beta})_{a+1}$ requires no lemmas to justify its
existence.

The case $\alpha = \infty$ is much easier, by virtue of being
completely general: over any partially ordered group~$\cQ$, the socle
along~$\cQ_+$ is $\hhom_\cQ\bigl(\kk[\cQ_+],\cM\bigr)\big/\cQ_+ =
\cM/\cQ_+$, so the isomorphism in question is by exactness of
quotient-restriction (Lemma~\ref{l:exact-qr}).
\end{proof}

\begin{remark}\label{r:second-isom-thm}
The isomorphism in the theorem can, if justified properly, be seen as
an instance of the second isomorphism theorem: if $\cN$ has submodules
$K$ and~$\cN'$, and $K' = K \cap \cN'$, then $\ker\bigl(\cN/\cN' \to
\cN/(\cN' + K)\bigr) = (\cN' + K)/\cN' \simto K/K'$.  The relevant
ambient module is $\cN = (\cM_{\preceq\beta})_{a-\sigma}$, with
submodules $K = \soc_\alpha\cM_{\preceq\beta}$ and~$\cN' =\nolinebreak
(\cM_{\prec\beta})_{a-\sigma}$.
\end{remark}

\begin{defn}\label{d:barcode}
Fix a module~$\cM$ over the real polyhedral group~$\RR$.
\begin{enumerate}
\item%
The \emph{total top} of~$\cM$ is $\Top\cM =
\bigoplus_{\beta\in\cbm}\top_\beta\cM$.

\item%
The \emph{graded total socle} of~$\cM$ is $\gr\Soc\cM =
\prod_{\alpha\in\cdm}\gr\soc_\alpha\cM$.

\item%
The \emph{functorial bar code} of~$\cM$ is the linear map
$$%
  \Top\cM \to \gr\Soc\cM
$$
obtained by composing the elder morphism of~$\cM$ in
Definition~\ref{d:elder-morphism} with the \emph{elder projection} in
Theorem~\ref{t:elder-projection}.
\end{enumerate}
\end{defn}

\begin{thm}\label{t:barcode}
Fix a module~$\cM$ over the real polyhedral group~$\RR$.  In any
decomposition of~$\cM$ as a direct sum of indicator subquotients
of\/~$\kk[\RR]$---that is, indicator modules on intervals---the left
endpoints of the intervals form a basis for the source vector
space~$\Top\cM$ and the right endpoints form a basis for the target
vector space~$\gr\Soc\cM$.
\end{thm}
\begin{proof}
Functoriality of bar codes (see Definitions~\ref{d:elder-morphism}
and~\ref{d:barcode}) reduces to the case of a single indicator
subquotient.  That case falls under Examples~\ref{e:birth-RR}
and~\ref{e:death-RR}.
\end{proof}

%%%%%%%%%%%%%%%%%%%%%%%%%%%%%%%%%%%%%%%%%%%%%%%%%%%%%%%%%%%%%%%%%%%%%%%%%
\addtocontents{toc}{\protect\setcounter{tocdepth}{1}}%%%%%%%%%%%%%%%%%%%%
\section{Future directions}\label{s:future}%%%%%%%%%%%%%%%%%%%%%%%%%%%%%%
%%%%%%%%%%%%%%%%%%%%%%%%%%%%%%%%%%%%%%%%%%%%%%%%%%%%%%%%%%%%%%%%%%%%%%%%%

%%%%%%%%%%%%%%%%%%%%%%%%%%%%%%%%%%%%%%%%%%%%%%%%%%%%%%%%%%%%%%%%%%%%%%%%%
\subsection{Fly wing implementation}\label{sub:fly-wing}\mbox{}%%%%%%%%%%

\noindent
The plan for statistical analysis of the fly wing dataset is to
summarize wings using biparameter persistence and then to
statistically analyze the \emph{rank function} $\RR^2 \times \RR^2 \to
\NN$ defined by $(\aa \preceq \bb) \mapsto \mathrm{rank}(H_\aa \to
H_\bb)$ \cite{multiparamPH,cagliari-difabio-ferri2010}.  In more
detail, the plan is: fly wing $\stackrel 1\goesto$ 2D persistence
module $\stackrel 2\goesto$ encoding by monomial matrix $\stackrel
3\goesto$ rank function.  There is potential loss at steps~1 and~3; in
contrast, step~2 is lossless by Theorem~\ref{t:syzygy}.

The question of loss in step~1 is one of statistical sufficiency.  It
is plausible that such loss does not occur locally---deforming a fly
wing should yield a nontrivial deformation of its biparameter
persistence module (Example~\ref{e:fly-wing-filtration})---although
perhaps globally it might.  Proofs of such statements would likely
rely on \qrcode s, since the different persistence modules in question
ought to have different birth or death posets.

The question of loss in step~3 amounts to a question about moduli: is
the variation in fly-wing modules solely from the placements of birth
upsets and death downsets in a fringe presentation
(Definition~\ref{d:monomial-matrix-fr}), or can the scalar entries
vary continuously?  There is reason to believe that continuous
variation of scalar entries does not occur---because the homology is
dimension~$0$ or codimension~$1$ in~$\RR^2$, both of which come with
canonical bases---but formulating such a statement precisely requires
theoretical development.

Doing statistics with rank functions requires effective data
structures for them.  This is a general problem, not limited to the
fly wing case although well exemplified by it.  Rank functions from an
$n$-parameter filtration are nonnegative integer-valued functions
on~$\RR^{2n}$ whose domains of constancy are---in the fly wing case,
as in many data science applications---semialgebraic.  Fringe
presentations and surrounding objects are suited for such functions.
Once rank functions are effectively represented, statistical analysis
can proceed largely as in the ordinary single-parameter case
\cite{robins-turner2016}.

%%%%%%%%%%%%%%%%%%%%%%%%%%%%%%%%%%%%%%%%%%%%%%%%%%%%%%%%%%%%%%%%%%%%%%%%%
\subsection{Computation}\label{sub:computation}\mbox{}%%%%%%%%%%%%%%%%%%%

\noindent
The plan for analysis of the fly wing dataset, in the persistence
setup from Example~\ref{e:fly-wing-filtration}, requires various
aspects of the theory to be made more explicitly algorithmic.  These
and other tasks are relevant in general, not merely for fly wings or
$n = 2$ parameters.
\begin{enumerate}
\item%
Compute a fringe presentation from given spline data.  This should
involve real semialgebraic geometry along with careful bookkeeping of
planar graph topology.

\item%
Compute rank functions from any of the data structures for
multiparameter persistence in Theorem~\ref{t:syzygy}.  Naive
algorithms and data structures are easy enough to write down when a
finite encoding is given as in Definition~\ref{d:encoding}, but
these---the data structures as well as the algorithms---are likely
suboptimal.

\item%
Define and devise algorithms to compute data structures for \qrcode s.
% Distinguishing between fly wings in particular means distinguishing
% between their \qrcode s.
This is almost analogous to computing flange presentations of finitely
determined modules, in that \qrcode s record the degrees of
generators, the degrees of cogenerators, and a linear map, just as
flange presentations do.  The differences are that the source vector
spaces in \qrcode s are (morally) entire graded pieces of the module
and the linear maps are filtered through these.

\item%
Any data structure for \qrcode s is automatically a data structure for
a single elder morphism, which after all is the \qrcode\ of a
subquotient, but it could be important for statistical purposes to
specify all elder morphisms at once.  The problem is that the source
and especially target vector spaces vary depending on the birth
parameter.  Thus the \emph{total elder morphism} is a family of elder
morphisms fibered over the set of parameters.  It suffices to fiber
over the birth~poset.

\item%
Algorithmically compute the equivalences in the syzygy theorems
(Theorems~\ref{t:syzygy} and~\ref{t:presentations-minimal}) over real
and discrete polyhedral groups.  Compute \qrcode s and elder morphisms
from fringe presentations or any of the other hallmarks of finitely
encoded modules.

\item%
Evaluate cogenerator functors
% (Definition~\ref{sub:closed-gen})
on downset modules
% (Example~\ref{e:indicator})
in such a way that the cohomology can be effectively calculated.
Dually, evaluate generator functors on upset modules.  These
operations are linchpins of the computational theory.  For example,
the socle of an $\RR^n$-module is extracted from a downset
copresentation
% $E^0 \to E^1$
by applying cogenerator functors to it, which yields a categorified
semialgebraic subset of~$\RR^n$ mapping to another categorified
semialgebraic subset of~$\RR^n$.  The kernel is the desired socle.
Part of the problem is to make these categorifications precise, in the
language of (say) constructible sheaves.
\end{enumerate}
Minimal flange presentations as well as injective or flat resolutions
of finitely determined $\ZZ^n$-modules are computable algorithmically
starting from generators and relations in the finitely generated case
\cite{injAlg}, and consequently for finitely determined modules using
the categorical equivalences between finitely generated and finitely
determined modules \cite[Definition~2.7, Table~1, and
Theorem~2.11]{alexdual}.  Note that some software exists for
multiparameter persistence of finitely generated $\ZZ^n$-modules; see
\cite{lesnick-wright2015} and \cite{topcat}, the latter being a
library for multiparameter persistent homology.
% by Oliver G\"afvert, KTH Royal Institute of Technology, Sweden

% future: alleviate problems with interleaving distance
% 
% consider citing [Magnus Botnan, Michael Lesnick, H. Bjerkevik] and
% [Oudot, Cochoy, Botnan]: interleaving distances on certain categories
% of $\RR^n$-modules
% 
% This is something to do, but it is a problem well known to the
% field; unless we can pinpoint a contribution to that problem---or
% how one should think about it---in this paper, let's leave it alone.

%%%%%%%%%%%%%%%%%%%%%%%%%%%%%%%%%%%%%%%%%%%%%%%%%%%%%%%%%%%%%%%%%%%%%%%%%
\subsection{Homological algebra}\label{sub:hom-alg}\mbox{}%%%%%%%%%%%%%%%

\subsubsection{Stability}
With discrete parameters,
% ``The need for a stabilisation procedure comes from the fact that
``algebraic invariants of multiparameter persistence modules such as
minimal number of generators, Betti tables, Hilbert polynomials
etc.~tend to change drastically when the initial data is altered even
slightly.  That is why the classical commutative algebra invariants
are not useful for data analysis.  For data analysis we need stable
invariants'' \cite{gafvert-chacholski2017}.  Real multiparameter
persistence provides a counterpoint: even though functorial \qrcode s
and elder morphisms are algebraic in nature, continuity of the
poset~$\RR^n$ allows them to move only slightly when the input data
are perturbed.  For example, wiggling a fly wing should nudge only the
upsets and downsets in a fringe presentation; it should not, in some
generic sense, alter the validity of a given finite encoding poset, as
long as one is willing to perturb the encoding poset morphism so as to
nudge the fibers.  That means the homological algebra of the perturbed
module~$\cM'$ is pulled back from the same homological algebra over
the encoding poset that governs the homological algebra of~$\cM$.
This needs to be made precise, particularly the word ``generic'' and
the concept of ``perturbation''---of data and modules as well as poset
morphisms---but the salient point is that when the scalar matrix in a
fringe presentation is fixed and only the upsets and downsets move,
the Hilbert function and rank function vary in
predictable~geometric~ways.  This is therefore a call for stability
theorems in real multiparameter~persistence.

The motion of Betti numbers in the discrete case corresponds in the
real case, at least in homological degree~$0$, to motion of the birth
poset---or, more accurately, motion of the constructible function on
the birth poset whose values are dimensions of generator functors
(Section~\ref{s:gen-functors}).  These are more subtle than Hilbert or
rank functions but should still enjoy some kind of upper
semicontinuity, in the finitely encoded context, along the lines of
upper semicontinuity in flat families
% \cite[Theorem~8.37]{cca}
in the discrete, finitely generated context.

%vspace{-.5ex}
\subsubsection{Indicator resolutions}
Regarding Conjecture~\ref{conj:indicator-dimension-n}, although
minimal upset and downset resolutions exist, it is unclear whether
they must terminate after finitely many steps.  And although finite
upset and downset resolutions exist---this much is part of the largely
unrestricted syzygy theorem over an arbitrary poset~$\cQ$---no uniform
bound on their length is known when~$\cQ$ has quotients with unbounded
order~dimension.

% compare ``no-moduli'' problem to [CSM14 (arXiv:1409.7936),
% Proposition~3.5], which concerns modules that are
% vector-space-ifications of filtrations of (finite) sets
% \comment{checked; doesn't seem to work...}

%vspace{-.5ex}
\subsubsection{Block decomposition of exact modules}%enlargethispage{.55ex}
The purpose of the papers by Crawley-Boevey \cite{crawley-boevey} and
by Cochoy and Oudot \cite{cochoy-oudot}, discussed in
Remark~\ref{r:prior-work-functors}, is bar-code decomposition of
$\RR$-finite modules over~$\RR$ and so-called ``exact'' $\RR^2$-finite
modules over~$\RR^2$, respectively.  The exactness condition specifies
that if $y \in \cM_\aa$ is the image of an element $y_i \in
\cM_{\aa-\epsilon_i\ee_i}$ for every choice of~$i$, then some class
$y' \in \cM_{\aa'}$, for $\aa' = \aa - \sum_i \epsilon_i\ee_i$,
satisfies $y' \mapsto y$.  Thinking in the context of arbitrary~$n$,
this condition ought to enable the proof of
Theorem~\ref{t:elder-projection} to be generalized to arbitrary~$n$;
see Example~\ref{e:elder} and the paragraph preceding it.
% Proof idea: Suppose an elder quotient socle element $s$ of
% degree~$\alpha$ maps infinitesimally along each axis~$i$ to an elder
% submodule element~$s_i$.  These elements all map to a ``join
% element'' in a degree that is the join of the degrees of the~$s_i$.
% (Technically, all these infinitesimal motions will be expressed by
% inverse limits.)
% 
% As $s_i$ is an elder submodule element, it is the image of some
% element of degree $\prec\beta$.  On the other hand, the join element
% is, too.  The idea would be to ``copy'' the situation near~$\alpha$
% to a situation behind~$\beta$---or at least a situation whose meet
% is behind~$\beta$---to produce an element to subtract from~$\wt s$.
% Problems to overcome:
%  - There's no guarantee that the $s_i$ map to a common element until
%    they're pushed all the way to the join element; in any lesser
%    degree, their common images can be off by something that dies
%    before getting to the join.
%  - When the picture is copied back to behind~$\beta$, the relative
%    positions of the~$s_i$ might change.

%%%%%%%%%%%%%%%%%%%%%%%%%%%%%%%%%%%%%%%%%%%%%%%%%%%%%%%%%%%%%%%%%%%%%%%%%
\subsection{Geometry of birth and death}\label{sub:geom-qr}\mbox{}%%%%%%%

\noindent
Birth and death posets of modules over real polyhedral groups should
inherit topologies from Definitions~\ref{d:sigma-neighborhood}
and~\ref{d:nearby}.  The death module should be a sheaf over the death
poset, and the \qrcode\ should take the birth module $\Gen\cM$ to
sections of the death module~$\Soc\cM$ in a well behaved---probably
constructible---manner.  Functorial \qrcode s are defined without this
extra topological nuance, but these extra structures should be
imposable on it.  Moreover, although these topological structures are
not required for recovery of a module from its \qrcode, they should
play a central role in the characterization of maps $\Gen\cM \to
\Soc\cM$ that arise as \qrcode s of finitely encoded modules.  Even in
the discrete case this characterization is not easy: any filtered map
from $\Gen\cM$ to $\Soc\cM$ ought to define a module with specified
tops and socles, but it is not a~priori clear what conditions on the
filtered map guarantee that the module thus defined is finitely
encoded.

%%%%%%%%%%%%%%%%%%%%%%%%%%%%%%%%%%%%%%%%%%%%%%%%%%%%%%%%%%%%%%%%%%%%%%%%%
\subsection{Multiscale intersection homology}\label{sub:PIH}\mbox{}%%%%%%

\noindent
The view of the particular form of multiparameter persistence in
Example~\ref{e:fly-wing-filtration} as a manifestation of persistent
intersection homology (Remark~\ref{r:biparameterPH}) has yet to be
formalized.  Once it is, properties such as stability and suitability
for data analysis in general remain to be investigated.

%%%%%%%%%%%%%%%%%%%%%%%%%%%%%%%%%%%%%%%%%%%%%%%%%%%%%%%%%%%%%%%%%%%%%%%%%
\addtocontents{toc}{\protect\setcounter{tocdepth}{2}}%%%%%%%%%%%%%%%%%%%%
%%%%%%%%%%%%%%%%%%%%%%%%%%%%%%%%%%%%%%%%%%%%%%%%%%%%
%%%%%%%%%%%%%%%%%%%%%%%%%%%%%%%%%%%%%%%%%%%%%%%%%%%%%%%%%%%%%%%%%%%%%%%%%

\addtocontents{toc}{\vspace{-1ex}}%%%%%%%%%%%%%%%%%%%%

%%%%%%%%%%%%%%%%%%%%%%%%%%%%%%%%%%%%%%%%%%%%%%%%%%%%%%%%%%%%%%%%%%%%%%%%%
\end{document}